%% file: main.tex
\documentclass[11
pt, reqno]{amsart}
\usepackage{amssymb,amsmath}
\usepackage[foot]{amsaddr}
\usepackage[dvipsnames]{xcolor}
\usepackage{graphicx}
\usepackage{float}
\usepackage[framemethod=tikz]{mdframed}
\usepackage[many]{tcolorbox}
\usepackage{tikz}
\usetikzlibrary{matrix, fit}
\usepackage{tikz-cd}
\usepackage[font=footnotesize]{caption}
\usepackage[backend=bibtex,url=false,style=alphabetic]{biblatex}
\usepackage{svg}
\usepackage{wrapfig}
\usepackage{floatflt}
\usepackage[bottom]{footmisc}
\definecolor{darkgreen}{rgb}{0,0.7,0}
\definecolor{darkblue}{rgb}{0,0,0.7}
\usepackage[colorlinks,hyperindex,linkcolor=darkblue,citecolor=darkgreen]{hyperref}
\usepackage{enumerate}

\makeatletter
\def\subsection{\@startsection{subsection}{2}%
  \z@{.9\linespacing\@plus.7\linespacing}
{.9\baselineskip}%
  {\normalfont\centering\itshape}%
}
\makeatother

\numberwithin{equation}{section}

\newcommand{\npboxed}[1]{%
\setlength{\fboxsep}{2pt}\setlength{\fboxrule}{0.4pt}%
\vspace{-2.4pt}\hspace{-2.4pt}\boxed{#1}\vspace{-2.4pt}\hspace{-2.4pt}
}

\def\tableaubox(#1)(#2)(#3:#4:#5){%
    \node[fit=(#1-#3-#4)(#1-#3-\the\numexpr#4+#5-1\relax), inner sep=0pt, draw] (#1-#3#4#5-box) {};
  \node[circle, fill=white, scale=0.75, inner sep=0.5pt, yshift=-3pt] at (#1-#3#4#5-box.north west) {#2};
}
\def\drawtableau(#1){%
\matrix (m) [matrix of nodes, nodes in empty cells, column sep=5pt, row sep=5pt] {#1};
\draw[thick] ([xshift=-5pt, yshift=1pt]m-3-1.north west) -- ([xshift=5pt, yshift=1pt]m-3-8.north east);
}

\input{macro}

\newcommand{\addEndSymbol}[2]{
\AtBeginEnvironment{#1}{%
  \pushQED{\qed}\renewcommand{\qedsymbol}{#2}%
}
\AtEndEnvironment{#1}{\popQED}
}

\newtheoremstyle{nparstyle}
  {8pt plus 2pt}
  {8pt plus 2pt}
  {}
  {0pt}
  {}
  {. }
  {0pt}
  {\thmnumber{#2}\thmnote{. \textbf{#3}}}
\theoremstyle{nparstyle}
\newtheorem{npar}{}[subsection]\addEndSymbol{npar}{$\triangle$}

\swapnumbers
\theoremstyle{definition}
\newmdtheoremenv[
  linewidth=1.6pt,
  hidealllines=true,
  leftline=true,
  innerleftmargin=5pt,
  innerrightmargin=0,
  innertopmargin=0pt,
  innerbottommargin=0pt,
  skipabove=10pt plus 4pt minus 0pt,
  skipbelow=10pt plus 4pt minus 0pt,
]{definition}[npar]{Definition}
\newtheorem{remark}[npar]{Remark}\addEndSymbol{remark}{$\triangle$}
\newtheorem{example}[npar]{Example}\addEndSymbol{example}{$\triangle$}
\newtheorem{examples}[npar]{Examples}\addEndSymbol{examples}{$\triangle$}

\theoremstyle{plain}
\newtheorem{lemma}[npar]{Lemma}
\newtheorem{proposition}[npar]{Proposition}
\newtheorem{conjecture}[npar]{Conjecture}
\newtheorem{theorem}[npar]{Theorem}
\newtheorem{corollary}[npar]{Corollary}

\newmdenv[
  linewidth=0pt,
  innerleftmargin=30pt,
  innerrightmargin=30pt,
  innertopmargin=0pt,
  innerbottommargin=0pt,
  skipabove=10pt plus 4pt minus 0pt,
  skipbelow=0pt,
]{indentedbox}
\newenvironment{highlight}[2]{\begin{indentedbox}\textbf{#1} (#2).\itshape}{\end{indentedbox}}

\makeatletter
\renewcommand{\tocsection}[3]{%
  \indentlabel{\@ifnotempty{#2}{\bfseries\ignorespaces#1 #2\quad}}\bfseries#3}
\renewcommand{\tocsubsection}[3]{%
  \indentlabel{\@ifnotempty{#2}{\ignorespaces#1 #2\quad}}#3}
\newcommand\@dotsep{4.5}
\def\@tocline#1#2#3#4#5#6#7{\relax
  \ifnum #1>\c@tocdepth 
  \else
    \par \addpenalty\@secpenalty\addvspace{#2}%
    \begingroup \hyphenpenalty\@M
    \@ifempty{#4}{%
      \@tempdima\csname r@tocindent\number#1\endcsname\relax
    }{%
      \@tempdima#4\relax
    }%
    \parindent\z@ \leftskip#3\relax \advance\leftskip\@tempdima\relax
    \rightskip\@pnumwidth plus1em \parfillskip-\@pnumwidth
    #5\leavevmode\hskip-\@tempdima{#6}\nobreak
    \leaders\hbox{$\m@th\mkern \@dotsep mu\hbox{.}\mkern \@dotsep mu$}\hfill
    \nobreak
    \hbox to\@pnumwidth{\@tocpagenum{\ifnum#1=1\bfseries\fi#7}}\par
    \nobreak
    \endgroup
  \fi}
\AtBeginDocument{%
\expandafter\renewcommand\csname r@tocindent0\endcsname{0pt}
}
\def\l@subsection{\@tocline{2}{0pt}{2.5pc}{5pc}{}}
\makeatother

\newcommand{\chaption}{section}
\newcommand{\Chaption}{Section}
\newcommand{\subchaption}{subsection}
\newcommand{\Subchaption}{Subsection}

\newcommand{\startSubchaption}[1]{\subsection{#1}}

\begin{document}

\title[Wonderful Blow-Ups of Weighted Building Sets]{Wonderful Blow-Ups of Weighted Building Sets And Configuration Spaces Of Filtered Manifolds}

\author{Aaron Gootjes-Dreesbach}
\address{Department of Mathematics, Utrecht University, Budapestlaan 6, 3584CD Utrecht, NL}
\email{math@awgd.org}

\begin{abstract}
Fulton and MacPherson famously constructed a configuration space that encodes infinitesimal collision data by  blowing up the diagonals~\cite{FM94}.
We observe that when generalizing their approach to configuration spaces of filtered manifolds (e.g. jet spaces or sub-Riemannian manifolds), these blow-ups have to be modified with weights in order for the collisions to be compatible with higher-order data.

In the present article, we provide a general framework for blowing up arrangements of submanifolds that are equipped with a \textit{weighting} in the sense of Loizides and Meinrenken~\cite{LM23}. We prove in particular smoothness of the blow-up under reasonable assumptions, extending a result of Li~\cite{Li09} to the weighted setting. Our discussion covers both spherical and projective blow-ups, as well as the (restricted) functoriality of the construction. 

Alongside a self-contained introduction to weightings, we also give a new characterization thereof in terms of their vanishing ideals and prove that cleanly intersecting weightings locally yield a weighting.

As our main application, we construct configuration spaces of filtered manifolds, including convenient local models. We also discuss a variation of the construction tailored to certain fiber bundles equipped with a filtration. This is necessary for the special case of jet configuration spaces, which we investigate in a future article~\cite{future}.
\end{abstract}

\maketitle

{\hypersetup{linkcolor=black}
\tableofcontents
}
\setlength{\parskip}{2.5mm plus0.4mm minus0.4mm}

\newpage
\section{Introduction}
\input{introduction}

\newpage
\section{Weightings}\label{sec:weightings}
\input{weightings}

\newpage
\section{Blowing up weighted submanifolds}\label{sec:weighted-submanifolds}
\input{weighted-submanifolds}

\newpage
\section{Blowing up building sets}\label{sec:building-sets}
\input{building-sets}

\newpage
\section{Blowing up weighted building sets}\label{sec:weighted-building-sets}
\input{weighted-building-sets}

\newpage
\section{Configuration spaces for filtered manifolds}\label{sec:filtered-manifolds}
\input{filtered-manifolds}

\appendix
\newpage
\section{Intersections of submanifolds}\label{app:manifold-intersections}
\input{manifold-intersections}

\newpage
\section{Smooth structure on the blow-up of a weighted submanifold}\label{app:proof-weighted-submanifolds}
\input{proof-weighted-submanifolds}

\newpage
\section{Smooth structure on the blow-up of a weighted building set}\label{app:proof-weighted-building-set}
\input{proof-weighted-building-set}

\newpage
\printbibliography

\end{document}

%% file: macro.tex
\newcommand{\N}{\mathbb{N}}
\newcommand{\Z}{\mathbb{Z}}
\newcommand{\R}{\mathbb{R}}
\newcommand{\bS}{\mathbb{S}}
\newcommand{\bP}{\mathbb{P}}

\newcommand{\cT}{\mathcal{T}}

\newcommand{\cI}{\mathcal{I}}
\newcommand{\cP}{\mathcal{P}}

\newcommand{\cW}{\mathcal{W}}
\newcommand{\cM}{\mathcal{M}}
\newcommand{\cH}{\mathcal{H}}
\newcommand{\cE}{\mathcal{E}}
\newcommand{\cG}{\mathcal{G}}
\newcommand{\cN}{\mathcal{N}}
\newcommand{\cS}{\mathcal{S}}

\newcommand{\cU}{\mathcal{U}}

\newcommand{\cD}{\mathcal{D}}

\newcommand{\x}[2]{x_{#1}^{(#2)}}

\newcommand{\fX}{\mathfrak{X}}

\newcommand{\Bl}{\operatorname{Bl}}
\newcommand{\PBl}{\operatorname{\bP Bl}}

\newcommand{\Arr}{\operatorname{Arr}}

\newcommand{\Conf}{\operatorname{Conf}}
\newcommand{\PConf}{\operatorname{\bP Conf}}

\newcommand{\supp}{\operatorname{supp}}
\newcommand{\codim}{\operatorname{codim}}

\newcommand{\Id}{\operatorname{Id}}
\newcommand{\rk}{\operatorname{rank}}

\newcommand{\fac}{\operatorname{fac}}

\newcommand{\join}{\vee}
\newcommand{\meet}{\wedge}

\newcommand{\gr}{\operatorname{gr}}
\newcommand{\upto}[1]{\underline{#1}}
\newcommand{\Rt}{\operatorname{Rt}}
\newcommand{\Ch}{\operatorname{Ch}}
\newcommand{\Ct}{\operatorname{Ct}}

%% file: introduction.tex
\begin{wrapfigure}{r}{0.3\textwidth}
    \centering
    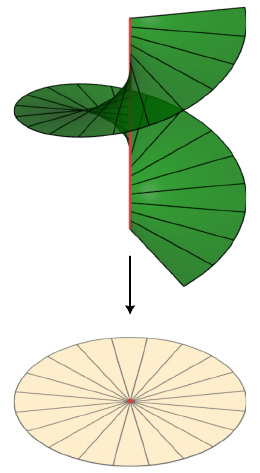
    \caption{Visualization of the blow-up of the origin of a disk. Points in the blow-up that are separated by an integer number of rotations are identified. The central spine thus can be identified with unit length vectors at the origin, and every point away from the origin has a unique preimage under the blow-down map.}
    \label{fig:helix-normal}
    \vspace{-10mm}
\end{wrapfigure}

\textit{Blowing up} a closed submanifold $N$ of an ambient manifold $M$ means to (smoothly) replace $N$ in $M$ with its unit normal bundle $\bS\nu N$. In the resulting space $\Bl_N M$, a sequence of points that approaches $N$ along some unit normal vector converges exactly to that vector. $\Bl_N M$ is a manifold with boundary equipped with a smooth and proper \textit{blow-down map} $b:\Bl_NM\to M$ that forgets the extra information. Compared to this \textit{spherical} blow-up, it is also natural to consider the \textit{projective} blow-up $\bP\Bl_N M$ that only remembers the normal direction in the projective bundle $\bP\nu N$.

Blow-ups can famously be used for \textit{desingularization}, see e.g. Hironaka's celebrated theorem~\cite{Hir64}.
Another classic application that is more relevant to our purposes is the construction of an ordered configuration space of pairs of points on a manifold $M$ that takes into account infinitesimal data of collisions. This is achieved by blowing up the diagonal $\Delta$ in $M^2$: The unit normal vectors in $\nu\Delta\simeq TM$ encode the direction along which two points of a configuration collided. This was generalized to configurations of higher multiplicities by Fulton and MacPherson~\cite{FM94}.

In this article, we introduce the necessary machinery to define a blown-up configuration space of jets in this spirit. The difficulty lies in the fact that jet space comes equipped with additional structure: The \textit{Cartan distribution} and the filtration of the tangent bundle that it induces by taking Lie brackets. This makes jet space an example of a \textit{filtered manifold}. For the bulk of this paper, we will work towards a configuration space for these filtered manifolds in general, leaving the application to jet spaces to a future article currently in preparation~\cite{future}.

\textbf{A motivating example.}
To get a flavour of what we are working towards, let us consider how to construct $J^{[2,2]}(\R^m,\R)$. This is the space of pairs of second-order jets from $\R^m$ to $\R$ that includes infinitesimal information for every collided pair\footnote{This definition serves as a motivation but may be updated in~\cite{future}. In fact, we suspect that a slightly modified blow-up involving higher weights may be better suited.}.

\textit{Adapted coordinates.}
We start innocently enough with the square $(J^2(\R^m,\R))^2$, whose elements we can naturally write as
$$(x_1,y_1,y'_1,y''_1;\;x_2,y_2,y'_2,y''_2)$$
for $x_i\in\R^m, y_i\in\R, y'_i\in\operatorname{Hom}(\R^m,\R)$ and $y''_i\in\operatorname{Sym}(\R^m\otimes\R^m,\R)$. Since we are interested in the behaviour at the diagonal $\Delta\subseteq(J^2(\R^m,\R))^2$, we pass to offset coordinates
$$\left(x_1,y_1,y'_1,y''_1;\;\Delta x,\Delta y, \Delta y', \Delta y''\right)$$
where we define $\Delta x\in\R^m, \Delta y\in\R, \Delta y'\in\operatorname{Hom}(\R^m,\R)$ and $\Delta y''\in\operatorname{Sym}(\R^m\otimes\R^m,\R)$ by the equations
\begin{equation}\label{eq:adapted-coords}
\begin{aligned}
x_2 &:= x_1 + \Delta x,\\
y_2 &:= y_1 + y'_1(\Delta x) + {\tfrac12} y''_1(\Delta x, \Delta x) + \Delta y,\\
y'_2 &:= y'_1 + y''_1(\Delta x,\cdot) + \Delta y',\\
y''_2 &:= y''_1 + \Delta y''.
\end{aligned}
\end{equation}
Not only are these coordinates adapted to the diagonal $\Delta$ in the sense that $\Delta$ is cut out by $\Delta x=\Delta y=\Delta y'=\Delta y''=0$, but the offsets have a natural interpretation as remainder terms of truncated Taylor series. In this way, the infinitesimal data encoded in the jets themselves is taken into account for the collisions that result from sending $\Delta x$ to zero while keeping the remaining coordinates fixed.

\textit{A first attempt at blowing up.}
Performing a classical (spherical) blow-up would now come down to making a substitution
\begin{equation}\label{eq:classic-blow-up}
\left(\Delta x,\Delta y, \Delta y', \Delta y''\right)=\lambda\cdot\left(\delta x,\delta y, \delta y', \delta y''\right),
\end{equation}
where $\lambda\in[0,\infty)$ and $\left(\delta x,\delta y, \delta y', \delta y''\right)$ is a unit vector, such that even when $\lambda=0$ we still remember along which directions the jets collided. The blow-down map to the original space $(J^2(\R^m,\R))^2$ is then given by inserting Equation~\eqref{eq:classic-blow-up} into Eqs.~\eqref{eq:adapted-coords}, resulting in e.g.
$$y_2 = y_1 + \lambda\, y'_1(\delta x) + {\tfrac12}\lambda^2\, y''_1(\delta x, \delta x) + \lambda\,\delta y.$$
This is not very satisfying: We would expect the last error term to be cubical in the parameter $\lambda$ that controls the collision, in line with the interpretation as a Taylor series. Indeed, making the offset term linear defeats the purpose of picking coordinates adapted to higher-order data, as the blow-up ends up being not sensitive to these.

\textit{Adjusting for weights.}
This is why we instead perform a \textit{weighted} blow-up according to the substitution
\begin{equation}\label{eq:weighted-blow-up}
\left(\Delta x,\Delta y, \Delta y', \Delta y''\right)=\left(\lambda\,\delta x,\tfrac13\lambda^3\,\delta y, \tfrac12\lambda^2\,\delta y', \lambda\,\delta y''\right).
\end{equation}
Aside from the formal similarity of the resulting expressions to Tayor expansions, the orders with which $\lambda$ appears can be justified by compatibility with the Lie filtration $H_\bullet\subseteq TJ^2(\R^m,\R)$ generated by the Cartan distribution: 
It turns out that the offset coordinates~\eqref{eq:adapted-coords} are chosen in such a way that they are adapted to the diagonal $\Delta$ as a filtered submanifold with respect to $H^2_\bullet$. The order in $\lambda$ with which we blow-up each of the coordinates normal to $\Delta$ is exactly the weight associated with the adapted coordinates in that direction.

\textit{Excluding collisions along the fibers.}
In order to have a well-defined fiber map to the blow-up $(\R^m)^{[2]}$ of the base, we further want to restrict our construction to those elements where $\delta x\neq 0$, i.e. to collisions that happen along a horizontal direction. This restriction allows us to rescale the new parameters using $\lambda$, yielding
$$\lambda\in[0,\infty),\quad \delta x\in\bS^{n-1},\quad \delta y\in\R,\quad \delta y'\in\operatorname{Hom}(\R^m,\R),\quad \delta y''\in\operatorname{Sym}(\R^m\otimes\R^m,\R)$$
with the normalization only applying to $\delta x$ instead of coupling $(\delta x,\delta y,\delta y',\delta y'')$.
Together with $(x_1,y_1,y_1',y_1'')$, these components assemble into elements of our blown-up 2-fold jet configuration space $J^{[2,2]}(\R^m,\R)$. 
The blow-down map is once again given by combining Eqs.~\eqref{eq:adapted-coords} and~\eqref{eq:weighted-blow-up}, resulting in
\begin{equation}\label{eq:blow-down}
\begin{aligned}
x_2 &:= x_1 + \lambda\,\delta x,\\
y_2 &:= y_1 + \lambda\,y'_1(\delta x) + {\tfrac12}\lambda^2\, y''_1(\delta x, \delta x) + \tfrac13 \lambda^3\,\delta y,\\
y'_2 &:= y'_1 + \lambda\, y''_1(\delta x,\cdot) + \tfrac12 \lambda^2\, \delta y',\\
y''_2 &:= y''_1 + \lambda\,\delta y''.
\end{aligned}
\end{equation}
As we hoped, projecting to $(x_1,\lambda,\delta x)$ gives a canonical projection map $J^{[2,2]}(\R^m,\R)\to (\R^m)^{[2]}$ into the blown-up 2-fold configuration space of the base $\R^m$.

\textit{Holonomic jet configuration space.}
Consider a smooth function $f:\R^m\to\R$ and a sequence of configurations $\{(x_{1,n},x_{2,n})\}_{n\in\N}$ in $\R^m$ with $x_{1,n}\neq x_{2,n}$ that converge to a point $(x_\infty,x_\infty)$ on the diagonal. Not every element of $J^{[2,2]}(\R^m,\R)$ can be reached as a limit of the sequence $(j^2f(x_{1,n}),j^2f(x_{2,n}))$: By comparing Equation~\eqref{eq:blow-down} to the Taylor series, we see that the parameters $\delta y, \delta y'$ and $\delta y''$ must all encode the third derivative of $f$ at the collision point $x_0$, just contracted with different powers of $\delta x$. This motivates the definition of the \textit{holonomic} jet configuration space
\begin{align*}
J^{[2,2]}_h(\R^m,\R):=\{ &(x_1,y_1,y'_1,y''_1;\lambda,\delta x,\delta y, \delta y',\delta y'') \in J^{[2,2]}(\R^m,\R) \;|\;\\
&\qquad\qquad\delta y = \delta y'(\delta x) \text{ and } \delta y'=\delta y''(\delta x,\cdot) \text{ whenever }\lambda=0 \}
\end{align*}
consisting of those configurations that can be reached as such limits. We claim that blown-up jet spaces like this are well-suited for the study of differential relations that involve more than one point, which we will consider in~\cite{future}.

\textbf{Generalizing to arbitrary jet bundles and higher multiplicities.}
The previous construction can be carried out, in equally explicit terms, for two-fold multijet space between Euclidean spaces at arbitrary order. However, we want to give a construction that is coordinate-independent and applies to jets of sections of arbitrary fiber bundles.
To this end, we adopt the recent differential-geometric framework of \textit{weightings} due to Loizides and Meinrenken~\cite{LM23,LM22}. This notion coincides with Melrose's \textit{quasi-homogeneous structures}~\cite{Mel91} that were also studied by Behr~\cite{Beh21}. Weightings provide a coordinate-independent way of assigning weights in normal directions of a submanifold and are the natural setting for the construction of weighted blow-ups. Our reliance on weightings is a reflection of the weighted nature of the error terms when truncating a Taylor series at finite order.

Moreover, a different and significant complication appears when moving from configurations of pairs to higher multiplicities: Instead of a single submanifold of $M^2$ representing collided pairs, in $M^s$ we have multiple diagonals, i.e. submanifolds cut out by requiring some subset of the components to agree.
The union of these diagonals is stratified with each stratum encoding which subset of the configuration is involved in the collision
(see Figure~\ref{fig:diagonals} on page~\pageref{fig:diagonals} for the concrete example of triples in an open interval). The achievement of Fulton and MacPherson was to perform a blow-up of this stratified diagonal in the unweighted case for arbitrary multiplicities\footnote{While they worked with projective blow-ups in an algebro-geometric language, their construction was also translated to spherical blow-ups in differential geometry by Axelrod and Singer~\cite{AS94}.}.
A series of papers further developed their techniques to define smooth blow-ups of sufficiently regular \textit{building sets} of subvarieties~\cite{CP95, MP98, Uly00, Hu03, Li09}. This article systematically extends this line of inquiry to building sets equipped with weightings, allowing us to perform a weighted blow-up of all strata of the diagonal.

\textbf{Blow-ups of building sets.}
A building set is a collection $\cG$ of closed submanifolds of some $M$ that all intersect cleanly (defined with more care for technical assumptions in Definition~\ref{def:building-set}). The archetypal example is the Fulton-MacPherson building set $\cG^{FM}$, which consists of all diagonals in $M^s$. It may be tempting to consider just the union of the submanifolds as the data required for a blow-up, but a building set also represents a choice of submanifolds that get blown-up even when they are strictly contained in some of the others. The \textit{graph-based} approach to defining the blow-up of $\cG$ is to take the closure of the image of the diagonal map
$$M\setminus\cup\cG\;\lhook\joinrel\longrightarrow\; \prod_{G\in\cG}\Bl_G M,$$
thus enforcing that the extra data of the individual blow-ups is consistent with each other. This approach was already introduced in~\cite{FM94} alongside an equivalent iterated blow-up. The insight of~\cite{CP95} generalized in~\cite{Li09} is that a sufficient (but not necessary) condition yielding smoothness of the blow-up is \textit{separation} of the building set: Roughly speaking, wherever some of the submanifolds intersect, the normal bundle of their intersection must be a direct sum of the individual normal bundles\footnote{Strictly speaking, it suffices to ask this for only the smallest of the intersecting submanifolds, see Definition~\ref{def:separated} for a precise statement.}. Since the blow-ups record exactly the normal data, this splitting allows them to coexist peacefully.

Our notion of \textit{weighted} building set (Definition~\ref{def:weighted-building-set}) equips each element of $\cG$ with a weighting and only assumes the following rather minimal consistency condition: Whenever some $G_1,...,G_n$ in $\cG$ intersect in a non-empty element $\cap_{i=1}^n G_i$ that also lies in $\cG$, then the weighting assigned to $\cap_{i=1}^n G_i$ is the clean intersection of the weightings associated to each $G_i$. As in the unweighted case, the blow-up is then defined with the graph-based approach, but now using weighted blow-ups for every element of $\cG$ (Definition~\ref{def:blow-up-weighted-building-set}). It turns out that separation of $\cG$ is no longer sufficient for a smooth structure: We also need some further compatibility of the weightings themselves instead of just the submanifolds along which they are defined. We call this condition \textit{uniform alignment} (Definition~\ref{def:building-properties}): Essentially, all the weightings that interact with each other\footnote{Concretely, this will mean those contained in the same \textit{nest} (Definition~\ref{def:nests}), a notion arising from the algebra of intersections of elements of $\cG$.} need to have local charts simultaneously adapted to all of them, and any non-zero weights assigned to a specific coordinate direction in these charts must match.

\textbf{Summary of Results.}
Our main result is the development of the notion of \textit{weighted building sets}, and the following regularity result about their blow-ups in a fairly broad class of sufficiently well-behaved situations:
\begin{highlight}{Main Theorem}{see~\ref{thm:manifold-structure}}
The (spherical) blow-up $\Bl_\cW M$ of a uniformly aligned weighting $\cW$ along a separated building set can be equipped with the structure of a smooth manifold with corners such that the blow-down map is smooth and proper. 
\end{highlight}
We provide explicit coordinates for this smooth structure (Definition~\ref{def:building-local-charts}), discuss the canonical stratification of the blow-up (Definition~\ref{def:ass-nest-and-strat}) and characterize the regularity of $\Bl_\cW M$ when seen as a subspace of $\prod_{G\in\cG}\Bl_{\cW_G} M$ (Proposition~\ref{prop:weak-submfd-structure}). 
We also provide a notion of morphisms between weighted building sets (Definition~\ref{def:morphism}). While the blow-up is not functorial, we describe how such a morphism at least induces a smooth map between the blow-ups on an open subset (Proposition~\ref{prop:morphism-induced}). Throughout, we provide a wealth of examples and counterexamples. As a side effect, we translate Li's regularity results for the unweighted case into differential-geometric language (Theorem~\ref{thm:li}).

As the main application of this framework, we build a blown-up $s$-fold configuration space $\Conf^{[s]}(M,H_\bullet)$ for every filtered manifold $(M,H_\bullet)$:
\begin{highlight}{Proposition}{see~\ref{def:filtered-weighted-building-set}, \ref{prop:fm-reg-unif-align}, and~\ref{def:filtered-fm}}
A Lie filtration $H_\bullet$ on a manifold $M$ induces a canonical uniformly aligned weighting $\cW^{FM}$ along the separated Fulton-MacPherson building set $\cG^{FM}$.
In particular, $$\Conf^{[s]}(M,H_\bullet):=\Bl_{\cW^{FM}}(M^s)$$ is a smooth manifold with corners.
\end{highlight}
A projective version of this space (Definition~\ref{def:filtered-fm}) trades in existence of the boundary for orbifold singularities. Due to the complexity of the general charts, we provide more convenient local models for both (Propositions~\ref{prop:FM-local-model} and~\ref{prop:FM-local-model-proj}). Having jet spaces in mind, we discuss a further adaptation of the configuration space for a Lie filtration on the total space of a fiber bundle $E\to M$ (\Subchaption~\ref{ssec:filtered-bundle-fulton-macpherson}).

To facilitate all of the above, we also establish some smaller results about weightings in general. Firstly, we give a new invariant characterization of weightings as subsets of the infinite-order tangent bundle $T^{(\infty)}M$ in terms of its vanishing ideal:
\begin{highlight}{Theorem}{see~\ref{thm:new-criterion}}
	A connected and embedded submanifold $\cW\subseteq T^{(\infty)}M$ is a weighting if and only if it is closed and its vanishing ideal $I_\cW$ satisfies
	\begin{equation*}
		I_\cW = \langle f^{(i)}\;|\; f\in C^\infty(M) \text{ with } f^{(j)}|_\cW=0\;\forall j\leq i \rangle.
    \end{equation*}
\end{highlight}
In other words, $I_\cW$ must already be generated by all $i$-th lifts $f^{(i)}$ of those functions $f\in C^\infty(M)$ to $T^{(\infty)}M$ whose lifts $f^{(j)}$ of order $j\leq i$ also vanish on $\cW$. In particular, this will allow us to conclude the following:
\begin{highlight}{Proposition}{see~\ref{prop:clean-locally-weighting}}
Let $\{\cW_\alpha\}_{\alpha\in A}$ be a family of weightings on $M$ that intersect cleanly in $\cW_0=\bigcap_{\alpha=1}^d\cW_\alpha$. Then $\cW_0$ is locally a weighting, i.e. around every point of $M$ there exists a neighbourhood $U$ such that $\cW_0\cap T^{(\infty)}U$ is a weighting.
\end{highlight}
We more broadly investigate several notions of compatibility between weightings, beyond the \textit{multiweightings} of~\cite{LM23}, with many examples. To express some of these results more conveniently, we introduce weightings as subsets of the infinite-order tangent bundle $T^{(\infty)}M$ but remain strictly equivalent to the definition of~\cite{LM23}. We also introduce a dual description of weightings as certain filtrations of higher \textit{co}tangent bundles. Finally, while the smooth structure of the blow-up of single weighted submanifolds is induced from that of the deformation space in~\cite{LM23}, we give a complementary direct description with explicit charts.

\textbf{Structure of this article.} This article is organized as follows: In \Chaption~\ref{sec:weightings}, we give a self-contained introduction to weightings while establishing a number of useful results that we require later on.
\Chaption~\ref{sec:weighted-submanifolds} details the construction of both spherical and projective blow-ups of a weighting along a single submanifold, including discussion of their smooth and orbifold structure.
In \Chaption~\ref{sec:building-sets}, we first introduce building sets and discuss the central notions of factors, flags and nests. We then give a differential-geometric account of Li's blow-up of building sets in preparation for the weighted case.
\Chaption~\ref{sec:weighted-building-sets} introduces weightings along an entire building set as well as their blow-ups. We also give charts that provide a smooth structure assuming sufficient regularity and examine the stratification of the blow-up.
\Chaption~\ref{sec:filtered-manifolds} applies this framework to construct configuration spaces for filtered manifolds. In particular, it establishes that a Lie filtration on a manifold induces a weighting along each diagonal of its $s$-fold power, and that these assemble into a well-behaved weighted building set. We also provide more convenient local models and further adapt the configuration space to a fiber bundle structure.
Appendix~\ref{app:manifold-intersections} briefly recaps notions of intersections of submanifolds that turn up in the discussion of compatibility of weightings. The prolonged proofs of smoothness of transition maps on the blow-ups of both a single weighted submanifold and an entire weighted building set are relegated to Appendices~\ref{app:proof-weighted-submanifolds} and~\ref{app:proof-weighted-building-set}, respectively.

\textbf{Notation \& Conventions.}
All our manifolds are smooth, Hausdorff and second-countable without boundary or corners unless explicitly stated otherwise.
We usually assume that the local charts $\chi:U\to\R^m$ of a given input manifold $M$ are surjective, but the charts that we build for our constructions (e.g. the blow-up $\Bl_\cW M$) are often not surjective. Whenever taking the intersection of a possibly empty set of sets, there is a natural choice of maximal set that we assign as the result.
For example, in the context of building sets, an empty intersection yields the full underlying manifold.
We point out differences to Loizides and Meinrenken in definitions and notation in Remark~\ref{rem:equiv-defs}. Particularly in regard to notions of regularity of submanifold intersection, compare our Definition~\ref{def:transversal}.
It will often be convenient to use for every natural number $s$ the shorthand $\upto{s}:=\{1,...,s\}$.

\textbf{Acknowledgments.} The author is deeply indebted to his advisor Álvaro del Pino Gómez for countless hours of helpful discussions and insightful feedback without which this work would not have been possible.

\textbf{Note.} The contents of this document will form part of the author's doctoral thesis and may be split into two articles in the future.

%% file: helix-normal.pdf_tex
\begingroup%
  \makeatletter%
  \providecommand\color[2][]{%
    \errmessage{(Inkscape) Color is used for the text in Inkscape, but the package 'color.sty' is not loaded}%
    \renewcommand\color[2][]{}%
  }%
  \providecommand\transparent[1]{%
    \errmessage{(Inkscape) Transparency is used (non-zero) for the text in Inkscape, but the package 'transparent.sty' is not loaded}%
    \renewcommand\transparent[1]{}%
  }%
  \providecommand\rotatebox[2]{#2}%
  \newcommand*\fsize{\dimexpr\f@size pt\relax}%
  \newcommand*\lineheight[1]{\fontsize{\fsize}{#1\fsize}\selectfont}%
  \ifx\svgwidth\undefined%
    \setlength{\unitlength}{124.72440945bp}%
    \ifx\svgscale\undefined%
      \relax%
    \else%
      \setlength{\unitlength}{\unitlength * \real{\svgscale}}%
    \fi%
  \else%
    \setlength{\unitlength}{\svgwidth}%
  \fi%
  \global\let\svgwidth\undefined%
  \global\let\svgscale\undefined%
  \makeatother%
  \begin{picture}(1,1.81818182)%
    \lineheight{1}%
    \setlength\tabcolsep{0pt}%
    \put(0,0){\includegraphics[width=\unitlength,page=1]{helix-normal.pdf}}%
    \put(0.43420223,0.7009502){\color[rgb]{0,0,0}\makebox(0,0)[t]{\lineheight{0.80000001}\smash{\begin{tabular}[t]{c}$b$\end{tabular}}}}%
    \put(0.25172216,1.68556992){\color[rgb]{0,0,0}\makebox(0,0)[t]{\lineheight{0.80000001}\smash{\begin{tabular}[t]{c}$\Bl_N(M)$\end{tabular}}}}%
    \put(0.09112237,0.46308538){\color[rgb]{0,0,0}\makebox(0,0)[t]{\lineheight{0.80000001}\smash{\begin{tabular}[t]{c}$M$\end{tabular}}}}%
  \end{picture}%
\endgroup%

%% file: weightings.tex
A \textit{weighting} $\cW$ along a submanifold $N\subseteq M$ is a way of assigning weights to coordinate directions normal to $N$. It provides a notion of functions vanishing on $N$ to some \textit{weighted} order, and it is the natural data to talk about weighted normal bundles and blow-ups.

Weightings were introduced by Loizides and Meinrenken in~\cite{LM23, LM22} in order to construct tangent groupoids in the sense of Van Erp-Yuncken~\cite{EY17} in the more general setting of singular Lie filtrations, building on constructions of Choi-Ponge~\cite{CP19c} and Haj-Higson~\cite{HH18}. 
They coincide with the earlier unpublished notion of \textit{quasi-homogeneous structures} due to Melrose~\cite{Mel91}, which was further developed by Behr~\cite{Beh21}.

We have multiple ambitions for this \chaption: To start, we give a succinct but self-contained introduction to weightings. Our perspective on the topic originates from Loizides and Meinrenken's higher tangent bundle approach, but implements it using a \textit{profinite} formalism provided by the infinite-order tangent bundle $T^{(\infty)}M$. We moreover establish a number of new results about weightings.
With Theorem~\ref{thm:new-criterion} we provide a new invariant characterization of the subsets of $T^{(\infty)}M$ that correspond to weightings in terms of their vanishing ideal. This is particularly convenient in our discussion of compatibility between multiple weightings defined over the same submanifold in \Subchaption~\ref{ssec:compatibility}. Here we introduce and discuss both weaker and stronger conditions than the \textit{multiweightings} from~\cite{LM23} that will become necessary for the weighted building sets of \Chaption~\ref{sec:weighted-building-sets}. Finally, since weightings inherently intertwine both sides of the differential-geometric duality between curves and functions, we outline the dual picture from the viewpoint of higher cotangent bundles.

This \chaption{} is subdivided as follows: In the first \subchaption, we give an introduction to higher tangent bundles, such that we can introduce weightings as subsets thereof in \Subchaption{}~\ref{ssec:weightings}. \Subchaption{}~\ref{ssec:weighted-normal} defines and discusses the weighted normal bundle, which will be essential in constructing blow-ups. \Subchaption{}~\ref{ssec:dual} introduces higher cotangent bundles and the dual portrayal of weightings as filtrations. The corresponding weighted conormal bundle is introduced in \Subchaption~\ref{ssec:weighted-conormal}. We discuss the linear data and compatibility between weightings in \Subchaption{}s~\ref{ssec:linear-data} and~\ref{ssec:compatibility} and conclude with an examination of their intersections in \Subchaption{}~\ref{ssec:intersection}.

\startSubchaption{Higher tangent bundles}

Let us recap the notion of \textit{higher} tangent bundles of a manifold:

\begin{definition}
Given a manifold $M$ and $r\in\N_0\cup\{\infty\}$, we define the \textbf{$r$-th order tangent bundle} $T^{(r)}M$ as the space of equivalence classes $[\gamma]$ of curves $\gamma:\R\to M$ whose derivatives at zero agree in local charts up to order $r$, i.e. $$[\gamma_1]=[\gamma_2]\quad:\iff\quad \gamma_1^{(i)}(0)=\gamma_2^{(i)}(0)\quad\text{for all }0\leq i\leq r.$$
We denote the natural quotient maps for $r_1\geq r_2$ as
$$p^{(r_2,r_1)}:T^{(r_1)}M\to T^{(r_2)}M$$
and write in particular $p^{(r)}:=p^{(0,r)}$ for the projection to $M=T^{(0)}M$.
\end{definition}

It of course holds that $T^{(1)}M=TM$.
For more context, we refer the interested reader to~\cite{Mo70a,Mi97,KMS13,LM23}.
Note that the $r$-th order tangent bundle is different from the \textit{$r$-fold tangent bundle} $$T^rM:= \underbrace{T...T}_{\text{$r$ times}}M,$$ but can be seen as a symmetric subset of it\footnote{For example, $T^{(2)}\R=\{(x,x',x'')\in\R^3\}$ sits in $T^2\R=\{(x,\Delta x, \delta x, \delta\Delta x)\in\R^4\}$ by setting $\Delta x=\delta x = x'$ and $\delta\Delta x=x''$.}.
Moreover, $T^{(r)}M$ is exactly the fiber $J^r_0(\R,M)$ over zero of the jet space of maps $\R\to M$ and thus carries a natural smooth structure for finite $r$.

\begin{npar}[Profinite structure on $T^{(\infty)}M$]\label{npar:profinite} We can also make sense of a smooth structure on $T^{(\infty)}M$ in a \textit{profinite} sense as follows:
\begin{enumerate}
    \item We equip $T^{(\infty)}M$ with the limit topology, i.e. a subset $U\subseteq T^{(\infty)}M$ is \textbf{open} if it is an arbitrary union of preimages $\left(p^{(r,\infty)}\right)^{-1}(U_r)$ for some collection of opens $U_r\subseteq T^{(r)}M$.
    \item A function $f:T^{(\infty)}M\to\R$ is \textbf{smooth} if every $p\in T^{(\infty)}M$ has an open neighbourhood $U$ and an $r\in\N_0$ such that $f$ is a composition of $p^{(r,\infty)}$ and a smooth function on $T^{(r)}M$. We write $C^\infty(T^{(\infty)}M)$ for the algebra of such smooth functions.
	\item We say a subset $Q\subseteq T^{(\infty)}M$ is an \textbf{embedded submanifold of codimension $d$} if there is an $r\in\N_0$ and an embedded submanifold $Q_r\subseteq T^{(r)}M$ of codimension $d$ such that $Q=\left(p^{(r,\infty)}\right)^{-1}(Q_r)$. The \textbf{vanishing ideal} of $Q$ is the ideal $$I_Q=\{f\in C^\infty(T^{(\infty)}M)\;|\; f|_Q=0\}.$$
\end{enumerate}
	We will not require more than the above basic definitions, but more context on profinite structures can be found e.g. in~\cite[][Appendix A]{AC12} or~\cite[][Appendices 2.A and 3.A]{Ac21}.
\end{npar}

\begin{npar}[Properties of $T^{(r)}N$]\ 
\begin{enumerate}
\item The $r$-th order tangent bundle comes with a fiber map
	\begin{align*}
		p^{(0,r)}: T^{(r)}M&\to T^{(0)}M=M\\ [\gamma] &\mapsto \gamma(0)
	\end{align*}
    that makes it into a fiber bundle.
\item Any function from $\R$ to $\R$ that preserves the origin acts on elements of $T^{(r)}M$ by precomposing. As a matter of fact, this action only depends on the $r$-th jet at the origin, so we have an action of the algebra 
\begin{equation}\label{eq:lambda-r}
\Lambda_r:=J^r_{0,0}(\R,\R)=\{j^rf(0)\;|\; f:\R\to\R \text{ smooth with } f(0)=0 \}.
\end{equation}

\item We will mostly be interested in the restriction of this action to $\R=\Lambda_1\subseteq\Lambda_r$. Concretely, $\lambda\in\R$ sends an equivalence class $[\gamma]$ to $$\lambda\cdot[\gamma]:=[\gamma(\lambda\cdot)].$$ For finite $r$, this $\R$-action makes $T^{(r)}M$ into a graded bundle in the sense of Grabowski-Rotkiewicz~\cite{GR12}.
\item This construction is functorial, i.e. any smooth map $\phi:M\to \widetilde M$ induces a $\Lambda_r$-equivariant map
	\begin{align*}
		T^{(r)}\phi:T^{(r)}M&\to T^{(r)}\widetilde M\\ [\gamma] &\mapsto [\phi\circ\gamma].
	\end{align*}
\item For every $f\in C^\infty(M)$ and $0\leq i\leq r$ it is convenient to define the \textbf{$i$-th lift} $f^{(i)}\in C^\infty(T^{(r)}M)$ of $f$ as
$$f^{(i)}([\gamma]) := \frac{1}{i!}\left.\frac{d^i}{dt^i}\right|_{t=0}f(\gamma(t)).$$
		It holds in particular that $f^{(0)}=f\circ p^{(0,r)}$. It is easy to check that $f^{(i)}(\lambda\cdot[\gamma])=\lambda^i f^{(i)}([\gamma])$. For simplicity, we suppress the dependence on $r$ in our notation $f^{(i)}$ for the lifts. This should not lead to ambiguities, as the $i$-th lifts to tangent bundles of different orders $r_1,r_2\geq i$ are exactly related by the quotient map $p^{(r_2,r_1)}$. While smoothness of the lift for finite $r$ is immediate, these relations also yield smoothness for $r=\infty$.
\item The lift of a product of functions $\{f_l\}_{l=1,...,p}$ satisfies a Leibniz rule
	\begin{equation}\label{eq:leibniz}
		\left(\prod\limits_{l=1}^p f_l\right)^{(i)} = \sum\limits_{|\vec{i}|=i} \prod\limits_{l=1}^p f_l^{(i_l)}
	\end{equation}
	where the sum is taken over multi-indices $\vec{i}=(i_1, ..., i_p)$.
\item While we do not have a need for it here, note that there also exists a notion of lift $X^{(-i)}$ for vector fields $X\in\fX(M)$ as well as more generally for tensor fields as observed in~\cite{Mo70b}, see also~\cite{LM23}.
\item Let $\chi=(x_1, ..., x_m):U\to\R^m$ be a coordinate chart on an open $U\subseteq M$ and $r$ be finite. Then the lifts of the components of $\chi$ assemble into coordinates
$$ T^{(r)}\chi=\left(\x{1}{0}, ..., \x{m}{0}, ..., \x{1}{r}, ..., \x{m}{r}\right) : T^{(r)}U\to \R^{m(r+1)} $$
on $T^{(r)}M$, where we make the canonical identification $T^{(r)}\R^m=\R^{m(r+1)}$ implicit. Conversely, any element $q\in T^{(r)}M$ can be regarded as arising from a curve $\gamma:\R\to U$ whose components under $\chi$ satisfy
\begin{equation}\label{eq:curve-TrM}
x_i(\gamma(t)) = \sum\limits_{j=0}^r t^j \x{i}{j}(q).
\end{equation}
In particular, the $\R$-action in local coordinates is given by
    \begin{equation*}
    \lambda\cdot(\x{1}{0}, ..., \x{m}{r}) := (\lambda^{0}\x{1}{0}, ..., \lambda^{r}\x{m}{r}).\qedhere
    \end{equation*}
\end{enumerate}
\end{npar}

For our later proof of Theorem~\ref{thm:new-criterion}, we want to put two small technical lemmas about derivatives of lifts on the record:

\begin{lemma}\label{lem:lift-lin-dependence}
Let $q\in T^{(r)}M$ lie in the zero section, i.e. $0\cdot q=q$, and consider a function $f:U\to\R$ and coordinates $x_a:U\to\R$ defined close to $p=p^{(0,r)}(q)\in M$. Then $$d_qf^{(i)}=\sum_a \left(\frac{\partial f}{\partial x_a}\right)^{(0)}\;d_q x_a^{(i)} \text{ for all $i\leq r$.}$$ In particular, the $i$-th lifts of a collection of functions on $U$ are linearly independent at $q$ if and only if the functions themselves are linearly independent at $p$.
\end{lemma}

\begin{proof}
Since the lifts of coordinates $\chi=(x_1, ..., x_m)$ on $M$ form coordinates on $T^{(r)}M$, we just need to show that
$$\frac{\partial}{\partial x_a^{(j)}} f^{(i)}(q) = \delta_{ij}\; \frac{\partial}{\partial x_a}f(p),$$
where $\delta_{ij}$ is the Kronecker delta.
When writing $f_\chi=f\circ\chi^{-1}$ for the representative of $f$ in the coordinates $\chi$, the left hand side evaluates to
$$\left.\frac{d}{ds}\right|_{s=0} \;\frac{1}{i!}\; \left.\frac{d^i}{dt^i}\right|_{t=0} f_\chi( \chi(p) + s t^j\,e_a),$$
where $e_a$ is the unit vector in the $a$-th direction and we crucially used that $q$ can be represented by the constant curve at $p$. Exchanging and evaluating the $s$-derivative, we obtain
$$\left(\frac{1}{i!}\; \left.\frac{d^i}{dt^i}\right|_{t=0} t^j \right)\;d_{\chi(p)}f_\chi(e_a)=\delta_{ij}\; \frac{\partial}{\partial x_a}f(p)$$
and are done.
\end{proof}

\begin{lemma}\label{lem:action-on-lift-deriv}
Let $q\in T^{(r)}M$, $X\in T_qT^{(r)}M$, $f\in C^\infty(M)$ and $i\in\N$. Then for all $\lambda\in\R$,
$$d_{\lambda\cdot q}f^{(i)}(\lambda\cdot X) = \lambda^i\; d_q f^{(i)}(X)$$
holds where $\lambda\cdot X$ is the action on $TT^{(r)}M$ induced by that on $T^{(r)}M$.
\end{lemma}

\begin{proof}
The vector $X$ can be represented as an equivalence class $[\gamma_\bullet]$ of a smooth family of paths $\{\gamma_s\}_{s\in\R}$ such that $q=[\gamma_0]$. It follows that $\lambda\cdot X = [\gamma_\bullet(\lambda \cdot)]$ and
\begin{align*}
d_{\lambda\cdot [\gamma_0]}f^{(i)}(\lambda\cdot [\gamma_\bullet]) &= \left.\frac{d}{ds}\right|_{s=0} \; f^{(i)}(\lambda\cdot[\gamma_s]) \\
&=  \left.\frac{d}{ds}\right|_{s=0} \;\lambda^i\; f^{(i)}([\gamma_s]) \\
&= \lambda^i\; d_{[\gamma_0]} f^{(i)}([\gamma_\bullet]),
\end{align*}
concluding our proof.
\end{proof}

\startSubchaption{Weightings}\label{ssec:weightings}

We are now ready to give a definition of weightings as subbundles of $T^{(\infty)}M$. While being phrased in profinite language, it is equivalent to the definition given in~\cite{LM23}.

\begin{definition}\label{def:weigthing} Let $M$ be a manifold of dimension $m$.
\begin{enumerate}
\item A \textbf{weight sequence $w$ of order $r\geq0$} is a sequence $w_1, ..., w_m$ of  non-negative integers bounded from above by $r$.
\item A \textbf{weighting} on $M$ is a subset $\cW\subseteq T^{(\infty)}M$ such that, around every $p\in M$, there are coordinates $(x_1, ..., x_m):U\to\R^m$ in which either $\cW\cap T^{(\infty)}U$ is empty or $$\cW \cap T^{(\infty)}U = \{q\in T^{(\infty)}M \;|\; \x{i}{j}(q)=0 \text{ for all $i$ and } j<w_i\}$$
	holds for some globally fixed weight sequence. We call such coordinates \textbf{adapted to $\cW$} and the tuple $(M,\cW)$ is a \textbf{weighted manifold}. $(M,\cW)$ is of \textbf{order} $r$ if its weight sequence is of order $r$.
\item A smooth map $\phi:(M,\cW)\to (\widetilde M,\widetilde\cW)$ between weighted manifolds is a \textbf{morphism of weighted manifolds} if $$T^{(\infty)}\phi(\cW)\subseteq \widetilde\cW.$$
\item The \textbf{support} of a weighting $\cW$ is given by $\supp\cW:=p^{(0,\infty)}(\cW)$. We say $\cW$ is a weighting \textbf{along a submanifold} $N\subseteq M$ if $N=\supp\cW.$
\end{enumerate}
\end{definition}

Intuitively, a weighting $\cW$ is a coordinate-invariant way to assign positive integer weights $w_i$ to all directions that are normal to $\supp\cW$. By construction, a weight is zero if and only if it is assigned to a direction tangent to $\supp\cW$. As a set, the weighting contains exactly the classes of curves based at $\supp\cW$ whose Taylor coefficients of order less than $w_i$ vanish along each coordinate direction $i$. The following Lemma follows immediately from the existence of adapted charts:

\begin{lemma}
Let $(M,\cW)$ be a weighted manifold. Then $\cW\subseteq T^{(\infty)}M$ and $\supp\cW\subseteq M$ are topologically closed and embedded submanifolds and the former is invariant under the reparametrization action.
\end{lemma}

This means that $\cW$ is the pullback of a submanifold at some finite level of $T^{(\infty)}M$. We can see that $(M,\cW)$ is of order $r$ exactly when this is possible with a submanifold of $T^{(r-1)}M$. It is also easy to see that that a weighted morphism maps the support of the domain into that of the codomain.

\begin{examples}\ \label{ex:weightings}
\begin{enumerate}
	\item Given a manifold $M$ we have the \textbf{initial weighting} $\cW^{i}=\emptyset$ with $\supp\cW^{i}=\emptyset$ and the \textbf{final weighting} $\cW^{f}=T^{(\infty)}M$ with $\supp\cW^{f}=M$. The initial weighting can be seen as arising from any weight sequence and is the only weighting with an ambiguous weight sequence. The final weighting arises from a sequence of only zeros. Any smooth map from $(M,\cW^i)$ or into $(M,\cW^f)$ can be upgraded to a morphism of weightings, and every chart is adapted to these weightings.
\item\label{ex:standard-weightings} On $M=\R^m$, for each given weight sequence $w_1, ..., w_m$ we have as local model the \textbf{standard weighting}
$$\cE^{(w_1, ..., w_m)}:=\{q\in T^{(\infty)}M \;|\; \x{i}{j}(q)=0 \text{ for all $i$ with } j<w_i\}$$
along the submanifold
$$\supp\cE^{(w_1, ..., w_m)}=\{(x_1, ..., x_m)\in\R^m \;|\; w_i\neq 0 \implies x_i=0 \}. $$ Concretely, the support of $\cE^{(0,1,2)}$ is the $x_1$-axis, and the elements of the weighting are equivalence classes
		$$[ t\mapsto (f_1(t), f_2(t), f_3(t)) ]$$
		where $f_2(0)=f_3(0)=f'_3(0)=0$.
The identity map $$\Id:(\R^m, \cE^{(w_1, ..., w_m)})\to(\R^m,\cE^{(\tilde w_1, ..., \tilde w_m)})$$ is a morphism of weighted manifolds if and only if $w_i\geq \tilde w_i$ for all $i$.
\item Given a closed submanifold $N$, we can define the \textbf{trivial weighting along $N$} by setting $\cW_N:=\left(p^{(0,\infty)}\right)^{-1}(N)$. Any chart of $M$ adapted to $N$ yields a chart adapted to this weighting for the weight sequence consisting of $\dim N$ zeros and $\codim N$ ones. One can see that the support $\supp{\cW_N}$ is just $N$, and conversely that any weighting of order 1 is a trivial weighting along its support. The identity map on $M$ is always a weighted morphism $(M,\cW)\to(M,\cW_N)$ for any weighting $\cW$ along $N$, while a map $\phi:M\to\widetilde{M}$ is a weighted morphism $(M,\cW_N)\to(\widetilde{M}, \cW_{\tilde N})$ between trivial weightings exactly when $\phi(N)\subseteq \tilde N$.

\item More generally for a closed submanifold $N$ and \textit{any} order $r$, there is a unique \textbf{maximal weighting of order $r$ along $N$} that can be written as the pull-back of the zero section of $T^{(r-1)}M|_N$ along $p^{(r-1)}$.

\item Any filtered submanifold $N$ of a filtered manifold $(M,H_\bullet)$ naturally comes with the structure of a weighting $\cW_{H_\bullet,N}$ with support $N$. It intuitively consists of equivalence classes of curves whose $i$-th derivative lies within the $i$-th degree of the filtration. We discuss this in detail in Section~\ref{ssec:weighting-filtered-submfd} as part of defining a blown-up configuration space for filtered manifolds. A central result of~\cite{LM22} is that this also holds in the wider context of singular Lie filtrations.

\end{enumerate}
Further examples can be constructed by considering products and pullbacks of weightings as discussed in~\cite{LM23}.
\end{examples}

A central object associated with a weighting is its structure filtration:

\begin{definition}\label{def:structure-filtration}
Let $\cW$ be a weighting on a manifold $M$. The \textbf{structure filtration} of $\cW$ is the filtration
$$ C^\infty(M)= \cW_0 \supseteq \cW_1 \supseteq \cW_2 \supseteq ... $$
given by ideals
	$$\cW_i:=\{f\in C^\infty(M)\;|\; f^{(j)}|_\cW=0 \text{ for all }j<i\}.$$
\end{definition}

Elements of degree $i$ in the structure filtration should be viewed as the functions that vanish to $i$-th order when taking the weights into account. 
That these are ideals follows by the Leibniz rule from Equation~\eqref{eq:leibniz}. It also implies that the structure filtration is multiplicative, i.e. $\cW_i\cdot\cW_{i'}\subseteq \cW_{i+i'}$ for all $i,i'\in\N$. The elements of $\cW_i$ are exactly the weighted morphisms from $(M,\cW)$ into the maximal weighting of order $i$ along $\{0\}\subset\R$.

\begin{examples}\ \label{ex:weightings-filtration}
\begin{enumerate}
\item For the initial weighting $\cW^i$, all degrees of the structure filtration are given by $C^\infty(M)$. For the final weighting $\cW^f$, all positive degrees of the structure filtration are given by $\{0\}$.
\item\label{ex:standard-weightings-filtration} For the standard weighting $\cE^w$ of a weight sequence $w=(w_1, ..., w_m)$ on $\R^m$, we prove in the next proposition that $\cE^w_i$ is the ideal in $C^\infty(\R^m)$ generated by
\begin{equation}\label{eq:local-model-filtration}
	x_1^{\alpha_1}...x_m^{\alpha_m} \qquad\text{where}\qquad \alpha_1 w_1 + ... + \alpha_m w_m \geq i.
\end{equation}
It is sufficient to consider only the minimal $(\alpha_1, ..., \alpha_m)$ satisfying the inequality to obtain a finite set of generators. Using the fact that weighted coordinates locally identify a weighting with a standard weighting, this means that the structure filtration of any weighting is \textit{locally} generated by such polynomials in the coordinate functions.

		Concretely for coordinates $(x,y,z)$ of $\R^3$ with weights $w=(0,1,2)$, we obtain the following filtration at low degrees:
		\begin{align*}
			\cE^w_0 &= C^\infty(\R^3), \\
		\cE^w_1 &= \langle y,z \rangle, \\
		\cE^w_2 &= \langle y^2, z \rangle, \\
		\cE^w_3 &= \langle y^3, yz, z^2 \rangle, \\
		\cE^w_4 &= \langle y^4, y^2z, z^2 \rangle.
		\end{align*}
	Note how the weight of a coordinate function is always the largest degree of the filtration in which it is still included.
\item The structure filtration of the trivial weighting $\cW_N$ along a manifold $N$ is given by powers $\cW_i=\cI^i$ of the vanishing ideal $\cI$ of $N$.\qedhere
\end{enumerate}
\end{examples}

\begin{proposition}\label{prop:filtration-standard-weighting}
	Let $w=(w_1, ..., w_m)$ be a weight sequence. Then the structure filtration of the standard weighting $\cE^w$ on $\R^m$ is given by
	$$\cE^w_i=\langle x^\alpha \;|\; \alpha\in\N^m \text{ with } \alpha w\geq i \rangle,$$
	where the right-hand side denotes the generated ideal in $C^\infty(\R^m)$.
\end{proposition}

\begin{proof}
It is easy to see that $x_a\in\cE^w_a\setminus\cE^w_{a+1}$ by definition of the standard weighting, and this implies by multiplicativity that any polynomial of the form above lies in the ideal $\cE^w_i$.

Conversely, we need to show that for any $f\in\cE^w_i$ we can produce functions $f_\alpha\in C^\infty(\R^m)$ such that 
	\begin{equation}\label{eq:hadamard-to-show}
		f(x)=\sum_{\alpha w\geq i} x^\alpha f_\alpha(x)
	\end{equation}
	with only finitely many $f_\alpha$ not vanishing everywhere. This is a weighted version of Hadamard's Lemma and can be proven similarly:
For any $x\in\R^m$, the curve
	\begin{align*}
		\gamma_x:\R&\to\R^m\\ t &\mapsto t\cdot x
	\end{align*}
along the weighted action on $\R^m$ is a representative of an element $[\gamma_x]$ in $\cE^w$. Evaluating $f$ along these curves yields functions $$g_x:=f\circ \gamma_x:\R\to\R.$$ By definition of the lifts and $f\in\cE^w_k$, we have $$\frac{g^{(k)}(0)}{k!} = f^{(k)}([\gamma_x]) = 0 \qquad\text{for all $k<i$}.$$
Thus the Taylor expansion of $g_x$ at $t=0$ to order $i-1$ vanishes, and the integral formula for the remainder at $t=1$ yields
	$$f(x)=g_x(1)=\int_0^1\frac{g_x^{(i)}(s)}{(i-1)!}(1-s)^{i-1}ds.$$
By complete induction over $i$, one can show easily that there exist functions $g^\alpha:\R^m\to\R$ for every multi-index $\alpha$ with $\alpha w\geq i$ such that
	$$g_x^{(i)}(t)=\sum_{\alpha w\geq i} x^\alpha t^{\alpha w-i}g^\alpha(t\cdot x),$$
	where only a finite number of $g^\alpha$ is not everywhere vanishing. Inserting this into the integral remainder formula and pulling out the sum and powers of $x$ yields
	the desired Equation~\eqref{eq:hadamard-to-show}.
\end{proof}

We observe that the $\cE^w_i$ with degree $i$ larger than the order $r$ of $w$ contain no more elements than minimally necessary due to multiplicativity, i.e. they are given by sums of products of functions whose degrees are smaller than $r$ but sum to $i$. With a partition of unity argument, we obtain the following corollary for general weightings:

\begin{corollary}\label{cor:mult-after-order}
Let $\cW$ be a weighting of order $r$ over $M$ and $i> r$. Then $\cW_i$ consists of all sums of products $\prod_{j=1}^{n}f_j$ where $f_j\in\cW_{i_j}$, $i_j\leq r$ and $\sum_{j=1}^n i_j=i$.
\end{corollary}

This shows that the structure filtration is generated as a multiplicative filtration by its degrees up to the order of the weighting and also ensures that our definition coincides with that of \cite{LM23}.

\begin{npar}[Remark]While we mostly take the higher tangent bundle point of view, for completeness we collect some important properties of the structure filtration:
\begin{enumerate}
\item One can recover a weighting from its structure filtration by the following formula:
\begin{equation}\label{eq:weighting-from-filtration}
 \cW = \{ q\in T^{(\infty)}M \;|\; \forall j<i, f\in \cW_i: f^{(j)}(q)=0\}
\end{equation}
\item The support is exactly the vanishing locus of the first degree of the structure filtration:
$$\supp\cW=\{p\in M \;|\; f(p)=0 \text{ for all } f\in\cW_1 \}$$
\item A smooth map $\phi:(M,\cW)\to (\widetilde M,\widetilde\cW)$ is a morphism of weighted manifolds if and only if $\cW_i\supseteq\phi^*\widetilde\cW_i$ holds for all $i\geq0$.
\end{enumerate}
These follow by straightforward computations, involving local coordinates adapted to the weighting where necessary.
\end{npar}

\begin{npar}[Equivalent definitions]\label{rem:equiv-defs} The previous remark suggests that one can equivalently define weightings using their structure filtration as data. In~\cite[][Def.~2.2]{LM23} Loizides and Meinrenken do just that: They introduce weightings as filtrations satisfying a local model similar to Equation~\eqref{eq:local-model-filtration}. They also invariantly characterize which filtrations correspond to weightings in~\cite[][Prop.~2.8]{LM23}. Our description as a subset of $T^{(\infty)}M$ is motivated by their local expression at finite order $r$ from~\cite[][Lem.~7.3]{LM23}. Finally, they also provide an invariant criterion for the higher tangent bundle viewpoint in~\cite[][Th.~8.4]{LM23}. Their $T_rM$, $Q$, $C^\infty_{M,(i)}$ and $N$ become $T^{(r)}M$, $p^{(r,\infty)}(\cW), \cW_i$ and $\supp\cW$ in our notation. Beyond this we differ in not requiring weight sequences and thus components of adapted coordinates to be ordered, which will be convenient to express alignment of multiple weightings.

Melrose's equivalent \textit{quasi-homogeneous structures} are defined by an invariant characterization of the structure filtration in~\cite[][Def.~1.15.11]{Mel91} with a local description similar to our Equation~\eqref{eq:local-model-filtration} following right after.
\end{npar}

Note that in terms of the vanishing ideal $I_\cW$ of a weighting $\cW$, coordinates $\{x_1,...,x_m\}$ are adapted to $\cW$ with weights $\{w_1,...,w_m\}$ if and only if
\begin{equation}\label{eq:adapted-coords-from-ideal}
I_\cW = \langle x_a^{(i)} \;|\; 1\leq a\leq m,\, i<w_a \rangle
\end{equation}
holds locally. This follows immediately from Definition~\ref{def:weigthing} of adapted coordinates. We conclude our discussion of weightings by proving a new characterization in this spirit without resorting to coordinates:

\begin{theorem}\label{thm:new-criterion}
	A connected embedded submanifold $\cW\subseteq T^{(\infty)}M$ is a weighting if and only if it is closed and
	its vanishing ideal satisfies
			\begin{equation}\label{eq:new-criterion}
				I_\cW = \langle f^{(i)} \in C^{\infty}(T^{(\infty)}M) \;|\; f\in C^\infty(M) \text{ with } f^{(j)}|_\cW=0\;\forall j\leq i \rangle.
			\end{equation}
\end{theorem}

In terms of the structure filtration of a weighting, Equation~\eqref{eq:new-criterion} is equivalent to saying that $I_\cW$ is generated by $\oplus_i \cW_i^{(i-1)}$, i.e. by the $(i-1)$-st lifts of elements of $\cW_i$ for any positive $i\in\N$.

\begin{proof}
It is clear that any weighting satisfies Equation~\eqref{eq:new-criterion} using local adapted coordinates and is closed by definition. Conversely we need to show for any closed $\cW$ satisfying this equation that around every point $p\in M$ we can produce coordinates adapted to $\cW$. Note first that due to $$f^{(i)}(\lambda\cdot q)=\lambda^i\,f^{(i)}(q),$$ such a $\cW$ must be invariant under the $\R$-action on $T^{(r)}M$. In particular, it is a graded subbundle. It follows that the projection $N:=p^{(0,\infty)}(\cW)$ is also a closed embedded submanifold. Thus any coordinates around $p\not\in N$ are adapted to $\cW$ when restricted to a small enough neighbourhood.

Consider now a point $p\in N$ and take a preimage $q_0\in\cW$ under $p^{(0,\infty)}$. By $\R$-invariance, we can pick $q_0$ in the zero section, i.e. $0\cdot q_0 = q_0$. Since $\cW$ is an embedded submanifold\footnote{Note that this is an embedded submanifold in the profinite sense of Paragraph~\ref{npar:profinite}. However, this makes no difference for the purposes of this proof, as any property we need can be proven by passing to a finite level $T^{(r)}M$ where $\cW$ becomes an honest embedded submanifold.}, it is cut out in some neighbourhood $U$ of $q_0$ by a collection of $k:=\codim \cW$ independent functions. Due to Equation~\eqref{eq:new-criterion}, we can further assume that there are functions\footnote{Since we do not yet know that $\cW$ is a weighting, it bears mentioning that Definition~\ref{def:structure-filtration} of $\cW_{d_a+1}$ makes sense for any subset $\cW\subseteq T^{(\infty)}M$.} $f_a\in \cW_{d_a+1}$ with $d_a\in\N_0$ such that $$\cW\cap U = \{q\in U\;|\;  f_a^{(d_a)}(q)=0 \text{ for all }a=1, ..., k\}.$$ Our goal is now to construct a set of functions $\{x_b\}_{b=1,...,\codim N}$ on $U$ (possibly shrinking it) and weights $w_b\in\N_{>0}$ such that $$\cW\cap U = \{q\in U\;|\;  x_b^{(i)}(q)=0 \text{ for all }b=1, ..., \codim N\text{ and }i<w_b\},$$ and subsequently extend these to a full set of adapted coordinates.

We construct the $x_a$ by proving the following statement inductively in $i\in\N_0$, starting at some large $i_{\text{max}}$ and decreasing: \textit{For every integer $i\in \N_0$, close to $q_0\in T^{(\infty)}M$ there is a minimal set $$\cM_i=\left\{m_{i,a}^{(d_{i,a})}\;|\; a=1,...,k\right\}$$ of functions $m_{i,a}\in C^\infty(U)$ lifted to degree $d_{i,a}$ such that: (1) $\cM_i$ cuts out $\cW$, (2) $m_{i,a}\in\cW_{d_{i,a}+1}$ and (3) if $i\leq j\leq d_{i,a}$, then there exists some other $a'\in\{1,...,k\}$ such that $m_{i,a}=m_{i,a'}$ and $d_{i,a'}=j$.}

The last condition says that any function that appears in $\cM_i$ lifted to any degree larger than $i$ also appears lifted to all other degrees above $i$.

For $i=0$, $\cM_0$ consists of the lifts of the $x_b$ we desire: Each function in the set must appear with all lifts up to some maximal weight. Since all $x_b^{(0)}$ lie in $\cM_0$, and clearly the projection $N$ is cut out by these, $b$ must indeed range from 1 to $\codim N$. When choosing $i_{\text{max}}$ large enough to make condition (3) trivial, our starting set $\cM_{i_{\text{max}}}:=\{f_a^{(d_a)}\;|\; a=1,...,k\}$ satisfies the statement.

Assuming now that the statement holds for some fixed $i$, we need to show that it also follows for $(i-1)$:
Take any function $f$ such that $f^{(i)}$ lies in $\cM_i$ but $f^{(i-1)}$ does not. By assumption, $f^{(i-1)}$ vanishes on $\cW$ and therefore its differential $d_{q_0} f^{(i-1)}$ is a linear combination of the derivatives of functions in $\cM_i$. By Lemma~\ref{lem:lift-lin-dependence}, $d_{q_0} f^{(i-1)}$ must be a linear combination of functions that are also lifted to the degree $(i-1)$. Since the coefficients in the Lemma are the same irrespective of the degree of the lift, we can conclude that $d_{p}f$ is a linear combination of differentials of functions whose $(i-1)$st lifts appear in $\cM_i$. By the same reasoning applied to $f^{(i)}\in\cM_i$ and since we know that $\cM_i$ is a minimal generating set, we conclude that $d_{p}f$ is \textit{not} a linear combination of differentials of functions whose $i$th lifts appear in $\cM_i$. Thus there appears a $g$ in the former linear combination such that $g^{(i-1)}\in\cM_i$ but $g^{(i)}\not\in\cM_i$. We can now remove this $g^{(i-1)}$ from $\cM_i$ and add in $f^{(i-1)}$ instead to obtain a new $\cM_i'$. This set clearly still vanishes on $\cW$ since $f\in\cW_{i+1}$. It is also still a minimal set that cuts out $\cW$ since it has the same number of elements as $\cM_i$ that have linearly independent differentials by construction.
After sufficiently many steps of this, we have constructed a new set $\cM_{i-1}$ that satisfies the statement for $(i-1)$.

We thus know that the functions $\{x_b\}_{b=1, ..., \codim N}$ exist. Their lifts cut out $\cW$ close to $q$ on the zero section, but by invariance of $\cW$ under the $\R$-action, they vanish on $\cW$ over the full fiber of a neighbourhood $V$ of $p\in N$. Moreover, the differentials of the lifts are linearly independent on the full fibers due to Lemma~\ref{lem:action-on-lift-deriv}. They thus form a minimal set of functions cutting out $\cW$ not only close to $q$ but on $\left(p^{(0,\infty)}\right)^{-1}(V)$.
They assemble into a map from $V$ to $\R^{\codim N}$ that is a submersion at $p$. By the submersion theorem, we can extend the $x_b$ to a full set $\{x_b\}_{b=1,...,m}$ of coordinates close to $p$. These coordinates are adapted to $\cW$.
\end{proof}

This argument of course also constructs local adapted charts around points $p\in\supp\cW$ without the closed and connectedness assumptions. However, these charts do not need to exist around $p\in\partial\supp\cW\setminus\supp\cW$ and may correspond to weight sequences that don't match (even after reordering) on different components of $\cW$. The latter phenomenon will show itself in particular in Proposition~\ref{prop:clean-locally-weighting}.

\startSubchaption{The weighted normal bundle}\label{ssec:weighted-normal}

To construct weighted blow-ups, one needs a weighted analogon to the normal bundle:

\begin{definition}\label{def:weighted-normal-bundle}
Let $(M,\cW)$ be a weighted manifold of order $r$. As a set, the \textbf{weighted normal bundle} is the quotient $\nu\cW:=\cW/\sim$ under the equivalence relation
$$ q_1\sim q_2 :\iff f^{(j)}(q_1)=f^{(j)}(q_2) \text{ for all } j\in\N_0, f\in\cW_j.$$ It comes with a canonical projection $\nu\cW\to \supp\cW$ induced by the map $T^{(\infty)}M\to M$.
\end{definition}

Classically, the normal bundle of a submanifold $N\subseteq M$ can be defined as the quotient of $TM|_N$ by $TN$. Here $TM|_N$ is generalized to $\cW$, representing all the higher tangent vectors based at $\supp\cW$ that are consistent with the weighting. In particular, we quotient out all the vectors tangent to $\supp\cW$ since they cannot be distinguished by elements of the structure filtration. A visualization of the resulting weighted normal vectors is depicted in Figure~\ref{fig:weighting}.

\begin{figure}
    \centering
    \Large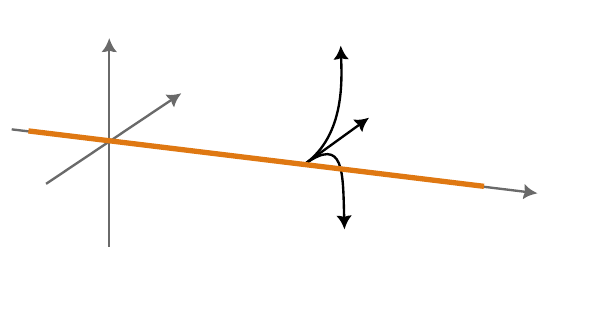
    \caption{The elements of the weighted normal bundle can be visualized as arrows based at the support and curved according to the weights, drawn here for the standard weighting $\cE^{(0,1,2)}$ along $N=\supp\cE^{(0,1,2)}$. A weighted vector in $\nu\cE^{(0,1,2)}$ with coordinates $(y_1,y_2,y_3)$ corresponds to the line $\{\lambda\cdot (y_1,y_2,y_3)\;|\;\lambda\in[0,1]\}$ traced out by the action associated with the weights. The coordinates correspond precisely to the end-point of the line, giving an identification $\nu{\cE^{(0,1,2)}}\simeq\R^3$.}
    \label{fig:weighting}
\end{figure}

\begin{npar}[Properties of weighted normal bundles]\label{props-weighted-normal} Let $\cW$ be of order $r$.
\begin{enumerate}
\item If the lifts $f^{(j)}$ for all $f\in\cW_j$ agree on some elements of $\cW$ up to $j=r$, then they must also agree for all $j>r$ by Corollary~\ref{cor:mult-after-order} and the product rule for lifts. We may thus equivalently define the weighted normal bundle as the quotient of $p^{(r,\infty)}(\cW)\subseteq T^{(r)}M$ under the relation
$$ q_1\sim q_2 :\iff f^{(j)}(q_1)=f^{(j)}(q_2) \text{ for all } 0\leq j\leq r, f\in\cW_j.$$
\item In local coordinates of $T^{(r)}M$ induced by a chart $\chi=(x_1, ..., x_m):U\to\R^m$ on $U\subseteq M$ adapted to the weighting $\cW$, it therefore suffices to consider only the local generators of $\cW_j$:
$$ q_1\sim q_2 \iff \x{i}{w_i}(q_1)=\x{i}{w_i}(q_2) \text{ for all } i=1, ..., m.$$
In other words, the weighted normal bundle only remembers along each coordinate direction the first coefficient of a curve that is not forced to vanish by virtue of sitting in $\cW$. 
\item The weighted normal bundle thereby inherits a canonical smooth structure from $T^{(r)}M$. In concrete terms, charts $\chi:(x_1, ..., x_m):U\to\R^m$ on $U\subseteq M$ induce charts
$$\chi^{(\cW)}=\left(x_1^{(w_1)}, ..., x_m^{(w_m)}\right):\nu\cW|_U\to \R^m$$
by lifting each coordinate to the degree of its weight.

\item For a trivial weighting $\cW_N$, this definition reduces to the \textit{normal} normal bundle $\nu N\to N$.

\item The $\R$-action of $T^{(\infty)}M$ descends to $\nu\cW$, making it a graded bundle in the sense of Grabowski-Rotkiewicz as well. In local coordinates, it is given by
    \begin{equation}\label{eq:action-normal}
    \lambda\cdot\left(\x{1}{w_1}, ..., \x{m}{w_m}\right) := \left(\lambda^{w_1}\x{1}{w_1}, ..., \lambda^{w_r}\x{m}{w_r}\right).
    \end{equation}
This allows us to derive the expression
\begin{equation}\label{eq:inverse-coords-normal}
\left(\chi^{(\cW)}\right)^{-1}(y_1, ..., y_m) = \left[t\mapsto \chi^{-1}(t\cdot(y_1, ..., y_m))\right]
\end{equation}
for the inverse of the charts.
\item The projection $\pi:\nu\cW \to \supp\cW$ is given in charts by
$$
\left( \chi\circ\pi\circ\left(\chi^{(\cW)}\right)^{-1} (y_1, ..., y_m) \right)_i = \begin{cases}
y_i &\text{if }w_i=0,\\
0 &\text{else}
\end{cases}
$$
and is in particular a smooth fibration. When identifying the zero section of $T^{(\infty)}M$ with $M$, it is the result of acting on $\nu\cW$ by zero.
\item The weighted normal bundle construction is functorial: Given a weighted morphism $\phi:(M,\cW)\to(\widetilde M,\widetilde \cW)$, we get induced maps
\begin{align*}
\nu_\phi: \nu\cW&\to \nu{\widetilde \cW}\\
[\gamma] &\mapsto [\phi\circ\gamma]
\end{align*}
by composing a representing curve with $\phi$.
If $\phi$ is the identity map on a fixed manifold $M$, we will also denote this by $$\nu_{\widetilde \cW,\cW}: \nu\cW\to \nu{\widetilde \cW}.$$ Any induced map intertwines the actions and thus is a map of graded bundles. If $\chi$ and $\widetilde\chi$ are adapted coordinates on $M$ and $\widetilde M$, then the induced map locally takes the form $$
\left( \widetilde\chi^{(\widetilde\cW)}\circ\nu_\phi\circ
\left(\chi^{(\cW)}\right)^{-1}(y_1, ..., y_m)\right)_i = \frac{1}{\tilde w_i!} \left.\frac{d^{\tilde w_i}}{dt^{\tilde w_i}}\right|_{t=0}\phi_{\widetilde\chi\chi}(t\cdot(y_1, ..., y_m))_i,
$$
where $\phi_{\widetilde\chi\chi}:=\widetilde\chi\circ\phi\circ\chi^{-1}$ is the local representation of $\phi$.\qedhere
\end{enumerate}
\end{npar}

\begin{figure}
    \centering
    \Large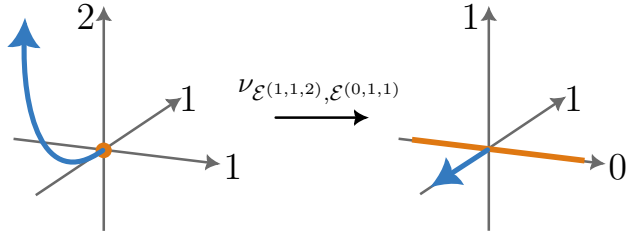
    \caption{Example of the map induced between the weighted normal bundles for weights $(1,1,2)$ and $(0,1,1)$ acting on an arbitrary blue weighted vector. The supports are the origin and the first axis, respectively, and highlighted in orange.}
    \label{fig:induced-example}
\end{figure}

\begin{example}\label{example:induced-map-weighted-normal-bundle}
Let $\cE$ and $\cE'$ be standard weightings on $\R^m$ with weights $w_i$ and $w'_i$ along the coordinates $x_i$ such that the identity $\Id:(\R^m,\cE)\to(\R^m,\cE')$ is a morphism of weightings, i.e. $w_i\geq w'_i$. In this case the induced map $\nu_{\cW',\cW}$ sends a weighted normal vector in $\nu\cW$ with coordinates $x_i^{(w_i)}$ to one in $\nu{\cW'}$ with coordinates $$x_i^{(w'_i)}=\begin{cases}
x_i^{(w_i)} &\text{if }w_i=w'_i,\\
0 &\text{else.}
\end{cases}$$
In other words, in coordinates the map $\nu_{\cW',\cW}$ is the canonical projection onto $$\{(y_1,...,y_m)\in\R^m\;|\; y_i=0 \text{ for all $i$ with }w_i\neq w'_i  \}.$$
An example for this is depicted in Figure~\ref{fig:induced-example}.
\end{example}

\startSubchaption{The dual filtration}\label{ssec:dual}

We now want to turn to the dual viewpoint, and study weightings in terms of higher order cotangent vectors. While this may be conceptually enlightening, our use of this is restricted to a convenient characterization of uniform alignment in Proposition~\ref{prop:char-uniform-align}. To start, we have objects dual to the higher tangent bundles:

\begin{definition}
Let $M$ be a manifold and $r\in\N_0\cup\{\infty\}$. The \textbf{$r$-th order cotangent bundle} is given by the jet space
$$T_{(r)}M:=\{ j^rf(p)\in J^r(M,\R)\;|\; f:M\to\R,\; p\in M\text{ with } f(p)=0\}$$
of functions on $M$ that vanish at the evaluation point.
It thus comes for each $r_1\geq r_2$ with smooth quotient maps
$$p_{(r_2,r_1)}:T_{(r_1)}M\to T_{(r_2)}M$$
and in particular the projection $p_{(r)}:=p_{(0,r)}$ to $T_{(0)}M=M$.
\end{definition}

Note that we can canonically identify $T_{(1)}M=T^*M$.

\begin{npar}[Properties of $T_{(r)}M$]\ 
\begin{enumerate}
\item Once again this carries a smooth structure for finite $r$ and otherwise a profinite structure, both induced by the smooth structure on finite jet space.

\item Unlike $T^{(r)}M$, $T_{(r)}M$ comes with the additional structure of a bundle of algebras with the multiplication determined by multiplying representative functions. Concrete expressions for finite $r$ can be obtained by the Leibniz rule.

\item Consider again the non-commutative algebra
$$\Lambda_r=J^r_{0,0}(\R,\R)=\{j^rf(0)\;|\;f(0)=0\}\simeq\R^r$$
where the multiplication is given by composing representatives. Note that this is different from the commutative multiplication given by multiplying representatives. This algebra acts on $T^{(r)}M$ and $T_{(r)}M$
by pre- and post-composition, respectively.

\item There are natural projections $\Lambda_{r}\to\Lambda_i$ respecting the algebra structure whenever $r\geq i$. Conversely, there is a natural inclusion $\Lambda_i\subseteq\Lambda_r$ respecting the algebra structure only when $i=1$, since otherwise composition creates higher-order terms.

\item Restricting the action to $\R=\Lambda_1\subseteq\Lambda_r$ just gives the scalar multiplication of the algebra structure on $T_{(r)}M$, making it much less interesting as a graded bundle.

\item Analogously to the situation for the cotangent bundle, this construction is not functorial, but does pointwise induce for every smooth map $\phi:M\to \widetilde M$ and $p\in M$ a smooth equivariant map
	\begin{align*}
		T_{(r)}\phi|_p:T_{(r)}\widetilde M|_{\phi(p)}&\to T_{(r)} M|_p\\ j^rf(\phi(p)) &\mapsto j^r(f\circ\phi)(p).
	\end{align*}
    If $\phi$ is a diffeomorphism, we do get a global $T_{(r)}\phi:T_{(r)}\widetilde M\to T_{(r)} M$.
    
\item Given a chart $\chi=(x_1,...,x_m):U\to\R^m$ on $U\subseteq M$, this can once again be used to induce a chart $$T_{(r)}\chi^{-1}:T_{(r)}U\to T_{(r)}\R^m$$ on $T_{(r)}M$ that sends $j^rf(p)$ to the jet of the local representative $f\chi^{-1}$ at $\chi(p)$. Here we use the canonical identification $$T_{(r)}\R^m\simeq\R^{m+{{r+m}\choose m}-1}$$ that arises from $J^r(\R^m,\R)\simeq\R^{m+{{r+m}\choose m}}$ and write
$$T_{(r)}\chi^{-1}=\left( x_{(1)},...,x_{(m)}, \{\partial x_{(\alpha)}\}_{1\leq|\alpha|\leq r} \right)$$
for the components, where the multi-index $\alpha\in\N_0^m$ labels the derivatives as usual. Since $x_{(j)}=x_j\circ p_{(r)}$, we may for convenience sometimes identify the $x_{(i)}$ with the original coordinates $x_i$ on $M$, and for concrete $\alpha$ like $(0,1,2)$ we may write $\partial x_{(\alpha)}=\partial x_1\partial^2 x_3$. Since derivatives are linear, these coordinates are vector bundle coordinates. The rank of the bundle is ${{r+m}\choose m}-1$.
\item An element $\lambda=(\lambda_1, ..., \lambda_r)\in\Lambda_r$ acts on $g\in T_{(r)}M$ as follows in the coordinates corresponding to small multi-indices:
\begin{align*}
\partial x_1(\lambda\cdot g) &=
    \lambda_1\cdot\partial{x_1}(g)\\
\partial x_1 \partial x_2 (\lambda\cdot g) &= 
    \lambda_1\cdot\partial{x_1}\partial x_2(g) \;+\; \lambda_2\cdot\partial{x_1}(g)\cdot \partial{x_2}(g)\\
\partial x_1 \partial x_2 \partial x_3 (\lambda\cdot g) &= 
    \lambda_1\cdot\partial{x_1}\partial x_2\partial x_3(g)\\
    &\quad +\lambda_2\cdot\left( \partial{x_1}\partial x_2(g)\cdot \partial{x_3}(g) +\partial{x_1}\partial x_3(g)\cdot \partial{x_2}(g) + \partial{x_2}\partial x_3(g)\cdot \partial{x_1}(g) \right)\\
    &\quad +\lambda_3\cdot\partial{x_1}(g)\cdot \partial{x_2}(g)\cdot \partial{x_3}(g)
\end{align*}
A general expression results from a multivariate version of Faà di Bruno's formula, see e.g.~\cite[][Cor. 2.10]{CS96}.\qedhere
\end{enumerate}
\end{npar}

The duality between these higher order objects is reflected in the following pairings:

\begin{definition}
Let $M$ be a manifold, $r\in\N_0\cup\{\infty\}$ and $i\leq r$. On the fibered product
$$
T_{(r)}M\times_M T^{(r)}M = \{ (j^rf(p),[\gamma])\in T_{(r)}M\times T^{(r)}M\;|\; p=\gamma(0) \}
$$
we define the maps
\begin{align*}
\langle\cdot,\cdot\rangle_i:\; T_{(r)}M\times_M T^{(r)}M\; &\to\; \Lambda_i\\
(j^rf(p),[\gamma])\; &\mapsto\; j^i(f\circ\gamma)(0),
\end{align*}
\end{definition}

\begin{npar}[Remarks]\ 
\begin{enumerate}
\item Since $T^{(r)}M$ is not a vector bundle for $r>1$, these pairings are of course not generally bilinear. However, they are linear in the first argument and $\langle\cdot,\cdot\rangle_1$ recovers the usual bilinear pairing of tangent and cotangent bundles.

\item These pairings encode the same structure as the lifts of functions, but more symmetrically:
Indeed the $j$-th component of $\langle j^rf(p),[\gamma]\rangle_i$ under the canonical identification $\Lambda_i\simeq \R^{i}$ is exactly the lift $f^{(j)}([\gamma])$.

\item The pairings are equivariant under the natural action of $\Lambda_r^2$ by pre- and post-composition, i.e. $$\langle\alpha\cdot j^rf(p), [\gamma]\cdot\beta\rangle_i=\alpha\cdot\langle j^rf(p), [\gamma]\rangle_i \cdot\beta$$
for $\alpha,\beta\in\Lambda_r.$\qedhere
\end{enumerate}
\end{npar}

The structure filtration represents the weighting in the dual picture. We can make this manifest by expressing it in terms of the infinite order cotangent bundle:

\begin{definition}\label{def:dual-filtration}
To every weighting $\cW\subseteq T^{(\infty)}M$ we assign the \textbf{dual filtration}
$$T_{(\infty)}M|_{\supp\cW}=j^{\infty}\cW_{0}\supseteq j^{\infty}\cW_{1}\supseteq ...$$
by setting for each $i\geq0$ 
$$j^{\infty}\cW_i:=\{j^{\infty}f(p) \;|\;  p\in \supp\cW \text{ and } f\in\cW_i \text{ with }f(p)=0\}.$$
\end{definition}

\begin{npar}[Properties of the dual filtration]\ 
\begin{enumerate}
\item The condition $f(p)=0$ in Definition~\ref{def:dual-filtration} is automatic for $i\geq1$ since $x\in\supp\cW$ guarantees that any $f\in\cW_i\subseteq\cW_1$ vanishes at $x$.
\item When fixing a basepoint $p\in M$, the fibers yield a multiplicative filtration $j^\infty\cW_\bullet|_p$ of ideals. This follows from the corresponding properties of the structure filtration.
\item Passing to $j^\infty\cW_\bullet$ does not lose any information as we can still go back by setting $\supp\cW=p_{(\infty)}(j^\infty\cW_0)$ and using the following formulas:
\begin{align*}
\cW&=\{q\in T^{(\infty)}M|_{\supp\cW}\;|\;  \langle g,q\rangle_i=0\text{ for all }i\in\N_0, \;g\in j^\infty\cW_{i}\text{ with }p_{(\infty)}(g)=p^{(\infty)}(q) \},\\
\cW_i &=\{f\in C^\infty(M)\;|\; j^\infty f(p)\in j^\infty\cW_i \text{ for all } p\in \supp\cW\}.
\end{align*}

\item Given that the pairing encodes lifts of functions, and $\cW_i$ consists of those functions whose lifts of degree below $i$ vanish, we have
\begin{equation}\label{eq:inclusion-equiv}
\begin{aligned}
j^{\infty}\cW_i\;\subseteq\;\{g\in T_{(\infty)}M|_{\supp\cW}\;|\; &\langle g,q\rangle_{i-1}=0\text{ for all }q\in\cW\\
&\text{ with } p_{(\infty)}(g)=p^{(\infty)}(q) \}
\end{aligned}
\end{equation}
for $i\geq1$. It is tempting to say that the inverse inclusion is also true, but for example the jet at the origin of the first coordinate function $x_1$ for the standard weighting $\cE^{(0,1)}$ lies in the right-hand side for $i=1$, but not in $j^\infty\cW_1$. This can be understood as a lack of holonomicity of that jet:
For any element $g$ of $j^\infty\cW_i$, there must exist a representing function $f$ whose close-by jets all pair to zero with elements of $\cW$, while the right-hand side asks this only pointwise.\qedhere
\end{enumerate}
\end{npar}

We have the following dual local model for weightings:

\begin{proposition}\label{prop:dual-weighting-local}
Let $\cW$ be a weighting over $M$ and consider weighted coordinates $\chi=(x_1,...,x_m)$ over $U\subseteq M$ with weight sequence $w\in\N_0^m$. It then holds that 
\begin{align*}
j^\infty\cW_i \cap T_{(\infty)}U &= \langle j^\infty x^\beta (p)\;|\;p\in\supp\cW\cap U,\; \beta\in\N_0^m\setminus\{0\} \text{ with } \beta w\geq i \rangle\\
&= \{ g\in T_{(\infty)}U|_{\supp\cW} \;|\; \partial x_{(\alpha)}(g)=0\text{ for all }\alpha\in\N_0^m\setminus\{0\} \text{ with }\alpha w<i  \}\\
&= \{ g\in T_{(\infty)}U \;|\; \partial x_{(\alpha)}(g)=0\text{ for all }\alpha\in\N_0^m\setminus\{0\} \text{ with }\alpha w<i\\
&\qquad\qquad\qquad\qquad\qquad\qquad\text{ and }x_{(j)}(g)=0\text{ for all $j$ with }w_j>0 \},
\end{align*}
where the angled brackets denote taking the generated ideal fiberwise.
\end{proposition}

\begin{proof}
The first equality holds using the local model for $\cW_i$ from Proposition~\ref{prop:filtration-standard-weighting}. The third equality holds because $\supp\cW$ is cut out in $U$ by $x_j=0$ for all $j$ with $w_j>0$.
We are thus left with the second equality. Using that polynomials span all jets, we consider the equivalent statement that for $p\in\supp\cW$ and $\beta\in\N_0^m\setminus\{0\}$
$$\beta w\geq i\;\iff\; \partial x_{(\alpha)}(j^\infty x^\beta(p))=0\;\text{ for all } \alpha\text{ with }|\alpha|\geq1,\;\alpha w<i.$$

\textit{Regarding $(\Rightarrow)$:} Consider a multi-index $\alpha$ with $|\alpha|\geq1$ and $\alpha w<i$.
If $\alpha\not\leq \beta$, it always holds that $\partial x_{(\alpha)}x^\beta$ vanishes and we are done. If instead $\alpha\leq\beta$, it is proportional to $x^{\beta-\alpha}(p)$. We must have $\beta w\geq i>\alpha w$, so there must be some $\beta_j>\alpha_j$ for $j\in\{1,...,m\}$ with $w_j\neq0$. $x^{\beta-\alpha}(p)$ then contains a positive power of $x_j(p)$, which vanishes by $p\in\supp\cW$.

\textit{Regarding $(\Leftarrow)$:} Assume this did not hold, i.e. $\beta w< i$. Then we can take the statement on the right hand side for $\alpha=\beta$ and obtain the contradiction $\partial x_{(\beta)}x^\beta(p)=0.$
\end{proof}

In particular, $j^\infty\cW_\bullet$ consists of subbundles of infinite rank (unless all weights vanish), but generally also infinite corank\footnote{One situation where the corank is finite is when $\supp\cW$ consists of a single point. Indeed, the holonomicity condition becomes trivial, making Equation~\eqref{eq:inclusion-equiv} into an equality. The fact that the tangential directions in $\supp\cW$ are what generates the infinite corank is also apparent in Example~\ref{ex:dual-filtration}.}.

\begin{example}\label{ex:dual-filtration}
Consider again the standard weighting $\cE:=\cE^{(0,1,2)}$ on $\R^3$ with coordinates $(x,y,z)$ from Example~\ref{ex:weightings}(2). By taking the local model of the structure filtration, we have for $p\in \supp\cE$ that
\begin{align*}
j^\infty\cE_0|_p &= T_{(\infty)}\R^3|_p,\\
j^\infty\cE_1|_p &= \langle\; j^\infty y(p),\; j^\infty z(p)\;\rangle,\\
j^\infty\cE_2|_p &= \langle\; j^\infty y^2(p),\; j^\infty z(p)\;\rangle,\\
j^\infty\cE_3|_p &= \langle\; j^\infty y^3(p),\; j^\infty yz(p),\; j^\infty z^2(p)\;\rangle,\\
j^\infty\cE_4|_p &= \langle\; j^\infty y^4(p),\; j^\infty y^2z(p),\; j^\infty z^2(p)\;\rangle.
\end{align*}
By the previous proposition, this is the same as
\begin{align*}
j^\infty\cE_0 &= \{ g\in T_{(\infty)}\R^3 \;|\; y(g)=z(g)=0 \},\\
j^\infty\cE_1 &= \{ g\in j^\infty\cE_0 \;|\; \partial_x^{n+1}(g)=0 \text{  for all } n\geq 0 \},\\
j^\infty\cE_2 &= \{ g\in j^\infty\cE_1 \;|\; \partial_x^n\partial_y(g)=0 \text{  for all } n\geq 0 \},\\
j^\infty\cE_3 &= \{ g\in j^\infty\cE_2 \;|\; \partial_x^n\partial_z(g)=\partial_x^n\partial^2_y(g)=0 \text{  for all } n\geq 0 \},\\
j^\infty\cE_4 &= \{ g\in j^\infty\cE_3 \;|\; \partial_x^n\partial^3_y(g)=\partial_x^n\partial_y\partial_z(g)=0 \text{  for all } n\geq 0 \}.
\end{align*}
where we use the notation $\left(x,y,z,\{\partial^{\alpha_1}_x\partial^{\alpha_2}_y\partial^{\alpha_3}_z\}_{\alpha\in\N_0^3\setminus\{0\}}\right)$ for the coordinates induced by $(x,y,z)$ on $T_{(\infty)}\R^3$.
\end{example}

\startSubchaption{The weighted conormal bundle}\label{ssec:weighted-conormal}

Dual to the weighted normal bundle, we have the following:

\begin{definition}
Let $\cW$ be a weighting. We define for $i\geq1$ the \textbf{degree $i$ weighted conormal bundle} as the quotient $$\nu^*_{(i)}\cW:=j^\infty\cW_i/\sim,$$
where we identify $j^\infty f_1(p)\sim j^\infty f_2(p)$ if and only if
$$\langle j^\infty f_1(p),[\gamma]\rangle_i\;=\;\langle j^\infty f_2(p),[\gamma]\rangle_i$$
for all $[\gamma]\in\cW$ with $\gamma(0)=p$.
The full \textbf{weighted conormal bundle} is the graded vector bundle 
$$\nu^*\cW:=\bigoplus_{i=1}^\infty \nu^*_{(i)}\cW.$$
\end{definition}

Intuitively, the weighted conormal bundle $\nu^*\cW$ contains the jets $j^\infty\cW_i$ of functions that vanish to a given weighted degree $i$, but only up to what can be distinguished by evaluation on elements of $\cW$. This is exactly dual to our definition of the weighted normal bundle $\nu\cW$ as the elements of $\cW$ up to what can be distinguished by elements of $\cW_i$.

Note that while each $j^\infty\cW_i$ is a bundle of algebra ideals, the multiplicativity of the filtration means that the product of any two weighted covectors vanishes in $\nu^*_{(i)}\cW$. This leaves $\nu^*_{(i)}\cW$ and $\nu^*\cW$ with only an interesting vector bundle structure. The grading on $\nu^*\cW$ refers to the fibers being graded vector spaces, not the weaker graded structure in the sense of Grabowski-Rotkiewicz as it exists on the dual $\nu\cW$.

Before giving more properties of the conormal bundle, we want to point out what its defining equivalence relation comes down to in local coordinates:

\begin{lemma}\label{lem:conormal-coords}
Let $\cW$ be a weighting with weighted coordinates $\chi=(x_1,...,x_m):U\to\R^m$ and fix $i\geq1$. Then elements $g_1$ and $g_2$ of $j^\infty\cW_i$ represent the same conormal in $\nu^*_{(i)}\cW$ if and only if
$$x_{(j)}(g_1)=x_{(j)}(g_2)$$
for all $j=1,...,m$ with $w_j=0$ and
$$\partial x_{(\alpha)}(g_1)= \partial x_{(\alpha)}(g_2)$$
for all multi-indices $\alpha$ with $\alpha w = i$ and such that $\alpha_j=0$ whenever $w_j=0$.
\end{lemma}

\begin{proof}
The first condition is simply checking that $g_1$ and $g_2$ both lie in the same fiber of $T_{(\infty)}M$, where all $x_{(j)}$ with $w_j\neq0$ must vanish anyway since elements of $j^\infty\cW_i$ must lie in the fibers over $\supp\cW$. Let us assume this is true for $p=p_{(\infty)}(g_1)=p_{(\infty)}(g_2)$ and consider the second condition.

Since both the pairing $\langle\cdot,\cdot,\rangle_i$ in the first entry and derivatives are linear, it is sufficient to consider the case where it holds for some $f\in\cW_i$ that $g_1=j^\infty f(p)$ and $g_2=j^\infty(0)(p)$. Using the definition of $\nu^*_{(i)}\cW$ and $\partial x_{(\alpha)}$, this leaves us with the following claim:
\begin{align*}
&f^{(i)}(q)=0 \text{ for all } q\in \cW|_p
\\
&\quad\iff\quad
\partial x^\alpha f(p) = 0 \text{ for all } \alpha\text{ with }\alpha w=i\text{ and }w_j=0\implies \alpha_j=0.
\end{align*}

Since $f\in\cW_i$, we can use Proposition~\ref{prop:filtration-standard-weighting} to locally write
\begin{equation}\label{eq:local-f-expr}
f(x)=\sum_{\alpha w\geq i} f_\alpha(x) \, x^\alpha
\end{equation}
for some smooth functions $f_\alpha\in C^\infty(U)$ with only finitely many $f_\alpha$ not vanishing. We can furthermore assume that every $f_\alpha$ such that there is a $j$ with $w_j=0$ and $\alpha_j>0$ vanishes, as these can be absorbed into terms with smaller $\alpha$.

Using this substitution, we can compute that the right-hand side of our claim is equivalent to
\begin{equation}\label{eq:lhs-to-show}
f_\alpha(p)\neq 0 \text{ for all } \alpha\text{ with }\alpha w=i\text{ and }w_j=0\implies \alpha_j=0.
\end{equation}

For the left-hand side, we can represent each $q$ by a curve $\gamma$ with $\gamma(0)=p$ and
\begin{equation}\label{eq:local-gamma-expr}
x_k(\gamma(t)) = \sum\limits_{j=w_k}^r t^j \x{k}{j}(q)
\end{equation}
according to Equation~\eqref{eq:curve-TrM} and the fact that $q\in\cW$ implies $\x{i}{j}(q)=0$ when $j<w_i$. This allows us to compute
\begin{align*}
f^{(i)}(q) &= \left.\frac{d^i}{dt^i}\right|_{t=0} f\circ\gamma(t) \\
&= \left.\frac{d^i}{dt^i}\right|_{t=0}  \sum_{\alpha w\geq i} f_\alpha(\gamma(t)) \;\prod_{k=1}^m x_k(\gamma(t))^{\alpha_k}\\
&= \left.\frac{d^i}{dt^i}\right|_{t=0} \sum_{\alpha w\geq i} f_\alpha(\gamma(t)) \;\prod_{k=1}^m \left(\sum_{j=w_k}^r t^j\, x_k^{(j)}(q) \right)^{\alpha_k} \\
&=i!\,\sum_{\alpha w= i} f_\alpha(p) \,  \prod_{k=1}^m \left(\sum_{j=w_k}^r x_k^{(j)}(q) \right)^{\alpha_k}.
\end{align*}
The first equality is just the definition of the $i$-th lift. The second and third equalities follow by substituting our expressions for $f$ and $\gamma$ from Eqs.\eqref{eq:local-f-expr} and~\eqref{eq:local-gamma-expr}. We can then collect the $t^j$ factors and observe that they start at order $\alpha w = i$. The fourth equality now follows by the product rule and all higher-order terms vanishing upon differentiating. We see that the last expression vanishes for \textit{all} $q\in\cW|_p$ if and only if Equation~\eqref{eq:lhs-to-show} holds, and can conclude the proof.
\end{proof}

\begin{npar}[Properties of the weighted conormal bundle]\ 
\begin{enumerate}
\item By Lemma~\ref{lem:conormal-coords}, a finite number of coordinates on $T_{(\infty)}M$ descend to coordinates on each $\nu^*_{(i)}\cW$. In particular, the degree $i$ weighted conormal bundle is a vector bundle of finite rank over $\supp\cW$ spanned by equivalence classes of jets of monomials
$x^\alpha$ where $\alpha w=i$ and $\alpha_j=0$ whenever $w_j=0$.

\item If $\cW_N$ is the trivial weighting over a closed submanifold $N$, then the usual conormal bundle $\nu^*N$ is canonically isomorphic to $\nu^*_{(1)}\cW_N$. In particular, the latter is of rank $\codim N$. We will see in Example~\ref{ex:conormal} that when the weighting is not trivial, the rank of $\nu^*_{(1)}\cW_N$ drops. The covectors of some coordinate directions move to higher degrees of the conormal bundle, making it necessary to consider the full conormal bundle as a graded object instead of just its first degree.

\item The pairing descends to well-defined maps
$$\langle\cdot,\cdot\rangle_i: \nu_{(i)}^*\cW\times_{\supp\cW}\nu\cW\to \Lambda_i.$$
In fact, their image is exactly the kernel $\ker(\Lambda_i\to\Lambda_{i-1})\simeq\R$, i.e. those elements of $\Lambda_i$ such that only the highest order derivative does not vanish. These maps are linear in the first argument and equivariant. They can further be assembled to a pairing
$$\langle\cdot,\cdot\rangle_\infty: \nu^*\cW\times_{\supp\cW}\nu\cW\to \Lambda_\infty$$
between the weighted conormal and normal bundles that is linear in the first argument and equivariant.

\item Indeed the weighted conormal and normal bundles are constructed by the smallest equivalence relations making the pairings 
$$\langle\cdot,\cdot\rangle_i: \nu_{(i)}^*\cW\times_{\supp\cW}\nu\cW\to \Lambda_i$$
non-degenerate in the following sense:
Whenever $\langle a,n \rangle_i = \langle b,n\rangle_i$ for all $n\in\nu\cW|_p$ for some $a,b\in\nu^*_{(i)}\cW|_p$, then already $a=b$. Similarly, whenever $\langle c,a \rangle_i = \langle c,b\rangle_i$ for all $c\in\nu^*_{(i)}\cW|_p$ for some $a,b\in\nu\cW|_p$, then also $a=b$.

\item As a consequence, one can show that
$$\nu^*_{(i)}\cW|_p=\operatorname{Hom}_i(\nu\cW|_p, \R),$$
where the right-hand side denotes the space of $i$-homogeneous maps with respect to the canonical $\R$-action on $\nu\cW$ and the multiplication on $\R$. As such,
$$\nu^*\cW|_p=\operatorname{Hom}(\nu\cW|_p,\Lambda_\infty)$$
where the right-hand side denotes equivariant maps intertwining the action on $\nu\cW$ with the graded action
$$s\cdot(\lambda_1, \lambda_2, \lambda_3, ...) = (s\lambda_1, s^2\lambda_2, s^3\lambda_3, ...)$$
of $s\in\R$ on $\Lambda_\infty$. Dually, we can write
$$\nu\cW|_p=\operatorname{Hom}(\nu^*\cW|_p,\Lambda_\infty),$$
where the right-hand side denotes algebra homomorphisms for the commutative multiplication on $\Lambda_\infty$.

\item As a quotient of a subset of $T_{(\infty)}M$, functoriality of the weighted conormal bundle fails in a similar way. Concretely, for every weighted morphism $\phi:(M,\cW)\to(\widetilde{M},\widetilde{\cW})$ and $p\in{M}$ we only get pointwise maps
\begin{align*}
\nu^*\widetilde{M}|_{\phi(p)} &\to \nu^*\cW|_{p}\\
[j^\infty f(\phi(p))] &\mapsto [j^\infty(f\circ\phi)(p)].
\end{align*}
If $\phi$ is a diffeomorphism, these do assemble into an global induced map $\nu^*\widetilde{M} \to \nu^*\cW$.\qedhere
\end{enumerate}
\end{npar}

\begin{example}\label{ex:conormal}
Consider again the standard weighting $\cE:=\cE^{(0,1,2)}$ on $\R^3$ with coordinates $(x,y,z)$. 
Compared to the dual filtration in Example~\ref{ex:dual-filtration}, taking the quotient has the effect of discarding all monomials with either a power of $x$ or whose weighted degree is strictly larger than required and we get the following basis:
\begin{align*}
\nu^*_{(1)}\cE &= \langle y\rangle_{C^\infty(\supp\cW)},\\
\nu^*_{(2)}\cE &= \langle y^2, z\rangle_{C^\infty(\supp\cW)},\\
\nu^*_{(3)}\cE &= \langle y^3, yz\rangle_{C^\infty(\supp\cW)},\\
\nu^*_{(4)}\cE &= \langle y^4, y^2z, z^2\rangle_{C^\infty(\supp\cW)}.
\end{align*}
Here, the angled brackets $\langle \cdot\rangle_{C^\infty(\supp\cW)}$ just denote the fiberwise linear span.
\end{example}

\begin{remark}
It may be tempting to define the weighted conormal bundle as the graded bundle of the filtration $j^\infty\cW_\bullet$. According to Equation~\eqref{eq:inclusion-equiv}, our definition of $\nu^*_{(i)}\cW$ is a quotient of $j^\infty\cW_i/j^\infty\cW_{i+1}$. This means that the pairing with $\nu\cW$ would still be well-defined, but not non-degenerate. Furthermore, the alternative definition would yield bundles of infinite rank at every degree: In the context of Example~\ref{ex:conormal}, the basis of e.g. $\nu^*_{(2)}\cE$ would additionally need to contain equivalence classes induced by monomials $x^ny^2$ and $x^nz$ for any $n\geq1$.
\end{remark}

\startSubchaption{Linear data of a weighting}\label{ssec:linear-data}

A weighting carries some data that is purely linear, taking the form of filtrations of the normal and conormal bundles of its support:

\begin{definition}
Let $\cW$ be a weighting over a manifold $M$. By applying the differential to the structure filtration, we obtain the \textbf{conormal filtration}
$$
\nu^* \supp\cW = d\cW_1\supseteq d\cW_2\supseteq \ldots \supseteq 0_{\supp\cW}$$
of $\cW$ by setting
$$d\cW_i:=\{d_pf\;|\; f\in\cW_i, p\in\supp\cW\}.$$

Dually, we define the \textbf{normal filtration}
$$ \nu\supp\cW\supseteq \ldots \supseteq \nu_{(-1)}\cW \supseteq \nu_{(0)}\cW = 0_{\supp_\cW} $$
by taking for all $p\in\supp\cW$ the annihilators
$$\nu_{(-i)}\cW|_p:=\{[v]\in\nu(\supp\cW)|_{p}\;|\; \omega(v)=0\text{ for all }\omega\in d\cW_{i+1}|_p \}.$$
\end{definition}

While the conormal filtration is only implicit in \cite{LM23}, we find it enlightening to introduce these on equal footing.

\begin{npar}[Properties of the (co)normal filtration]\label{props-conormal-filtration} Let $\cW$ be of order $r$.\nopagebreak
\begin{enumerate}\nopagebreak
\item In local coordinates $\{x_j\}_{j=1,...,m}$ of $M$ adapted to $\cW$ with weights $w_j$, Proposition~\ref{prop:filtration-standard-weighting} immediately gives 
\begin{align*}
d\cW_i|_p &= \langle d_px_j\;|\; w_j\geq i\rangle,\\
\nu_{(-i)}\cW|_p &= \langle [\partial x_j(p)] \;|\; 0<w_j\leq i\rangle.
\end{align*}
Whenever $i$ is larger than the order of $\cW$, we have that $d\cW_i=0$ and $\nu_{(-i)}\cW=\nu\supp\cW$. We also see that the set of weights and their multiplicity is already encoded in the linear data.
\item It follows that the filtrations $d\cW_i$ and $\nu_{(-i)}\cW$ are by smooth vector subbbundles of finite rank.
\item The space of sections of the conormal filtration is given by
\begin{equation*}
\Gamma(d\cW_i)\simeq \frac{\cW_i}{\cW_1^2\cap\cW_i}.
\qedhere
\end{equation*}
\end{enumerate}
\end{npar}

Note that the linear data is not in general sufficient to recover the weighting:

\begin{example}
Consider the standard weighting $\cE:=\cE^{(1,3)}$ on $\R^2$ with standard coordinates $(x,y)$ and the diffeomorphism $(u,v):\R^2\to\R^2$ defined by
$$u(x,y):= x,\qquad v(x,y):=y+x^2.$$
Let $\cW$ be the pull-back of $\cE$ along $(u,v)$. Thus $(u,v)$ are weighted coordinates for $\cW$, which is the subset of $T^{(\infty)}\R^2$ cut out by $$u^{(0)}=v^{(0)}=v^{(1)}=v^{(2)}=0.$$
A quick calculation shows this is equivalent to
$$x^{(0)}=y^{(0)}=y^{(1)}=y^{(2)}+2\left(x^{(1)}\right)^2=0$$
in the coordinates $(x,y)$, and we see in particular that $\cE\neq\cW$. However, $du=dx$ and $dv=dy$ holds on $\{0\}=\supp\cE=\supp\cW$, such that the conormal filtrations of $\cE$ and $\cW$ must be the same.
\end{example}

It is no coincidence that our counterexample is of order three:

\begin{proposition}
Let $\cW$ and $\widetilde{\cW}$ be weightings of order $r\leq2$ over the same manifold with the same conormal (or equivalently normal) filtration. Then $\cW=\widetilde{\cW}$.
\end{proposition}

\begin{proof}
If $r=0$, then both $\cW$ and $\widetilde{\cW}$ must be the final weighting. If $r=1$, weightings are already determined by their support, so they must match. So we really only need to work for the case $r=2$: By taking linear combinations we can find adapted coordinates $x_j$ of $\cW$ and $\tilde{x}_j$ of $\widetilde{\cW}$ close to the support such that
$$x_j(p)=\tilde{x}_j(p) \quad\text{ and }\quad d_px_j=d_p\tilde{x}_j \quad\text{ for all }p\in\supp\cW$$
and the weight sequence $w$ matches both. Thus for $i\leq1$, any $x_j^{(i)}$ vanishes exactly when $\tilde{x}_j^{(i)}$ does. Since these are exactly the conditions that appear in the local description of $\cW$ and $\widetilde{\cW}$ when $r\leq2$, we must have $\cW=\widetilde{\cW}$.
\end{proof}

Using only the linear data of the weighting, we can define linearized versions of the weighted conormal and normal bundles:

\begin{definition}
Let $\cW$ be a weighting over a manifold $M$. The \textbf{linearized weighted conormal bundle} is defined as the associated graded bundle of the conormal filtration, i.e.
$$\nu^*_\text{lin}\cW:=\gr d\cW_\bullet = \bigoplus_{i=1}^\infty d\cW_i/d\cW_{i+1},$$
and the \textbf{linearized weighted normal bundle} is the associated graded bundle over the normal filtration, i.e.
$$\nu_\text{lin}\cW:=\gr \nu_{(\bullet)}\cW = \bigoplus_{i=1}^\infty \nu_{(-i)}\cW/\nu_{(-i+1)}\cW.$$
\end{definition}

\begin{npar}[Properties of the linearized bundles]\label{props-linearized-bundles}\ 
\begin{enumerate}
\item In local coordinates $\{x_j\}_{j=1,...,m}$ of $M$ adapted to $\cW$ with weights $w_j$ we have for every $p\in\supp\cW$ and $i\geq1$ that
\begin{align*}
(d\cW_i/d\cW_{i+1})|_p &= \langle [d_px_j]\;|\; w_j=i\rangle,\\
(\nu_{(-i)}\cW/\nu_{(-i+1)}\cW)|_p &= \langle [\partial x_j(p)]\;|\; w_j=i\rangle.
\end{align*}
We see in particular that $\nu^*_\text{lin}\cW\simeq\nu^*\supp\cW$ and $\nu_\text{lin}\cW\simeq\nu\supp\cW$, but not necessarily canonically so.
\item For a pair $\cW\subseteq\widetilde{\cW}$ of weightings, one can easily check using $\supp\cW\subseteq\supp\widetilde{\cW}$ and $\cW_i\supseteq\widetilde{\cW}_i$ that there are well-defined linear canonical maps
$$\nu_\text{lin}\cW\to\nu_\text{lin}\widetilde{\cW} \qquad\text{and}\qquad \nu^*_\text{lin}\widetilde{\cW}|_{\supp\cW}\to \nu^*_\text{lin}\cW$$
that are compatible with the grading.
\item Loizides and Meinrenken show in \cite[][Prop. 4.4]{LM23} that $\nu_\text{lin}\cW$ is the linearization of $\nu\cW$ as a graded bundle in the sense of Grabowski-Rotkiewicz. In particular, $\nu\cW\simeq\nu_\text{lin}\cW$ always holds as graded bundles, but not canonically.

\item Similarly, $\nu^*_\text{lin}\cW$ can be seen as the linearization of $\nu^*\cW$ in the sense that canonically $$\nu^*_\text{lin}\cW\simeq\frac{\nu^*\cW}{j^\infty(\cW_1^2)}.$$
For dimensional reasons, neither $\nu^*\cW$ nor $\nu^*_{(1)}\cW$ are in general isomorphic to $\nu^*_\text{lin}\cW$.

\item If the weighting $\cW_N$ is the trivial weighting over $N$, then canonically both $$\nu\cW_N\simeq\nu_\text{lin}\cW_N\simeq\nu N$$ and
\begin{equation*}
\nu^*_{(1)}\cW_N\simeq\nu^*_\text{lin}\cW_N\simeq\nu^* N.\qedhere
\end{equation*}
\end{enumerate}
\end{npar}

While linear data is not sufficient to check whether two weightings are the same in general, it \textit{does} suffice by dimensional reasons if we additionally know that one weighting is contained in the other:

\begin{lemma}\label{lem:linear-check}
Let $\cW\subseteq\widetilde{\cW}$ be a pair of connected weightings and assume that one (equivalently both) of the canonical maps
$$\nu_\text{lin}\cW\to\nu_\text{lin}\widetilde{\cW} \qquad\text{and}\qquad \nu^*_\text{lin}\widetilde{\cW}|_{\supp\cW}\to \nu^*_\text{lin}\cW$$
is an isomorphism. Then it must hold that $\cW=\widetilde{\cW}$.
\end{lemma}

\begin{proof}
Since the canonical maps are compatible with the gradings, we know that the ranks of each degree of the linearized bundles must match. By \ref{props-linearized-bundles}(1), this means that $\cW$ and $\widetilde{\cW}$ can be equipped with the same weight sequence and in particular have the same codimension in $T^{(\infty)}M$. If they are connected, this and $\cW\subseteq\widetilde{\cW}$ means they must coincide.
\end{proof}

\startSubchaption{Compatibility between weightings}\label{ssec:compatibility}

We ultimately want to consider blow-ups of families of weightings over the same manifold. In order for these to be well-behaved, we will need to assume compatibility conditions.

\begin{definition}\label{def:weighting-compat}
Let $\{\cW_\alpha\}_{\alpha=1,...,d}$ be a family of weightings on $M$. We say these weightings
\begin{enumerate}
    \item \textbf{intersect cleanly} if they intersect cleanly when regarded as submanifolds of $T^{(r)}M$, where $r$ is at least as large as the order of each $\cW_\alpha$,
    \item are \textbf{aligned} if around every point $p\in M$ there exists a chart $\chi:U\to\R^m$ that is simultaneously adapted to every weighting $\cW_\alpha$,
    \item are \textbf{uniformly aligned} if the weights $w_{\alpha,i}$ under these simultaneously aligned charts further satisfy for every $\alpha,\beta\in A$ and $i\in\{1,...,m\}$ that $w_{\alpha,i}=w_{\beta,i}$ whenever both sides are non-zero.
\end{enumerate}
\end{definition}

To avoid ambiguity, clean intersection of submanifolds is defined in Appendix~\ref{app:manifold-intersections}. There, we also define stronger notions (intersection like coordinate subspaces and transversal intersection) and recall a criterion for clean intersections in terms of vanishing ideals, all of which will be relevant in this \subchaption.

Note that uniform alignment corresponds to existence of \textit{uniform} weights $w_i$ such that $w_{\alpha,i}\in\{w_i,0\}$. It is easy to see that uniformly aligned weightings are always aligned, and aligned weightings always intersect cleanly.

The intersection properties of weightings are often related to those of their supports and vice versa:

\begin{lemma}\ \label{lem:intersection-props-supports}
\begin{enumerate}
\item The supports of any aligned family of weightings must intersect like coordinate subspaces.
\item The supports of any cleanly intersecting family of weightings must themselves intersect cleanly.
\item Every family of weightings with transversely intersecting supports is uniformly aligned.
\end{enumerate}
\end{lemma}

The third statement can be understood as a consequence of weightings with transversely intersecting supports divvying up the normal bundle between themselves.

\begin{proof}\ 

\textit{Regarding (1)}:
This follows immediately from the simultaneous local model of an aligned family of weightings.

\textit{Regarding (2)}:
Let $r$ be the order at which $\cW^{(r)}_\alpha:= p^{(r)}(\cW_\alpha)$ are cleanly intersecting submanifolds. Since the $\cW_\alpha$ are weightings, $\cW^{(r)}_\alpha$ must contain the zero section in $T^{(r)}M$ over $\supp\cW_\alpha$. For any vector $v\in \cap_{\alpha\in A} T_p\supp\cW_\alpha$, we can thus find a preimage $v_0$ in $\cap_{\alpha\in A}T_{0_p}\cW^{(r)}_\alpha$ under $p^{(r)}$ by lifting along the zero section. Due to cleanness of the intersection, this preimage lies in $T_{0_p}\left(\cap_{\alpha\in A}\cW^{(r)}_\alpha\right)$. Its image $v$ therefore lies in $T_{p}\left(\cap_{\alpha\in A}\supp\cW_\alpha\right)$.

\textit{Regarding (3)}:
Let $p$ be a point in the intersection of some of the supports. Consider adapted local coordinates around $p$ for each of the weightings. The coordinate directions with non-vanishing weights must all have linearly independent derivatives at $p$ by assumption, and thus can be extended to a full set of coordinates around $p$. Since each weighting is cut out by lifts of coordinate components with non-vanishing weights, these new coordinates are aligned with each of the weightings. The alignment is uniform, since any component appears in the vanishing ideal of at most one of the supports.
\end{proof}

Let us consider a number of examples and counterexamples:

\begin{examples}\ \label{ex:weighting-compatibility}

\begin{enumerate}
    \item The initial and final weighting are uniformly aligned with any other weighting.
    \item Any family of standard weightings on $\R^m$ is always aligned, and uniformly aligned exactly if there is no direction that gets assigned two different, non-zero weights.

    \item As a consequence of Lemma~\ref{lem:intersection-props-supports}(2), trivial weightings along any collections that do not intersect cleanly give a counterexample to cleanly intersecting weightings.

    \item As a consequence of Lemma~\ref{lem:intersection-props-supports}(1), considering trivial weightings along $n+1$ different lines in $\R^n$ that intersect in a single point (where $n>1$) yields an example of cleanly intersecting weightings that are not aligned.

    \item In Def.~\ref{def:filtered-weighted-building-set}, we will see that the collection of diagonals in the $s$-fold power of a filtered manifold supports a collection of weightings that is not, in general, aligned. Consider as a simple case the configuration space $M^3$ of points on the interval $M=(0,1)$ and observe that the three diagonals $$\Delta_i=\{(x_1,x_2,x_3)\;|\;x_k=x_l \text{ for }k,l\neq i\} \qquad\text{for }i\in\{1,2,3\}$$ intersect cleanly in the \textit{small} diagonal $\{(x,x,x)\;|\; x\in M\}$, but not like coordinate subspaces. The point of the notion of \textit{building set} from \Chaption~\ref{sec:building-sets} will be that blow-ups are still well-behaved when sufficiently many subsets of the collection are aligned, such as any pair of diagonals in the case of $M^3$.

    \item\label{ex:not-aligned} In the previous examples, alignment fails already since the supports do not intersect like coordinate subspaces. But it can also fail due to incompatibility of the higher-order data:
    Consider the weightings $\cW_A,\cW_B$ and $\cW_C$ on $\R^3$ from Figure~\ref{fig:not-aligned}. Observe that each of these has a single coordinate direction with a maximal weight 2. Assume there were adapted coordinates around the point $p$ in the support of all three weightings. For each $I\in\{A,B,C\}$, the maximal weight forces the existence of one coordinate direction whose differential at $p$ is proportional to $d_p f$ for any $f\in\cW_{I,2}$. Since the weightings are standard weightings, we can pick for $f$ in particular the functions $x_3, x_2$ and $x_2+x_3$, respectively. Thus $\cW_A$ forces the existence of a coordinate component with differential proportional to $d_p x_3$, $\cW_B$ forces one with $d_p x_2$ and $\cW_3$ forces $d_p(x_2+x_3)$. None of these covectors are proportional to each other and thus must arise from three different coordinate components. But since they are linearly dependent, this is impossible! This also provides a counterexample to the converse directions of Lemma~\ref{lem:intersection-props-supports}(1) and~(2).
    \item One can easily pick standard weightings such as $\{\cE^{(0,1,1)},\cE^{(1,1,0)}\}$ to obtain examples of uniformly aligned families that do not have transversely intersecting supports. These are counterexamples to the converse direction of Lemma~\ref{lem:intersection-props-supports}(3).\qedhere
\end{enumerate}
\end{examples}

\begin{figure}
    \centering
    \Large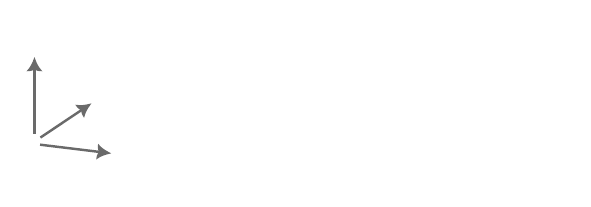
    \caption{We define three weightings supported over the coordinate subspaces $\textcolor{purple}{A}, \textcolor{ForestGreen}{B}$ and $\textcolor{olive}{C=\{0\}}$ as follows: $\textcolor{purple}{\cW_A:=\cE^{(0,0,2)}}$, $\textcolor{ForestGreen}{\cW_B:=\cE^{(0,2,1)}}$ and $\textcolor{olive}{\cW_C:=R\cdot\cE^{(1,1,2)}}$, where $R$ is a rotation of 45 degrees around the $x_1$-axis. These weightings intersect cleanly and have supports that intersect like coordinate subspaces, but are not aligned.}
    \label{fig:not-aligned}
\end{figure}

\begin{remark}
For comparison, the \textbf{multi-weightings} from ~\cite{LM23} are equivalent to a finite family $\{\cW_\alpha\}_{\alpha=1..k}$ of weightings over a manifold $M$ satisfying the following: There is a tableau $\{w_{\alpha j}\}_{\alpha=1..k,j=1..m}$ of non-negative weights such that around every $p\in M$ there is a neighbourhood $U$, an open $V\subseteq\R^m$ and charts $\chi:U\to V$ such that $\chi_*(\cW_\alpha|_U)=\cE^{(w_{\alpha1},...,w_{\alpha m})}|_V$. In other words, the table of weights gives a simultaneous local model.

This is subtly stronger than just alignment of all $\cW_\alpha$: For example, consider the submanifolds
$$A=\{(x,y)\in\R^2\,|\, x=2\}\cup S^1 \quad\text{and}\quad B=\{(x,y)\in\R^2\,|\,y=2\}\cup S^1.$$ The trivial weightings $\cW_A$ and $\cW_B$ are aligned but are not a multi-weighting as the local models near the points $(1,0)$ and $(2,2)$ in $A\cap B$ do not match. In higher dimensions, this phenomenon also appears when assuming all supports are connected.

It should be noted that when submanifolds intersect cleanly in the language of Loizides and Meinrenken, we would instead say that they intersect like coordinate subspaces, compare Definition~\ref{def:transversal}.
\end{remark}

We conclude this \subchaption{} by observing that uniform alignment of nested weightings can be tested using only linear data. 

\begin{lemma}\label{lem:linear-unif-alignment-check}
Let $\cW\subseteq\widetilde{\cW}$ be weightings such that the canonical map $$\nu^*_{\mathrm{lin}}\widetilde{\cW}|_{\supp\cW}\to \nu^*_{\mathrm{lin}}\cW$$ is an injection. Then $\cW$ and $\widetilde{\cW}$ are uniformly aligned. 

Concretely, consider adapted coordinates $\{\tilde x_a\}_{a=1,...,m}$ of $\widetilde{\cW}$ with weights $\tilde w_a$. Close to any point $p\in\supp{\cW}$, the set $\{\tilde x_a\;|\; \tilde w_a\geq 1\}$ can be extended to adapted coordinates of ${\cW}$.
\end{lemma}

\begin{proof}
Take adapted coordinates $\{\tilde x_a\}$ of $\widetilde{\cW}$ as in the statement. Without loss of generality, reorder the $\tilde x_a$ such that the components with non-vanishing weights appear first. After reordering it thus holds for $k:=\codim\supp\widetilde{\cW}$ that $\tilde w_a\geq 1$ if and only if $a\leq k$. Now pick any coordinates $\{x_a\}_{a=1,...,m}$ around $p\in\supp{\cW}$ that are adapted to ${\cW}$ with weights ${w}_a$.

We know that the equivalence classes of $\{d_p x_a\}_{a=1,...,m}$ forms a basis of $\nu^*_{\mathrm{lin}}\cW$ and by the injectivity assumption that the equivalence classes of $\{d_p \tilde x_a\}_{a=1,...,k}$ are a linearly independent set. We can thus reorder the $x_a$ such that
$$\hat x_a:=\begin{cases}\tilde x_a & \text{if }a\leq k,\\ x_a & \text{else} \end{cases}$$
for $1\leq a\leq m$ still yields a basis. Moreover, the differentials must also be linearly independent before taking equivalence classes, which means that the $\hat x_a$ form a coordinate chart close to $p$. It is clear that these coordinates are adapted to $\widetilde{\cW}$ since we preserved all components that are normal to $\supp\widetilde{\cW}$. We will show that the coordinates are also adapted to $\cW$.

As an intermediary step, consider the local weighting $\hat\cW$ that is defined by declaring the $\hat x_a$ to be adapted coordinates with weights
$$\hat w_a:=\begin{cases}w_a & \text{if }a\leq k,\\ \tilde w_a & \text{else}. \end{cases}$$
We know that $\cW\subseteq\hat\cW$ since $\hat\cW$ is constructed by some coordinates adapted to $\cW$, and some coordinates adapted to $\widetilde{\cW}\supseteq\cW$. By Lemma~\ref{lem:linear-check}, it suffices to check that the canonical map
$$\phi: \nu^*_\text{lin}\hat\cW|_{\supp\cW}\to\nu^*\cW$$
is an isomorphism\footnote{We can assume connectedness when applying the Lemma since we argue locally.} to conclude the proof by noting $\cW=\hat\cW$. However, we chose the $\hat x_a$ precisely such that the equivalence classes of their differentials form a basis in $\nu^*\cW$, so $\phi$ must indeed be an isomorphism.
\end{proof}

\startSubchaption{Intersection of weightings}\label{ssec:intersection}

It is natural to wonder whether the intersection of weightings can be a weighting itself. Let us first build intuition with examples:

\begin{examples}\ \label{ex:weighting-intersections}

\begin{enumerate}
\item Any intersection of standard weightings on $\R^n$ is easily seen to be the weighting where each coordinate direction gets assigned the maximum of all original weights along this direction.

\item However, in general we must concede that the intersection can entirely fail to be a weighting: Consider the standard weighting $\cE^{(0,0,1)}$ on $\R^3$ as well as $\Phi^*\cE^{(0,0,1)}$ where $\Phi$ is the diffeomorphism $$\Phi(x_1,x_2,x_3):=(x_1,x_2,x_3+x_1 x_2).$$ A brief calculation shows that the projection of their intersection to $\R^3$ gives the subset $\{x_3=x_1x_2=0\}$. This would have to be the support of the intersected weighting but is not even a submanifold.

\item\label{ex:clean-not-weighting} This is even true when the weightings intersect cleanly: Construct two weightings $\cW_C$ and $\cW_D$ supported over two points $p_1, p_2 \in \R^2$ as follows: $\cW_C$ assigns the weight 1 to the horizontal $x_1$-axis at both points, and the weight 2 to the vertical $x_2$-axis. $\cW_D$ makes the same assignments over $p_1$, but flips the weights over $p_2$. Their intersection is clean and given by the union of a weighting for weights $(1,2)$ over $p_1$ and a weighting for $(2,2)$ over $p_2$. Since the weight sequences do not match, this set is not globally a weighting.

\item\label{ex:weighting-not-clean} Conversely, if the intersection is a weighting, it does not need to be clean: Consider on $\R^2$ the weightings $\cE^{(0,2)}$ and $\Psi^*\cE^{(0,1)}$, where $\Psi$ is the diffeomorphism $$\Psi(x_1,x_2):=(x_1,x_2+x_1^2).$$ These are cut out of $T^{(\infty)}\R^2$ by the equations
    $$x_2^{(0)}=x_2^{(1)}=0 \qquad\text{and}\qquad x_2^{(0)}+\left(x_1^{(0)} \right)^2=0,$$ respectively. Intersecting these evidently gives the weighting $\cE^{(1,2)}$ with the vanishing ideal $$I_{\cE^{(1,2)}}=\langle x_2^{(0)}, x_2^{(1)}, x_1^{(0)} \rangle.$$ By Lemma~\ref{lem:clean-vanishing-ideal}, this would need to be included in $$I_{\cE^{(0,2)}}+I_{\Psi^*\cE^{(0,1)}}=\langle x_2^{(0)}, x_2^{(1)}, \left(x_1^{(0)}\right)^2 \rangle$$ to be a clean intersection, but $x_1^{(0)}$ is clearly missing.
    \qedhere
\end{enumerate}
\end{examples}

The problem of mismatching weight sequences from Example~\ref{ex:weighting-intersections}~(\ref{ex:clean-not-weighting}) turns out to be the only pitfall to keep in mind, as a clean intersection of weightings always yields a weighting locally:

\begin{proposition}\label{prop:clean-locally-weighting}
Let $\{\cW_\alpha\}_{\alpha\in A}$ be a family of weightings on $M$ that intersect cleanly in $\cW_0=\bigcap_{\alpha=1}^d\cW_\alpha$. Then $\cW_0$ is locally a weighting, i.e. around every point of $M$ there exists a neighbourhood $U$ such that $\cW_0\cap T^{(\infty)}U$ is a weighting. 

Furthermore, there exist coordinates $\chi=(x_1,...,x_m)$ adapted to $\cW_0$ such that each component $x_i$ is also part of an adapted coordinate system for some $\cW_\alpha$ with the same weight assigned to it. If $\cW_0$ is connected, then it is also globally a weighting.
\end{proposition}

\begin{proof}
By Theorem~\ref{thm:new-criterion}, the vanishing ideals of the $\cW_\alpha$ are given by $$I_{\cW_\alpha}=\langle f^{(i)}\;|\; f\in C^\infty(M) \text{ with } f^{(j)}|_{\cW_\alpha}=0\;\forall j\leq i \rangle.$$ By Lemma~\ref{lem:clean-vanishing-ideal}, we know that
\begin{equation}\label{eq:intersection-ideal}
    I_{\cW_0} = \sum_{\alpha\in A} I_{\cW_\alpha} = \langle f^{(i)}\;|\; f\in C^\infty(M), \alpha\in A \text{ with } f^{(j)}|_{\cW_\alpha}=0\;\forall j\leq i \rangle.
\end{equation}
Again by Theorem~\ref{thm:new-criterion}, to conclude that $\cW_0$ is locally a weighting we only need to check that
$$I_{\cW_0}\subseteq \langle f^{(i)}\;|\; f\in C^\infty(M) \text{ with } f^{(j)}|_{\cW_0}=0\;\forall j\leq i \rangle$$
as the other inclusion is automatic. But this follows immediately from Equation~\eqref{eq:intersection-ideal} because $\cW_0\subseteq\cW_\alpha$.

Note that since the vanishing ideal of each $I_{\cW_\alpha}$ locally is generated by lifts of components of adapted coordinates, Lemma~\ref{lem:clean-vanishing-ideal} implies that $\cW_0$ is locally cut out by a collection of these. The proof of Theorem~\ref{thm:new-criterion} proceeds by picking a subset to build adapted coordinates for $\cW_0$, so the "Furthermore.." part of the proposition follows immediately. Finally, if $\cW_0$ is connected, then Theorem~\ref{thm:new-criterion} gives globally consistent local models, i.e. with the same weight sequences.
\end{proof}

We can check the cleanliness of an intersection that results in a weighting using their structure filtrations:

\begin{lemma}\label{lem:reg-intersection}
Let $\{\cW_\alpha\}_{\alpha=1,...,d}$ be a family of weightings on $M$ that intersect in a weighting $\cW_0=\bigcap_{\alpha=1}^d\cW_\alpha$. Then the intersection is clean if and only if
\begin{equation}\label{eq:reg-int}
\cW_{0,i} = \sum_{|\vec{i}|= i}\prod_{\alpha=1}^d \cW_{\alpha,i_\alpha}
\end{equation}
holds for the corresponding structure filtrations.
\end{lemma}

\begin{proof}\ \\
\textit{Assuming clean intersection}, we need to show the '$\subseteq$' inclusion of Equation~\eqref{eq:reg-int}, as the other inclusion is automatic by $\cW_0\subseteq\cW_\alpha$, thus $\cW_{\alpha,i}\subseteq\cW_{0,i}$, and multiplicativity of $\cW_{0,i}$. It suffices to argue locally, so we can assume according to Prop.~\ref{prop:filtration-standard-weighting} that $\cW_{0,i}$ is generated by products $x_1^{\beta_1}...x_m^{\beta_m}$ of some local coordinates with $\beta_1 w_1+...+\beta_m w_m\geq i$. Proposition~\ref{prop:clean-locally-weighting} allows us to choose coordinates adapted to $\cW_0$ such that each $x_j$ is part of coordinates adapted to some $\cW_\alpha$ with the same weight. In particular, $x_j^{\beta_j}\in\cW_{\alpha,(\beta_j w_j)}$, and the generators $x_1^{\beta_1}...x_m^{\beta_m}$ are included in the right-hand side of Equation~\eqref{eq:reg-int}. It does not matter if $\beta_1 w_1+...+\beta_m w_m$ is strictly larger than $i$ since $\cW_{\alpha,i_\alpha}\subseteq\cW_{\alpha,i_\alpha-1}$.

\textit{Assuming Equation~\eqref{eq:reg-int}}, due to Lemma~\ref{lem:clean-vanishing-ideal} we only need to show that $I_{\cW_0}\subseteq \sum_\alpha I_{\cW_\alpha}$ since the other inclusion is automatic. By Theorem~\ref{thm:new-criterion}, we only need to show this inclusion for the generators $f^{(i)}\in I_{\cW_0}$ where $f\in\cW_{0,i+1}$. According to Equation~\ref{eq:reg-int}, we can decompose $f$ into $$f=\sum_{|\vec{i}|\geq i+1}\prod_{\alpha=1}^d f_{\alpha,i_\alpha} \qquad\text{with } f_{\alpha,i_\alpha}\in\cW_{\alpha,i_\alpha}.$$
Per the Leibniz rule for lifts, we obtain $$f^{(i)}=\sum_{|\vec{i}|\geq i+1}\sum_{|\vec{j}|= i}\prod_{\alpha=1}^d f_{\alpha,i_\alpha}^{(j_\alpha)}.$$ In each term of this sum, there must be at least one $\alpha$ such that $j_\alpha<i_\alpha$. For this $\alpha,$ it holds that $f_{\alpha,i_\alpha}^{(j_\alpha)}\in I_{\cW_\alpha}$ by Theorem~\ref{thm:new-criterion}, and we are done.
\end{proof}

For an intersection weighting $\cW_0=\bigcap_{\alpha=1}^d\cW_\alpha$ the identity can be seen as a weighted morphism $(M,\cW_0)\to(M,\cW_\alpha)$ for all $\alpha$. If the intersection is clean, we can check whether a weighted normal vector of $\cW_0$ vanishes by looking at the normal vectors it induces along each submanifold:

\begin{lemma}\label{lem:reg-intersection-normals}
Let $\{\cW_\alpha\}_{i=\alpha...d}$ be a collection of weightings on $M$ that intersect cleanly in a weighting $\cW_0$. Then any $[n]\in\nu_{\cW_0}M$ with $\nu_{\cW_\alpha,\cW_0}([n])=0$ for all $\alpha=1...d$ must vanish.
\end{lemma}

\begin{proof}
Given an $n\in \cW_0$ as in the statement of the Lemma, we want to show that for all $i\in\N_0$ and $f\in\cW_{0,i}$ it follows that $f^{(i)}(n)=0$. By Lemma~\ref{lem:reg-intersection}, such an $f$ can be decomposed as
$$f=\sum\limits_{|\vec{i}|=i}\prod\limits_{\alpha=1}^d f_{\alpha,i_\alpha} \quad\text{where}\quad f_{\alpha,i_\alpha}\in \cW_{\alpha,i_\alpha}.$$

We compute by inserting this decomposition and using the Leibniz rule for lifts:
\begin{align*}
f^{(i)}(n) &= \sum\limits_{|\vec{i}|=i} \left( \prod\limits_{\alpha=1}^d f_{\alpha,i_\alpha} \right)^{(i)}(n) \\
&= \sum\limits_{|\vec{i}|=i} \sum\limits_{|\vec{j}|=i} \prod\limits_{\alpha=1}^d f_{\alpha,i_\alpha}^{(j_\alpha)}(n) \\
&= \sum\limits_{|\vec{j}|=i} \prod\limits_{\alpha=1}^d f_{\alpha,j_\alpha}^{(j_\alpha)}(n) \\
&= 0.
\end{align*}
The third equality follows since any term where $\vec{j}\neq\vec{i}$ must contain a factor of the form $f^{(j_\alpha)}_{\alpha, i_\alpha}(n)$ with $j_\alpha<i_\alpha$, and such factors vanish since $n\in\cW_0\subseteq\cW_\alpha$ and $f_{\alpha,i_\alpha}\in \cW_{\alpha,i_\alpha}$ (compare with Equation~\eqref{eq:weighting-from-filtration}).
The final equality results from the assumption that $\nu_{\cW_\alpha,\cW_0}([n])=0$, i.e. that for all $g\in \cW_{\alpha,j}$ we have $g^{(j)}(n)=0.$ It follows by setting $j=j_\alpha$ and $g=f_{\alpha,j_\alpha}$.
\end{proof}

The weightings $\cE^{(0,2)}$ and $\Psi^*\cE^{(0,1)}$ from Ex.~\ref{ex:weighting-intersections}~(\ref{ex:weighting-not-clean}) that fail to intersect cleanly also yield a counterexample to the conclusion of this Lemma: The non-zero normal vector $[t\mapsto (t,0)]$ in $\nu_{\cE^{(1,2)}}$ induces vanishing vectors along both of the original weightings.

%% file: weighting.pdf_tex
\begingroup%
  \makeatletter%
  \providecommand\color[2][]{%
    \errmessage{(Inkscape) Color is used for the text in Inkscape, but the package 'color.sty' is not loaded}%
    \renewcommand\color[2][]{}%
  }%
  \providecommand\transparent[1]{%
    \errmessage{(Inkscape) Transparency is used (non-zero) for the text in Inkscape, but the package 'transparent.sty' is not loaded}%
    \renewcommand\transparent[1]{}%
  }%
  \providecommand\rotatebox[2]{#2}%
  \newcommand*\fsize{\dimexpr\f@size pt\relax}%
  \newcommand*\lineheight[1]{\fontsize{\fsize}{#1\fsize}\selectfont}%
  \ifx\svgwidth\undefined%
    \setlength{\unitlength}{283.46456693bp}%
    \ifx\svgscale\undefined%
      \relax%
    \else%
      \setlength{\unitlength}{\unitlength * \real{\svgscale}}%
    \fi%
  \else%
    \setlength{\unitlength}{\svgwidth}%
  \fi%
  \global\let\svgwidth\undefined%
  \global\let\svgscale\undefined%
  \makeatother%
  \begin{picture}(1,0.55)%
    \lineheight{1}%
    \setlength\tabcolsep{0pt}%
    \put(0,0){\includegraphics[width=\unitlength,page=1]{weighting.pdf}}%
    \put(0.77459999,0.25948706){\color[rgb]{0.8745098,0.47058824,0.07058824}\makebox(0,0)[t]{\lineheight{0.80000001}\smash{\begin{tabular}[t]{c}$N$\end{tabular}}}}%
    \put(0.92942566,0.19959797){\color[rgb]{0,0,0}\makebox(0,0)[t]{\lineheight{0.80000001}\smash{\begin{tabular}[t]{c}0\end{tabular}}}}%
    \put(0.31745495,0.3710075){\color[rgb]{0,0,0}\makebox(0,0)[t]{\lineheight{0.80000001}\smash{\begin{tabular}[t]{c}1\end{tabular}}}}%
    \put(0.15681987,0.44877159){\color[rgb]{0,0,0}\makebox(0,0)[t]{\lineheight{0.80000001}\smash{\begin{tabular}[t]{c}2\end{tabular}}}}%
    \put(0,0){\includegraphics[width=\unitlength,page=2]{weighting.pdf}}%
  \end{picture}%
\endgroup%

%% file: induced-example.pdf_tex
\begingroup%
  \makeatletter%
  \providecommand\color[2][]{%
    \errmessage{(Inkscape) Color is used for the text in Inkscape, but the package 'color.sty' is not loaded}%
    \renewcommand\color[2][]{}%
  }%
  \providecommand\transparent[1]{%
    \errmessage{(Inkscape) Transparency is used (non-zero) for the text in Inkscape, but the package 'transparent.sty' is not loaded}%
    \renewcommand\transparent[1]{}%
  }%
  \providecommand\rotatebox[2]{#2}%
  \newcommand*\fsize{\dimexpr\f@size pt\relax}%
  \newcommand*\lineheight[1]{\fontsize{\fsize}{#1\fsize}\selectfont}%
  \ifx\svgwidth\undefined%
    \setlength{\unitlength}{283.46456693bp}%
    \ifx\svgscale\undefined%
      \relax%
    \else%
      \setlength{\unitlength}{\unitlength * \real{\svgscale}}%
    \fi%
  \else%
    \setlength{\unitlength}{\svgwidth}%
  \fi%
  \global\let\svgwidth\undefined%
  \global\let\svgscale\undefined%
  \makeatother%
  \begin{picture}(1,0.35)%
    \lineheight{1}%
    \setlength\tabcolsep{0pt}%
    \put(0,0){\includegraphics[width=\unitlength,page=1]{induced-example.pdf}}%
    \put(0.38439881,0.09291341){\color[rgb]{0,0,0}\makebox(0,0)[t]{\lineheight{0.80000001}\smash{\begin{tabular}[t]{c}1\end{tabular}}}}%
    \put(0.32533434,0.18325145){\color[rgb]{0,0,0}\makebox(0,0)[t]{\lineheight{0.80000001}\smash{\begin{tabular}[t]{c}1\end{tabular}}}}%
    \put(0.1887518,0.29176818){\color[rgb]{0,0,0}\makebox(0,0)[t]{\lineheight{0.80000001}\smash{\begin{tabular}[t]{c}2\end{tabular}}}}%
    \put(0,0){\includegraphics[width=\unitlength,page=2]{induced-example.pdf}}%
    \put(0.89624151,0.09305706){\color[rgb]{0,0,0}\makebox(0,0)[t]{\lineheight{0.80000001}\smash{\begin{tabular}[t]{c}0\end{tabular}}}}%
    \put(0.83717703,0.18339508){\color[rgb]{0,0,0}\makebox(0,0)[t]{\lineheight{0.80000001}\smash{\begin{tabular}[t]{c}1\end{tabular}}}}%
    \put(0.70059445,0.29191182){\color[rgb]{0,0,0}\makebox(0,0)[t]{\lineheight{0.80000001}\smash{\begin{tabular}[t]{c}1\end{tabular}}}}%
    \put(0,0){\includegraphics[width=\unitlength,page=3]{induced-example.pdf}}%
    \put(0.50042622,0.20798548){\color[rgb]{0,0,0}\makebox(0,0)[t]{\lineheight{0.80000001}\smash{\begin{tabular}[t]{c}{\normalsize$\nu_{\cE^{(1,1,2)},\cE^{(0,1,1)}}$}\end{tabular}}}}%
  \end{picture}%
\endgroup%

%% file: not-aligned.pdf_tex
\begingroup%
  \makeatletter%
  \providecommand\color[2][]{%
    \errmessage{(Inkscape) Color is used for the text in Inkscape, but the package 'color.sty' is not loaded}%
    \renewcommand\color[2][]{}%
  }%
  \providecommand\transparent[1]{%
    \errmessage{(Inkscape) Transparency is used (non-zero) for the text in Inkscape, but the package 'transparent.sty' is not loaded}%
    \renewcommand\transparent[1]{}%
  }%
  \providecommand\rotatebox[2]{#2}%
  \newcommand*\fsize{\dimexpr\f@size pt\relax}%
  \newcommand*\lineheight[1]{\fontsize{\fsize}{#1\fsize}\selectfont}%
  \ifx\svgwidth\undefined%
    \setlength{\unitlength}{283.46456693bp}%
    \ifx\svgscale\undefined%
      \relax%
    \else%
      \setlength{\unitlength}{\unitlength * \real{\svgscale}}%
    \fi%
  \else%
    \setlength{\unitlength}{\svgwidth}%
  \fi%
  \global\let\svgwidth\undefined%
  \global\let\svgscale\undefined%
  \makeatother%
  \begin{picture}(1,0.35)%
    \lineheight{1}%
    \setlength\tabcolsep{0pt}%
    \put(0,0){\includegraphics[width=\unitlength,page=1]{not-aligned.pdf}}%
    \put(0.19116018,0.07661604){\color[rgb]{0,0,0}\makebox(0,0)[lt]{\lineheight{0.80000001}\smash{\begin{tabular}[t]{l}$x_1$\end{tabular}}}}%
    \put(0.37449814,0.20989769){\color[rgb]{0,0,0}\makebox(0,0)[lt]{\lineheight{0.80000001}\smash{\begin{tabular}[t]{l}2\end{tabular}}}}%
    \put(0.6563681,0.22658634){\color[rgb]{0,0,0}\makebox(0,0)[lt]{\lineheight{0.80000001}\smash{\begin{tabular}[t]{l}2\end{tabular}}}}%
    \put(0.66693268,0.13922985){\color[rgb]{0,0,0}\makebox(0,0)[lt]{\lineheight{0.80000001}\smash{\begin{tabular}[t]{l}1\end{tabular}}}}%
    \put(0.76794212,0.22064497){\color[rgb]{0,0,0}\makebox(0,0)[lt]{\lineheight{0.80000001}\smash{\begin{tabular}[t]{l}1\end{tabular}}}}%
    \put(0.81701164,0.1940551){\color[rgb]{0,0,0}\makebox(0,0)[lt]{\lineheight{0.80000001}\smash{\begin{tabular}[t]{l}2\end{tabular}}}}%
    \put(0.68694489,0.19016524){\color[rgb]{0,0,0}\makebox(0,0)[lt]{\lineheight{0.80000001}\smash{\begin{tabular}[t]{l}1\end{tabular}}}}%
    \put(0.3457003,0.06567846){\color[rgb]{0.7372549,0.24705882,0.16862745}\makebox(0,0)[lt]{\lineheight{0.80000001}\smash{\begin{tabular}[t]{l}$A$\end{tabular}}}}%
    \put(0.72179466,0.13507143){\color[rgb]{0.1372549,0.6,0.21960784}\makebox(0,0)[lt]{\lineheight{0.80000001}\smash{\begin{tabular}[t]{l}$B$\end{tabular}}}}%
    \put(0.57459599,0.14099392){\color[rgb]{0.70980392,0.56470588,0.05882353}\makebox(0,0)[lt]{\lineheight{0.80000001}\smash{\begin{tabular}[t]{l}$C$\end{tabular}}}}%
    \put(0.16013432,0.16426495){\color[rgb]{0,0,0}\makebox(0,0)[lt]{\lineheight{0.80000001}\smash{\begin{tabular}[t]{l}$x_2$\end{tabular}}}}%
    \put(0.042053,0.26110895){\color[rgb]{0,0,0}\makebox(0,0)[lt]{\lineheight{0.80000001}\smash{\begin{tabular}[t]{l}$x_3$\end{tabular}}}}%
    \put(0,0){\includegraphics[width=\unitlength,page=2]{not-aligned.pdf}}%
  \end{picture}%
\endgroup%

%% file: weighted-submanifolds.tex
\startSubchaption{Spherical blow-ups}

Blowing up a submanifold means to replace it with an \textit{exceptional divisor} that encodes along which (possibly weighted) normal direction one can move onto it. On the level of sets, we can thus take the following definitions:

\begin{definition}\label{def:blow-up-weighting}
Let $(M,\cW)$ be a weighted manifold. The \textbf{(spherical) blow-up of $\cW$} is given by
$$\Bl_\cW(M):= M\setminus\supp\cW \quad\sqcup\quad \bS\nu\cW,$$
where the \textbf{weighted unit normal bundle} is defined as the quotient $$\bS\nu\cW:=(\nu\cW \setminus \supp\cW)/\R_{>0}.$$ In this context, $\bS\nu\cW$ is also referred to as the \textbf{exceptional divisor} of $\Bl_\cW(M)$. If $\cW$ is the trivial weighting along a closed submanifold $N$, we may denote the blow-up as $\Bl_N(M)$.

A blow-up comes equipped with the \textbf{blow-down map} $$b_\cW:\Bl_\cW(M)\to M$$
that sends any point in $M\setminus\supp\cW$ to itself and unit normal vectors in $\bS\nu\cW$ to the point in $\supp\cW$ at which they are based.
\end{definition}

\noindent\begin{minipage}{\textwidth}
\begin{minipage}{0.65\textwidth}
\begin{npar}[Visualization]
As Figure~\ref{fig:helix-weighted} shows, the weighted blow-up can be visualized analogously to the unweighted blow-up of the origin of the disk from Figure~\ref{fig:helix-normal}: We consider the surface traced out vertically by the weighted radial rays coming out of the origin and identify any two points that are an integer number of rotations apart. The red central spine now stands in one-to-one correspondence with the exceptional divisor $\bS\nu\cW$ since slicing the surface at a fixed height gives a weighted ray. The points away from the spine are just a copy of $M\setminus\supp\cW$, and the blow-down map is the vertical projection. 

This is a good topological picture to understand the behaviour of limits. In contrast to the unweighted case, it does not yield a smooth structure as the depicted surface is not smoothly embedded in $\R^3$ even before taking the quotient: While the tangent spaces along the spine are almost everywhere spanned by the vertical axis and the horizontal axis that has weight one, there are two singular points where the surface 'flips' around the spine. We will assign a smooth structure that makes this blow-up diffeomorphic to an annulus, but this structure cannot be faithfully depicted with an embedded surface such that the blow-down map is the vertical projection.
\end{npar}
\end{minipage}\hfill%
\begin{minipage}{0.4\textwidth}
    \centering
    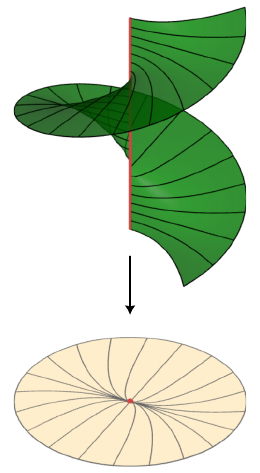
    \captionof{figure}{Visualization of the blow-up of the origin of a disk equipped with weights $(1,2)$.}
    \label{fig:helix-weighted}
\end{minipage}
\end{minipage}

The reader in a rush may skip ahead directly to Definition~\ref{def:smooth-blow-up} of the smooth charts on the blow-up. However, since these involve somewhat convoluted formulas, we first sketch a path to motivate our choices:
In Paragraph~\ref{par:smooth-weighted-normal-local}, we describe how we can equip $\bS\nu\cW$ with a smooth structure in the local case, i.e. the case where $M=\R^m$ and $\cW$ is a standard weighting. In Paragraph~\ref{smooth-structure-local-case}, we argue that this smooth structure on the exceptional divisor can be extended to a smooth structure on the blow-up. It in particular provides expressions for coordinates that serve as motivation for our definition of charts for general $M$ and $\cW$.

\begin{npar}[Smooth structure on $\bS\nu\cW$ in the local case]\label{par:smooth-weighted-normal-local}
The local model of a weighted manifold is a standard weighting $\cW=\cE^{(w_1, ..., w_m)}$ on Euclidean space $M=\R^m$. The support of this weighting is $$N=\{(x_1, ..., x_m) \;|\; x_i=0 \text{ for all $i$ with } w_i\neq 0\},$$
and the weighted normal bundle can be viewed as points $n=\left(\x{1}{w_1}, ..., \x{m}{w_m}\right)$ in $\R^m$ interpreted as curved arrows $\{\lambda\cdot n\,|\, \lambda\in[0,1]\}$ traced out by the action from Equation~\eqref{eq:action-normal}. The unit normal bundle is given by quotienting out the action of $\R_{>0}$ on non-zero graded vectors, i.e. $$\bS\nu\cW=(\R^m\setminus N)/\R_{>0}.$$
It can be visualized as a cylinder around $N$. To equip it with charts we can imitate projective coordinates and consider for every sign $s\in\{\pm1\}$ and direction $x_h$ normal to $N$ (or, equivalently, $w_h>0$) the hyperplane $\{x_h=s\}$. On the open subset $U_{hs}\subseteq \bS\nu\cW$ that consists of orbits intersecting this hyperplane, we use the remaining $\{x_i\}_{i\neq h}$ at the intersection point as coordinates:
    \begin{equation}\label{coords-local-unit-normal}
    \R^{m-1} \ni (y_1, ..., y_{h-1}, y_{h+1}, ... y_m) \leftrightarrow
        [y_1, ..., s, ..., y_m] \in \bS\nu\cW
    \end{equation}
Fig.~\ref{fig:coords} depicts this situation for the case $m=3$ with weights $(0,1,2)$ and the hyperplane $x_3=1$.
\end{npar}

\begin{figure}
    \centering
    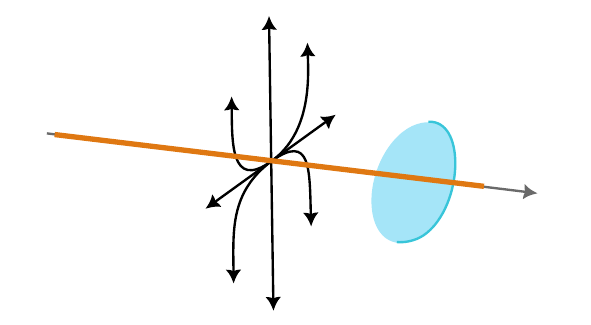
    \caption{The weighted unit normal bundle $\bS\nu_\cW M$ for the standard weighting $\cW=\cE^{(0,1,2)}$ on $\R^m$. By identifying the unique intersection points with the weighted rays, we can use the green hyperplane $\{x_3=1\}$ to parametrize all points on the upper half of the cylinder.}
    \label{fig:coords}
\end{figure}

\begin{npar}[Smooth structure on $\Bl_\cW(M)$ in the local case]\label{smooth-structure-local-case}
We want to extend the parametrization in Equation~\eqref{coords-local-unit-normal} from the exceptional divisor to the full blow-up $\Bl_\cW\R^m$. To do this, we add back in a parameter $y_h\in[0,\infty)$ that pushes away from the exceptional divisor according to the action~\eqref{eq:action-normal}  whenever it is non-zero:
    \begin{equation}\label{coords-local}
    \R^m_{h} \ni (y_1, ..., y_m) \leftrightarrow
    \begin{cases}
        y_h\cdot(y_1, ..., s, ..., y_m) &\text{for } y_h>0,\\
        [y_1, ..., s, ..., y_m] &\text{for } y_h=0,
    \end{cases}
    \end{equation}
    where the $s$ is always inserted at the $h$-th position and $\R^m_{h}$ is the half-space $\R^{h-1}\times[0,\infty)\times\R^{m-h}$.
When $y_h=0$, the equivalence class in the bottom expression yields an element of $\bS\nu\cW$, and otherwise $y_h\cdot(y_1, ..., s, ..., y_m)$ gives a point in $\R^m\setminus N$.

Let us invert Equation~\eqref{coords-local} to obtain explicit expressions for the components $y_i$ of the charts. When given an element $(x_1,...,x_m)$ of $\R^m\setminus N$, then by definition of the action we must find $(y_1,...,y_m)$ with $y_h>0$ such that
$$
x_i = y_h^{w_i}\cdot y_i \qquad\text{for $i\neq h$ and}\qquad x_h = y_h^{w_h}\cdot s.
$$
Solving for the $y_i$ we get
\begin{equation}\label{eq:coord-rel1}
    y_i = x_i \cdot (s\,x_h)^{-w_i/w_h} \qquad\text{for $i\neq h$ and}\qquad y_h = (s\, x_h)^{1/w_h}.
\end{equation}
These coordinates only make sense when we assume that $s\,x_h>0$, since we may be taking a root.
When given a weighted unit normal vector $n=[v_1,...,v_n]\in\bS\nu\cW$ instead, we must assume that $s\,v_h>0$ is non-zero in order to represent $n$ as $$n=\left[v_1\cdot (s\,v_h)^{-w_1/w_h}, \;\ldots, \;s, \;\ldots,\; v_m\cdot (s\,v_h)^{-w_m/w_h}\right].$$
Now Equation~\eqref{coords-local} allows us to read off
\begin{equation}\label{eq:coord-rel2}
    y_i = v_i \cdot (s\,v_h)^{-w_i/w_h} \qquad\text{for $i\neq h$ and}\qquad y_h = 0.
\end{equation}
The open conditions $s\,x_h>0$ and $s\,v_h>0$ determine the domain on which we can define our charts. Luckily, any point in $\Bl_\cW\R^m=\R^m\setminus N\cup\bS\nu\cW$ lies in one of the domains for some choice of $h$ and $s$.
\end{npar}

Equations~\eqref{eq:coord-rel1} and~\eqref{eq:coord-rel2} coax us into defining the following coordinates in the general case:

\begin{definition}\label{def:smooth-blow-up}
    Let $\cW$ be a weighting over a manifold $M$ of dimension $m$ and $\chi=(x_1,...,x_m):U\to\R^m$ an adapted coordinate chart on $U\subseteq M$ for the weight sequence $w_1, ..., w_m$ such that $U$ intersects $\supp\cW$. Fix a sign $s\in\{\pm1\}$ and an $h\in\{1, ..., m\}$ such that $w_h> 0$, i.e. such that the $x_h$-axis is normal to $\supp\cW$.
    
    Define a subset $U_{hs}\subseteq \Bl_\cW U$ as $$
    U_{hs}=\{ p\in U\setminus\supp\cW \;|\; s\,x_h(p)> 0 \} \;\cup\; \{ [n]\in\bS\nu\cW|_U \;|\; s\,x_h^{(w_h)}(n)> 0  \} $$
    as well as the half-space model
    $$\R^m_h:=\{(y_1, ..., y_m)\in\R^m \;|\; y_h\geq 0 \}.$$
    We can now define the \textbf{induced chart}
    $$\chi_{hs}=\left(x_{1:hs}, ..., x_{m:hs}\right):U_{hs}\to\R^m_h$$
    on $\Bl_\cW(M)$ by setting 
    $$ x_{i:hs}(p):=
    \begin{cases}
    x_i(p)\cdot(s\,x_h(p))^{-w_i/w_h}&\text{for } p\in U\setminus\supp\cW,\\
    \x{i}{w_i}(n)\cdot\left(s\,\x{h}{w_h}(n)\right)^{-w_i/w_h}&\text{for } p=[n]\in \bS\nu\cW|_U
    \end{cases}
    $$
    for $i\neq h$, as well as
    $$ x_{h:hs}(p):=  
    \begin{cases}
    \left(s\,x_h(p)\right)^{1/w_h}&\text{for } p\in U\setminus\supp\cW,\\
    0 &\text{for } p=[n]\in \bS\nu\cW|_U.
    \end{cases}
    $$
    If the weighting is not clear from context, we may add a superscript $(\cW)$ to our notation and write $\chi^{(\cW)}_{hs}$, $x^{(\cW)}_{i:hs}$ and $U^{(\cW)}_{hs}$.
\end{definition}

The notation for $x_{i:hs}$ is inspired by the fact that it is given by weighted quotients of either $x_i$ and $s\,x_h$ or $\x{i}{w_i}$ and $s\,\x{h}{w_h}$. Note that when $w_i=0,$ we have $x_{i:hs}=x_i\circ b_\cW$. The component $x_{h:hs}$ is special in that it can be thought of as the distance to the exceptional divisor.

\begin{npar}[Blow-down map in coordinates]\label{par:blow-down-coords}
It is often convenient to have the analogous expression to Equation~\eqref{coords-local} for the inverse of the charts $\chi_{hs}$, namely
    \begin{equation}\label{eq:inv-coords-local}
    \left(\chi_{hs}\right)^{-1}(y_1, ..., y_m) =
    \begin{cases}
        \chi^{-1}(y_h\cdot(y_1, ..., s, ..., y_m)) &\text{for } y_h>0,\\
        [\left(\chi^{(\cW)}\right)^{-1}(y_1, ..., s, ..., y_m)] &\text{for } y_h=0.
    \end{cases}
    \end{equation}
For example, it allows us to read off the local expression
\begin{equation}\label{eq:local-blow-down}
\chi\circ b_\cW \circ \chi_{hs}^{-1} (y_1, ..., y_m) = y_h \cdot (y_1, ..., s, ..., y_m)
\end{equation}
for the blow-down map, using that the projection to the base of the weighted normal bundle $\nu\cW$ is given by the action of zero under Equation~\eqref{eq:action-normal}.
\end{npar}

These charts provide a smooth structure:

\begin{theorem}\label{thm:weighted-blow-up-atlas}
For every weighted manifold $(M,\cW)$, the blow-up $\Bl_\cW(M)$ can be equipped with the structure of a smooth manifold with boundary such that the blow-down map $b_\cW:\Bl_\cW(M)\to M$ is smooth and proper and the boundary of $\Bl_\cW(M)$ is the exceptional divisor.

Concretely, we can build a smooth atlas by taking all charts as in Definition~\ref{def:smooth-blow-up} together with the charts of $M$ whose domains do not intersect $\supp\cW$.
\end{theorem}

Loizides and Meinrenken prove a similar statement by first considering the smooth structure on the \textit{weighted deformation space} and then taking a quotient. For the sake of completeness and in preparation for later constructions, we prove smoothness of transitions functions directly in Appendix~\ref{app:proof-weighted-submanifolds}.

One should think of the blow-down map as an essential part of the data of a blow-up:

\begin{example}\label{ex:multiple}
The weightings $\cE^{(1,1)}$ and $\cE^{(2,2)}$ on $\R^2$ can easily be seen to have diffeomorphic blow-ups, namely the annulus with one open and one closed boundary. However, one cannot find a diffeomorphism that also intertwines the blow-down maps. Indeed the comparison map
$$\Phi:\Bl_{\cE^{(2,2)}}(\R^2)\to\Bl_{\cE^{(1,1)}}(\R^2)$$
that is defined as the unique continuous extension of the identity on $\R^2\setminus\{0\}$ intertwines the blow-down maps, but is just a smooth bijection. See also Example~\ref{example:induced-maps}(2).
\end{example}

\startSubchaption{Induced map between blow-ups}

Note that the blow-up is not functorial, but we can still get an induced map defined on an open subset:

\begin{definition}
Let $\phi:(M,\cW)\to(\widetilde M,\widetilde \cW)$ be a morphism of weighted manifolds and define the open subset $\Bl_\phi (M)\subseteq \Bl_\cW (M)$ as
$$ \Bl_\phi (M) := M \setminus \phi^{-1}(\supp\widetilde \cW) \cup \{[n]\in \bS\nu_\cW M \;|\; \nu_\phi(n)\neq 0 \}. $$

On this subset we define the induced map $b_\phi: \Bl_\phi (M) \to \Bl_{\widetilde \cW} (M)$ by
\begin{align*}
b_\phi(p) &:= \phi(p) & &\text{for } p\in M\setminus \phi^{-1}(\supp\widetilde \cW), \\
b_\phi([n]) &:= [\nu_\phi(n)] & &\text{for } [n]\in \bS\nu_\cW M \cap \Bl_\phi (M).
\end{align*}
If $\phi$ is the identity map on a fixed manifold $M$, then we write $\Bl_{\widetilde \cW,\cW}(M)$ and $b_{\widetilde \cW,\cW}$ instead of $\Bl_\phi(M)$ and $b_\phi$.
\end{definition}

This induced map between weighted blow-ups is already discussed in Remark~6.3(c) of~\cite{LM23}.

\begin{examples}\ \label{example:induced-maps}
\begin{enumerate}
    \item If we consider again the weightings $\cE^{(1,1,2)}\subseteq\cE^{(0,1,1)}$ from Figure~\ref{fig:induced-example} we see that $\Bl_{\cE^{(0,1,1)},\cE^{(1,1,2)}}(\R^3)$ is missing two kinds of points from the ambient $\Bl_{\cE^{(1,1,2)}}(\R^3)$: On one hand, any $p\in\{(\lambda,0,0)\in\R^3\;|\:\lambda\neq0\}$ on the first axis, and on the other hand any weighted normal vector at the origin that is tangent to the third axis. These are exactly the points that lack data to determine a reasonable image in $\Bl_{\cE^{(0,1,1)}}(\R^3)$.
\item The weightings $\cE^{(2,2)}\subset \cE^{(1,1)}$ from Example~\ref{ex:multiple} provide an exceptional case in that the map induced on weighted normal bundles by the inclusion vanishes entirely. Thus $\Bl_{\cE^{(1,1)},\cE^{(2,2)}}(\R^2)$ is the bulk $\R^2\setminus\{0\}$ and the induced map 
$$\Bl_{\cE^{(2,2)}}(\R^2)\supset \R^2\setminus\{0\}\to \Bl_{\cE^{(1,1)}}(\R^2)$$ is the identity. Interestingly, its extension $\Phi$ from Example~\ref{ex:multiple} gives an isomorphism between weighted normal bundles when restricting to the exceptional divisor, which is decidedly \textit{not} the (vanishing) induced map on the weighted normal bundles.
\qedhere
\end{enumerate}
\end{examples}

Openness of $\Bl_\phi M$ and smoothness of $b_\phi$ are established in Corollary~\ref{cor:domain-induced-open}
and Proposition~\ref{prop:induced-map-smooth} of the Appendix. We also provide a local expression of the induced map in Lemma~\ref{lem:local-rep-induced-map}.

The induced map between blow-ups allows us to understand the relationship between the limit points in different blow-ups of the same sequence away from the exceptional divisors:

\begin{lemma}\label{lem:seq-compat-weightings}
Fix an ambient manifold $M$, two weightings $\cW,\widetilde \cW$ and a sequence $p_n\in M\setminus(\supp\cW\cup \supp{\widetilde \cW})$ such that $p$ and $\tilde p$ are the limits of $p_n$ when viewed as a sequence in $\Bl_\cW(M)$ and $\Bl_{\widetilde \cW}(M)$, respectively. Then $b_{\cW}(p)=b_{\widetilde \cW}(\tilde p)$.

If further $\cW\subseteq\widetilde \cW$ and $p\in\Bl_{\widetilde \cW,\cW}(M)$, then $\tilde p=b_{\widetilde \cW,\cW}(p)$.
\end{lemma}

\begin{proof}
The statements follow by continuity of $b_\cW, b_{\widetilde \cW}$ and $b_{\widetilde \cW,\cW}$ as well as uniqueness of limit points in a manifold topology.
\end{proof}

Note that the limit point $\widetilde p$ does not get determined by $p$ if the limit of the sequence $p_n$ does not lie within $\Bl_{\widetilde \cW,\cW}(M)$.

\startSubchaption{Projective blow-ups}

To conclude our discussion of weighted blow-ups of single submanifolds, we must briefly speak of the \textit{projective} blow-ups that arise from quotienting out what remains of the $\R$-action on the exceptional divisor:

\begin{definition}
Let $(M,\cW)$ be a weighted manifold. The \textbf{projective blow-up of $\cW$} is given by
$$
\PBl_\cW(M):= M\setminus \supp\cW \quad\cup\quad \bP\nu_\cW M,
$$
where the weighted projective normal bundle $\bP\nu_\cW M$ is defined as the quotient $(\nu_\cW\setminus\supp\cW)/\R$ of the full $\R^*$-action. We define the \textbf{blow-down map} $$ pb_\cW:\PBl_\cW(M)\to M $$
and the blow-up $\PBl_N(M)$ of a trivial weighting analogously to Def.~\ref{def:blow-up-weighting}.
\end{definition}

The result of taking this quotient is visualized in Figure~\ref{fig:projective-helix}. This procedure can introduce singularities, but we can still write down orbifold charts analogously to the manifold charts of the spherical blow-up:

\begin{figure}
    \centering
    \Large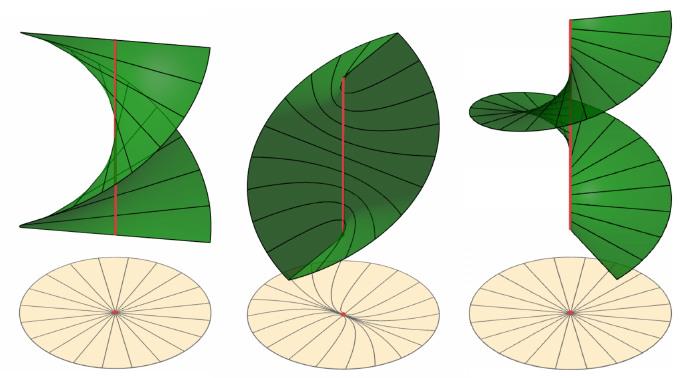
    \caption{Visualization of the projective blow-ups of the origin on the disk when equipped with weights $(1,1), (1,2)$ and $(2,2)$, respectively. For the trivial weights, projectivisation corresponds to gluing the exceptional divisor onto itself, thereby removing one of the boundaries of the annulus and resulting in a Möbius strip. Equivalently, one can immediately consider the surface traced out by the weighted lines through the origin instead of just the one-sided weighted rays. For the weights $(1,2)$, this process topologically results in a disk, with the exceptional divisor homeomorphic to a closed interval. While the edges at the top and bottom of the surface are mostly an artifact of this visualization, there really are orbifold singularities at the endpoints of the exceptional divisor. For the weights $(2,2)$, the fact that all weights are even means that the orbits under the full $\R$-action and the original $\R_{>0}$-action coincide and projectivization has no effect.}
    \label{fig:projective-helix}
\end{figure}

\begin{definition}
    Let $\cW$ be a weighting over a manifold $M$ of dimension $m$ and $\chi=(x_1,...,x_m):U\to\R^m$ an adapted coordinate chart on $U\subseteq M$ for the weight sequence $w_1, ..., w_m$. Fix a sign $s\in\{\pm1\}$ and an $h\in\{1, ..., m\}$ such that $w_h> 0$.
    
    \textit{If $w_h$ is odd,} define the subset $\bP U_{hs}\subseteq \Bl_\cW U$ as $$
    \bP U_{hs}=\{ p\in U\setminus\supp\cW \;|\; s\,x_h(p)\neq 0 \} \;\cup\; \{ [n]\in\bP\nu_\cW U \;|\; s\,x_h^{(w_h)}(n)\neq 0  \}$$
    and let $G_{hs}$ be the trivial group that acts trivially on $\R^m$.
    
    \textit{If $w_h$ is even,} let $$
    \bP U_{hs}=\{ p\in U\setminus\supp\cW \;|\; s\,x_h(p)> 0 \} \;\cup\; \{ [n]\in\bP\nu_\cW U \;|\; s\,x_h^{(w_h)}(n)\neq 0  \}$$
    and define $G_{hs}:=\Z_2$. Let the non-trivial element of $\Z_2$ act on $(y_1, ..., y_m)\in\R^m$ by flipping the sign of $y_h$ and of all $y_i$ with $w_i$ odd.

    We can now define orbifold parametrizations
    $$\phi_{hs}:\R^m\to \bP U_{hs}$$
    by setting$$
        \phi_{hs}(y_1, ..., y_m) := \begin{cases}
        \chi^{-1}(y_h\cdot(y_1, ..., s, ..., y_m)) &\text{for } y_h\neq0,\\
        [\left(\chi^{(\cW)}\right)^{-1}(y_1, ..., s, ..., y_m)] &\text{for } y_h=0.
        \end{cases}
    $$
\end{definition}

It is easy to check that $\phi_{hs}$ maps surjectively into $\bP U_{hs}$. The parity of $w_h$ determines whether $\{y_h\leq0\}\subseteq\R^m$ gets folded onto the image of $\{y_h\geq0\}$ by $\phi_{hs}$, warranting the case distinction in the definition of $\bP U_{hs}$.
Moreover, $\phi_{hs}$ is invariant under the action of $G_{hs}$ on its domain.

The proof that these really are orbifold charts proceeds largely analogously to the smooth structure in the spherical case and is thus omitted here.

\begin{proposition}
Let $(M,\cW)$ be a weighted manifold and consider the projective blow-up $\PBl_\cW(M)$.
\begin{enumerate}
    \item $\PBl_\cW(M)$ has orbifold singularities at exactly those points of the exceptional divisor that are equivalence classes of curves that are contained in a coordinate subspace spanned by directions with even weights. Their isotropy group is $\Z_2$.
    \item If the weights are all odd, then $\PBl_\cW(M)$ is a smooth manifold without boundary.
    \item If the weights are all even, then $\PBl_\cW(M)$ has a canonical structure of a smooth manifold with the exceptional divisor as boundary.
    \item If the weights have mixed parity, then the orbifold singularities of $\PBl_\cW(M)$ cannot be seen as boundary points of a smooth manifold.
\end{enumerate}
\end{proposition}

\begin{proof}
\textit{Regarding (1):} The singular locus consists of the points whose isotropy group is non-trivial. The isotropy group of a point $p\in\PBl_\cW(M)$ can be calculated as the stabilizer of any point $p'\in\phi_{hs}^{-1}(p)$ under the action of $G_{hs}$ for any choice of orbifold chart. If $w_h$ is odd, $G_{hs}$ and thus the isotropy group are both trivial. If $w_h$ is even, the set of points invariant under $G_{hs}$ is given by
$$\{(y_1, ..., y_m)\in\R^m\;|\; y_h=0 \land w_i\text{ odd} \implies y_i = 0\}.$$
This is the set of points with non-trivial stabilizer $G_{hs}=\Z_2$ and thus the singular locus. The condition $y_h=0$ corresponds to points on the exceptional divisor and the remaining conditions describe equivalence classes of curves as in the statement of the proposition.

\textit{Regarding (2):} Assuming all weights are odd, the singular locus is empty according to part (1). The orbifold structure is thus smooth.

\textit{Regarding (3):} Assuming all weights are even, the action of $G_{hs}$ in any chart flips the single component $y_h$ such that taking the quotient $\R^m/G_{hs}$ is equivalent to restricting the parametrization $\phi_{hs}$ to the subspace
    $$\R^m_h:=\{(y_1, ..., y_m)\in\R^m \;|\; y_h\geq 0 \}.$$
This yields a parametrization of a manifold with boundary (namely the spherical blow-up $\Bl_\cW(M)$).

\textit{Regarding (4):} We once again consider local charts. Since we know that points in charts with $w_h$ odd are smooth, we restrict our attention to charts with $w_h$ even. By reordering the components, we can without loss of generality assume that the action of $G_{hs}$ is given by
$$(y_1,..., y_m)\mapsto (y_1, ..., y_{h-1}, -y_h, ..., -y_m),$$
where $h<m$ by assumption of mixed parity weights. The points invariant under this action take the form $(y_1, ..., y_{h-1}, 0, ..., 0)$ and thus form a submanifold of codimension $m-h+1>1$. This gets homeomorphically mapped to subset of $\bP U_{hs}$ of the same codimension. We cannot equip $\bP U_{hs}$ with the structure of a manifold with this subset as boundary, as it would have to have codimension 1.
\end{proof}

%% file: helix-weighted.pdf_tex
\begingroup%
  \makeatletter%
  \providecommand\color[2][]{%
    \errmessage{(Inkscape) Color is used for the text in Inkscape, but the package 'color.sty' is not loaded}%
    \renewcommand\color[2][]{}%
  }%
  \providecommand\transparent[1]{%
    \errmessage{(Inkscape) Transparency is used (non-zero) for the text in Inkscape, but the package 'transparent.sty' is not loaded}%
    \renewcommand\transparent[1]{}%
  }%
  \providecommand\rotatebox[2]{#2}%
  \newcommand*\fsize{\dimexpr\f@size pt\relax}%
  \newcommand*\lineheight[1]{\fontsize{\fsize}{#1\fsize}\selectfont}%
  \ifx\svgwidth\undefined%
    \setlength{\unitlength}{124.72440945bp}%
    \ifx\svgscale\undefined%
      \relax%
    \else%
      \setlength{\unitlength}{\unitlength * \real{\svgscale}}%
    \fi%
  \else%
    \setlength{\unitlength}{\svgwidth}%
  \fi%
  \global\let\svgwidth\undefined%
  \global\let\svgscale\undefined%
  \makeatother%
  \begin{picture}(1,1.81818182)%
    \lineheight{1}%
    \setlength\tabcolsep{0pt}%
    \put(0,0){\includegraphics[width=\unitlength,page=1]{helix-weighted.pdf}}%
    \put(0.47852607,0.7009502){\color[rgb]{0,0,0}\makebox(0,0)[rt]{\lineheight{0.80000001}\smash{\begin{tabular}[t]{r}$b_\cW$\end{tabular}}}}%
    \put(0.05095605,1.68556992){\color[rgb]{0,0,0}\makebox(0,0)[lt]{\lineheight{0.80000001}\smash{\begin{tabular}[t]{l}$\Bl_\cW(M)$\end{tabular}}}}%
    \put(0.09112237,0.46308538){\color[rgb]{0,0,0}\makebox(0,0)[t]{\lineheight{0.80000001}\smash{\begin{tabular}[t]{c}$M$\end{tabular}}}}%
  \end{picture}%
\endgroup%

%% file: coords.pdf_tex
\begingroup%
  \makeatletter%
  \providecommand\color[2][]{%
    \errmessage{(Inkscape) Color is used for the text in Inkscape, but the package 'color.sty' is not loaded}%
    \renewcommand\color[2][]{}%
  }%
  \providecommand\transparent[1]{%
    \errmessage{(Inkscape) Transparency is used (non-zero) for the text in Inkscape, but the package 'transparent.sty' is not loaded}%
    \renewcommand\transparent[1]{}%
  }%
  \providecommand\rotatebox[2]{#2}%
  \newcommand*\fsize{\dimexpr\f@size pt\relax}%
  \newcommand*\lineheight[1]{\fontsize{\fsize}{#1\fsize}\selectfont}%
  \ifx\svgwidth\undefined%
    \setlength{\unitlength}{283.46456693bp}%
    \ifx\svgscale\undefined%
      \relax%
    \else%
      \setlength{\unitlength}{\unitlength * \real{\svgscale}}%
    \fi%
  \else%
    \setlength{\unitlength}{\svgwidth}%
  \fi%
  \global\let\svgwidth\undefined%
  \global\let\svgscale\undefined%
  \makeatother%
  \begin{picture}(1,0.55)%
    \lineheight{1}%
    \setlength\tabcolsep{0pt}%
    \put(0,0){\includegraphics[width=\unitlength,page=1]{coords.pdf}}%
    \put(0.84755829,0.25712742){\color[rgb]{0.8745098,0.47058824,0.07058824}\makebox(0,0)[t]{\lineheight{0.80000001}\smash{\begin{tabular}[t]{c}$N$\end{tabular}}}}%
    \put(0.57436977,0.09626488){\color[rgb]{0.05490196,0.52156863,0.6627451}\makebox(0,0)[lt]{\lineheight{0.80000001}\smash{\begin{tabular}[t]{l}$\bS\nu_\cW M$\end{tabular}}}}%
    \put(0.78291763,0.43220155){\color[rgb]{0.04705882,0.58431373,0.27058824}\makebox(0,0)[t]{\lineheight{0.80000001}\smash{\begin{tabular}[t]{c}$x_3=1$\end{tabular}}}}%
    \put(0.92942566,0.19959797){\color[rgb]{0,0,0}\makebox(0,0)[t]{\lineheight{0.80000001}\smash{\begin{tabular}[t]{c}0\end{tabular}}}}%
    \put(0.15681987,0.44877159){\color[rgb]{0,0,0}\makebox(0,0)[t]{\lineheight{0.80000001}\smash{\begin{tabular}[t]{c}2\end{tabular}}}}%
    \put(0,0){\includegraphics[width=\unitlength,page=2]{coords.pdf}}%
    \put(0.32210243,0.40961946){\color[rgb]{0,0,0}\makebox(0,0)[t]{\lineheight{0.80000001}\smash{\begin{tabular}[t]{c}1\end{tabular}}}}%
    \put(0,0){\includegraphics[width=\unitlength,page=3]{coords.pdf}}%
  \end{picture}%
\endgroup%

%% file: projective-helix.pdf_tex
\begingroup%
  \makeatletter%
  \providecommand\color[2][]{%
    \errmessage{(Inkscape) Color is used for the text in Inkscape, but the package 'color.sty' is not loaded}%
    \renewcommand\color[2][]{}%
  }%
  \providecommand\transparent[1]{%
    \errmessage{(Inkscape) Transparency is used (non-zero) for the text in Inkscape, but the package 'transparent.sty' is not loaded}%
    \renewcommand\transparent[1]{}%
  }%
  \providecommand\rotatebox[2]{#2}%
  \newcommand*\fsize{\dimexpr\f@size pt\relax}%
  \newcommand*\lineheight[1]{\fontsize{\fsize}{#1\fsize}\selectfont}%
  \ifx\svgwidth\undefined%
    \setlength{\unitlength}{325.98425197bp}%
    \ifx\svgscale\undefined%
      \relax%
    \else%
      \setlength{\unitlength}{\unitlength * \real{\svgscale}}%
    \fi%
  \else%
    \setlength{\unitlength}{\svgwidth}%
  \fi%
  \global\let\svgwidth\undefined%
  \global\let\svgscale\undefined%
  \makeatother%
  \begin{picture}(1,0.55652174)%
    \lineheight{1}%
    \setlength\tabcolsep{0pt}%
    \put(0,0){\includegraphics[width=\unitlength,page=1]{projective-helix.pdf}}%
  \end{picture}%
\endgroup%

%% file: building-sets.tex
There are two approaches to generalize blow-ups to collections of submanifolds: The \textit{graph} blow-up takes the closure of the bulk of the manifold embedded into the product of all individual blow-ups (see our Definitions~\ref{def:blow-up-building-set} and~\ref{def:blow-up-weighted-building-set}). The \textit{iterated} blow-up proceeds instead by blowing up smaller submanifolds first and lifting the remaining submanifolds at each step. Both approaches were already used in Fulton and MacPherson's seminal paper~\cite{FM94} about compactified configuration spaces. They coincide in this context, provided the iterated blow-ups are performed in a sensible order. An equivalence between graph and iterated blow-up for a different specific construction in the non-weighted smooth setting can also be found in~\cite{AMN21}.

While we expect the same to hold true with reasonable assumptions in the weighted setting, we will focus only on the graph blow up. This is for two reasons: First, as we have concrete applications in mind, we want to obtain concrete coordinates instead of an iterative construction. Second, this allows us to avoid having to define weightings and their blow-ups in the category of manifolds with corners, which would necessarily be produced in intermediate steps\footnote{Weighted blow-ups for manifolds with corners are discussed in~\cite{Beh21} in the language of Melrose.}.

One insight of~\cite{FM94} is that it may be necessary to also blow up some, but not necessarily all of the intersections of the submanifolds under consideration in order to obtain a smooth result.
We call such a choice a \textit{building set} $\cG$ in order to differentiate it from the full \textit{arrangement} $\Arr_\cG$ that it generates by taking intersections. This notion first appeared in~\cite{CP95} and was developed to further generality in \cite{MP98} and finally \cite{Li09}.

In this \chaption, we hope to provide a self-contained pedagogic introduction to these ideas. Moreover, the papers above deal only with projective blow-ups in an algebro-geometric setting\footnote{The special case of the Fulton-MacPherson compactification was adapted to the smooth setting in~\cite{AS94}, replacing projective with spherical blow-ups. Similarly, \cite{Ga03} can be seen to provide a translation of~\cite{MP98}.}, while we adapt definitions to discuss both spherical and projective blow-ups in the smooth category. The central tools this \chaption{} provides are a characterization of nests in Proposition~\ref{prop:char-nest} and the construction of adapted coordinates to nests of separated building sets in Proposition~\ref{prop:nest-coords}.
While the former is in the same spirit as similar characterizations in the settings of \cite{CP95,MP98,Li09}, we believe they are new as stated here.

In \Subchaption~\ref{ssec:building-sets}, we start by giving basic definitions and examples of building sets and arrangements. \Subchaption~\ref{ssec:factors} discusses factors, flags and nests as essential algebraic notions needed to discuss the blow-up. We conclude in \Subchaption~\ref{ssec:blow-ups} with a qualitative discussion of non-weighted blow-ups in preparation for the more general weighted constructions of the next \chaption.

\startSubchaption{Building sets and arrangements}\label{ssec:building-sets}

\begin{definition}\label{def:building-set}
    A \textbf{building set}\index{building set} over a manifold $M$ is a collection $\cG$ of non-empty and (topologically) closed submanifolds of positive codimension such that all $\cS\subseteq\cG$ intersect cleanly in a connected submanifold $\cap\cS$.
    We say that $$\Arr_\cG:=\{\cap\cS\;|\; \cS\subseteq\cG \}$$ is the \textbf{arrangement}\index{arrangement} induced by $\cG$, where by convention $\cap\emptyset= M$.
\end{definition}

\begin{remark}\label{remark:building-set-def}
Let us break down the details of this definition:
The elements of $\cG$ are closed subsets in order for blow-ups over each individual element of $\cG$ to be well-defined. Excluding the empty set and codimension zero sets is just a matter of convention as they do not affect the blow-up.
We assume connectedness of $\cap\cS$ for all $\cS\subseteq\cG$ in an effort to reduce the considerable notational burden of later chapters. This is sufficient for our applications, but one may readily generalize our results by first performing blow-ups separately on an open cover of $M$ on which the restricted building sets do satisfy connectedness and then gluing the pieces together (compare~\cite[][appendix 5.4]{Li09} for the non-weighted case).
\end{remark}

\begin{figure}
    \centering
    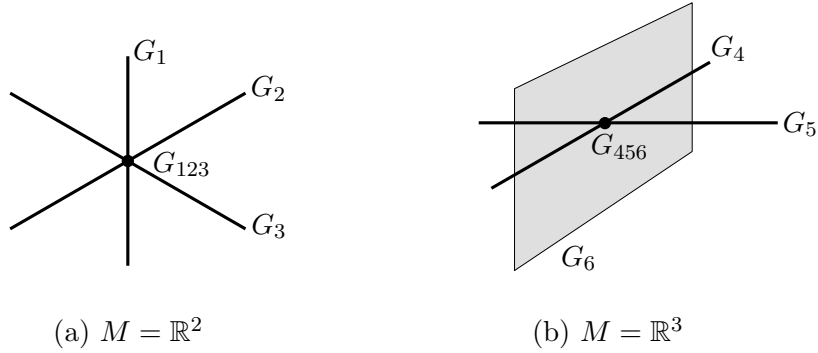
    \caption{We will use different collections of these submanifolds of $M$ as examples of building sets. $G_4$ and $G_5$ are spanned by the $x_2$ and $x_1$ axes, respectively, while $G_6$ is spanned by $x_2$ and $x_3$. $G_{123}$ and $G_{456}$ are the intersections of the other submanifolds.}
    \label{fig:low-dim-examples}
\end{figure}

\begin{examples}\ 
\begin{enumerate}
    \item We will be considering subsets of the collections of submanifolds of $\R^2$ and $\R^3$ depicted in Figure~\ref{fig:low-dim-examples} as examples throughout this section.
    \item Given a manifold $X$ and $s\geq 1$, we view $M:=X^s$ as a space of ordered configurations allowing collisions. For any $I\subseteq\{1, ..., s\}$ containing at least two elements, we define a corresponding \textit{diagonal} $$\Delta_I:=\{(p_1, ..., p_n)\in X^s\;|\; \forall i,l\in I: p_i=p_l\}.$$
The collection of all such diagonals is the \textbf{Fulton-MacPherson building set} $\cG^{FM}$. Note that for $s\geq4$, the induced arrangement is strictly larger than $\cG^{FM}$, also containing \textit{polydiagonals} like $\Delta_{\{1,2\}}\cap\Delta_{\{3,4\}}$. This building set is studied in more detail in \Subchaption~\ref{ssec:fulton-macpherson}.
    \item We call a building set \textbf{closed} if every non-empty intersection of elements in the building set is also contained in it, i.e. $\cG=\Arr_\cG\setminus\{\emptyset, M\}$. We can take the closure of any building set by adding intersections.\qedhere
\end{enumerate}
\end{examples}

Note that nested submanifolds have decreasing dimension, ensuring that a building set is at least locally finite:

\begin{lemma}\label{lem:dim-shrinking}
    For any two elements $A,B$ of a building set $\cG$, $A\subsetneq B$ implies that $\dim A<\dim B$.
\end{lemma}

\begin{proof}
Assume that the dimensions of $A\subsetneq B$ match. Since $A$ is closed in $B$, it must be a union of connected components of $B$. By the connectedness assumption, this implies the contradiction $A=B$.
\end{proof}

The induced arrangement carries some structure:

\begin{npar}[Algebraic structure on $\Arr_\cG$]
	$\Arr_\cG$ is a bounded lattice under set inclusion, where join and meet are given by
 \begin{align*}
\cap\cS_1\join\cap\cS_2 &:= \cap(\cS_1 \cup \cS_2),\\
\cap\cS_1\meet\cap\cS_2 &:= \cap(\cS_1 \cap \cS_2),
 \end{align*}
respectively. $M$ is the maximal element and $\cap\cG$ is the minimal element which may or may not be empty. We will commonly need for any subset $\cP\subseteq\cG$ and element $S\in\Arr_\cG$ the sets
\begin{align*}
\cP_{\geq S} &:= \{P\in\cP \;|\; S\subseteq P\},\\
\max\cP &:= \{P\in\cP \;|\; \not\exists P'\in\cP: P\subsetneq P' \},\\
\min\cP &:= \{P\in\cP \;|\; \not\exists P'\in\cP: P'\subsetneq P \}.\qedhere
\end{align*}
\end{npar}

\startSubchaption{Factors, flags and nests}\label{ssec:factors}

Consider a submanifold $S\in\Arr_\cG$ from the arrangement produced by $\cG$ and pick a point $p$ on it that is generic, i.e. not contained in any smaller $S'\in\Arr_\cG$. If we want to understand the blow-up of $\cG$ close to $p$, not all elements of $\cG$ are relevant: To start, only the elements of $\cG_{\geq S}$ that contain our generic $p$ have any chance of having an effect. To understand the regularity of a simultaneous blow-up of an arrangement, we need to check how nicely these relevant $\cG_{\geq S}$ intersect.

Assume for now that the $\cG_{\geq S}$ intersected transversely. The latter is equivalent to linear independence of their conormal bundles. As the blow-up construction involves unit normal vectors, this linear independence will provide a splitting of a given normal vector of $S$ into those of each $G\in\cG_{\geq S}$. This \textit{separation} of the data intuitively allows the blow-ups of all the relevant submanifolds in $\cG$ to be performed independently of each other.

But it is not necessary that \textit{all} elements of $\cG_{\geq S}$ intersect transversely. For example, the blow-ups of any two nested submanifolds are automatically compatible despite failing to be transverse. In effect, we will see that only the \textit{minimal} elements in $\cG_{\geq S}$ can cause trouble. Consequently we make the following definition:

\begin{definition}\label{def:factors}\label{def:separated}
Let $\cG$ be a building set.
\begin{enumerate}
	\item For every $S\in\Arr_\cG$ we define the set of \textbf{$\cG$-factors of $S$} as $$\cG_S:=\min\cG_{\geq S}.$$
	\item We say $\cG$ is \textbf{separated} if the collection $\cG_S$ intersects transversely for all $S\in\Arr_\cG$.
\end{enumerate}
\end{definition}

We will indeed see that separation of the building set is a sufficient condition for a manifold structure on its blow-up. Note that $\cG_G=\{G\}$ for $G\in\cG$ and that $S=\cap\cG_S$.

In~\cite{CP95} and~\cite{Li09}, conditions analogous to separation are taken to be part of the definition of building sets. As the blow-up can be very regular without making this assumption (see Example~\ref{ex:blow-ups}(1)), we find it more natural to view it as an independent property.

\begin{examples}\ 
\begin{enumerate}
    \item For the submanifolds in Figure~\ref{fig:low-dim-examples}, separated building sets are given by $\{G_1,G_2\}$ and $\{G_5,G_6\}$. The sets $\{G_1,G_2,G_3\}$ and $\{G_4,G_5\}$ are not separated.
    \item For a closed building set $\cG$, the $\cG$-factors always consist of the element itself, making $\cG$ trivially separated. The closures of the non-separated examples above yield separated building sets $\{G_1,G_2,G_3,G_{123}\}$ and $\{G_4,G_5,G_{456}\}$ this way.
\qedhere
\end{enumerate}
\end{examples}

When blowing up a single submanifold $N$, points on the exceptional divisor encode along which non-vanishing normal vectors we can approach $N$. When blowing up a whole building set $\cG$, a point should similarly determine such normal vectors for each $G\in\cG$, but these vectors are not independent choices. The relations between them will give the exceptional divisor a stratified structure. To understand this structure, we make the following two definitions:

\begin{definition}\label{def:nests}
Let $\cG$ be a building set over $M$.
\begin{enumerate}
	\item A \textbf{$\cG$-flag} of length $l$ is a chain $\emptyset \subsetneq S_1\subsetneq S_2 \subsetneq ... \subsetneq S_l \subsetneq M$ of elements of $\Arr_\cG$. We allow the trivial flag with $l=0$.
	\item A \textbf{$\cG$-nest} is a subset $\cN\subseteq\cG$ such that $\cN=\bigcup_{i=1}^l \cG_{S_i}$ for some $\cG$-flag $S_1\subsetneq ... \subsetneq S_l$. The trivial flag induces the empty nest.
\end{enumerate}
\end{definition}

A $\cG$-flag can be intuitively understood as recording all the possible ways one can approach points in $\cup\cG$ from $M\setminus \cup\cG$. However, not all the elements of $\Arr_\cG$ get blown up, but only those in $\cG$. This means that flags hold redundant data. Passing to the induced $\cG$-nest fixes this, and we will discuss in Remark~\ref{rem:strata} that nests indeed index a stratification of the blow-up.

\begin{examples}\ 
\begin{enumerate}
    \item Consider the building set $\cG=\{G_1,G_2\}$ from Figure~\ref{fig:low-dim-examples}. The five possible non-empty flags are $$G_1,\quad G_2, \quad G_{123},\quad G_{123}\subsetneq G_1,\quad G_{123}\subsetneq G_2.$$ The first two of these represent moving onto $G_i\setminus G_{123}$ from the bulk. The third approaches $G_{123}$ directly along a tangent vector in neither $TG_1$ nor $TG_2$. The last two flags represent moving onto $G_{123}$ while asymptotically approaching $G_1$ or $G_2$. The three possible non-empty nests are $$\{G_1\},\quad \{G_2\},\quad \{G_1,G_2\}$$ as the last three flags induce the same nest.
    \item For any closed building set (such as $\{G_1,G_2,G_{123}\}$), nests and flags are in 1-to-1 correspondence.
    \item The naming of nests is motivated by the Fulton-MacPherson building set, as $\cG^{FM}$-nests stand in 1-to-1 correspondence with nested subsets of $$\cI=\{I\subseteq\{1,...,s\}\;|\; |I|\geq2 \},$$ see Lemma~\ref{fulton-mac-nest}. Examples~\ref{ex:fm-flags-nests} show that one can understand nests as encoding all possible essentially different orders of point collisions.
    \qedhere
\end{enumerate}
\end{examples}

The definition of a $\cG$-nest is intuitively motivated but cumbersome in practice. We will now work towards Proposition~\ref{prop:char-nest}, which gives a much more convenient characterization in the separated case. This ultimately comes down to $\cG$-factors interacting nicely with the lattice structure on $\Arr_\cG$. To start, we have:

\begin{lemma}\label{lem:separated-g-factors}
Let $\cG$ be a separated building set and $G_1\neq G_2$ two $\cG$-factors of some $\emptyset\neq S\in\Arr_\cG$. Then there is no $G_0\in\cG$ containing both $G_1$ and $G_2$.
\end{lemma}

\begin{proof}
By separation of $\cG$, $G_1$ and $G_2$ are transverse, and their intersection is non-empty since $S\neq\emptyset$. Any $G_0$ containing them must thus have codimension zero and cannot be contained in $\cG$.
\end{proof}

Furthermore, every subset of $\cG$ gets partitioned by the $\cG$-factors of its intersection:

\begin{lemma}\label{lem:g-factors-partition}
Let $\cP$ be a subset of a separated building set $\cG$. Then the sets $$\{\cP_{\geq G}\;|\; G\in\cG_{\cap\cP}\}$$ form a partition of $\cP$. Each of the $\cP_{\geq G}$ is non-empty and satisfies $\cap\cP_{\geq G}=G.$ 
\end{lemma}

\begin{proof}
For every $P\in\cP$, it holds that $P\in\cG$ and $P\supseteq\cap\cP$ such that in particular $P\in\cP_{\geq G}$ for some $\cG$-factor $G$ of $\cap\cP$. If this $P$ was also in $\cP_{\geq G'}$ for a different $\cG$-factor $G'$, then Lemma~\ref{lem:separated-g-factors} gives a contradiction, so the $\cP_{\geq G}$ form a partition.

We need to work a bit harder to show that the elements $\cP_{\geq G}$ of the partition intersect in $G$. First note that separation of $\cG$ implies that we have a direct sum
\begin{equation}
\nu^*_{\cap\cP}M=\bigoplus_{G'\in\cG_{\cap\cP}}\nu^*_{G'} M|_{\cap\cP}.
\end{equation}
of the conormal bundles of $\cG$-factors. Here we may assume that $\cap\cP\neq\emptyset$ since the statements to prove are otherwise trivial. Since the conormal bundle of an intersection $\cap\cP$ is always the sum of the original conormal bundles, we get 
\begin{equation}\label{eq:dir-sum-normal}
\sum_{P\in\cP}\nu^*_P M|_{\cap\cP}=\bigoplus_{G\in\cG_{\cap\cP}}\nu^*_{G} M|_{\cap\cP}.
\end{equation}

Fixing some $G\in\cG_{\cap\cP}$, the inclusion $G\subseteq\cap\cP_{\geq G}$ is trivial and implies that we have 
\begin{equation}\label{eq:normal-bundle-p}
\nu^*_G M\supseteq \sum_{P\in\cP_{\geq G}}\nu^*_P M|_{G}.
\end{equation}
When splitting the left-hand side of Equation~\eqref{eq:dir-sum-normal} according to the partition, each group of terms thus includes into a single term of the direct sum on the right, implying that the inclusion is in fact an equality. This means on one hand that $\cap\cP_{\geq G}=G.$ On the other hand, if $\cP_{\geq G}$ was empty, this would mean that the rank of $\nu^*_G M$ is zero. But by definition of building sets, the codimension of $G$ is positive, and we also have a contradiction to $\cP_{\geq G}=\emptyset$.
\end{proof}

The previous Lemma is the key ingredient for the following equivalence:

\begin{proposition}\label{prop:char-nest}
Let $\cG$ be a building set. Consider the following statements for any subset $\cN\subseteq\cG$ with $\cap\cN\neq\emptyset$:
\begin{enumerate}
\item Every subset $\cP\subseteq\cN$ of (pairwise) non-comparable elements satisfies $\cP=\cG_{\cap\cP}$.
\item $\cN$is a $\cG$-nest.
\item Every subset $\cP\subseteq\cN$ of (pairwise) non-comparable elements with $|\cP|\geq2$ satisfies $\cap\cP\not\in\cG$.
\end{enumerate}
It always holds that $(1)\Rightarrow(2)\Rightarrow(3)$. If $\cG$ is separated, it also holds that $(3)\Rightarrow(1)$ and all statements are equivalent.
\end{proposition}

\begin{proof}\ 

\textit{Regarding (1)$\Rightarrow$(2):} If $\cN$ is empty, it is the empty nest induced by the trivial flag. Otherwise define $\cN_1:=\cN$ and iteratively set $\cM_i:=\min\cN_i$ and $\cN_i:=\cN_{i-1}\setminus\cM_{i-1}$ until no elements are left. We claim that the $S_i:=\cap\cN_i=\cap\cM_i$ form a flag whose induced nest is $\cN$.

Indeed since $S_i=\cap\cM_i$, and the elements in $\cM_i$ are non-comparable by construction, (1) yields $\cM_i=\cG_{S_i}$. Every element of $\cN$ is minimal at some point during this process, thus $\cN=\bigcup_i \cM_i=\bigcup_i \cG_{S_i}$ follows. It only remains to be seen that the $S_i$ form a flag, as we already saw that they induce $\cN$.

Clearly the $S_i$ are nested in each other as intersections of the shrinking set $\cN_i$. They are strict subsets:
If it was true that $S_i=S_{i+1}$, then $\cM_i=\cG_{S_i}=\cG_{S_{i+1}}=\cM_{i+1}$ would follow, which is patently false. We also have that $S_1=\cap\cN_1=\cap\cN$ is not empty by assumption.

\textit{Regarding (2)$\Rightarrow$(3):} Let  $\cP\subseteq\cN$ consist of (pairwise) non-comparable elements with $|\cP|\geq2$. Assume $\{S_i\}_{i=1...l}$ is a flag inducing $\cN$, i.e. $\cN=\bigcup_{i=1}^l \cG_{S_i}$. Thus for every $P\in\cP$ there must be an $i_P\in\N$ such that $P$ is a $\cG$-factor of $S_{i_P}$. Let $m:=\min_{P\in\cP} i_P$. Since the $S_i$ are nested, $S_m\subseteq\cap\cP$. There must be a $P_0\in\cP$ that is a $\cG$-factor of $S_m$, and as the elements of $\cP$ are incomparable, $\cap\cP\subsetneq P_0$. If $\cap\cP$ was an element of $\cG$, then we have a contradiction to $P_0$ being a $\cG$-factor of $S_m$ since $\cap\cP$ is smaller.

\textit{Regarding (3)$\Rightarrow$(1):} Let $\cP\subseteq\cN$ be a subset of non-comparable elements. By Lemma~\ref{lem:g-factors-partition}, $\cP$ is partitioned into non-empty $\cP_{\geq G}$ for each $G\in\cG_{\cap\cP}$. If $|\cP_{\geq G}|=1$ for all $G$, we would be done as this implies $\cP_{\geq G}=\{G\}$ and $\cP=\cG_{\cap \cP}$. So assume instead that there is a subset with $|\cP_{\geq G}|\geq2$. Applying (3) to it would yield $\cap\cP_{\geq G}=G\not\in\cG$, a contradiction.
\end{proof}

Note that (1)$\Rightarrow$(2) can be modified to produce all flags that induce a given nest by choosing, at every step, different \textit{subsets} of minimal elements to intersect and remove from the nest.

The separation assumption is really necessary for equivalence:

\begin{example}
Consider $\cG=\{G_1,G_2,G_3\}$ for the submanifolds from Figure~\ref{fig:low-dim-examples}. As a counterexample to (2)$\Rightarrow$(1), $\cG$ itself is a nest induced by the flag $S_1=G_{123}$, but $\cP=\{G_1,G_2\}\subseteq\cG$ consists of non-comparable elements and $\cP\neq\cG_{\cap\cP}$. It also does not hold that (3)$\Rightarrow$(2): The intersection of $\cP$ is not in $\cG$ but $\cP$ is not a nest.
\end{example}

The previous proposition gives us some insight into the structure of nests in the separated case:

\begin{corollary}
Let $\cG$ be a separated building set and $\cN\subseteq\cG$ a non-empty nest. Then $\cN$ is a tree in the sense that whenever $G_1\subsetneq G_0$ and $G_2\subsetneq G_0$ for $G_0,G_1,G_2\in\cN$, then either $G_1\subseteq G_2$ or $G_2\subseteq G_1$.
\end{corollary}

\begin{proof}
Assume there existed $G_0,G_1,G_2\in\cN$ as in the statement and assume that $G_1$ and $G_2$ were not comparable. Then by part (1) of Proposition~\ref{prop:char-nest}, $G_1$ and $G_2$ are the $\cG$-factors of $G_1\cap G_2$. By Lemma~\ref{lem:separated-g-factors}, there cannot be a $G_0$ containing both of them, and we have a contradiction.
\end{proof}

Finally, we want to show that for every nest of a separated building set we can find coordinates adapted to all of its elements:

\begin{proposition}\label{prop:nest-coords}
Let $\cG$ be a separated building set and $\cN$ a $\cG$-nest. Then around every point in $\cap\cN$ there are coordinates of $M$ in which each element of $\cN$ is given by a coordinate subspace.
\end{proposition}

\begin{proof}
By the first characterization of nests in Proposition~\ref{prop:char-nest} and separation, any subset of incomparable elements in $\cN$ intersects transversely. Applying this to the set $\cN_0:=\max\cN$ of maxima of $\cN$, we can immediately find locally defined functions that are linearly independent close to a given $p\in\cap\cN$ and such that each of the maxima is cut out by a subset. Now pick some maximal element $G_1$ of $\cN\setminus\cN_0$. The functions cutting out the elements of $(\cN_0)_{\geq G_1}$ all vanish on $G_1$ and can be minimally extended to a set that cuts out $G_1$, still linearly independently at $p$. The set $\{G_1\}\cup(\cN_0\setminus(\cN_0)_{\geq G_1})$ consists of incomparable elements, thus these newly added functions are also linearly independent of the rest. We can now define $\cN_1:=\{G_1\}\cup\cN_0$, pick some $G_2\in\cN\setminus\cN_1$, and continue this process until no elements of $\cN$ are left. In a final step, these functions form components of a submersion at $p$ which can be extended to a full set of coordinates by the submersion theorem.
\end{proof}

\startSubchaption{Blow-ups}\label{ssec:blow-ups}

We have now assembled everything we need to define blow-ups of building sets, discuss examples and provide context from the literature. In particular, this \subchaption{} states the central regularity result that the blow-up of a separated building set is smooth. We do not give a proof since the more general results about blow-ups of \textit{weighted} buildings sets of the following \chaption{} specialize to the building sets discussed here.

\begin{definition}\label{def:blow-up-building-set}
	Let $\cG$ be a non-empty building set over $M$. Then the closure of the image of the diagonal inclusion map
	$$ \iota: M\setminus\cup\cG \;\lhook\joinrel\longrightarrow \prod_{G\in\cG}\Bl_G (M)$$
	is called the \textbf{blow-up along $\cG$} and denoted by $\Bl_\cG (M)$. The \textbf{blow-down map} $$b_\cG:\Bl_\cG (M)\to M$$ is the map that sends $\{p_G\}_{G\in\cG}$ to $b_G(p_G)$ for any $G\in\cG$.
We also define the \textbf{projective blow-up} $\PBl_\cG(M)$ analogously, by replacing each $\Bl_G(M)$ with $\PBl_G(M)$.
\end{definition}

Points in the image of the diagonal map are completely determined by the value of the blow-down map. In contrast, for every $G\in\cG$ with $b_\cG(p)\in G$, a point $p$ in the boundary of the image of $\iota$ additionally encodes a normal direction to $G$ at $b_\cG(p)$. As per Lemma~\ref{lem:seq-compat-weightings} there are relations constraining this data whenever a direction normal to some $G_1\in\cG$ projects to a non-zero normal direction along a larger $G_2\in\cG$.

Our guiding question is the following: \textit{When is $\Bl_\cG(M)$ a manifold with corners in some canonical way?} It certainly holds in many situations:

\begin{figure}
    \centering
    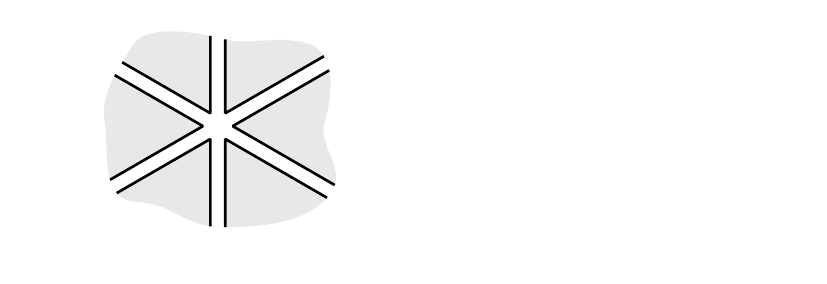
    \caption{Blow-ups of the building sets $\cG_a:=\{G_1,G_2,G_3\}$ and $\cG_b:=\{G_1,G_2,G_3,G_{123}\}$ from Figure~\ref{fig:low-dim-examples}. The blow-down map acts by collapsing all boundary components back into $\cup\cG_a=\cup\cG_b$.}
    \label{fig:plane-blown-up}
\end{figure}

\begin{examples}\label{ex:blow-ups}\ 
\begin{enumerate}
    \item The blow-ups of two different building sets $\cG_a$ and $\cG_b$ in the Euclidean plane are depicted in Figure~\ref{fig:plane-blown-up} and both canonically carry a smooth structure with corners. Note how the inclusion of $G_{123}$ in $\cG_b$ removes the intersection between the boundary components corresponding to the other submanifolds.
    \item The projective blow-ups of $\cG_a$ and $\cG_b$ are canonically diffeomorphic to $\R^2$ and $\Bl_{\{0\}}\R^2$, respectively. In the spherical case, points in the exceptional divisor still carry orientation data. For projective blow-ups, hyperplanes in the building set have no effect when constructing projective blow-ups as this data is quotiented out.
    \item Analogously to (1), the blow-up of $\cG_c=\{G_5,G_6\}$ from Figure~\ref{fig:low-dim-examples} can be visualized as the result of cutting out small open neighbourhoods of both $G_5$ and $G_6$.
\qedhere
\end{enumerate}
\end{examples}

But $\Bl_\cG(M)$ is not always smooth:

\begin{examples}\label{ex:not-mfd}\ 
\begin{enumerate}
\item Similarly to~\ref{ex:blow-ups}(i), one may consider a collection $\cG$ of hypersurfaces in $\R^3$ that all intersect in the origin. If each component of $\R^3\setminus\cup\cG$ touches at most three of the hypersurfaces away from the origin, we obtain a manifold with corners. If, however, the surfaces are arranged in such a way that some component is bounded by at least four of them, we have a singularity modelled by a pyramid at the origin.
\item While the previous example can be equipped with a (non-canonical!) structure of a manifold with corners (or even just a boundary) by smoothing edges, this is not always possible: Consider the building set $\cG:=\{G_4,G_5\}$ with the submanifolds of $\R^3$ from Figure~\ref{fig:low-dim-examples}. Then $\Bl_\cG(\R^3)$ is a subset of $\Pi:=\Bl_{G_4}(\R^3)\times\Bl_{G_5}(\R^3)$. Assume this was a manifold with corners whose underlying topology coincides with the subspace topology. Since the image of the diagonal map is locally homeomorphic to $\R^3$, its boundary as a manifold with corners must be locally homeomorphic to $\R^2$ (though not diffeomorphic, as it may have edges and corners).

$G_4$ and $G_5$ are spanned by the $x_2$ and $x_1$-axis, respectively, so we can consider local coordinates $(a_0,a_2,a_3)$ on $\Bl_{G_4}(\R^3)$ induced by choosing the hyperplane $x_1=1$, as well as coordinates $(b_0,b_1,b_3)$ on $\Bl_{G_5}(\R^3)$ induced by $x_2=1$. In particular, the blow-down maps in these coordinates are given by
$$
\begin{pmatrix} a_0\\a_2\\a_3 \end{pmatrix}\mapsto
\begin{pmatrix} a_0\\a_2\\a_0 a_3 \end{pmatrix}
\qquad\text{and}\qquad
\begin{pmatrix} b_0\\b_1\\b_3 \end{pmatrix}\mapsto
\begin{pmatrix} b_1\\b_0\\b_0 b_3 \end{pmatrix},
$$
i.e., $a_0$ and $b_0$ are the non-negative parameters controlling distance to the divisor. The part of the image of the diagonal map that is covered by these coordinates is thus cut out by the equations
$$b_1=a_0, \qquad a_2=b_0, \qquad a_0a_3=b_0b_3$$
where $a_0, b_0>0.$ We discard the superfluous parameters $b_1$ and $a_2$ to see that over the domain of these charts, taking the closure adds a boundary homeomorphic to
\begin{align*}
B=\{ (a_0,b_0,a_3,b_3)\in [0,\infty)^2\times\R^2 \;&|\; (a_0=0\lor b_0=0)\\
&\land (a_0>0\implies a_3=0)\\
&\land(b_0>0\implies b_3=0) \}.
\end{align*}
This space is not locally homeomorphic to $\R^2$, yielding a contradiction\footnote{
For example, the plane $P=\{(0,0,a_3,b_3)\in B\;|\;a_3^2+b_3^2<\epsilon\}$ is an arbitrarily small non-open subset homeomorphic to $\R^2$. This is not possible for a subset of $\R^2$.}.
\qedhere
\end{enumerate}
\end{examples}

In these counterexamples, the building sets were not separated. Indeed one otherwise finds:

\begin{theorem}[\cite{Li09}]\label{thm:li}
    For any separated building set $\cG$ over $M$, $\Bl_\cG (M)$ can canonically be equipped with the structure of a smooth manifold with corners such that $b_\cG$ is smooth, proper and a diffeomorphism away from the exceptional divisor.
\end{theorem}

While Li's theorem was originally phrased in the language of algebraic geometry, the differential-geometric formulation above is a direct corollary of our more general Theorem~\ref{thm:manifold-structure} from the weighted setting combined with Lemma~\ref{lem:triv-separated-is-unif}. Note that the building set $\cG_a$ in Example~\ref{ex:blow-ups}(1) is not separated, so this theorem is not sharp.

\begin{npar}[Applications]
Li points out that a number of constructions in the literature can be viewed as blow-ups of different building sets:
Hu~\cite{Hu03} studies the results of blowing up closed building sets. By adding the polydiagonals to the Fulton-MacPherson building set, we obtain a closed building set whose blow-up was studied by Ulyanov~\cite{Uly00}. Between these extreme cases, for every connected graph with $s$ labeled vertices one can find a building set over the same arrangement that generates Kuperberg-Thurston's compactification~\cite{KT99}.
The moduli space of rational curves with a number of marked points studied in~\cite{Ka92} can also be constructed as a blown-up building set. Finally, we want to point out that~\cite{Ro14} studies \textit{weighted compactifications of configurations spaces} by assigning rational weights to every point in a configuration and constructing a building set adapted to these weights that is contained in the Fulton-MacPherson building set. Note that this is different from assigning weights to the normal directions of collided configurations and performing a weighted blow-up as we will study in the next \chaption{}.
\end{npar}

Let us briefly preview the structure of the exceptional divisor:

\begin{npar}[Stratification of $\Bl_\cG(M)$]\label{rem:strata}
In the more general weighted setting, we will see that there is a canonical stratification of $\Bl_\cG(M)$ (Definition~\ref{def:ass-nest-and-strat}). In general, this is a weak notion of stratification (compare Proposition~\ref{prop:stratification}), but in a regular situation (e.g., in the context of Li's theorem) the stratification is very well-behaved: In local charts it matches the stratification of the standard corner model (Proposition~\ref{prop:strat-local-coords}) and the strata are labeled by $\cG$-nests (Proposition~\ref{prop:ass-nest-is-nest}).
The bulk of the blow-up (i.e., everything away from the exceptional divisor) constitutes the stratum labeled by the empty set. The non-empty nests $\cN$ label strata whose points can be reached as limit points of curves in the bulk that asymptotically approach the exceptional divisor as encoded in $\cN$.
\end{npar}

\begin{example}
Consider again the blow-ups of two different building sets $\cG_a$ and $\cG_b$ in the Euclidean plane from Figure~\ref{fig:plane-blown-up}. First consider $\cG_a$: The edges $b_{\cG_a}^{-1}(G_i\setminus G_{123})$ are exactly the strata labeled by the nests $\{G_i\}$ and the corners $b_{\cG_a}^{-1}(G_{123})$ are correspond to the nest $\{G_1,G_2,G_3\}$. The blow-up of $\cG_b$ has additional edges since $G_{123}$ is included: The edges $b_{\cG_b}^{-1}(G_i\setminus G_{123})$ are still labeled by $\{G_i\}$, and the new edges at the origin correspond to $\{G_{123}\}$. All of these result from moving onto the exceptional divisor along a straight path. The corners on the other hand are now labeled by $\{G_i,G_{123}\}$ and result from moving into the origin asymptotically along one of the $G_i$.
\end{example}

%% file: low-dim-examples.pdf_tex
\begingroup%
  \makeatletter%
  \providecommand\color[2][]{%
    \errmessage{(Inkscape) Color is used for the text in Inkscape, but the package 'color.sty' is not loaded}%
    \renewcommand\color[2][]{}%
  }%
  \providecommand\transparent[1]{%
    \errmessage{(Inkscape) Transparency is used (non-zero) for the text in Inkscape, but the package 'transparent.sty' is not loaded}%
    \renewcommand\transparent[1]{}%
  }%
  \providecommand\rotatebox[2]{#2}%
  \newcommand*\fsize{\dimexpr\f@size pt\relax}%
  \newcommand*\lineheight[1]{\fontsize{\fsize}{#1\fsize}\selectfont}%
  \ifx\svgwidth\undefined%
    \setlength{\unitlength}{396.8503937bp}%
    \ifx\svgscale\undefined%
      \relax%
    \else%
      \setlength{\unitlength}{\unitlength * \real{\svgscale}}%
    \fi%
  \else%
    \setlength{\unitlength}{\svgwidth}%
  \fi%
  \global\let\svgwidth\undefined%
  \global\let\svgscale\undefined%
  \makeatother%
  \begin{picture}(1,0.34285714)%
    \lineheight{1}%
    \setlength\tabcolsep{0pt}%
    \put(0.24682737,0.01421904){\color[rgb]{0,0,0}\makebox(0,0)[t]{\lineheight{0.80000001}\smash{\begin{tabular}[t]{c}(a) $M=\R^2$\end{tabular}}}}%
    \put(0.70218293,0.01421904){\color[rgb]{0,0,0}\makebox(0,0)[t]{\lineheight{0.80000001}\smash{\begin{tabular}[t]{c}(b) $M=\R^3$\end{tabular}}}}%
    \put(0.36377844,0.11609832){\color[rgb]{0,0,0}\makebox(0,0)[lt]{\lineheight{0.80000001}\smash{\begin{tabular}[t]{l}$G_3$\end{tabular}}}}%
    \put(0.25143691,0.28174039){\color[rgb]{0,0,0}\makebox(0,0)[lt]{\lineheight{0.80000001}\smash{\begin{tabular}[t]{l}$G_1$\end{tabular}}}}%
    \put(0.36330457,0.24702935){\color[rgb]{0,0,0}\makebox(0,0)[lt]{\lineheight{0.80000001}\smash{\begin{tabular}[t]{l}$G_2$\end{tabular}}}}%
    \put(0.27023877,0.17548385){\color[rgb]{0,0,0}\makebox(0,0)[lt]{\lineheight{0.80000001}\smash{\begin{tabular}[t]{l}$G_{123}$\end{tabular}}}}%
    \put(0.80000431,0.28511231){\color[rgb]{0,0,0}\makebox(0,0)[lt]{\lineheight{0.80000001}\smash{\begin{tabular}[t]{l}$G_4$\end{tabular}}}}%
    \put(0,0){\includegraphics[width=\unitlength,page=1]{low-dim-examples.pdf}}%
    \put(0.6578931,0.08572145){\color[rgb]{0,0,0}\makebox(0,0)[lt]{\lineheight{0.80000001}\smash{\begin{tabular}[t]{l}$G_6$\end{tabular}}}}%
    \put(0.8682773,0.21353621){\color[rgb]{0,0,0}\makebox(0,0)[lt]{\lineheight{0.80000001}\smash{\begin{tabular}[t]{l}$G_5$\end{tabular}}}}%
    \put(0.68592938,0.19181707){\color[rgb]{0,0,0}\makebox(0,0)[lt]{\lineheight{0.80000001}\smash{\begin{tabular}[t]{l}$G_{456}$\end{tabular}}}}%
    \put(0,0){\includegraphics[width=\unitlength,page=2]{low-dim-examples.pdf}}%
  \end{picture}%
\endgroup%

%% file: plane-blown-up.pdf_tex
\begingroup%
  \makeatletter%
  \providecommand\color[2][]{%
    \errmessage{(Inkscape) Color is used for the text in Inkscape, but the package 'color.sty' is not loaded}%
    \renewcommand\color[2][]{}%
  }%
  \providecommand\transparent[1]{%
    \errmessage{(Inkscape) Transparency is used (non-zero) for the text in Inkscape, but the package 'transparent.sty' is not loaded}%
    \renewcommand\transparent[1]{}%
  }%
  \providecommand\rotatebox[2]{#2}%
  \newcommand*\fsize{\dimexpr\f@size pt\relax}%
  \newcommand*\lineheight[1]{\fontsize{\fsize}{#1\fsize}\selectfont}%
  \ifx\svgwidth\undefined%
    \setlength{\unitlength}{396.8503937bp}%
    \ifx\svgscale\undefined%
      \relax%
    \else%
      \setlength{\unitlength}{\unitlength * \real{\svgscale}}%
    \fi%
  \else%
    \setlength{\unitlength}{\svgwidth}%
  \fi%
  \global\let\svgwidth\undefined%
  \global\let\svgscale\undefined%
  \makeatother%
  \begin{picture}(1,0.34285714)%
    \lineheight{1}%
    \setlength\tabcolsep{0pt}%
    \put(0.26635626,0.01421904){\color[rgb]{0,0,0}\makebox(0,0)[t]{\lineheight{0.80000001}\smash{\begin{tabular}[t]{c}(a) $\Bl_{\cG_a}\R^2$\end{tabular}}}}%
    \put(0,0){\includegraphics[width=\unitlength,page=1]{plane-blown-up.pdf}}%
    \put(0.73211166,0.01421904){\color[rgb]{0,0,0}\makebox(0,0)[t]{\lineheight{0.80000001}\smash{\begin{tabular}[t]{c}(b) $\Bl_{\cG_b}\R^2$\end{tabular}}}}%
    \put(0,0){\includegraphics[width=\unitlength,page=2]{plane-blown-up.pdf}}%
  \end{picture}%
\endgroup%

%% file: weighted-building-sets.tex
We want to generalize the blow-ups of building sets of the last \chaption{} by equipping each element with a weighting, forming a \textit{weighted} building set.

We define weightings along building sets rigorously in \Subchaption~\ref{ssec:weighted-building-sets} along with appropriate notions of triviality and (uniform) alignment. In \Subchaption~\ref{ssec:weighted-building-blow-ups}, we define their blow-ups and discuss how, close to any given point in the blow-up, only a subset of the building set is relevant (the \textit{control set} of that point). This also naturally yields a stratification of the blow-up.  We furthermore restate our main Theorem in~\ref{thm:manifold-structure}. In \Subchaption~\ref{ssec:perspectives}, we introduce the notion of \textit{good perspective} as the data required to define charts and show that their domains cover the blow-up when the weighting is uniformly aligned and defined over a separated building set. \Subchaption~\ref{ssec:charts} defines the induced charts and gives crucial ingredients for the proof of the main theorem. \Subchaption~\ref{ssec:smooth-structure} discusses the relation of the smooth structure to both the stratification and the embedding into the product of individual blow-ups. Finally, we introduce morphisms of weighted building sets and the maps they induce between blow-ups in \Subchaption~\ref{ssec:morphisms}.

\startSubchaption{Weighted building sets}\label{ssec:weighted-building-sets}

\begin{definition}\label{def:weighted-building-set}
	Let $\cG$ be a building set on a manifold $M$. A \textbf{weighting $\cW$ along $\cG$} consists of weightings $\cW_G\subseteq T^{(\infty)} M$ along each submanifold $G\in\cG$ such that $\cap_{P\in\cP}\cW_P$ is a clean intersection equal to $\cW_{\cap\cP}$ for all $\cP\subseteq\cG$ with $\cap\cP\in\cG$. We also say that $(\cW,\cG)$ is a \textbf{weighted building set}.
\end{definition}

We see that for all $G\subseteq G'$, the identity is a weighted morphism from $(M,\cW_G)$ to $(M,\cW_{G'})$ since considering $\cP=\{G,G'\}$ yields $\cW_G\subseteq\cW_{G'}$. Note that a weighting along $\cG$ is determined by the weightings over the elements of $\cG$ that are \textit{independent} in the sense that they cannot be written as a nontrivial intersection of subsets of $\cG$. However, not all choices of weightings on the independent submanifolds can be extended to a weighting along the full building set since we require compatibility when taking intersections.

\begin{npar}[Notation]
From here on whenever we fix a weighting $\cW$ along a building set $\cG$, we index all objects relating to $\Bl_{\cW_G}$ with only $G$ for brevity. So for any $G,G'\in\cG$ we will write $b_G$ and $\nu_{G',G}$ and $b_{G',G}$, etc, as shorthand for $b_{\cW_G}$ and $\nu_{\cW_{G'},\cW_G}$ and $b_{\cW_{G'},\cW_G}$.
\end{npar}

Our notions of compatibility of weightings let us consider weighted building sets of different regularity:

\begin{definition}\label{def:building-properties}
Let $\cW$ be a weighting over a building set $\cG$. We say it is...
\begin{enumerate}
\item \textbf{trivial} if $\cW_G$ is a trivial weighting for all $G\in\cG$.
\item \textbf{aligned} if for every $\cG$-nest $\cN$, the weightings $\{\cW_G\}_{G\in\cN}$ are aligned.
\item \textbf{uniformly aligned} if for every $\cG$-nest $\cN$, the weightings $\{\cW_G\}_{G\in\cN}$ are uniformly aligned.
\end{enumerate}
\end{definition}

We want to stress that (uniform) alignment of a weighting over a building set is a weaker condition than (uniform) alignment when viewed as a collection of weightings, i.e. in the sense of Definition~\ref{def:weighting-compat}. For example, the trivial weighting along the Fulton-MacPherson building set satisfies only the former. We will argue in the following \subchaption{} that these weaker conditions only on nests are sufficient, luckily.

We first give a convenient mental model for the data of an aligned weighted building set:

\begin{npar}[Tableau of weights]\label{par:tableau}
For an aligned weighting $\cW$ over a building set $\cG$, consider a $\cG$-nest $\cN$ and a point $p\in\cap\cN$. By alignment, there exists a chart $\chi=(x_1,...,x_m):U\to\R^m$ of $M$ around $p$ that is aligned with each $\{\cW_N\}_{N\in\cN}$. This chart assigns to each submanifold $N\in\cN$ and each coordinate direction $x_i$ a weight $w_{N,i}$. If $\cG$ is separated, we can display this data conveniently in a \textbf{weight tableau} as follows: We arrange the non-vanishing weights for each coordinate direction in columns and associate them with the elements of $\cN$ by drawing horizontal boxes. 

For example, a possible weight tableau over a nest $\cN=\{A,B,C,D\}$ on an eight-dimensional manifold could look like this:
\begin{center}
\begin{tikzpicture}
  \matrix (m) [matrix of nodes, nodes in empty cells, column sep=5pt, row sep=5pt] {
    1 & 2 & 3 & 1 &  &  &  &  \\
    1 & 2 & 3 & 1 & 2 & 1 & 2 &  \\
    $x_1$ & $x_2$ & $x_3$ & $x_4$ & $x_5$ & $x_6$ & $x_7$ & $x_8$ \\
  };
  \draw[thick] ([xshift=-5pt, yshift=1pt]m-3-1.north west) -- ([xshift=5pt, yshift=1pt]m-3-8.north east);

  \tableaubox(m)(A)(1:1:2)
  \tableaubox(m)(B)(1:3:2)
  \tableaubox(m)(C)(2:1:5)
  \tableaubox(m)(D)(2:6:2)
\end{tikzpicture}
\end{center}
Each element of $\cN$ is represented by a box of weights covering a subset of the coordinates, where every weight outside of the box is zero by convention. The box as a whole thus represents the normal directions to the submanifold it is associated with. We arrange the coordinate components in such a way that these boxes are connected and vertically stacked without overhang. 
Due to separation, this is always possible: Any two incomparable elements in a nest must be $\cG$-factors by Proposition~\ref{prop:char-nest}, and thereby intersect transversely. This means that any two $N,N'$ with $w_{N,i}\neq0\neq w_{N',i}$ must satisfy either $N\subseteq N'$ or $N\supseteq N'$.

Conveniently, the partial order of elements of $\cN$ can be read off of the tableau: $N\subseteq N'$ holds precisely when the box for $N'$ is stacked somewhere above $N$. Uniform alignment of the weighting is reflected here by all non-zero weights in a given column being the same.
\end{npar}

For the rest of this \subchaption, we will investigate the relationship between triviality, (uniform) alignment and separation of the underlying building set. The main upshot will be that uniform alignment can be checked with a simple criterion in the separated case (Proposition~\ref{prop:char-uniform-align}).

Clearly, any uniformly aligned weighting is also aligned. Moreover, we have:

\begin{lemma}\label{lem:triv-separated-is-unif}
    Any trivial weighting $\cW$ along a separated building set $\cG$ is uniformly aligned.
\end{lemma}

\begin{proof}
Proposition~\ref{prop:nest-coords} provides adapted coordinates to every nest $\cN$ of $\cG$. Since $\cW$ is trivial, any chart adapted to $G\in\cN$ is also adapted to $\cW_G$. Uniform alignment follows immediately since trivial weightings can only have weights zero or one.
\end{proof}

Conversely, uniform alignment implies neither triviality nor separation of the building set, as the pair of uniformly aligned standard weightings $\{\cE^{(0,2,1)},\cE^{(1,2,0)}\}$ shows. Separation alone does not even suffice for alignment:

\begin{example}
Consider again the weightings from Figure~\ref{fig:not-aligned}. Since they intersect cleanly, they can be viewed as a weighting along the building set $\{A,B,C\}$ formed by their supports. This building set is separated since it is closed. Moreover, the flag $C\subseteq B\subseteq A$ induces all of $\{A,B,C\}$ as a nest, but as we argued in Example~\ref{ex:weighting-compatibility}(\ref{ex:not-aligned}), the weightings over this nest are not aligned.
\end{example}

For arbitrary weighted building sets, checking uniform alignment using Definition~\ref{def:building-properties} is hard: It requires either carefully constructing adapted charts or showing non-existence, possibly for a large number of nests. With Lemma~\ref{lem:linear-unif-alignment-check}, we provided a method to check uniform alignment of pairs of weightings based on just linear data, without the need to construct coordinates explicitly. We can use this to modify the proof of Proposition~\ref{prop:nest-coords} to find a condition sufficient for uniform alignment in the separated weighted case. The crucial step is the following Lemma: For any collection of weightings over supports that intersect transversely together with a smaller weighting that satisfies a linear compatibility condition, it constructs adapted charts.

\begin{lemma}\label{lem:extending-coordinates}
Consider a weighting $\{\cW_\alpha\}_{\alpha=1,...,d}$ over a collection of supports $\{N_\alpha\}$ that intersect transversely. Let $\cW_0\subseteq\cap_\alpha\cW_\alpha$ be a weighting with support $N_0\subseteq \cap_\alpha N_\alpha$ such that the map $$\bigoplus_\alpha\nu^*_\mathrm{lin}\cW_\alpha|_{N_0}\to\nu^*_\mathrm{lin}\cW_0$$ is an injection. Then $\cW_0$ and all $\cW_\alpha$ are uniformly aligned.

More concretely, given weighted coordinates $\{x_{\alpha,j}\}_{j=1,...,n_\alpha}$ of each $\cW_\alpha$  with weights $w_{\alpha,j}$ around some fixed $p\in N_0$, we can extend the functions $\{x_{\alpha,j}\;|\; w_{\alpha,j}\geq 1\}$ to a weighted chart of $\cW_0$ around $p$.
\end{lemma}

\begin{proof}
Take coordinates $\{x_{\alpha,j}\}_{j=1,...,n_\alpha}$ adapted to the $\cW_\alpha$ as in the statement. We already saw in Lemma~\ref{lem:intersection-props-supports}(3) that we can take a subset of these to obtain coordinates $\{\tilde x_a\}_{a=1,...,n}$ in which $\cW_\alpha$ are all uniformly aligned. It immediately follows in these coordinates that $\widetilde{\cW}:=\cap_\alpha\cW_\alpha$ is a weighting that is uniformly aligned with all $\cW_\alpha$. It is easy to see that $$\nu^*_\text{lin}\widetilde{\cW}\simeq \bigoplus_\alpha \nu^*_\text{lin}\cW_\alpha$$ holds. In particular, 
$$\nu^*_\text{lin}\widetilde{\cW}|_{N_0}\to \nu^*_\mathrm{lin}\cW_0$$ must be an injection, so we find ourselves in the situation of Lemma~\ref{lem:linear-unif-alignment-check}. Using arbitrary coordinates $\{x_a\}_{a=1,...,m}$ adapted to $\cW_0$, we can construct coordinates $\{\hat x_a\}_{a=1,...,m}$ as in the proof of the lemma. All weightings are uniformly aligned in these coordinates.
\end{proof}

We conclude:

\begin{proposition}\label{prop:char-uniform-align}
    Let $\cW$ be a weighting along a separated building set $\cG$. Then $\cW$ is uniformly aligned if and only if for every $G\in\cG$ and $S\in\Arr_\cG$ with $G\subseteq S$ the map
    $$\bigoplus_{H\in\cG_S}\nu^*_\mathrm{lin}\cW_H|_G\to\nu^*_\mathrm{lin}\cW_G$$ is an injection.
\end{proposition}

\begin{proof}\ 

\textit{Regarding "If":}
We can exactly follow the proof of Proposition~\ref{prop:nest-coords} with one caveat: Since the weights are not trivial, functions that cut out a weighting over a larger support cannot automatically be extended to a collection that cuts out a smaller weighting. However, the condition on weighted conormal bundles guarantees that we can use Lemma~\ref{lem:extending-coordinates} for this step.

\textit{Regarding "Only if":}
The submanifolds $G\subseteq S$ form a flag inducing the nest $\cN:=\{G\}\cup\cG_S$. By uniform alignment we can find coordinates adapted to all involved weightings. In these, we use the uniformity of weights to conclude that each $\nu^*_\text{lin}\cW_H|_G\to\nu^*_\text{lin}\cW_G$ is injective. Separation of $\cG$ gives transversality of $\cG_S$, which ensures that the images of these component maps are linearly independent.
\end{proof}

\startSubchaption{Blow-ups}\label{ssec:weighted-building-blow-ups}

To define the blow-up of a weighted building set, we follow the graph-based strategy:

\begin{definition}\label{def:blow-up-weighted-building-set}\nobreak%
	Let $\cW$ be a weighting along a non-empty\footnotemark{} building set $\cG$ on $M$. Then the closure of the image of the diagonal inclusion map
	$$ \iota: M\setminus\cup\cG \;\lhook\joinrel\longrightarrow \prod_{G\in\cG}\Bl_{\cW_G} (M)$$
	is called the \textbf{(spherical) blow-up along $\cW$} and denoted by $\Bl_\cW (M)$. The \textbf{exceptional divisor} now consists of the limit points not in the image of $\iota.$ The \textbf{blow-down map} $$b_\cW:\Bl_\cW (M)\to M$$ is the map that sends $\{p_G\}_{G\in\cG}$ to $b_G(p_G)$ where $G\in\cG$ is picked arbitrarily.
    
 We also define the \textbf{projective blow-up} $\PBl_\cW(M)\xrightarrow{pb_\cW} M$ by replacing each $\Bl_{\cW_G}(M)$ with $\PBl_{\cW_G}(M)$.
\end{definition}
\addtocounter{footnote}{-1}\footnotetext{We can otherwise set $\Bl_\emptyset(M)=M$ by convention, where $b_\emptyset$ is the identity.}

Note that, equivalently, $\PBl_\cW (M)$ is the image of $\Bl_\cW (M)$ under the map assembled from the individual quotient maps $\Bl_{\cW_G}(M)\to\PBl_{\cW_G}(M)$. One can thus understand its topology as a quotient of the spherical blow-up\footnote{This is indeed how we approach the projective case in our application to configuration spaces of filtered manifolds, compare Propositions~\ref{prop:FM-local-model-proj} and~\ref{prop:FM-local-bundle-model-proj}}.
For this reason, we will focus on spherical blow-ups in the remainder of this \chaption.

Our definition of the blow-up effectively subtracts the supports $\cup\cG$ from $M$ and replaces them with all limit points we can reach in $\prod_{G\in\cG}\Bl_{\cW_G} (M)$ by sequences in the bulk. This forces relations between the components:

\begin{lemma}\label{lem:limit-point-coherence}
	Let $\cW=\{\cW_G\}_{G\in\cG}$ be a weighted building set. Then every $\{p_G\}_{G\in\cG}\in \Bl_\cW (M)$ satisfies the following:
	\begin{enumerate}
		\item There exists a $p\in M$ such that $\forall G\in\cG: b_G(p_G)=p$.
		\item If $G\subset G'$ in $\cG$ and $p_G\in\Bl_{G',G}(M)$, then $p_{G'}=b_{G',G}(p_G)$.
	\end{enumerate}
\end{lemma}

\begin{proof}
Both points follow by applying Lemma~\ref{lem:seq-compat-weightings} pairwise.
\end{proof}

\begin{examples}
The inverse is not generally true:
\begin{enumerate}
    \item For the trivial weighting above the building set $\{G_4,G_5\}$ of submanifolds from Figure~\ref{fig:low-dim-examples}, the point $([e_3], -[e_3])$ satisfies the relations from Lemma~\ref{lem:limit-point-coherence} trivially but does not lie in $\Bl_\cG(M)$. This weighting is uniformly aligned, but the building set is not separated.

    \item Consider the standard weightings $\cE^{(0,1)},\cE^{(1,0)}$ and $\cE^{(2,2)}$ over the closed building set consisting of the two axes and the origin of $\R^2$. A weighted unit normal vector in the exceptional divisor of $\cE^{(2,2)}$ completely determines the components with regard to the other weightings. This is \textit{not} due to the induced maps between weighted normal bundles: These both vanish by Example~\ref{example:induced-map-weighted-normal-bundle}. Instead, there is a canonical isomorphism $$\nu\cE^{(2,2)}\to \nu\cE^{(0,1)}\oplus\nu\cE^{(1,0)}$$ that arises analogously to the isomorphism in Example~\ref{example:induced-maps}(2). However, the relations from Lemma~\ref{lem:limit-point-coherence} would already be satisfied for any three weighted unit normal vectors over the origin since the condition on induced maps is trivially satisfied.
    This is an aligned weighting over a separated building set, but not uniformly aligned. 
    \qedhere
\end{enumerate}
\end{examples}

Due to these relations between components, given a point $p=\{p_G\}_{G\in\cG}\in \Bl_\cW (M)$ we can find a minimal subset $\cG_p\subseteq\cG$ that suffices to determine the remaining components for any point close to $p$:

\begin{definition}\label{def:ass-nest-and-strat}
Let $\cW=\{\cW_G\}_{G\in\cG}$ be a weighted building set. We assign to every $p=\{p_G\}_{G\in\cG}\in \Bl_\cW (M)$ its \textbf{control set}
$$ \cG_p := \{G\in\cG \;|\; b_\cW(p)\in G \text{ and } \nexists H\in\cG: H\subsetneq G \text{ and } p_H\in \Bl_{G,H}(M) \}. $$
Conversely, we define for any $\cN\subseteq\cG$ $$ \Bl_\cW(M)_\cN := \{p\in \Bl_\cW (M) \;|\; \cG_p = \cN \}. $$
\end{definition}

This can be seen to define a topological stratification of the blow-up in the following sense:

\begin{proposition}\label{prop:stratification}
Let $\cW=\{\cW_G\}_{G\in\cG}$ be a weighted building set. A sequence of points in $\Bl_\cW(M)_\cN$ can converge to a point in $\Bl_\cW(M)_{\cN'}$ only if $\cN\subseteq\cN'$. In other words,
\begin{equation}\label{eq:stratification}
\overline{\Bl_\cW(M)_\cN}\subseteq \bigcup\limits_{\cN'\supseteq \cN} \Bl_\cW(M)_{\cN'}.
\end{equation}
\end{proposition}

\begin{proof}
    Let $p_n$ be a sequence of points with $\cG_{p_n}=\cN$ that converge to a point $p$ with $\cG_p=\cN'$. Pick an arbitrary $G\in\cN$. It holds for all $n$ that $b(p_n)\in G$ and $\nexists H\in\cG: H\subsetneq G \text{ and } (p_n)_H\in \Bl_{G,H}(M)$. Since the submanifolds in $\cG$ are closed, $b(p)=\lim_n b(p_n)\in G$ as well. If there was an $H$ as above such that $p_H\in \Bl_{G,H}(M)$, then as $\Bl_{G,H}(M)$ is open (Corollary~\ref{cor:domain-induced-open}) this would have to be true for some $(p_n)_H$ already, yielding a contradiction. Together, this implies $G\in\cN'$, i.e. $\cN\subseteq\cN'$. 
\end{proof}

\begin{examples}
We consider again the blow-ups from Figure~\ref{fig:plane-blown-up}:
\begin{enumerate}[(a)]
\item For the building set $\cG_a=\{G_1,G_2,G_3\}$, the corners of the blow-up are exactly the stratum labeled by the set $\{G_1,G_2,G_3\}$. Each edge that gets blown-down to $G_i$ is in the stratum labeled by $\{G_i\}$. The bulk, i.e. all points away from the exceptional divisor, are always labeled by the empty set. The opposite inclusion of Equation~\eqref{eq:stratification} fails here, since the closure of the stratum for $\{G_i\}$ does not include all corners.
\item For the building set $\cG_a=\{G_1,G_2,G_3,G_{123}\}$, the line segments that get blown-down to some $G_i$ are still in the $\{G_i\}$ stratum. We now additionally have circle segments in the boundary, which correspond to the stratum labeled by $\{G_{123}\}$. The $\{G_{123}\}$ and $\{G_i\}$ strata exactly intersect in the $\{G_{123},G_i\}$ stratum in the corners. As opposed to the previous example, Equation~\eqref{eq:stratification} is a true equality now. That this holds generally for uniformly aligned weightings along separated building sets will turn out to be an immediate corollary of Proposition~\ref{prop:strat-local-coords}.
\qedhere
\end{enumerate}
\end{examples}

We do not expect strata labeled by arbitrary subsets $\cN\subseteq\cG$ to be non-empty.
Indeed for weightings over separated building sets, the control sets must all be nests:

\begin{proposition}\ \label{prop:ass-nest-is-nest}
Let $\cW=\{\cW_G\}_{G\in\cG}$ be a weighting over a separated building set. Then for every $p\in \Bl_\cW (M)$, the control set $\cG_p$ is a $\cG$-nest.
\end{proposition}

\begin{proof}
Let $p=\{p_G\}_{G\in\cG}\in \Bl_\cW (M)$. We will use the third characterization of nests from Proposition~\ref{prop:char-nest} to conclude that $\cG_p$ is a $\cG$-nest. We thus assume towards contradiction that there exists a subset $\cP\subseteq\cG_p$ of incomparable elements with $|\cP|\geq2$ such that $\cap\cP\in\cG$. By definition of weighted building sets, $\cW_{\cap\cP}=\cap_{P\in\cP}\cW_P$ must be a clean intersection of weightings. Since $\cap\cP\in\cG$, then by definition of $\cG_p$, $p_{\cap\cP}\not\in\Bl_{P,\cap\cP}(M)$ for all $P\in\cP$. Moreover, we must have $p_{\cap\cP}=[n]$ for some weighted normal vector $n\in\nu\cW_{\cap\cP}$ since we lie on the exceptional divisor. $p_{\cap\cP}\not\in\Bl_{P,\cap\cP}(M)$ then means that $\nu_{P,\cap\cP}(n)=0$ for all $P\in\cP$. We are now exactly in the situation of Lemma~\ref{lem:reg-intersection-normals}, which yields that $n=0$. However, this is impossible for a representative of an element in the blow-up and we have a contradiction.
\end{proof}

We will see the opposite implication that every $\cG$-nest can be realized as the control set of some point in the blow-up at least under the additional assumption of uniform alignment in Corollary~\ref{cor:nests-are-control-sets}.

If $\cG$ is not separated, we can find examples where the control sets fail to be nests:

\begin{example}
Consider three different planes $A,B,C$ in $\R^3$ that intersect in a common line $A\cap B\cap C$ as well as a point $D$ on that line. Equip $A$ and $B$ with the unique trivial weightings $\cW_A$ and $\cW_B$,  and equip $C$ and $D$ with the unique maximal weightings $\cW_C$ and $\cW_D$ of order 2. Together, these form a weighting $\cW$ over the building set $\cG=\{A,B,C,D\}$. Let $p$ be the limit in $\Bl_\cW(M)$ of a sequence of points that converges to $D$ along a line that is not contained in any of $A,B$ or $C$. Then $\cG_p=\{A,B,D\}$, as the induced map from $D$ to $C$ is the only one that does not vanish by Example~\ref{example:induced-map-weighted-normal-bundle}. However, $\cG_p$ is not a nest: $A$ and $B$ can both be in a nest only if $A\cap B$ is part of the flag that induces it. But $C$ is also a $\cG$-factor of $A\cap B$ and would thus need to be included as well.
\end{example}

Due to difficulty in finding a trivially weighted counterexample, we suspect the following:

\begin{conjecture}
Let $\cW=\{\cW_G\}_{G\in\cG}$ be the trivial weighting over any building set. Then for every $p\in \Bl_\cW (M)=\Bl_\cG(M)$, the control set $\cG_p$ is a $\cG$-nest.
\end{conjecture}

The previous Proposition~\ref{prop:ass-nest-is-nest} is the motivation for our definition of (uniform) alignment of weighted building sets: Close to $p$, the data in the blow-ups of the nest $\cG_p$ suffices to determine the rest, so it is natural that we only ever need consistency conditions on weightings that appear in the same nest. With this knowledge, let us turn to investigating the smoothness of the blow-up.

We already saw in \ref{ex:not-mfd} that we cannot in general expect even the spherical blow-up to inherit a manifold structure if $\cG$ is not separated. Alignment and separation are also not sufficient to avoid singularities:

\begin{example}\label{ex:not-unif-sing}
Consider the standard weightings $\cW_A=\cE^{(1,1,0)}$ and $\cW_B=\cE^{(1,2,1)}$ on $\R^3$. Since the weights increase, we have $\cW_A\supseteq\cW_B$. It follows that they form an aligned weighting over the closed building set $\cG=\{A,B\}$ consisting of the third axis $A$ and the origin $B$. By the choice of hyperplanes $x_1=1$ and $x_2=1$, respectively, we obtain local coordinates $(a_0,a_2,a_3)$ on $\Bl_{\cW_A}(\R^3)$ and $(b_0,b_1,b_3)$ on $\Bl_{\cW_B}(\R^3)$ with $a_0,b_0\geq 0$ whose blow-down maps are locally given by
$$
\begin{pmatrix} a_0\\a_2\\a_3 \end{pmatrix}\mapsto
\begin{pmatrix} a_0\\a_0 a_2\\a_3 \end{pmatrix}
\qquad\text{and}\qquad
\begin{pmatrix} b_0\\b_1\\b_3 \end{pmatrix}\mapsto
\begin{pmatrix} b_0 b_1\\b_0^2\\b_0 b_3 \end{pmatrix}.
$$
A quick calculation\footnote{The image of the diagonal map is given by the equations $a_0=b_0b_1, a_0a_2=b_0^2$ and $a_3=b_0b_3$ for $a_0,b_0\geq 0$. One readily sees that combining the first two equations also yields $a_2b_1=b_0$. To restrict to the boundary, we send either $a_0$ or $b_0$ to zero, but it follows in both cases that the other must vanish as well. To avoid confusion, note that there actually is a boundary away from the exceptional divisor of $B$, but this is not covered by the chosen chart domain on $\Bl_{\cW_B}(\R^3)$.} shows that the points in the boundary of the image of the diagonal map are cut out in the coordinates $(a_0,a_2,a_3;b_0,b_1,b_3)$ of $\Bl_{\cW_A}(\R^3)\times\Bl_{\cW_B}(\R^3)$ by the following equations:
$$
a_0=a_3=b_0=0 \text{ and } a_2b_1=0.
$$
For the weighted blow-up to be a manifold with corners even just topologically, this boundary would need to locally be homeomorphic to $\R^2$, which at the origin it clearly is not. The singular point can, for example, be reached as the limit along the curve $\gamma(t):=(t^3,t^4,0)$ in $\R^3$.
\end{example}

Our main theorem asserts that uniform alignment \textit{and} separation of the building set is sufficient for a smooth structure:

\begin{theorem}\label{thm:manifold-structure}
Let $\cW$ be a uniformly aligned weighting along a separated building set $\cG$. 
Then $\Bl_\cW(M)$ can be equipped with the structure of a smooth manifold with corners such that the blow-down map $b_\cW:\Bl_\cW (M)\to M$ is smooth and proper.
\end{theorem}

In the following two \subchaption{}s, we will discuss choices involved in writing down charts as well as the charts themselves. These are prerequisites for the proof of the main theorem in Appendix~\ref{app:proof-weighted-building-set}.

\startSubchaption{Good perspectives}\label{ssec:perspectives}

We will build coordinates close to $p\in\Bl_{\cW}(M)$ by mixing and matching different coordinate components of charts on each factor $\Bl_{\cW_G}(M)$, all of which are induced by the same chart $\chi:U\to\R^m$ on $M$. For this to work, that chart $\chi$ must be aligned with all the weightings that are relevant in that region - with the weightings over $\cG_p$, that is. We then still have to pick coordinate hyperplanes for each element of the nest to determine charts. These choices should be compatible in the following sense:

\begin{definition}\label{def:good-perspective}
Fix a weighting $\cW$ along a building set $\cG$. Consider the tuple $(\chi,\cN,\mathbf{h},\mathbf{s})$ of a $\cG$-nest $\cN$, a coordinate chart $\chi=(x_1, ..., x_m):U\to\R^m$ on $M$ and maps $\mathbf{h}:\cN\to \{1,...,m\}$ and $\mathbf{s}:\cN\to\{\pm 1\}$.
We say these form a \textbf{good perspective} if the following hold:
\begin{enumerate}[(1)]
\item $\chi$ is aligned with all weightings in $\cW$ along $\cN$.
\item For every $N\in\cN$, it holds that
$$N=\max\{N'\in\cN\;|\;  w_{N',\mathbf{h}(N)}\neq 0  \},$$
where existence of the maximum is part of the condition and $w_{N,i}$ is the weight assigned to the $i$-th coordinate of $\chi$ under $\cW_N$.
\end{enumerate}
\end{definition}

Condition (2) implies that the $\mathbf{h}(N)$-th unit vector is perpendicular to $\chi(N\cap U)$ for each $N\in\cN$, but also that we never pick a hyperplane for $N$ that could be used for a larger submanifold $N'$. Note that it follows that $\mathbf{h}$ is injective, and we thus choose different coordinate hyperplanes for each element of the nest.

\begin{example}\label{ex:weighttableau}
Let $(\chi,\cN,\mathbf{h},\mathbf{s})$ be a good perspective on some aligned weighted building set $(\cW,\cG)$ on $M$. If $\cG$ is separated, we can visualize our choice of $\mathbf{h}$ using the weight tableau from Paragraph~\ref{par:tableau} by drawing an additional box around the relevant weight for each element of $\cN$. A possible good perspective for the example from that paragraph would then look as follows:
\begin{center}
\begin{tikzpicture}
  \matrix (m) [matrix of nodes, nodes in empty cells, column sep=5pt, row sep=5pt] {
    \npboxed{1} & 2 & \npboxed{3} & 1 &  &  &  &  \\
    1 & 2 & 3 & 1 & \npboxed{2} & \npboxed{1} & 2 &  \\
    $x_1$ & $x_2$ & $x_3$ & $x_4$ & $x_5$ & $x_6$ & $x_7$ & $x_8$ \\
  };
  \draw[thick] ([xshift=-5pt, yshift=1pt]m-3-1.north west) -- ([xshift=5pt, yshift=1pt]m-3-8.north east);

  \tableaubox(m)(A)(1:1:2)
  \tableaubox(m)(B)(1:3:2)
  \tableaubox(m)(C)(2:1:5)
  \tableaubox(m)(D)(2:6:2)
\end{tikzpicture}
\end{center}
Condition (2) of Definition~\ref{def:good-perspective} means that these boxed weights must always be located at the very top of each column.
\end{example}

One may worry that existence of good perspectives is already ruled out for nests where a box is entirely covered by smaller boxes stacked on top. However, at least for the separated case, such nests do not exist. In fact, when making a choice for $\mathbf{h}$, the following lemma exploits that condition (2) prevents us from `using up' all the normal directions of a larger submanifold for the hyperplanes associated with smaller submanifolds:

\begin{lemma}\label{lem:existence-perspective-for-nest}
Let $\cW$ be an aligned weighting along a separated building set $\cG$. For every $\cG$-nest $\cN$ and point $p\in\cap\cN$, there exists a good perspective $(\chi,\cN, \mathbf{h}, \mathbf{s})$ such that $p$ is covered by $\chi$.
\end{lemma}

\begin{proof}
By alignment, we can pick some chart $\chi$ around a given $p\in\cap\cN$ aligned with all the weightings along $\cN$. Assume towards contradiction that in the weight tableau of $\cN$, there existed an $N\in\cN$ completely covered by higher boxes $\{N_i\}_{i\in I}$. In particular, $N=\cap_{i\in I}N_i$. However, by separation of the building set we can use the third characterization of Proposition~\ref{prop:char-nest} to conclude that $\cap_{i\in I}\not\in\cG$, a contradiction to $N\in\cN\supseteq\cG$. Consequently every box has at least one coordinate direction not covered by a higher box that can be used to construct $\mathbf{h}$. We can make an arbitrary choice for the signs $\mathbf{s}$ since they do not appear in the two conditions for a good perspective.
\end{proof}

We define corresponding chart domains for good perspectives as well as short-hands for the induced component coordinates:

\begin{definition}
Fix a good perspective $(\chi,\cN,\mathbf{h},\mathbf{s})$ on a weighting $\cW$ along $\cG$.
Recall that $U^{(\cW_N)}_{\mathbf{h}(N)\mathbf{s}(N)}\subseteq\Bl_{\cW_N}(M)$ are the domains of the induced charts from Definition~\ref{def:smooth-blow-up}. Consider the canonical projections $$\pi_N:\prod_{G\in\cG}\Bl_{\cW_G}(M)\to\Bl_{\cW_N}(M)$$
for every $N\in\cN$.

We define the open $U^0_{\chi\cN \mathbf{hs}}\subseteq\prod_{G\in\cG}\Bl_{\cW_G}(M)$ as the intersection
$$ U^0_{\chi\cN \mathbf{hs}} := \bigcap\limits_{N\in\cN} \pi^*_N U^{(\cW_N)}_{\mathbf{h}(N)\mathbf{s}(N)} \;\cap\;
\bigcap_{G\in \cG\setminus\cN}\bigcup_{G\supseteq N\in\cN} \pi^*_N\Bl_{\cW_G,\cW_N}(M).$$
On this set, we can further define the pull-backs of the induced coordinates 
$$ x_{i:N}:=x^{(\cW_N)}_{i:\mathbf{h}(N)\mathbf{s}(N)}\circ\pi_N.$$
Finally, the \textbf{chart domain} of the perspective is given by $$U_{\chi\cN \mathbf{hs}} := U^0_{\chi\cN \mathbf{hs}} \cap \Bl_\cW(M),$$
where we can additionally set
$$x_{i:\emptyset} := x_i\circ b_\cW.$$
\end{definition}

By construction $U_{\chi\cN \mathbf{hs}}$ is open. The first term in the definition of $U^0_{\chi\cN \mathbf{hs}}$ guarantees well-definition of the $x_{i:N}$. The second term involving $\Bl_{\cW_G,\cW_N}(M)$ restricts to a set where $p\in U_{\chi\cN \mathbf{hs}}$ is entirely determined by its components over the nest $\cN$ due to Lemma~\ref{lem:limit-point-coherence}. As an immediate consequence, we have:

\begin{lemma}\label{lem:control-set-in-nest}
For any point $p$ in the chart domain $U_{\chi\cN \mathbf{hs}}$ of a good perspective, it holds that $\cG_p\subseteq\cN$.
\end{lemma}

\begin{example}\label{ex:coordinatetableau}
Let us again consider the good perspective $(\chi,\cN,\mathbf{h},\mathbf{s})$ from Example~\ref{ex:weighttableau}. We can visualize all the component coordinates at a point $p\in U_{\chi,\cN\mathbf{hs}}$ in a corresponding \textbf{tableau of coordinates}:
\begin{center}
\begin{tikzpicture}
  \matrix (m) [matrix of nodes, nodes in empty cells, column sep=5pt, row sep=5pt] {
    $\npboxed{x_{1:A}}$ & $ x_{2:A}$ & $\npboxed{x_{3:B}}$ & $ x_{4:B}$ &  &  &  &  \\
    $ x_{1:C}$ & $x_{2:C}$ & $ x_{3:C}$ & $x_{4:C}$ & $\npboxed{x_{5:C}}$ & $\npboxed{x_{6:D}}$ & $ x_{7:D}$ &  \\
    $x_{1:\emptyset}$ & $x_{2:\emptyset}$ & $x_{3:\emptyset}$ & $x_{4:\emptyset}$ & $ x_{5:\emptyset}$ & $ x_{6:\emptyset}$ & $x_{7:\emptyset}$ & $ x_{8:\emptyset}$ \\
  };
  \draw[thick] ([xshift=-5pt, yshift=1pt]m-3-1.north west) -- ([xshift=5pt, yshift=1pt]m-3-8.north east);

  \tableaubox(m)(A)(1:1:2)
  \tableaubox(m)(B)(1:3:2)
  \tableaubox(m)(C)(2:1:5)
  \tableaubox(m)(D)(2:6:2)
\end{tikzpicture}
\end{center}
The bottom row corresponds to the coordinates on $M$ after blowing down. Similarly to how we do not draw vanishing weights, all the coordinates $x_{i:N}$ of components in directions with vanishing weight $w_{N,i}$ are just given by $x_{i:\emptyset}$ and can thus be skipped. Note that we expect there to be further relations between these coordinates - we will later construct charts on the blow-up by selecting one coordinate in each column such that this subset determines all others.
\end{example}

Let us investigate whether we can cover the blow-up $\Bl_\cW(M)$ by chart domains of good perspectives. Unwittingly, we have already seen an example where this is not the case:

\begin{example}
Consider again the blow-up of Example~\ref{ex:not-unif-sing}. The charts we constructed on the two components then correspond to choices $\mathbf{h}=(1,2)$ and $\mathbf{s}=(1,1)$ and the following tableau:
\begin{center}
\begin{tikzpicture}
  \matrix (m) [matrix of nodes, nodes in empty cells, column sep=5pt, row sep=5pt] {
    \npboxed{1} & 1 &  \\
    1 & \npboxed{2} & 1 \\
    $x_1$ & $x_2$ & $x_3$ \\
  };
  \draw[thick] ([xshift=-5pt, yshift=1pt]m-3-1.north west) -- ([xshift=5pt, yshift=1pt]m-3-3.north east);

  \tableaubox(m)(A)(1:1:2)
  \tableaubox(m)(B)(2:1:3)
\end{tikzpicture}
\end{center}
\vspace{-5mm}
We see that this is not a good perspective since the boxed 2 does not sit at the top of the column. In these coordinates, we had found a singular point $p=(p_A,p_B)\in\Bl_{\cW_A}(M)\times\Bl_{\cW_B}(M)$ - namely, the pair of weighted unit normal vectors in the exceptional divisors of $\cW_A$ and $\cW_B$ with coordinates $p_A=[1,0,0]$ and $p_B=[0,1,0]$, respectively. It can be reached as the limit of the curve $t\mapsto(t^3,t^4,0)$ in the bulk as $t$ goes to zero. We see that it is impossible to choose $\mathbf{h}$ and $\mathbf{s}$ such that they form a good perspective with a chart domain that covers $p$: Since the coordinates of $p_A$ and $p_B$ vanish in all but one component, requiring $p_I\in U^{(\cW_I)}_{\mathbf{h}(I)\mathbf{s}(I)}$ for $I=A,B$ forces our choice of $\mathbf{h}$ and $\mathbf{s}$.
\end{example}

It turns out that this phenomenon could only appear due to the lack of uniform alignment:

\begin{lemma}\label{lem:perspective-covers}
Assume $\cW$ is a uniformly aligned weighting over a separated building set.
Then every point in the blow-up lies in the chart domain for some good perspective.
\end{lemma}

\begin{proof}
Let $p=\{p_G\}_{G\in\cG}\in\Bl_\cW(M)$ be arbitrary. Set $\cN$ to be the control set $\cG_p$. Since this is a nest according to Lemma~\ref{prop:ass-nest-is-nest}, uniform alignment implies that there is a chart $\chi=(x_1, ..., x_m):U\to\R^m$ of $M$ around $b(p)$ that is aligned with all weightings over $\cN$. We now claim that for every $N\in\cN$, there is an $h\in\{1,...,m\}$ such that
\begin{enumerate}[(a)]
\item the weight $w_{N,h}$ doesn't vanish,
\item there is no proper superset $N'\supseteq N$ in $\cN$ whose weight $w_{N',h}$ also doesn't vanish, and
\item $p_N\in U^{(\cW_N)}_{h+}\cup U^{(\cW_N)}_{h-}$ for the chart domains from Definition~\ref{def:smooth-blow-up}.
\end{enumerate}
Assuming, for now, that this is true, we can define the maps $\mathbf{h}:\cN\to\{1,...,m\}$ and $\mathbf{s}:\cN\to\{\pm1\}$ by assembling these individual $h$'s. Condition (c) tells us that we can pick suitable signs $\mathbf{s}(N)$ to guarantee $$ p\in \bigcap\limits_{N\in\cN} \pi^*_N U^{(\cW_N)}_{\mathbf{h}(N)\mathbf{s}(N)}.$$ By condition (a) and (b) we can establish the property (2) of Definition~\ref{def:good-perspective} and deduce that $(\chi,\cN,\mathbf{h},\mathbf{s})$ form a good perspective. Finally, by the definition of the control set, for every $G\in\cG\setminus\cN$ there must exist a subset $N\in\cN$ with $p_N\in\Bl_{\cW_G,\cW_N}(M)$ and thus $p\in U_{\chi\cN\bf hs}.$

\textit{Proof of the claim:} Assume this was false, i.e. there existed an $N\in\cN$ such that for each $h\in\{1,...,m\}$ at least one of the conditions (a), (b) or (c) fails. Note that by construction of $\cN=\cG_p$, $p_N=[n]$ for some non-zero weighted normal vector $n$. A weighted normal vector is zero exactly when the fiber coordinates vanish, i.e. $x_h^{(w_{N,h})}(n)=0$ for all $h$ which do not satisfy (a). But we are going to show that if (b) or (c) fail, it must hold that $x_h^{(w_{N,h})}(n)=0$ and thus find a contradiction to $n$ being non-zero.

Assume (b) fails, so there is a superset $N'\supseteq N$ in $\cN$ with $w_{N',h}\neq 0$. Then we must have $p_N=[n]\not\in\Bl_{\cW_{N'},\cW_N}(M)$ by construction of $\cN$ and thereby $\nu_{{N'},N}(n)=0.$ Since the weight along $x_h$ of neither $N$ nor $N'$ vanish, by uniform alignment they must be equal. But then by Example~\ref{example:induced-map-weighted-normal-bundle}, 
\begin{align*}
x_h^{(w_{N,h})}(n) &= x_h^{(w_{N',h})}\left(\nu_{\cW_{N'},\cW_N}(n)\right)\\
&=x_h^{(w_{N',h})}(0) = 0.
\end{align*}

Assume instead (c) fails. By definition of the chart domains for $p_N=[n]$ in the exceptional divisor, this is equivalent to $x_h^{(w_{N,h})}(n)$ being neither larger nor smaller than zero and we are done.
\end{proof}

\startSubchaption{Charts}\label{ssec:charts}

We are finally ready to define our charts:

\begin{definition}\label{def:building-local-charts}
Fix a weighting $\cW$ along a building set $\cG$ and a good perspective $(\chi,\cN,\mathbf{h},\mathbf{s})$. Define the corner model
$$ \R^m_\mathbf{h} := \{(y_1, ..., y_m)\in\R^m \;|\; \forall N\in\cN: y_{\mathbf{h}(N)}\geq0\},$$
the \textbf{selector map}
\begin{align*}
\mu: \{1, ..., m\} &\to\cN\cup\{\emptyset\}\\
i&\mapsto\max\left(\{N\in\cN\;|\; w_{N,i}\neq 0 \text{ and }\mathbf{h}(N)\neq i\}\cup\{\emptyset\}\right)
\end{align*}
and the \textbf{induced chart}
$$ \chi_{\bf hs}=(x_{1:\bf hs}, ..., x_{m:\bf hs}):U_{\chi\cN\bf hs}\to \R^m_\mathbf{h} $$
by setting
$$ x_{i:\bf hs}:= \begin{cases}
    \left(\mathbf{s}(N)\,x_{i:\mu(i)}\right)^{1/w_{N,i}} & \text{if }\exists N\in\cN:i=\mathbf{h}(N),\\
    x_{i:\mu(i)} &\text{else.}
\end{cases}$$
For each $N\in\cN$, we may also write $t_N := x_{\mathbf{h}(N):\bf hs}$ for the \textbf{control parameter} associated with $N$ to emphasize its significance.
\end{definition}

\begin{examples}\label{ex:building-coords}\ 
\begin{enumerate}
\item Consider again the tableau of coordinates from Example~\ref{ex:coordinatetableau}:
\begin{center}
\begin{tikzpicture}
  \matrix (m) [matrix of nodes, nodes in empty cells, column sep=5pt, row sep=5pt] {
    $\npboxed{x_{1:A}}$ & $\bf x_{2:A}$ & $\npboxed{x_{3:B}}$ & $\bf x_{4:B}$ &  &  &  &  \\
    $\bf x_{1:C}$ & $x_{2:C}$ & $\bf x_{3:C}$ & $x_{4:C}$ & $\npboxed{x_{5:C}}$ & $\npboxed{x_{6:D}}$ & $\bf x_{7:D}$ &  \\
    $x_{1:\emptyset}$ & $x_{2:\emptyset}$ & $x_{3:\emptyset}$ & $x_{4:\emptyset}$ & $\bf x_{5:\emptyset}$ & $\bf x_{6:\emptyset}$ & $x_{7:\emptyset}$ & $\bf x_{8:\emptyset}$ \\
  };
  \draw[thick] ([xshift=-5pt, yshift=1pt]m-3-1.north west) -- ([xshift=5pt, yshift=1pt]m-3-8.north east);

  \tableaubox(m)(A)(1:1:2)
  \tableaubox(m)(B)(1:3:2)
  \tableaubox(m)(C)(2:1:5)
  \tableaubox(m)(D)(2:6:2)
\end{tikzpicture}
\end{center}
Within every column, we have set those entries in bold that are selected by $\mu$ and go into the definition of our new coordinates $x_{i:\mathbf{hs}}$. This is always either the highest non-vanishing entry (which is then exactly equal to $x_{i:\mathbf{hs}}$), or the one right below in case the highest entry is boxed (in which case we take a root with a sign for $x_{i:\mathbf{hs}}$). The central idea of our coordinates is the following: There are enough relations forced between the different components of a point $p=\{p_G\}_{G\in\cG}$ in the blow up $\Bl_\cW(M)$ such that the bold coordinates suffice to determine all the other coordinates in the tableau. Since these determine all the points over the control set, they fix $p$ entirely. Our coordinates are just the bold entries rescaled for convenience.
\item Let us be more concrete for a simpler example of a good perspective $(\chi,\cN,\mathbf{h},\mathbf{s})$ with the following tableaus of weights and coordinates:

\begin{center}
\begin{tikzpicture}
  \matrix (m) [matrix of nodes, nodes in empty cells, column sep=5pt, row sep=5pt] {
    $1$ & $2$ & $\npboxed{1}$ &  &  &  \\
    $1$ & $2$ & $1$ & $2$ & $\npboxed{3}$ & \\
    $x_{1}$ & $x_{2}$ & $x_{3}$ & $x_{4}$ & $ x_{5}$ & $x_{6}$ \\
  };
  \draw[thick] ([xshift=-5pt, yshift=1pt]m-3-1.north west) -- ([xshift=5pt, yshift=1pt]m-3-6.north east);

  \tableaubox(m)(B)(1:1:3)
  \tableaubox(m)(A)(2:1:5)
\end{tikzpicture}\qquad
\begin{tikzpicture}
  \matrix (m) [matrix of nodes, nodes in empty cells, column sep=5pt, row sep=5pt] {
    $\bf x_{1:B}$ & $\bf x_{2:B}$ & $\npboxed{x_{3:B}}$ & &  &  \\
    $x_{1:A}$ & $x_{2:A}$ & $\bf x_{3:A}$ & $\bf x_{4:A}$ & $\npboxed{x_{5:A}}$ & \\
    $x_{1:\emptyset}$ & $x_{2:\emptyset}$ & $x_{3:\emptyset}$ & $x_{4:\emptyset}$ & $\bf x_{5:\emptyset}$ & $\bf x_{6:\emptyset}$ \\
  };
  \draw[thick] ([xshift=-5pt, yshift=1pt]m-3-1.north west) -- ([xshift=5pt, yshift=1pt]m-3-6.north east);

  \tableaubox(m)(B)(1:1:3)
  \tableaubox(m)(A)(2:1:5)
\end{tikzpicture}
\end{center}
Say for simplicity that $\mathbf{s}(A)=\mathbf{s}(B)=1$. In this situation, Definition~\ref{def:building-local-charts} yields the following charts:
\begin{align*}
x_{1:\mathbf{hs}} &:= x_{1:B}\in\R,\\
x_{2:\mathbf{hs}} &:= x_{2:B}\in\R,\\
t_B:= x_{3:\mathbf{hs}} &:= x_{3:A}\in[0,\infty),\\
x_{4:\mathbf{hs}} &:= x_{4:A}\in\R,\\
t_A:=x_{5:\mathbf{hs}} &:= \sqrt[3]{x_{5:\emptyset}}\in[0,\infty),\\
x_{6:\mathbf{hs}} &:= x_{6:\emptyset}\in\R.
\end{align*}
Using the local expressions for the component-wise blow-down maps from Equation~\eqref{eq:local-blow-down}, we see that
$$x_{3:\emptyset}=(x_{3:B})^1\cdot \mathbf{s}(B)=(x_{5:A})^1\cdot x_{3:A} \qquad\text{and}\qquad x_{5:\emptyset}=(x_{5:A})^3\cdot \mathbf{s}(A)$$
must hold in the bulk of the blow-up. Using that the signs in our good perspective are trivial, we can solve for the parameters $x_{5:A}$ and $x_{3:B}$ that control the distance to the exceptional divisors of the individual blow-ups, obtaining
\begin{equation}\label{eq:example-controls-pars}
x_{5:A} = t_A \qquad\text{and}\qquad x_{3:B}=t_A t_B.
\end{equation}
Thus the distance to the exceptional divisor in the $N$-th component of $p=\{p_N\}_{N\in\cN}\in\Bl_{\cW}(M)$ is given by the product of the control parameters of all $N'\subseteq N$.
Using Eqs.~\eqref{eq:example-controls-pars} and~\eqref{eq:local-blow-down} together, we can write the overall blow-down map in terms of the new coordinates $\chi_{\mathbf{hs}}$ as follows:
\begin{equation*}\label{eq:example-blowdown}\begin{aligned}
x_{1:\emptyset} &= (t_At_B)\cdot x_{1:\mathbf{hs}},\\
x_{2:\emptyset} &= (t_At_B)^2\cdot x_{2:\mathbf{hs}},\\
x_{3:\emptyset} &= (t_At_B)\cdot 1,\\
x_{4:\emptyset} &= t_A^2\cdot x_{4:\mathbf{hs}},\\
x_{5:\emptyset} &= t_A^3\cdot 1,\\
x_{6:\emptyset} &= x_{6:\mathbf{hs}}.
\end{aligned}\end{equation*}
Observe how the components normal only to $A$ get acted on by $t_A$, while the components normal to both $A$ and $B$ get acted on by the product $t_At_B$. In either case we act by the uniform weight along each given direction.
\qedhere
\end{enumerate}
\end{examples}

\begin{npar}[Iterated blow-ups]\label{rem:iterative}
While we have defined the blow-up $\Bl_\cW(M)$ with the graph-based approach for the purposes of this article, we expect this to be equivalent to an iterated approach. Concretely, $\Bl_\cW(M)$ could be constructed by blowing up the elements of $\cG$ in any order such that smaller submanifolds appear before larger ones, where at every step we extend all larger submanifolds into the exceptional divisor by taking a closure\footnote{Note that this requires a more general notion of weighting along submanifolds $N$ of a manifold with corners. Compatibility between the weight assignment and the corners must here be guaranteed by a simultaneous local model.}. Indeed, this is already the case for the results in~\cite{FM94} and~\cite{Li09} that we are generalizing.

To motivate our choices in Definition~\ref{def:building-local-charts}, let us demonstrate how to obtain the same charts by an iterated blow-up in the case of Example~\ref{ex:building-coords}(2). Say we first perform the blow-up of $A$ to obtain coordinates $(x_{1:A}, ..., x_{m:A})$. According to the blow-down map from Equation~\eqref{eq:local-blow-down}, $B\setminus A$ is then cut out, at least away from the exceptional divisor, by the equations $$x_{5:A}\cdot x_{1:A} = (x_{5:A})^2\cdot x_{2:A} = x_{5:A}\cdot x_{3:A} = 0.$$
Taking the closure gives the closed submanifold $\overline{B\setminus A}$ cut out by $$ x_{1:A}=x_{2:A}=x_{3:A}=0,$$ including a boundary at $x_{5:A}=0$. We can naturally assign the weights 1, 2, and 1, respectively, to the normal directions. The induced coordinates on the weighted blow-up of $\overline{B\setminus A}$ now match exactly those we constructed before.
\end{npar}

We will spend the rest of this \subchaption{} arguing how the $x_{i:\mathbf{hs}}$ determine all other entries of the tableau of coordinates.

To start, we have relations between the coordinates of a stacked pair of submanifolds $A\subsetneq B$ in the tableau:
\begin{center}
\begin{tikzpicture}
  \matrix (m) [matrix of nodes, nodes in empty cells, column sep=5pt, row sep=5pt] {
    $\vdots$ \\
    $\;$ & $\;$ & $\;$ & $\npboxed{\times}$ & $\dots$\\
    $\star$ & $\star$ & $\star$ & $\;$ & $\;$ & $\;$ & $\npboxed{\phantom{A}}$ & $\dots$ \\
    $\vdots$ \\
  };

  \tableaubox(m)(A)(3:1:7)
  \tableaubox(m)(B)(2:1:4)
\end{tikzpicture}
\end{center}
The following Lemma tells us how the coordinates marked by stars are determined by the control parameter $t_B$ of $B$ and the coordinates directly above, while the coordinate marked by a cross is determined by $t_B$ and the boxed coordinate in the row below. In the former case we can even allow $A=\emptyset$ to get the coordinates after blowing down. In fact, in an iterated approach these conditions can be understood as a \textit{partial} blown-down of one step in the sequence of blow-ups.

\begin{lemma}\label{lem:coord-rels-pair}
Fix a uniformly aligned weighting $\cW$ along a separated building set $\cG$ and a good perspective $(\chi,\cN,\mathbf{h},\mathbf{s})$. 
Consider $A\subsetneq B$ in $\cN\cup\{\emptyset\}$ such that there exists no $C\in\cN$ with $A\subsetneq C\subsetneq B$. Then it holds for $w_{B,i}\neq0$ that
\begin{equation}\label{eq:coords-rels-pair1}
x_{i:A} = t_B^{w_{B,i}}\,\cdot\,\begin{cases}\mathbf{s}(B)&\text{for }i=\mathbf{h}(B),\\x_{i:B}&\text{for }i\neq\mathbf{h}(B),\end{cases}
\end{equation}
and for $A\neq\emptyset$ that
\begin{equation}\label{eq:coords-rels-pair2}
x_{\mathbf{h}(B):B}=t_B\,\cdot\, x_{\mathbf{h}(A):A}.
\end{equation}
\end{lemma}

\begin{proof}
We first show Equation~\eqref{eq:coords-rels-pair2} for $A\neq\emptyset$. In the bulk away from the exceptional divisor it must hold by the coordinate expressions for blow-down maps from Equation~\eqref{eq:local-blow-down} that
$$
x_{\mathbf{h}(B):\emptyset} = \mathbf{s}(B)\,\cdot\, x_{\mathbf{h}(B):B}^{w_{B,\mathbf{h}(B)}} = x_{\mathbf{h}(B):A}\,\cdot\, x_{\mathbf{h}(A):A}^{w_{A,\mathbf{h}(B)}}.
$$
By definition of $t_B,$ we can substitute $x_{\mathbf{h}(B):A}=\mathbf{s}(B)\,\cdot\, t_B^{w_{B,\mathbf{h}(B)}}$, take a $w_{B,\mathbf{h}(B)}$-th root and cancel the sign to obtain the desired Equation~\eqref{eq:coords-rels-pair2}. Here we make use of the weights matching by uniform alignment. Since the blow-up is defined as a closure, the result must also hold on the exceptional divisors.

To show Equation~\eqref{eq:coords-rels-pair1} for $i\neq\mathbf{h}(B)$ we proceed analogously, but starting by inspection of the $i$-th coordinate in the bulk: If $A=\emptyset$, this immediately implies
$$
x_{i:A} = x_{\mathbf{h}(B):B}^{w_{B,i}}\,\cdot\, x_{i:B} = t_{B}^{w_{B,i}}\,\cdot\, x_{i:B}.
$$
If on the other hand $A\neq\emptyset,$ we get
$$
x_{i:\emptyset} = x_{\mathbf{h}(B):B}^{w_{B,i}}\,\cdot\, x_{i:B} = x_{\mathbf{h}(A):A}^{w_{A,i}}\,\cdot\, x_{i:A}.
$$
We can then substitute Equation~\eqref{eq:coords-rels-pair2} and cancel factors to obtain the desired formula. The proof for $i=\mathbf{h}(B)$ follows analogously with $\mathbf{s}(B)$ replacing each $x_{i:B}$, and we are done.
\end{proof}

Applying this Lemma repeatedly gives all entries of the tableau in terms of $\chi_{\mathbf{hs}}$:

\begin{proposition}\label{prop:comp-coords}
Fix a uniformly aligned weighting $\cW$ along a separated building set $\cG$ and a good perspective $(\chi,\cN,\mathbf{h},\mathbf{s})$. Then for each $N\in\cN\cup\{\emptyset\}$, the entry $x_{i:N}$ depends on components $x_{i:\mathbf{hs}}$ (including $t_N=x_{\mathbf{h}(N):\mathbf{hs}}$) of the chart $\chi_{\bf hs}$ as follows:

For any $i\neq\mathbf{h}(N)$ with $w_{N,i}\neq0$ or $N=\emptyset$,
\begin{equation}\label{eq:coord-lemma-two}
x_{i:N} = \prod_{N\subsetneq \hat{N}\in\cN} \left(t_{\hat{N}}\right)^{w_{\hat{N},i}} \cdot
\begin{cases}
 \mathbf{s}(N_0) &\text{if }i=\mathbf{h}(N_0) \text{ for } N_0\in\cN\setminus\{N\}, \\
 x_{i:\bf hs}&\text{else.}
\end{cases}
\end{equation}
For all $N\neq\emptyset$,
\begin{equation}\label{eq:coord-lemma-three}
x_{\mathbf{h}(N):N} = \prod_{N\supseteq \check{N}\in\cN} t_{\check{N}}.
\end{equation}
Finally for any $i$ with $w_{N,i}=0$ or $N=\emptyset$,
\begin{equation}\label{eq:coord-lemma-one}
x_{i:N} = x_{i:\emptyset}.
\end{equation}
\end{proposition}

\begin{proof}
We first prove Equation~\eqref{eq:coord-lemma-one}. It is trivial for $N=\emptyset$. When $N$ is non-empty, it is an immediate consequence of $w_{N,i}=0$ by Paragraph~\ref{par:blow-down-coords}, as the right-hand side is also the $i$-th component of the blow-down map in the $N$-th component.

To prove the remaining two equations by induction, note first that since the underlying building set $\cG$ is separated, the set of subsets $\check{N}\in\cN$ of any fixed $N$ is a chain. Similarly, the set of supersets $\hat{N}\in\cN$ of any fixed $N$ whose weight $w_{\hat{N},i}$ does not vanish is also a chain. We can thus write
$$ \check{N}_{0}\subsetneq \check{N}_1\subsetneq ... \subsetneq \check{N}_{\check{k}} = N \subsetneq \hat{N}_{\hat{k}} \subsetneq ... \subsetneq \hat{N}_{1}. $$
These are exactly the factors that contribute to each of the products in the statement of the lemma.

We get Equation~\eqref{eq:coord-lemma-three} for $x_{\mathbf{h}(N):N}$ for the special case $\check{k}=0$ (i.e., $N$ is a minimum of the nest) by noting that then $\mu(\mathbf{h}(N))=\emptyset$ and thus
$$t_N=x_{\mathbf{h}(N):\bf hs} = x_{\mathbf{h}(N):N}$$
by inspecting the definitions. The general formula for $\check{k}>0$ follows by induction using Equation~\eqref{eq:coords-rels-pair2} as the induction step.

Similarly consider the $\hat{k}=0$ case of Equation~\eqref{eq:coord-lemma-two}, i.e. there is no larger $\hat{N}_1$ with $w_{\hat{N}_1,i\neq 0}$. We need to show $x_{i:N}=x_{i:\bf hs}$ for every $i\neq\mathbf{h}(N)$ with $w_{N,i\neq0}$. But in this case $\mu(i)=N$ (even if $N=\emptyset$), and the equation holds by definition of $x_{i:\bf hs}$.

Now consider the $\hat{k}=1$ case of Equation~\eqref{eq:coord-lemma-two}. We can have two situations: Since $w_{\hat{N}_1,i}\neq 0$ we could have either $\mathbf{h}(\hat{N}_1)=i$ or not. If equality does hold, then we need to show
\begin{equation}\label{eq:target}
x_{i:N} =t_{\hat{N}_1}^{w_{\hat{N}_1,i}} \,\cdot\, \mathbf{s}(\hat{N}_1).
\end{equation}
If $N=\emptyset$, this holds by definition of the coordinates, so assume $N\neq\emptyset$.
In this case, by definition the following equations must hold in the bulk away from both exceptional divisors:
$$
x_{i:\emptyset}= x_{\mathbf{h}(N):N}^{w_N,i} \,\cdot\, x_{i:N} = x_{i:\hat{N}_1}^{w_{\hat{N}_1,i}} \,\cdot\, \mathbf{s}(\hat{N}_1).
$$
Using Equation~\eqref{eq:coords-rels-pair2} for $N\subsetneq \hat{N}_1$ with $\mathbf{h}(\hat{N}_1)=i$, this implies
$$
x_{\mathbf{h}(N):N}^{w_N,i} \,\cdot\, x_{i:N} = x_{\mathbf{h}(N):N}^{w_{\hat{N}_1,i}}\,\cdot \, t_{\hat{N}_1}^{w_{\hat{N}_1,i}} \,\cdot\, \mathbf{s}(\hat{N}_1).
$$
By uniform alignment, we can cancel the first factors to get Equation~\eqref{eq:target} in the bulk.  But by continuity and the definition of the blow-up as a closure, this equation must also hold on the exceptional divisors.
If on the other hand $\mathbf{h}(\hat{N}_1)\neq i$, then we can first apply Equation~\eqref{eq:coords-rels-pair1} and then the above $\hat{k}=0$ case to conclude
$$
x_{i:N}=t_{\hat{N}_1}^{w_{\hat{N}_1,i}} \,\cdot\, x_{i:\hat{N}_1} = t_{\hat{N}_1}^{w_{\hat{N}_1,i}} \,\cdot\, x_{i:\bf hs}.
$$
In either case, Equation~\eqref{eq:coord-lemma-two} holds for $\hat{k}=1$. Any larger $\hat{k}$ can now be reached by induction using Equation~\eqref{eq:coords-rels-pair1}, as we can't have $i=\mathbf{h}(\hat{N})$ for any superset that itself has a superset with non-vanishing weight along $i$.
\end{proof}

The previous proposition can be expressed more succinctly:

\begin{corollary}\label{cor:comp-coords}
Fix a uniformly aligned weighting $\cW$ along a separated building set $\cG$ and a good perspective $(\chi,\cN,\mathbf{h},\mathbf{s})$. For $i=1,...,m$ and $N\in\cN$, the coordinate $x_{i:N}$ of a point $p\in\Bl_\cW(M)$ parametrized by $\chi_{\mathbf{hs}}(p)=(y_1, ..., y_m)\in\R^m_{\mathbf{h}}$ is given by
$$
x_{i:N}\circ\chi^{-1}_{\bf hs}(y_1, ..., y_m) = \begin{cases}
\prod\limits_{N\supseteq \check{N}\in\cN} y_{\mathbf{h}(\check{N})} &\text{if } i=\mathbf{h}(N), \\
\prod\limits_{\substack{\hat{N}\in\cN\\ w_{N,i}=0 \lor N\subsetneq \hat{N}}} \left(y_{\mathbf{h}(\hat{N})}\right)^{w_{\hat{N},i}} \mathbf{s}(N_0) &\text{if }i=\mathbf{h}(N_0) \text{ for } N_0\in\cN\setminus\{N\}, \\
\prod\limits_{\substack{\hat{N}\in\cN\\  w_{N,i}=0 \lor N\subsetneq \hat{N}}} \left(y_{\mathbf{h}(\hat{N})}\right)^{w_{\hat{N},i}} y_i &\text{else.}
\end{cases}
$$
\end{corollary}

\begin{proof}
This follows immediately from Eqns.~\eqref{eq:coord-lemma-two} and~\eqref{eq:coord-lemma-three}, and using Equation~\eqref{eq:coord-lemma-one} as an intermediate step when $w_{N,i}=0$. Note that in this case the condition $N\subsetneq \hat{N}$ is dropped when determining which $\hat{N}$ to range over in the lower two products, as $\emptyset\subsetneq\hat{N}$ is trivially true.
\end{proof}

At this point we have collected all the ingredients to give a proof that the charts yield a smooth structure, i.e. our main Theorem~\ref{thm:manifold-structure}. The interested reader can find this in Appendix~\ref{app:proof-weighted-building-set}, but we do not require the details of the proof in what follows.

\startSubchaption{Smooth structure}\label{ssec:smooth-structure}

Now that we know that the blow-up supports a smooth structure under reasonable assumptions, we will discuss compatibility of the stratification with it as well as how the structure relates to the ambient manifold $\prod_{G\in\cG}\Bl_{\cW_G} (M).$

As for the stratification, it coincides under the charts with the canonical stratification of the corner model $\R_\mathbf{h}^m$:

\begin{proposition}\label{prop:strat-local-coords}
Fix a
uniformly aligned weighting $\cW$ along a separated building set $\cG$ and any good perspective $(\chi,\cN,\mathbf{h},\mathbf{s})$ covering $U_{\chi\cN\mathbf{hs}}$. Then the stratification from Definition~\ref{def:ass-nest-and-strat} gets mapped to the canonical stratification of the corner model $\R^{m}_{\mathbf{h}}$. In particular, any point $p\in U_{\chi\cN\mathbf{hs}}$ satisfies
$$ \cG_p = \{N\in \cN \;|\; t_N(p)=0\}. $$
\end{proposition}

Note that due to the regularity of the corner model, all Whitney conditions are also satisfied.

\begin{proof}
By definition of the control set, we need to show for all $N\in\cN$ that
$$ t_N=0 \iff b_\cW(p)\in N \text{ and } \nexists H\in\cG: H\subsetneq N \land p_H\in \Bl_{N,H}(M).$$
We can equivalently write that no such $H$ exists in $\cN$ on the right-hand side, as $p\in U_{\chi\cN\mathbf{hs}}$ implies that every $p_H$ for $H\in\cG\setminus\cN$ is determined by some $p_{H'}$ for $H'\in\cN$.
This allows us to rewrite the right-hand side using local coordinate expressions induced by $\chi$ for the blow-down map and the subset $\Bl_{N,H}(M)$:
$$
x_{\mathbf{h}(N):\emptyset}(p)= 0 \text{ and } \nexists H\in\cG: H\subsetneq N \land x_{\mathbf{h}(N):H}(p)>0.
$$
Given Equation~\eqref{eq:coord-lemma-two}, this is exactly equivalent to the left-hand side $t_N=0$.
\end{proof}

\begin{corollary}\label{cor:nests-are-control-sets}
For a uniformly aligned weighting $\cW$ along a separated building set $\cG$ and an arbitrary nest $\emptyset\neq\cN\subseteq\cG$, the stratum $\Bl_{\cW}(M)_\cN$ is non-empty.
\end{corollary}

\begin{proof}
Take some $p\in\cap\cN$. By Lemma~\ref{lem:existence-perspective-for-nest}, we can find $\mathbf{h}$ and $\mathbf{s}$ to obtain a good perspective $(\chi,\cN,\mathbf{h},\mathbf{s})$. By Proposition~\ref{prop:strat-local-coords}, the origin $q=\chi_{\mathbf{hs}}^{-1}(0)$ in the induced coordinates now satisfies $\cG_q=\cN$.
\end{proof}

To discuss the relation to the ambient manifold structure on $\prod_{G\in\cG}\Bl_{\cW_G} (M)$, we temper our expectations with an example:

\begin{example}
Consider the uniformly aligned weighting $\cW$ along the separated building set $\cG=\{A,B\}$ in $\R^2$ with the following weight tableau in standard coordinates $(x_1,x_2)\in\R^2$:

\begin{center}
\begin{tikzpicture}
  \matrix (m) [matrix of nodes, nodes in empty cells, column sep=5pt, row sep=5pt] {
    \npboxed{2} & 0 \\
    2 & \npboxed{1} \\
    $x_1$ & $x_2$\\
  };
  \draw[thick] ([xshift=-5pt, yshift=1pt]m-3-1.north west) -- ([xshift=5pt, yshift=1pt]m-3-2.north east);

  \tableaubox(m)(A)(2:1:2)
  \tableaubox(m)(B)(1:1:1)
\end{tikzpicture}
\end{center}
\vspace{-5mm}
The boxed weights correspond to a good perspective of $\mathbf{h}(A)=2$, $\mathbf{h}(B)=1$ and, for example, signs $\mathbf{s}(A)=\mathbf{s}(B)=1$. According to Corollary~\ref{cor:comp-coords}, the coordinates on the components $\Bl_A(M)$ and $\Bl_B(M)$ are determined by the coordinates $t_B=x_{1:\mathbf{hs}}$ and $t_A=x_{2:\mathbf{hs}}$ on $\Bl_\cW(M)$ as follows:
\begin{equation}\label{eq:comp-ex}
\begin{pmatrix}x_{1:A}\\x_{2:A}\end{pmatrix} = 
\begin{pmatrix}t_B^2\\t_A\end{pmatrix}, \qquad
\begin{pmatrix}x_{1:B}\\x_{2:B}\end{pmatrix} = 
\begin{pmatrix}t_A t_B\\t_A\end{pmatrix}.
\end{equation}
As a sanity check, here are the coordinates after blowing down expressed in any of the other charts:
$$
\begin{pmatrix}x_{1:\emptyset}\\x_{2:\emptyset}\end{pmatrix} = 
\begin{pmatrix}t_A^2 t_B^2\\t_A\end{pmatrix} = 
\begin{pmatrix}x_{2:A}^2x_{1:A}\\x_{2:A}\end{pmatrix} = 
\begin{pmatrix}x_{1:B}^2\\x_{2:B}\end{pmatrix}.
$$
Equation~\eqref{eq:comp-ex} tells us that in these coordinates, the blow-up sits in $\Bl_A(M)\times\Bl_B(M)$ like the subset $$\{(t_B^2,t_A,t_At_B,t_A)\;|\; t_A,t_B\geq0\}\subset \R\times [0,\infty)^2\times\R.$$ Dropping the superfluous fourth component and reparametrizing in terms of the first two gives the subset $$\{(x,y,\sqrt{x}y)\;|\;x,y\geq0\},$$
which at $x=y=0$ has a smooth singularity like the Whitney umbrella - in particular the limits of the tangent planes along the two boundaries of the corner do not match.
\end{example}

We learn from this example that we cannot expect the blow-up to satisfy any reasonable smooth embedding condition in the ambient manifold with corners. However, the singularity we found is contained in a small part of the boundary, not topological in nature and luckily the worst that can happen. To make this precise, we recall the notion of a \textit{weak submanifold}, which was introduced by Ammann, Mougel and Nistor~\cite{AMN21} for the same purpose in the unweighted setting:

\begin{definition}
    A subset $S$ of a manifold $M$ \textit{without} corners or boundary is a \textbf{weak submanifold of $M$} if around every $p\in S$ there is a smooth chart $\phi: U\to \R^n$ of $M$ and $0\leq l\leq m\leq n$ such that $\phi(S\cap U)=(\R^m_l\times\{0\})\cap \phi(U)$.

    A subset $S$ of a manifold $M$ \textit{with corners} is a \textbf{weak submanifold of $M$} if around every $p\in S$ there is a smooth chart $\phi:U\to\R^n_k=[0,\infty)^k\times\R^{n-k}$ of $M$ such that $\phi(S\cap U)$ is a weak submanifold of $\R^n$ in the sense above.

\end{definition}

This definition boils down to a subset $S$ which in local coordinates itself looks like a manifold with corners. Compared to other common notions of submanifolds of manifolds with corners, this is not requiring any further compatibility between the corners of $S$ and those of the ambient $M$. This notion is exactly what is satisfied by the image of an embedding of manifolds with corners (see \cite[][Proposition 2.13]{AMN21}). Note that it also makes sense to speak of a \textit{topological} weak submanifold by replacing smooth charts with continuous charts.

\begin{proposition}\label{prop:weak-submfd-structure}
Let $\cW$ be a uniformly aligned weighting along a separated building set $\cG$.
\begin{enumerate}
\item Topologically, $\Bl_\cW(M)$ is a weak submanifold of $\prod_{G\in\cG}\Bl_{\cW_G} (M)$.
\item Consider some $p\in U_{\chi\cG_p\mathbf{hs}}$. Then $\Bl_\cW(M)$ is a weak manifold in the smooth category around $p$ if and only if there exists no pair $A\subsetneq B$ in $\cG_p$ with $w_{B,\mathbf{h}(B)}>1$ and $t_A(p)=t_B(p)=0$.
\end{enumerate}
\end{proposition}

\begin{proof}
Note that around any point $p\in\Bl_\cW(M)$ it suffices to check whether the projection to $\prod_{G\in\cG_p}\Bl_{\cW_G} (M)$ is embedded since all other components in $\prod_{G\in\cG}\Bl_{\cW_G} (M)$ are smoothly determined by these. Consider a point $p\in U_{\chi\cG_p\mathbf{hs}}$.

\textit{Regarding (1)}: To show that the projection to $\prod_{G\in\cG_p}\Bl_{\cW_G} (M)$ is topologically an embedded weak submanifold, it is sufficient to show that for each coordinate component $x_{j:\mathbf{hs}}$ there is some $N\in\cG_p$ such that $x_{j:N}$ coincides with $x_{j:\mathbf{hs}}$ up to a homeomorphism of $\R$. We claim that this is true for any $N$ if $\mu(j)=\emptyset$, and otherwise true for $N=\mu(j)$.

\textit{If $\mu(j)=\emptyset$,} then we know that either (i) $j=\mathbf{h}(N_0)$ for some minimal element $N_0$ of $\cG_p$ or (ii) for all $N'\in\cG_p$ that $j\neq\mathbf{h}(N')$ and that the weight $w_{N',j}$ vanishes. In either of these cases, Corollary~\ref{cor:comp-coords} tells us that $x_{j:N}=x_{j:\mathbf{hs}}$, and we are done.

\textit{If $\mu(j)\neq\emptyset$} we set $N=\mu(j)$. We now have either (i) for all larger $\hat N\supsetneq N$ the weight $w_{\hat N,j}$ vanishes or (ii) there exists some $N_0$ of $\cG_p$ with $j=\mathbf{h}(N_0)$. Corollary~\ref{cor:comp-coords} yields $x_{j:N} = x_{j:\mathbf{hs}}$ in case (i) and $x_{j:N}= x_{j:\mathbf{hs}}^{w_{N,j}}\cdot\mathbf{s}(N_0)$ in case (ii). Since $w_{N,j}\geq 1$, we still have equality up to a homomorphism.

\textit{Regarding (2)}:
We first argue that we obtain a smooth weak submanifold whenever no $A\subsetneq B$ as in (2) exist. We can follow the same strategy as for (1), except that we now need to ensure that the derivative of $x_{j:N}$ by $x_{j:\mathbf{hs}}$ does not vanish close to $p$. The only case where this is not automatic is case (ii) when $\mu(j)\neq\emptyset$ and $t_{N_0}=x_{j:\mathbf{hs}}$ vanishes at $p$, since the weight in $x_{j:N}= x_{j:\mathbf{hs}}^{w_{N,j}}\cdot\mathbf{s}(N_0)$ may be larger than one. In this case (i.e. assuming $t_{N_0}$ vanishes at $p$ and $w_{N,j}=w_{N_0,j}>1$), by applying the condition in (2) to any pair $A\subsetneq N_0$ we get that all $t_{A}$ cannot vanish at $p$. 
While the derivative of the component $x_{j:N}$ by $x_{j:\mathbf{hs}}=t_{N_0}$ vanishes, we claim that it now does not for the component $x_{j:N_0}$ (and this is sufficient to still have an immersion and thus locally an embedding). Calculating it using our favorite Corollary~\ref{cor:comp-coords} gives $x_{j:N_0}=\prod_{A\subseteq N_0} t_{A}$. Since none of the factors vanish, the derivative by $t_{N_0}$ does not either.

We are left to show that we have a singularity at $p$ whenever there exist $A\subsetneq B$ with $w_{B,\mathbf{h}(B)}>1$ and $t_A(p)=t_B(p)=0$. To see this, it is sufficient that $t_B=x_{\mathbf{h}(B):\mathbf{hs}}$ does not immerse into any of the component coordinates $x_{j:N}$ at $p$. Inspecting all the possible expressions for $x_{j:N}$ in Corollary~\ref{cor:comp-coords}, $t_B$ almost always appears raised to the power $w_{B,\mathbf{h}(B)}>1$, so cannot immerse when $t_B(p)=0$. It only appears linearly when $j=\mathbf{h}(N)$ for some $N\supseteq B$. All these $x_{j:N}$ are given by a product that also involves $t_A$. Since this vanishes at $p$, $t_B$ cannot immerse into these $x_{j:N}$ either and we are done.
\end{proof}

\startSubchaption{Morphisms of weighted building sets}\label{ssec:morphisms}

While one may consider different notions of morphism between weighted building sets, we give a fairly broad definition sufficient for our applications:

\begin{definition}\label{def:morphism}
    Let $(M,\cG,\cW)$ and $(\widetilde{M},\widetilde{\cG},\widetilde{\cW})$ be weighted building sets. A \textbf{morphism of weighted building sets} $$(\Phi,\phi):(M,\cG,\cW)\to(\widetilde{M},\widetilde{\cG},\widetilde{\cW})$$ is a smooth map $\Phi:M\to\widetilde{M}$ and a map $\phi:\widetilde{\cG}\to\cG$ such that for all $\widetilde{G}\in\widetilde{\cG}$, $\Phi$ is a weighted morphism between $(M,\cW_{\phi(\widetilde{G})})$ and $(\widetilde{M},\widetilde{\cW}_{\widetilde{G}})$.
\end{definition}

Given some $\Phi$, a natural candidate for $\phi$ is given by $\phi(\widetilde{G}):=\Phi^{-1}(\widetilde{G})$ if the latter lies in $\cG$. Yet in \Subchaption~\ref{ssec:filtered-bundle-fulton-macpherson} we will see in our application to configuration spaces of fiber bundles a case where a different $\phi$ appears naturally.

For this reason, we do not make a fixed choice of $\phi$. Instead, our definition captures what data we need to make the construction of an induced map between the blow-ups possible. The concrete role of $\phi$ is to pick out which component of the domain has to determine a given component of the codomain:

\begin{definition}
Let $(\Phi,\phi):(M,\cG,\cW)\to(\widetilde{M},\widetilde{\cG},\widetilde{\cW})$ be a morphism of weighted building sets and consider the canonical projections $$\pi_H:\prod_{G\in{\cG}}\Bl_{{\cW}_G}(M)\to\Bl_{{\cW}_H}(M)$$
for every $H\in{\cG}$. Define the open subset
$$\Bl_{(\Phi,\phi)}(M):= \bigcap_{\widetilde{G}\in\widetilde{\cG}} \pi_{\phi(\widetilde{G})}^*\Bl_{\Phi^{\widetilde{G}}}(M)\;\;\subseteq\;\; \Bl_{\cW}(M)$$
where $\Phi^{\widetilde{G}}$ denotes $\Phi$ when viewed as a weighted morphism between $(M,\cW_{\phi(\widetilde{G})})$ and $(\widetilde{M},\widetilde{\cW}_{\widetilde{G}})$, as well as
the \textbf{induced map}
\begin{align*}
    b_{(\Phi,\phi)}:\Bl_{(\Phi,\phi)}(M) &\to \Bl_{\widetilde{\cW}}(\widetilde{M})\\
    \{p_G\}_{G\in\cG} &\mapsto \{b_{\Phi^{\widetilde{G}}}(p_{\phi(\widetilde{G})})\}_{\widetilde{G}\in\widetilde{\cG}}.
\end{align*}
\end{definition}

This map is continuous by construction. When our main theorem yields a smooth structure on domain and codomain, we also have smoothness:

\begin{proposition}\label{prop:morphism-induced}
Let $(\Phi,\phi):(M,\cG,\cW)\to(\widetilde{M},\widetilde{\cG},\widetilde{\cW})$ be a morphism of weighted building sets. Assume that $\cG$ and $\widetilde{\cG}$ are separated and that $\cW$ and $\widetilde{\cW}$ are uniformly aligned. Then $b_{(\Phi,\phi)}$ is smooth with respect to the canonical structure of manifolds with corners on $\Bl_{(\Phi,\phi)}(M)\subseteq\Bl_{\cW}(M)$ and $\Bl_{\widetilde{\cW}}(\widetilde{\cW})$.
\end{proposition}

\begin{proof}
Take some point $p\in\Bl_{(\Phi,\phi)}(M)$ and its image $\tilde{p}:=b_{(\Phi,\phi)}(p)$ and find good perspectives such that $p\in U_{\chi\cG_p\mathbf{hs}}$ and $\tilde{p}\in U_{\tilde{\chi}\widetilde{\cG}_{\tilde{p}}\mathbf{\tilde{h}\tilde{s}}}$ for charts $\chi=(x_1,...,x_m)$ and $\tilde{\chi}=(\tilde{x}_1, ..., \tilde{x}_m)$. We have to show for all $\tilde{N}\in\widetilde{\cG}_{\tilde{p}}$ and $i\in\{1,...,m\}$ that
$$\tilde{y}_i:=\tilde{x}_{i:\tilde{N}}\circ b_{(\Phi,\phi)}\circ \chi^{-1}_{\mathbf{hs}}(y_1,...,y_m)\qquad\text{for }i=1,...,m$$ depends smoothly on $(y_1,...,y_m)$ close to $\chi_{\mathbf{hs}}(p)$. We know that there is an $N\in\cG_p$ such that $N\subseteq\phi(\tilde{N})$ and $p_N\in\Bl_{\phi(\tilde{N}),N}(M)$ since $\cG_p$ is a control set. Let $\tilde{s}=\mathbf{\tilde{s}}(\tilde{N})$ and $\tilde{h}=\mathbf{\tilde{h}}(\tilde{N})$ and write
\begin{align*}
\tilde{y}_i
&= \tilde{x}^{(\widetilde{\cW}_{\tilde{N}})}_{i:\tilde{h}\tilde{s}}\circ \pi_{\tilde N}\circ b_{(\Phi,\phi)} \circ \chi^{-1}_{\mathbf{hs}}(y_1,...,y_m)\\
&= \tilde{x}^{(\widetilde{\cW}_{\tilde{N}})}_{i:\tilde{h}\tilde{s}}\circ b_{\Phi^{\tilde{N}}}\circ \pi_{\phi(\tilde{N})} \circ \chi^{-1}_{\mathbf{hs}}(y_1,...,y_m)\\
&= \tilde{x}^{(\widetilde{\cW}_{\tilde{N}})}_{i:\tilde{h}\tilde{s}}\circ b_{\Phi^{\tilde{N}}}\circ b_{\phi(\tilde{N}),N}\circ \pi_{N} \circ \chi^{-1}_{\mathbf{hs}}(y_1,...,y_m)\\
&= \tilde{x}^{(\widetilde{\cW}_{\tilde{N}})}_{i:\tilde{h}\tilde{s}}\circ b_{\left(\Phi^{\tilde{N}}\circ\Id\right)}\circ \pi_{N} \circ \chi^{-1}_{\mathbf{hs}}(y_1,...,y_m)
\end{align*}
where the identity is understood as a weighted morphism $\Id:(M,\cW_N)\to(M,\cW_{\phi(\tilde{N})})$.

Finally we define the shorthand
$$z_l:=x_{l:N}\circ\chi^{-1}_{\mathbf{hs}}(y_1,...,y_m) \qquad\text{for }l=1,...,m.$$
We have now arranged it so that the dependence of each $z_l$ on $(y_1,...,y_m)$ is smooth by Corollary~\ref{cor:comp-coords} and the dependence of each $\tilde{y}_i$ on $(z_1,...,z_m)$ is smooth by Proposition~\ref{prop:induced-map-smooth} as the local representation of an induced map.
\end{proof}

%% file: filtered-manifolds.tex
As an application of blow-ups along weighted building sets we present a blown-up configuration space analogous to that of Fulton and MacPherson, but adapted to a filtered structure on a manifold.

In \Subchaption{}s~\ref{ssec:filtered-submanifolds} and~\ref{ssec:weighting-filtered-submfd} we recall the definition of filtered manifolds and discuss how submanifolds adapted to this filtration induce weightings. In \Subchaption~\ref{ssec:fulton-macpherson} we introduce the Fulton-MacPherson building set on $M^s$ in detail, and in \Subchaption{}~\ref{ssec:fulton-macpherson-weights} we equip this building set with weightings induced by the filtered structure on $M$. In \Subchaption{}~\ref{ssec:filtered-fulton-macpherson} we discuss the resulting weighted blow-up. Finally, we conclude in \Subchaption~\ref{ssec:filtered-bundle-fulton-macpherson} with a discussion on how this construction should be adapted when the original space is the total space of a fiber bundle to prepare for a future application to jet bundles.

\startSubchaption{Filtered manifolds}\label{ssec:filtered-submanifolds}

\begin{definition}[\cite{Mo93}]\label{def:filtered-manifold}
    A \textbf{filtered manifold} of \textbf{depth $r\in\N$} is a manifold $M$ equipped with a \textbf{Lie filtration}\index{Lie filtration}, i.e. a filtration of $TM$ by subbundles
    $$ TM= H_{-r} \supseteq ... \supseteq H_{-1}\supseteq H_0 = 0_M $$
    such that the Lie bracket of two sections of $H_{-i}$ and $H_{-j}$, respectively, lands in $H_{-i-j}$.
\end{definition}

\begin{examples}\ 
\begin{enumerate}
\item The trivial filtration on $M$ is given by $TM=H_{-1}\supseteq H_0=0_M$.

\item A distribution $\cD$ is \textit{equiregular} if taking Lie brackets generates a non-decreasing sequence of subbundles of $TM$, and \textit{bracket-generating} if they eventually span all of $TM$. An equiregular bracket-generating distribution thus turns $M$ into a filtered manifold with $H_{-1}=\cD$. Contact, even-contact and Engel manifolds are particular examples. Even if $\cD$ is not bracket-generating, we by convention define a filtered structure on $M$ by adding $TM$ as the highest degree of the filtration after $\cD$ stabilizes.
\item We want to point out a concrete example that may be of particular interest: A \textit{constant-rank sub-Riemannian structure} on a manifold $M$ is a distribution $\cD\subseteq TM$ equipped with a smooth family of scalar products (see e.g.~\cite{ABB19}). If it is equiregular, we again obtain a filtered structure. 
\item For manifolds $M$ and $N$, the \textit{jet bundle} $J^k(M,N)$ of order $k\in\N_0$ canonically carries a regular bracket-generating distribution, the \textit{Cartan distribution}. It is spanned by the tangent spaces to all holonomic sections. Defining a blown-up configuration space of jets is our motivation for this theory and will be discussed in a future article~\cite{future}.
\item Filtered manifolds come up in the study of operators that are elliptic in a weighted sense.
A central role is played by the tangent groupoid of a filtered manifold, which was constructed in generality by van Erp and Yuncken~\cite{EY17}. It can be seen as the deformation to the normal cone analog to our configuration space in the special case of pairs ($s=2$).\qedhere
\end{enumerate}
\end{examples}

The following definition will be useful for both Lie and more general vector bundle filtrations:

\begin{definition}
    For any filtration $F_{-r}\supsetneq ... \supsetneq F_0 = 0_M$ of vector bundles over a manifold $M$, we define its \textbf{associated graded bundle} as the vector bundle
    $$\gr(F_\bullet) := F_{-1} \oplus  F_{-2}/F_{-1}\oplus ... \oplus F_{-r}/F_{-r+1}.$$
    This is trivially a graded vector bundle in the sense of Grabowski-Rotkiewicz by letting scalars $\lambda\in\R$ act on each degree $\gr(F_\bullet)_{-i}:= F_{-i}/F_{-i+1}$ by multiplying with $\lambda^{i}$. We say the \textbf{weight sequence} of $F_\bullet$ is the sequence
    $$ (w_1, ..., w_{\rk F_{-r}}) = (1, ..., 1, 2, ..., 2, ..., r, ..., r)  $$
    in which each natural number $n$ is repeated $\dim\gr(F_\bullet)_{-n}$ times.
\end{definition}

For a Lie filtration, the Lie bracket equips the fibers of the associated graded bundle with the structure of a Lie algebra.

\startSubchaption{Weighting induced by a filtered submanifold}\label{ssec:weighting-filtered-submfd}

In a sense, a filtration provides weights to \textit{all} directions, \textit{everywhere} in the manifold. For submanifolds $N$ adapted to the filtration, we recall in this subsection that this data naturally restricts to a weighting along $N$ by considering only the weights normal to $N$. This was established by Loizides and Meinrenken for the more general \textit{singular} Lie filtrations in their follow-up paper~\cite{LM22}.
While our construction of a blown-up configuration space should generalize to their setting, we restrict our attention here to the more concrete filtered manifolds.

We first recall what it means for a submanifold to be adapted to a Lie filtration:

\begin{definition}\label{def:filtered-submanifold}
    A \textbf{filtered submanifold} $N$ of a filtered manifold $(M,H_\bullet)$ is an embedded submanifold such that the quotients $$F_{-i}:=\left(H_{-i}|_N+TN\right)/TN$$ have constant rank and thus form filtration of the normal bundle $\nu N$ by subbundles. The \textbf{weight sequence} of $N$ is defined to be the weight sequence of $F_\bullet$.
\end{definition}

Note that our definition of a filtered submanifold is equivalent to the more common condition that the intersections $H_i\cap TN$ have constant rank, which is used e.g. in~\cite{HH18} and under the name \textit{equiregular submanifold} in~\cite{Gr96}. By construction, we have the relation
\begin{equation}\label{eq:filtration-quotient}
\gr(F_\bullet) = \gr(H_\bullet)|_N \,/\, TN,
\end{equation}
where the quotient on the right is performed at every degree.

A Lie filtration gives a natural notion of functions that vanish to a given order on a filtered submanifold:

\begin{definition}\label{def:filtered-submfd-weighting}
Given a filtered submanifold $N$ of $(M,H_\bullet)$, we define a filtration $\cW_{H_\bullet, N, \bullet}$ of $C^\infty(M)$ by setting
\begin{align*}
\cW_{H_\bullet, N, k}=\{f\in C^\infty(M) \;|\;  &X_1...X_l f|_N=0 \text{ for all } l\in\N, \alpha\in\N^l \\
&X_i\in\Gamma(H_{-\alpha_i}) \text{ with } |\alpha|< k \}
\end{align*}
and define the corresponding set $\cW_{H_\bullet,N}\subseteq T^{(\infty)}M$ as in Equation~\eqref{eq:weighting-from-filtration}.
\end{definition}

This is indeed a weighting:

\begin{theorem}[see Theorem 4.1 of~\cite{LM22}]\label{thm:filtered-is-weighting}
Let $(M,H_\bullet)$ be a filtered manifold of depth $r$ and $N$ a filtered submanifold. Then the set $\cW_{H_\bullet,N}\subseteq T^{(\infty)}M$ is a weighting of order $r$ for a weight sequence $(w_1, ..., w_m)$ consisting of $\dim N$ zeros followed by the weight sequence of $N$.

For any adapted coordinates $(z_1, ..., z_m)$ for these weights, it holds that
    \begin{equation}\label{eq:hh-coords}
H_{-i}|_n+T_nN = \{v\in T_n M \;|\; dz_j(v) = 0 \text{ for all $j$ with } w_j>i\}.     
    \end{equation}
\end{theorem}

\begin{remark}
To connect to the language of~\cite{LM22}, recall that $H_\bullet$ induces a singular Lie filtration $$\fX(M)=\cH_{-r}\supseteq...\supseteq \cH_{-1}\supseteq\cH_0=0$$
of the sheaf $\fX(M)$ of vector fields on $M$ by setting $\cH_{-i}:=\{X\in\fX(M)\;|\; X(M)\subseteq H_{-i}\}$. Their notion of $N$ being \textit{$\cH$-clean} generalizes the filtered submanifold condition.

In the case of regular Lie filtrations, the statement essentially already follows from a normal form theorem due to Haj-Higson~\cite[][Lemma 7.5]{HH18}, based on earlier work by Choi-Ponge~\cite{CP19}. This establishes that for every frame adapted to the filtered submanifold, we can find a coordinate system that is adapted to the submanifold and the normal directions of the frame, with a sensible behavior at higher orders.
\end{remark}

The linearized weighted normal bundle of this weighting can be characterized conveniently:

\begin{lemma}\label{lem:weighted-normal-is-graded}
For a closed filtered submanifold $N$ of $(M,H_\bullet)$ let $F_\bullet$ be as in Definition~\ref{def:filtered-submanifold}. It holds canonically that
$$ \nu_\mathrm{lin}{\cW_{H_\bullet,N}} \simeq \gr(F_\bullet) $$
as graded vector bundles over $N=\supp\cW_{H_\bullet, N}$.
\end{lemma}

\begin{proof}
We claim that the desired isomorphism is induced by the identity on representatives, i.e. by sending any given $[v]\in \nu_\mathrm{lin}{\cW_{H_\bullet,N}}$ for $v\in TM|_{N}$ to $[v]\in \gr(F_\bullet)$. Well-definition and bijectivity can be checked locally for any choice $\{x_1,...,x_m\}$ of local coordinates on $M$ that are adapted to $N$ as a filtered submanifold (and thus also adapted to $\cW_{H_\bullet,N}$) by comparing the expressions in Equation~\eqref{eq:hh-coords} and in Paragraph~\ref{props-linearized-bundles}(1).
\end{proof}

By Paragraph~\ref{props-linearized-bundles}(3), this also means that 
\begin{equation}\label{eq:noncanonical-weighted-normal}
\nu{\cW_{H_\bullet,N}} \simeq \gr(F_\bullet)
\end{equation}
holds, but this may depend on a choice of coordinates.

One consequence of this is a simplified check for adaptation of coordinates:

\begin{proposition}\label{prop:adapted-to-filtered-sub}
Let $N$ be a filtered submanifold of $(M,H_\bullet)$ with the weight sequence $(w_1, ..., w_{\codim N})$. Consider coordinates $$\chi=(y_1,...,y_{\dim N}, z_1, ..., z_{\codim N}):U\to \R^m$$ of $M$ such that the following holds:
\begin{enumerate}
    \item For all $j\in\{1, ..., m\}$, we have $z_j\in\cW_{H_\bullet, N, w_j}.$
    \item For all $i\in\{0, ..., r\}$ and $n\in N\cap U$ we have
    \begin{equation*}
H_{-i}|_n+T_nN = \{v\in T_n M \;|\; dz_j(v) = 0 \text{ for all $j$ with } w_j>i\}.
    \end{equation*}
\end{enumerate}
Then these coordinates are adapted to $\cW_{H_\bullet,N}$ with weights zero and $w_j$ for each $y_j$ and $z_j$, respectively.
\end{proposition}

\begin{proof}
Let $\cW$ be the weighting induced by equipping $\chi$ with the given weights. Our goal is to show that, locally, $\cW=\cW_{H_\bullet,N}$. By condition (1) and Definition~\ref{def:filtered-submfd-weighting} we see straight away that $\cW_k\subseteq\cW_{H_\bullet,N,k}$ or equivalently $\cW\supseteq\cW_{H_\bullet,N}$. We can thus check equality of the weightings using just the linear data by Lemma~\ref{lem:linear-check}. But due to Lemma~\ref{lem:weighted-normal-is-graded} and the coordinate expression for the normal filtration from Paragraph~\ref{props-conormal-filtration}(1), this comes down exactly to condition (2).
\end{proof}

\startSubchaption{The Fulton-MacPherson building set}\label{ssec:fulton-macpherson}

We interpret $M^s$ as an ordered configuration space of $s$ points in $M$ that allows collisions.
The diagonals of $M^s$ form the primordial example $\cG^{FM}$ of a building set, which we will recall and discuss in this \subchaption. While this already underlies the construction of Fulton and MacPherson in~\cite{FM94}, the conception as a building set originates in~\cite{CP95,Li09}. We introduce the convenient notion of \textit{factorization} $\operatorname{fac}\cP$ and provide examples and self-contained proofs of the central properties of $\cG^{FM}$. Purely for convenience, we assume here that $M$ is connected (compare~\ref{remark:building-set-def}).

\begin{definition}\label{def:filtered-building-set}
For any natural number $s\geq1$ define the indexing set $$\cI:=\{I\subseteq\upto{s} \;|\; |I|\geq2\}$$
and for any $I\in\cI$ the \textbf{diagonal} $$\Delta_I=\{(p_1,...,p_s)\in M^s \;|\; p_i=p_j\text{ for all }i,j\in I\}.$$
Together, they form the \textbf{Fulton-MacPherson building set} $$\cG^{FM}:= \{\Delta_I\;|\; I\in\cI\}.$$
\end{definition}

As a reminder, $\upto{s}:=\{1,...,s\}$. We want to point out that when $|M|>1$,
\begin{equation}\label{eq:nesting-indices}
    I_1\subsetneq I_2 \iff \Delta_{I_1}\supsetneq\Delta_{I_2}
\end{equation}
holds. We will need some more notation and language:

\begin{definition} Let $\cP$ be a subset of $\cI$.
\begin{enumerate}
\item We define the \textbf{polydiagonal} $\Delta_\cP:=\cap_{I\in\cP}\Delta_I.$ By convention, $\Delta_\emptyset=M^s.$
\item We say $\cP$ is \textbf{connected} if any two numbers in $\cup\cP$ can be connected by a sequence $a_1, ..., a_k$ in $\{1, ..., s\}$ satisfying the following: For any two successive elements $a_i,a_{i+1}$ there is an element of $\cP$ containing them.
\item We define the \textbf{factorization} $\fac\cP\subseteq\cI$ of $\cP$ to consist of the unions of the largest connected subsets of $\cP$.
\end{enumerate}
\end{definition}

Clearly all elements of $\Arr_{\cG^{FM}}$ can be written as $\Delta_\cP$, but not necessarily uniquely. We can view $\fac\cP$ as a canonical choice that results by merging all elements of $\cP$ that have any overlap. Before showing that $\cG^{FM}$ is indeed a building set as discussed in Section~\ref{sec:building-sets}, we offer some examples:

\begin{figure}
    \centering
    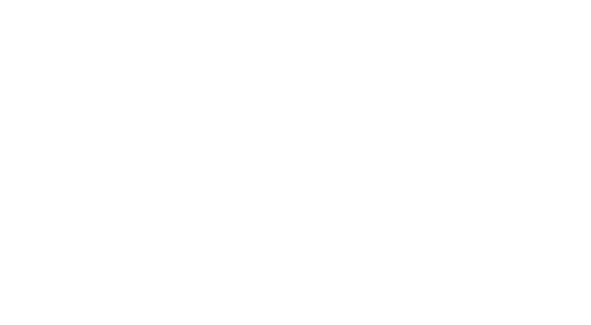
    \caption{The elements of $\cG^{FM}=\{ \textcolor{olive}{\Delta_{23}}, \textcolor{ForestGreen}{\Delta_{12}}, \textcolor{purple}{\Delta_{13}}, \Delta_{123} \}$ depicted for $s=3$ and $M=(0,1)$.}
    \label{fig:diagonals}
\end{figure}

\begin{examples}\label{ex:fm-building-set}\ 
\begin{enumerate}
\item For $s=1$, $\cG^{FM}=\emptyset$ is just the trivial building set.
\item For $s=2$, $\cG^{FM}=\{\Delta_{\{1,2\}}\}$ contains one unique diagonal.
\item For $s=3$, the arrangement $\Arr_{\cG^{FM}}$ induced by $\cG^{FM}$ is the following lattice:
\begin{equation*}
\begin{tikzpicture}
\node[shape=rectangle, draw=black] (123) at (0,0) {123};
\node[shape=rectangle, draw=black] (12) at (-2,1.3) {12} edge (123);
\node[shape=rectangle, draw=black] (23) at (0,1.3) {23} edge (123);
\node[shape=rectangle, draw=black] (13) at (2,1.3) {13} edge (123);
\node (root) at (0,2.6) {$M^3$} edge (12) edge (23) edge (13);
\end{tikzpicture}
\end{equation*}
For readability, here and in any concrete examples that follow we label diagonals $\Delta_I$ with the string of digits in $I$ so that e.g. $12=\Delta_{\{1,2\}}$. We will always omit any edges that are implied by transitivity and put elements of the building set in boxes (as opposed to elements of the full arrangement).

When taking $M$ to be the open interval $(0,1)$, the diagonals are depicted in Figure~\ref{fig:diagonals}. The intersection $$12\cap23\cap13=123$$ of three surfaces in a line is not like an intersection of coordinate subspaces and certainly not transversal. However, this only means that $\cG':=\{\Delta_{12},\Delta_{23},\Delta_{13}\}$ is not a separated building set: Since $\{\Delta_{123}\}$ is also included in $\cG^{FM}$, it is its only $\cG^{FM}$-factor and thereby trivially a transverse intersection.
\item At $s=4$, for the first time $\cG^{FM}$ is not a closed building set, i.e. the induced arrangement $\Arr_{\cG^{FM}}$ is strictly larger than $\cG^{FM}\cup\{M\}$. For example, the polydiagonal $$12\cap34=\Delta_{\{\{1,2\}, \{3,4\}\}}$$ is not a diagonal. The lattice $\Arr_{\cG^{FM}}$ quickly grows in complexity:
\begin{equation*}
\begin{tikzpicture}
\node[shape=rectangle, draw=black] (1234) at   (0,  0) {1234} edge (1234);
\node[shape=rectangle, draw=black] (234) at  (1.6,1.5) {234} edge (1234);
\node[shape=rectangle, draw=black] (134) at  (4.8,1.5) {134} edge (1234);
\node[shape=rectangle, draw=black] (124) at (-1.6,1.5) {124} edge (1234);
\node[shape=rectangle, draw=black] (123) at (-4.8,1.5) {123} edge (1234);
\node (12x34) at (-3.2,1.5) {$12\cap34$} edge (1234);
\node (13x24) at (0,1.5) {$13\cap24$} edge (1234);
\node (14x23) at (3.2,1.5) {$14\cap23$} edge (1234);
\node[shape=rectangle, draw=black] (12) at (-5,3.5) {12} edge (123) edge (124) edge (12x34);
\node[shape=rectangle, draw=black] (23) at (-3,3.5) {23} edge (123) edge (234) edge (14x23);
\node[shape=rectangle, draw=black] (24) at (-1,3.5) {24} edge (124) edge (234) edge (13x24);
\node[shape=rectangle, draw=black] (13) at  (1,3.5) {13} edge (123) edge (134) edge (13x24);
\node[shape=rectangle, draw=black] (34) at  (3,3.5) {34} edge (134) edge (234) edge (12x34);
\node[shape=rectangle, draw=black] (14) at  (5,3.5) {14} edge (124) edge (134) edge (14x23);
\node (root) at (0, 5) {$M^4$} edge (12) edge (23) edge (24) edge (13) edge (34) edge (14);
\end{tikzpicture}
\end{equation*}
Note that this graph always incorporates the lattices for lower $s$ as subgraphs.\qedhere
\end{enumerate}
\end{examples}

\begin{proposition}\label{lem:fulton-mac-nice-building}
Let $M$ be connected. Then $\cG^{FM}$ is a separated building set over $M^s$. Furthermore, the set of $\cG^{FM}$-factors of any $\Delta_\cP\in\Arr_{\cG^{FM}}$ is given by
$$\{\Delta_I\in\cG^{FM}\;|\; I\in\fac\cP\}.$$
\end{proposition}

\begin{proof}
The diagonals are clearly non-empty and of positive codimension. The intersection $$\Delta_\cP=\{(p_1,...,p_s)\in M^s\;|\; p_i=p_j \text{ for all }I\in\cP, \text{ and }i,j\in I\}$$ is evidently clean and furthermore connected since $M$ is connected, so $\cG^{FM}$ is a building set.

For every $I\in\fac\cP$, we have $\Delta_\cP\subseteq\Delta_I$. To show that $\Delta_I$ is a $\cG^{FM}$-factor, assume there was a $I'\in\cI$ such that $\Delta_\cP\subseteq\Delta_{I'}\subsetneq\Delta_I$. In order for the first inclusion to hold, $I'$ must be the union of some connected subset of $\cP$. By definition of the factorization as containing the largest such subsets and Equation~\eqref{eq:nesting-indices}, this is a contradiction to the second (strict) inclusion.

To see that $\cG^{FM}$ is separated, note that by construction the elements of a factorization $\fac\cP$ are pairwise disjoint. This immediately yields that the intersection of the corresponding diagonals is transverse.
\end{proof}

Let us characterize $\cG^{FM}$-nests for this concrete building set.

\begin{definition}
We call a subset $\cN\subseteq\cI$ a \textbf{nest} if for all $A,B\in\cN$ with $A\cap B\neq\emptyset$ we have either $A\subseteq B$ or $A\supseteq B$.
\end{definition}

\begin{lemma}\label{fulton-mac-nest}
The map that assigns to every $\cP\subseteq\cI$ the subset $$\widehat{\cP}:=\{\Delta_I\;|\; I\in\cP \}\subseteq\cG^{FM}$$ is a one-to-one correspondence of subsets, and $\cN\subseteq\cI$ is a nest if and only if $\widehat{\cN}$ is a $\cG^{FM}$-nest.
\end{lemma}

\begin{proof}
The map is a one-to-one correspondence of subsets of due to Equation~\eqref{eq:nesting-indices}. We prove the two directions of the equivalence separately.

\textit{Regarding $(\Rightarrow)$:} Let $\cN$ be a nest and $\cP\subseteq\cN$ a non-empty subset of non-comparable elements corresponding to a non-empty subset $\widehat{\cP}\subseteq\widehat{\cN}$ of non-comparable elements. In this case, $\cP=\fac\cP.$ Since $\widehat{\fac\cP}=\cG^{FM}_{\widehat{\cP}}$ by Lemma~\ref{lem:fulton-mac-nice-building}, we are done by the third characterization of nests in Proposition~\ref{prop:char-nest}.

\textit{Regarding $(\Leftarrow)$:} Let $\widehat{N}$ be a $\cG^{FM}$-nest and take two elements $A,B\in\cN$ with $A\cap B\neq\emptyset$. Assuming towards contradiction that $A$ and $B$ are non-comparable means that $\widehat{P}:=\{\Delta_A,\Delta_B\}$ is a non-empty subset of $\widehat{N}$ of non-comparable elements. Its intersection $\cap\widehat{P}=\Delta_{\{A,B\}}$ is just $\Delta_{A\cup B}$ since $A$ and $B$ overlap. This is a contradiction to the second characterization of nests in Proposition~\ref{prop:char-nest}.
\end{proof}

\begin{npar}[Nests as trees]\label{rem:nests-as-trees}
As Fulton and MacPherson already point out, nests also stand in one-to-one correspondence with (not necessarily connected) trees whose leaves are labeled by elements of $\upto{s}$. For example, the nest $$\cN=\{\{1,2,3\},\;\{5,6\},\;\{7,8,9\},\;\{5,6,7,8,9\}\}$$
for $s=9$ corresponds to the following tree:
\begin{equation*}
\begin{tikzpicture}
\node (1) at (0.9*1, 0) {1};
\node (2) at (0.9*2, 0) {2};
\node (3) at (0.9*3, 0) {3};
\node (4)[circle, fill=black, inner sep=0.5pt] at (0.9*4, 1) {};
\node (4l) at (0.9*4, 0.7) {4};
\node (5) at (0.9*5, 0) {5};
\node (6) at (0.9*6, 0) {6};
\node (7) at (0.9*7, 0) {7};
\node (8) at (0.9*8, 0) {8};
\node (9) at (0.9*9, 0) {9};

\node (123) at (0.9*2, 1) {};
\node (56) at (0.9*5.5, 0.7) {};
\node (789) at (0.9*8, 0.7) {};
\node (56789) at (0.9*7, 1) {};

\draw (1) -- (123.center);
\draw (2) -- (123.center);
\draw (3) -- (123.center);
\draw (5) -- (56.center);
\draw (6) -- (56.center);
\draw (56.center) -- (56789.center);
\draw (789.center) -- (56789.center);
\draw (7) -- (789.center);
\draw (8) -- (789.center);
\draw (9) -- (789.center);
\end{tikzpicture}
\end{equation*}
Assigning to every non-leaf node of this tree all labels that lie under it at any level recovers all elements of the nest $\cN$.
\end{npar}

Note that the minimal element $\Delta_{\{1, ..., s\}}$ of $\cG^{FM}$ is commonly called the \textit{small diagonal} in contrast to the \textit{large diagonal} $\cup\cG^{FM}$. The latter contains all collided configurations, including the partial collisions. The following examples show that flags and nests can be understood to encode all essentially different orders of such collisions, motivating the stratification of the exceptional divisor in Proposition~\ref{prop:stratification}:

\begin{examples}\label{ex:fm-flags-nests}\ 
\begin{enumerate}
\item For $s=3$, we have the following possible $\cG^{FM}$-flags up to permutation of the labels:
\begin{align*}
F_1: &\qquad \emptyset \\
F_2: &\qquad \emptyset\subseteq12 \\
F_3: &\qquad \emptyset\subseteq123 \\
F_4: &\qquad \emptyset\subseteq123\subseteq12
\end{align*}
The first non-zero step of each flag specifies which points are collided, while the full flag determines an order of collision. Concretely, $F_3$ could be interpreted as having all points collide \textit{at the same time}. The flag $F_4$ on the other hand corresponds to points one and two colliding first, and point three joining them after. Viewing these collisions more concretely as paths that move from the bulk of $M^3$ onto the small diagonal $\Delta_{123}$, $F_4$ would correspond to a path that asymptotically approaches $\Delta_{12}$.

Since $\cG^{FM}$ is closed for $s=3$, nests stand in one-to-one correspondence with these flags and encode the same data.

\item For $s=4$, we have the following possible $\cG^{FM}$-flags up to permutation of the labels:
\begin{align*}
F_1: &\qquad \emptyset & F_7: & \qquad \emptyset\subseteq1234\subseteq123 \\
F_2: &\qquad \emptyset\subseteq12 & F_8: & \qquad \emptyset\subseteq1234\subseteq123\subseteq12  \\
F_3: &\qquad \emptyset\subseteq123 & F_9: & \qquad \emptyset\subseteq12\cap34  \\
F_4: &\qquad \emptyset\subseteq1234 & F_{10}: & \qquad \emptyset\subseteq12\cap34\subseteq12  \\
F_5: &\qquad \emptyset\subseteq123\subseteq12 & F_{11}: & \qquad \emptyset\subseteq1234\subseteq12\cap34 \\
F_6: &\qquad \emptyset\subseteq1234\subseteq12 & F_{12}: & \qquad \emptyset\subseteq1234\subseteq12\cap34\subseteq12
\end{align*}
Flags $F_1$ to $F_8$ are analogous to the $s=3$ case. The remaining flags are interesting as they now contain the polydiagonal $12\cap34$. $F_{11}$ can, for example, be understood as a simultaneous collision of 12 and 34 followed by a collision of all of them. $F_{12}$ on the other hand has the collision of 34 happen after that of 12. For our purposes, the order of these initial collisions does not matter. Passing to nests forgets this superfluous information: By replacing true polydiagonals with their $\cG^{FM}$-factors, both $F_{11}$ and $F_{12}$ yield the same nest $\{1234,12,34\}.$ Except for a similar phenomenon for $F_{9}$ and $F_{12}$, all other flags produce different nests.
Note that nests of the closure of the Fulton-MacPherson building set (i.e. the one that also contains polydiagonals) would correspond exactly to flags, thus remembering full orders of collisions.\qedhere
\end{enumerate}
\end{examples}

\startSubchaption{The weighted Fulton-MacPherson building set}\label{ssec:fulton-macpherson-weights}

From here on out we fix a connected filtered manifold $(M,H_\bullet)$ of degree $r$ and $s\geq1$ an integer to define the \textit{weighted} Fulton-MacPherson building set induced by the filtration on $M$. We first obtain a filtration $H^s_\bullet$ on $M^s$ by setting $$H^s_i:=(H_i)^{\oplus s}.$$
We show that the diagonals are filtered submanifolds in order to consider the induced weightings along $\cG^{FM}$:

\begin{lemma}\label{lem:diags-are-filtered}
The diagonals $\Delta_I\in\cG^{FM}$ are filtered submanifolds of $(M^s, H^s_\bullet)$.
\end{lemma}

\begin{proof}
    This is a direct consequence of the product structure of $H^s$.
    We need to show that for every fixed $i\in\{1, .., r\}$, the dimension of the vector space $$H^s_{-i}|_p\cap T_p\Delta_I$$ is constant in $p\in\Delta_I$. By definition of the diagonals, this intersection is equal to $$ \left\{ (X_1, ..., X_s)\in \left(H_{-i}|_p\right)^s \;|\; X_i=X_j\text{ for all }i,j\in I \right\}. $$ We can count the dimension $(1+s-|I|)\cdot\rk H_{-i}$ independently of $p$.
\end{proof}

\begin{definition}\label{def:filtered-weighted-building-set}
We define the \textbf{weighted Fulton-MacPherson building set $\cW^{FM}$} on $M^s$ by setting $$\cW^{FM}_{\Delta_I} := \cW_{H^s_\bullet,\Delta_I}\qquad\text{ for each }\Delta_I\in\cG^{FM}.$$
\end{definition}

It is not automatic that this is a weighted building set: To start, we need to show that for every collection of diagonals $\cP\subseteq\cI$ that intersect in a single diagonal $\Delta_{\cup\cP}$, the corresponding intersection of $\{\cW^{FM}_{\Delta_I}\}_{I\in\cP}$ equals $\cW^{FM}_{\Delta_{\cup\cP}}$. This by itself is clear by Definition~\ref{def:filtered-submfd-weighting}. It is not immediately clear that this intersection is clean, which we will see as a side-effect of constructing convenient coordinates (Proposition~\ref{prop:fm-reg-unif-align}).

As we have seen in Example~\ref{ex:fm-building-set}(3), the elements of $\cG^{FM}$ do not necessarily intersect like coordinate subspaces. As a collection of weightings, $\cW^{FM}$ thus has no chance of being aligned. However, we will show that the nests \textit{are} uniformly aligned, allowing us to conclude that $\cW^{FM}$ is uniformly aligned as a weighted building set (Lemma~\ref{lem:fm-nests-aligned}). This is the main task ahead of us for the rest of this subsection.

Due to Proposition~\ref{prop:char-uniform-align}, uniform alignment comes down to compatibility of the linear data that we already extracted in Lemma~\ref{lem:weighted-normal-is-graded}. While this is also at the core of what follows (see Lemma~\ref{lem:subtree-adapted}), we want to provide explicit adapted coordinates on $M^s$ that make use of the product structure.

Let us start by defining coordinates adapted to the weighting along the diagonal of $M^2$ as a stepping stone towards the general case:

\begin{definition}\label{def:offset-coords}
Let $(M,H_\bullet)$ be a filtered manifold and $p\in \Delta:=\Delta_{\{1,2\}}\subsetneq M\times M$. We say that coordinates $$(x_1, ..., x_m, \Delta x_1, ..., \Delta x_m): U\to \R^{2m}$$
are \textbf{offset coordinates} if
\begin{enumerate}
    \item they are adapted to $\cW_{H^2_\bullet, \Delta}$ such that $(x_1,...,x_m)$ and $(\Delta x_1, ..., \Delta x_m)$ respectively get assigned vanishing weights and the weight sequence of $H_\bullet$, and
    \item $(x_1, ..., x_m)$ depend only on the first factor of $M\times M$.
\end{enumerate}
\end{definition}

Such coordinates always exist:

\begin{lemma}\label{lem:prod-coords}
For every filtered manifold $(M,H_\bullet)$ and point $p\in \Delta_{\{1,2\}}\subsetneq M\times M$, there exist offset coordinates of $M\times M$ close to $p$.
\end{lemma}

\begin{proof}
Take some adapted coordinates $$\chi=(y_1,...,y_m, z_1, ..., z_m):U\to \R^{2m}$$
of $\cW_{H^2_\bullet, \Delta}$ around $p$ that must exist by Theorem~\ref{thm:filtered-is-weighting}. By reordering we can assume without loss of generality that $(y_1,...,y_m)$ are the directions with vanishing weights. Due to the product structure of $H^2_\bullet$, the remaining directions must be assigned weights from the weight sequence of $H_\bullet$.
 
Let $\iota:M\hookrightarrow M\times M$ be the diagonal inclusion and $\pi_1:M\times M\to M$ be the projection to the first factor. We propose
$$ x_l:= y_l\circ\iota\circ\pi_1 \qquad\text{and}\qquad \Delta x_l :=z_l $$
as candidates for offset coordinates that obviously satisfy condition (2). These really are coordinates in a small enough neighbourhood: Since all $dz_l$ vanish on $T\Delta,$ by dimensional reasons $(dy_1,...,dy_m)$ must be injective on $T\Delta$ for $\chi$ to be coordinates. But on $T\Delta$, $d(\iota\circ\pi_1)$ is an isomorphism, so $(dx_1,...,dx_m)$ is also injective on it.

Furthermore, the modified coordinates are still adapted to $\cW_{H^2_\bullet, \Delta}$: By definition, this only depends on the coordinate directions normal to $N$, which remain untouched.
\end{proof}

We want to combine such offset coordinates to construct convenient charts close to any given diagonal of $M^s$ for general $s$. We require some more combinatorial machinery for this:

\begin{definition}
A \textbf{forest} $\prec$ is a structure of a partial order on $\upto{s}$ such that for every $k\in\upto{s}$ there is at most one $k'\in\upto{s}$ with $k\prec k'$ and no $l\in\upto{s}$ with $k\prec l\prec k'$. If it exists, this element $k'$ is called the \textbf{parent} of $k$ and will be denoted $\operatorname{par}^\prec(k)$.
The sets of \textbf{children} and \textbf{roots}, i.e. elements with and without parents, are denoted by $\Ch^\prec$ and $\Rt^\prec\subseteq\upto{s}.$
A subset $T\subseteq\upto{s}$ is a \textbf{$\prec$-subtree} if for all $k_1, k_2\in T$ the following holds:
\begin{enumerate}
\item For all $k_0\in\upto{s}$ with $k_1\prec k_0\prec k_2$ also $k_0\in T$, and
\item there exists a $p\in T$ such that $k_1\prec p$ and $k_2\prec p$.
\end{enumerate}
In particular, for every $k_0\in\upto{s}$ we have the \textbf{descendant subtree}
$$T_{\preceq k_0}=\{k\in\upto{s}\;|\; k\preceq k_0\}.$$
\end{definition}

The name is, of course, chosen since the Hasse diagram of $(\upto{s},\prec)$ consists of a number of trees under each root. A $\prec$-subtree is a connected subgraph of the Hasse diagram. Given such a forest, our idea is to build coordinates on $M^s$ such that each child component is specified using an offset relative to their parent:

\begin{definition}\label{def:forest-offset-coords}
Fix $s\geq2$. Given a point $p\in M^s$, let $\prec$ be a forest such that $p_k=p_l$ if and only if $k$ and $l$ are in a common subtree. Consider the projections
\begin{align*}
\pi_k: M^s&\to M, &   \rho_l: M^s&\to M^2,   \\
(q_1, ..., q_s) &\mapsto q_k,
&
(q_1, ..., q_s) &\mapsto (q_{\operatorname{par}^\prec(l)}, q_l).
\end{align*}

For all roots $k\in\Rt^\prec$ pick arbitrary coordinates
$$(\hat x_{k1}, ..., \hat x_{km}): U_k\to \R^{m}$$
on open $U_k\subseteq M$ containing $p_k$. For all children $l\in\Ch^\prec$, pick offset coordinates
$$(\hat x_{l1}, ..., \hat x_{lm}, \Delta \hat x_{l1}, ..., \Delta \hat x_{lm}): U_l\to \R^{2m}$$
on open $U_l\subseteq M$ containing $\pi_l(p)$.

Define $U\subseteq M^s$ as the open set
$$\bigcap_{k\in\Rt^\prec} \pi_k^{-1}(U_k) \;\cap\; \bigcap_{l\in\Ch^\prec} \pi_l^{-1}(U_l) $$
and shrink it further around $p$ until it contains no diagonal $\Delta_I$ with $p\not\in\Delta_I$.
We now say that the collection
$$\left(\{x_{k1}, ..., x_{km}\}_{k\in\Rt^\prec}, \{\Delta x_{l1}, ..., \Delta x_{lm}\}_{l\in\Ch^\prec}\right): U\to \R^{sm}$$
defined by the pullbacks $x_{kj}:=\hat x_{kj}\circ \pi_k$ and $\Delta x_{lj}:=\Delta \hat x_{lj} \circ \rho_l$  form \textbf{$\prec$-offset coordinates}.
\end{definition}

It is clear that these maps form coordinates: For any point $q\in U$, the roots $q_k$ can be recovered smoothly from $(x_{k1}, ..., x_{km})(q)$ by construction, and every child component $q_l$ can be recovered smoothly from $(\Delta x_{l1}, ..., \Delta x_{lm})(q)$ once $q_{\operatorname{par}^\prec(l)}$ is determined.

Moreover, these coordinates are adapted to the filtration and all diagonals formed by subtrees of the forest:

\begin{lemma}\label{lem:subtree-adapted}
Fix a forest $\prec$ with $\prec$-offset coordinates
$$\left(\{x_{k1}, ..., x_{km}\}_{k\in\Rt^\prec}, \{\Delta x_{l1}, ..., \Delta x_{lm}\}_{l\in\Ch^\prec}\right): U\to \R^{sm}$$
on $U\subseteq M^s$. Then these coordinates are adapted to $\cW_{\Delta_I}^{FM}$ for any $\prec$-subtree $I$.

In particular, if $i$ is the maximal element of $I$ we have for $q\in U$ that
\begin{equation}\label{eq:cut-out-diag}
q\in\Delta_I \qquad\iff\qquad (\Delta x_{l1}, ..., \Delta x_{lm})(q)=0 \text{ for all } l\in I\setminus\{i\}.
\end{equation}
Each $\Delta x_{lj}$ for $l\in I\setminus\{i\}$ gets assigned the $j$-th weight of the weight sequence of $H_\bullet$, while all other coordinates (i.e. $x_{kj}$ for $k\in\Rt^\prec$ and $\Delta x_{lj}$ for $l\in\Ch^\prec\setminus (I\setminus \{i\})$) are tangential to $\Delta_I$ and thus have vanishing weight.
\end{lemma}

\begin{proof}
Let $I$ be a $\prec$-subtree with maximal element $i$. $\Delta_I$ can be written as the intersection of all $$\Delta_l:=\Delta_{\{l,\operatorname{par}^\prec(l)\}}\qquad \text{ for }l\in L:=\Ch^\prec\;\cap\; (I\setminus\{i\}).$$ This intersection is transverse because $\prec$ is a forest. Moreover, the $\prec$-offset coordinates are adapted to each $\cW^{FM}_{\Delta_l}$: To this end, note that $\cW^{FM}_{\Delta_l}$ is the pull-back of the weighting $\cW_{H^2_\bullet,\Delta}$ along the map $\rho_l:M^s\to M^2$ from Definition~\ref{def:forest-offset-coords}. The coordinates $(\Delta x_{l1}, ..., \Delta x_{lm})$ normal to $\Delta_l$ are exactly constructed by this pull-back, and there is no additional condition on the remaining coordinates tangential to $\Delta_l$.

It now follows by Lemma~\ref{lem:extending-coordinates} that these coordinates are also adapted to the intersection $\cW^{FM}_{\Delta_I}=\cap_{l\in L}\cW^{FM}_{\Delta_l}$ if we can show that the map $$\bigoplus_{l\in L}\nu^*_\mathrm{lin}\cW^{FM}_{\Delta_l}|_{\Delta_I}\to\nu^*_\mathrm{lin}\cW^{FM}_{\Delta_I}$$ is an injection. By Equation~\eqref{eq:filtration-quotient} and Lemma~\ref{lem:weighted-normal-is-graded}, this is just the map $$\bigoplus_{l\in L}\left(\gr(H_\bullet)|_{\Delta_I}/{T\Delta_l}\right)\to\gr(H_\bullet)|_{\Delta_I}/{T\Delta_I}.$$ Its injectivity at each degree of the associated graded bundle is a direct consequence of transversality of the intersection $\Delta_I=\cap_{l\in L}\Delta_l$.

The weight assignments follow immediately from the weights in Lemma~\ref{lem:extending-coordinates}.
\end{proof}

We can use this to build weighted coordinates aligned with two different diagonals:

\begin{lemma}\label{lem:fm-pairs-aligned}
Let $I,J\in\cI$. Then the weightings $\cW^{FM}_{\Delta_I}$ and $\cW^{FM}_{\Delta_J}$ are uniformly aligned. If further $I\cap J\neq\emptyset$ holds, then they intersect cleanly in $\cW^{FM}_{\Delta_{I\cup J}}$.
\end{lemma}

\begin{proof}
If $I$ and $J$ are disjoint, then $\Delta_I$ and $\Delta_J$ intersect transversely and the weightings over them must be uniformly aligned by Lemma~\ref{lem:intersection-props-supports}(3).

If $I$ and $J$ do overlap, let $p=(p_1, ..., p_s)\in \Delta_I\cap \Delta_J=\Delta_{I\cup J}$ be arbitrary and pick $i\in I\cap J$. Define a forest by letting each element of $(I\cup J)\setminus\{i\}$ be a child of $i$ and making all other elements of $\upto{s}$ roots. Then both $I$ and $J$ are $\prec$-subtrees and we obtain their uniform alignment in $\prec$-offset coordinates by Lemma~\ref{lem:subtree-adapted}. Furthermore, $I\cup J$ is also a $\prec$-subtree, making these coordinates adapted to $\cW^{FM}_{\Delta_{I\cup J}}$ as well. Since the weights match in all directions normal to $\Delta_{I\cup J}$, we can conclude that $\cW^{FM}_{\Delta_I}\cup\cW^{FM}_{\Delta_J}=\cW^{FM}_{\Delta_{I\cup J}}$ holds and is a clean intersection.
\end{proof}

To generalize this strategy to nests, we need forests compatible with them.

\begin{definition}
We say a nest $\cN\subseteq\cI$ is \textbf{covered} by a forest $\prec$ if
\begin{enumerate}
\item every element of the nest is a subtree of the forest and conversely
\item for every $k\in\upto{s}$ the descendant subtree $T_{\preceq k}\in\cN$ is in the nest if and only if $k$ has at least one child.
\end{enumerate}
In this case, we define for every $N\in\cN$ the set of \textbf{$N$-controls} as
$$\Ct^\prec_N:=\{n\in N\;|\; \text{$N$ is the smallest element of $\cN$ containing $n$ as a non-root}\}$$
\end{definition}

Clearly, $\{\Rt^\prec\}\cup\{\Ct^\prec_N\}_{N\in\cN}$ forms a partition of $\upto{s}$.

\begin{example}\label{ex:fm-nest-forest}
Consider again for $s=9$ the nest $$\cN=\{\{1,2,3\},\;\{5,6\},\;\{7,8,9\},\;\{5,6,7,8,9\}\}$$ from Paragraph~\ref{rem:nests-as-trees}. As previously discussed, a nest encodes \textit{asymptotic orders} of collisions - in this case, $\{5,6\}$ and $\{7,8,9\}$ first collide within themselves, and subsequently the two resulting points collide.
The following is a forest that covers this nest, with the roots circled:
\vspace{5mm}
\begin{equation*}
\begin{tikzpicture}
\node (1) at (0.9*1, -0.9) {1};
\node[circle, draw, minimum size=5mm, inner sep=0mm] (2) at (0.9*2, 0) {2};
\node (3) at (0.9*3, -0.9) {3};
\node[circle, draw, minimum size=5mm, inner sep=0mm] (4) at (0.9*4, 0) {4};
\node (5) at (0.9*4, -1.8) {5};
\node (6) at (0.9*5, -0.9) {6};
\node[circle, draw, minimum size=5mm, inner sep=0mm] (7) at (0.9*6, 0) {7};
\node (8) at (0.9*7, -0.9) {8};
\node (9) at (0.9*8, -0.9) {9};

\draw (2) -- (1);
\draw (2) -- (3);
\draw (6) -- (5);
\draw (7) -- (6);
\draw (7) -- (8);
\draw (7) -- (9);
\end{tikzpicture}
\end{equation*}

This would correspond to offset coordinates
$$(\Delta x_1, x_2, \Delta x_3, x_4, \Delta x_5, \Delta x_6, x_7, \Delta x_8, \Delta x_9 )$$
where each $\Delta x_k$ is relative to the component labeled by the parent of $k$. These coordinates allow us to describe collisions in the order the nest encodes: We first send $\Delta x_5$ to zero to collide $\{5,6\}$, then send $\Delta x_8$ and $\Delta x_9$ to zero to collide $\{7,8,9\}$ and finally send $\Delta x_6$ to zero to collide the two groups. That this is possible for every covering forest depends crucially on their defining condition (1). Note that the sets of controls
$$
\Ct^\prec_{\{1,2,3\}} = \{1,3\}, \qquad
\Ct^\prec_{\{5,6\}} = \{5\}, \qquad
\Ct^\prec_{\{7,8,9\}} = \{8,9\}, \qquad
\Ct^\prec_{\{5,6,7,8,9\}} = \{6\}
$$
label exactly the parameters that need to vanish at every step of this process, which is why they will eventually appear in our local model for the blow-up.

Note that the tree corresponding to $\cN$ from Paragraph~\ref{rem:nests-as-trees} is related, but distinct from the forest above. Indeed the same nest may be covered by many different forests and vice versa. Given the condition on descendant subtrees\footnote{In the definition of a covering forest, it would have sufficed to only ask that every element of the nest is a subtree in order to produce weighted coordinates. We add the extra condition on descendant subtrees purely for convenience such that we can encode the data in a single picture, as the forests we produce in Lemma~\ref{lem:covering-forest-exists} satisfy it anyway.}, one may try to recover the nest from the forest as the set $$\{T_{\preceq k_0}\;|\; k_0 \text{ has at least one child} \}\subseteq\cI.$$ However, for the above forest this would fail to produce the element $\{7,8,9\}\in\cN$. In order to simultaneously encode forest and nest in a single picture, we can insert extra intermediary nodes with the same label as their parents:
\vspace{5mm}
\begin{equation*}
\begin{tikzpicture}
\node (1) at (0.9*1, -0.9) {1};
\node[circle, draw, minimum size=5mm, inner sep=0mm] (2) at (0.9*2, 0) {2};
\node (3) at (0.9*3, -0.9) {3};
\node[circle, draw, minimum size=5mm, inner sep=0mm] (4) at (0.9*4, 0) {4};
\node (5) at (0.9*4, -1.8) {5};
\node (6) at (0.9*5, -0.9) {6};
\node[circle, draw, minimum size=5mm, inner sep=0mm] (7) at (0.9*6, 0) {7};
\node (77) at (0.9*7, -0.9) {7};
\node (8) at (0.9*6, -1.8) {8};
\node (9) at (0.9*8, -1.8) {9};

\draw (2) -- (1);
\draw (2) -- (3);
\draw (6) -- (5);
\draw (7) -- (6);
\draw[dashed] (7) -- (77);
\draw (77) -- (8);
\draw (77) -- (9);
\end{tikzpicture}
\end{equation*}

We can recover the forest by collapsing the dashed lines again, but now we can also recover the nest as exactly the set of all labels of descendant subtrees rooted at any node with children. We can also read off the controls of a given $N\in\cN$ in this picture: They are exactly the first-degree children of the root of the corresponding tree that are not connected by a dashed line. 
\end{example}

It is always possible to find a covering forest:

\begin{lemma}\label{lem:covering-forest-exists}
Every nest $\cN\subseteq\cI$ is covered by some forest $\prec$.
\end{lemma}

\begin{proof}
We can construct a forest covering a given nest $\cN$ as follows: Let $N_1, ..., N_{|\cN|}$ be an ordering of all elements of the nest such that for all $i<j$ if holds that either $N_i\subsetneq N_j$ or $N_i$ and $N_j$ are incomparable. We start with a forest $\prec_0$ where all elements are incomparable roots. Given a forest $\prec_i$ we iteratively construct a new forest $\prec_{i+1}$ as follows: We pick an element $o_i\in\Rt^{\prec_i}\cap N_i$ and define $\prec_{i+1}$ as the smallest partial order stronger than $\prec_{i}$ such that additionally $o\prec_{i+1}o_i$ for all $o\in(\Rt^{\prec_i}\cap N_i)\setminus\{o_i\}$. In other words, we attach all the trees growing out of the other elements of $\Rt^{\prec_i}\cap N_i$ as subtrees under $o_i$. For the final forest $\prec:=\prec_{|\cN|}$, we have $N_i=\{k\in\upto{s}\;|\; k\preceq o_i\}$ by construction. Every other subtree of descendants is rooted at a leaf node and thus not an element of $\cI$.
\end{proof}

Consequently, we can build aligned charts over nests:

\begin{lemma}\label{lem:fm-nests-aligned}
Let $\cN\subseteq\cI$ be a nest. Then the weightings $\{\cW^{FM}_{\Delta_I}\;|\; I\in\cN\}$ are uniformly aligned.
\end{lemma}

\begin{proof}
It suffices to find weighted coordinates uniformly aligned with all $\cN$ close to any point $p=(p_1,...,p_s)\in\Delta_\cN$, as for any other point $p'\in M^s$ there is some maximal subnest $\cN'\subseteq\cN$ with $p'\in\Delta_{\cN'}$. The previous Lemma allows us to pick a forest $\prec$ that covers $\cN$. By Lemma~\ref{lem:subtree-adapted}, any $\prec$-offset coordinates are aligned with all $\{\cW^{FM}_{\Delta_I}\}_{I\in\cN}$. The weight of any direction $\Delta x_{lk}$ under $\cW^{FM}_{\Delta_I}$ is either zero or $w_k$ depending on whether $l$ lies in a given $I\in\cN$, thus the alignment is uniform.
\end{proof}

We can now conclude that the Fulton-MacPherson set is quite well-behaved.

\begin{proposition}\label{prop:fm-reg-unif-align}
$\cW^{FM}$ is a uniformly aligned weighted building set.
\end{proposition}

\begin{proof}
    For $\cW^{FM}$ to be a weighting along $\cG^{FM}$ at all, we need for any $\cP\subseteq\cI$ with $\Delta_\cP=\Delta_I$ for some $I\in\cI$ that
    \begin{equation}\label{eq:int-to-show}
        \bigcap_{P\in\cP} \cW^{FM}_{\Delta_P} = \cW^{FM}_{\Delta_I},
    \end{equation}
    and that this intersection is clean. To this end pick a sequence of elements $P_1, ..., P_k\in\cP$ and $I_1, ..., I_k$ such that $$\bigcap_{i=1}^l \Delta_{P_i} = \Delta_{I_l} \qquad\text{and}\qquad I_k=I,$$
    i.e. write $\Delta_I$ as a sequence of intersections such that at every stage we produce a diagonal. In particular, we now have $\Delta_{I_{l+1}}=\Delta_{I_l}\cap\Delta_{P_{l+1}}$ and the pair of weightings over these diagonals intersect cleanly in the weighting over $\Delta_{I_{l+1}}$ by Lemma~\ref{lem:fm-pairs-aligned}. By induction, we can conclude that Equation~\eqref{eq:int-to-show} holds. Cleanness of the intersections at every step guarantees cleanness of the intersection as a whole.
    
    Uniform alignment is exactly the content of Lemma~\ref{lem:fm-nests-aligned}.
\end{proof}

\startSubchaption{Blowing up configuration space}\label{ssec:filtered-fulton-macpherson}

\begin{definition}\label{def:filtered-fm}
Let $(M,H_\bullet)$ be a filtered manifold. We define the \textbf{$s$-fold Fulton-MacPherson configuration space} as $$ \Conf^{[s]}(M,H_\bullet):=\Bl_{\cW^{FM}}(M^s)\xrightarrow{b}M^s.$$

We analogously define the \textbf{$s$-fold projective Fulton-MacPherson configuration space} $ \PConf^{[s]}(M,H_\bullet)$ with the projective blow-up. If the filtration is trivial, we write $\Conf^{[s]}(M)$ and $\PConf^{[s]}(M)$, respectively.
\end{definition}

\begin{npar}[Notation]
For every point $p\in \Conf^{[s]}(M,H_\bullet)$ we let $\cN_p\subseteq\cI$ be the nest that corresponds to $\cG^{FM}_p$ under the correspondence of Lemma~\ref{fulton-mac-nest}.
\end{npar}

By Lemma~\ref{lem:fulton-mac-nice-building}, Proposition~\ref{prop:fm-reg-unif-align} and Theorem~\ref{thm:manifold-structure}, we know that $ \Conf^{[s]}(M,H_\bullet)$ is a smooth manifold with corners while its quotient $\PConf^{[s]}(M,H_\bullet)$ may have some singularities. The nests associated with each point give a stratification.

\begin{npar}[Comparison to the classic Fulton-MacPherson configuration space]
If the filtration $H_\bullet$ is trivial, then all weightings in $\cW^{FM}$ are also trivial, such that $\Conf^{[s]}(M,H_\bullet)$ and $\PConf^{[s]}(M,H_\bullet)$ coincide with the classic Fulton-MacPherson configuration spaces $\Conf^{[s]}(M)$ and $\PConf^{[s]}(M)$, respectively. If $H_\bullet$ is not trivial, the homotopy type of the spherical $\Conf^{[s]}(M,H_\bullet)$ still agrees with that of $\Conf^{[s]}(M)$: Since they are manifolds with corners that agree away from their boundary, one can always consider a small homotopy that retracts into the bulk. Even though there is no canonical comparison map that is a diffeomorphism (compare Example~\ref{ex:multiple}), we still expect them to be diffeomorphic since the local models agree. For the projective $\PConf^{[s]}(M,H_\bullet)$, it is not hard to see that not even the homotopy type generally matches that of $\PConf^{[s]}(M)$ (compare Figure~\ref{fig:projective-helix}).
\end{npar}

As we have seen in the previous section, we must make a choice of a forest covering the nest $\cN_p$ to construct adapted coordinates covering some point $p\in\Conf^{[s]}(M,H_\bullet)$. We could then go ahead and build local coordinates on the blow-up according to Def.~\ref{def:building-local-charts}. But these will have considerable notational complexity! We can make our life slightly easier by instead giving local diffeomorphisms to a model that is more complicated than the corner model:

\begin{proposition}\label{prop:FM-local-model}
Let $p=(p_1, ..., p_s)\in M^s$ and take a nest $\cN\subseteq\cI$ with $p\in\Delta_\cN$, a forest $\prec$ covering the nest and $\prec$-offset coordinates 
$$\left(\{x_{k1}, ..., x_{km}\}_{k\in\Rt^\prec}, \{\Delta x_{l1}, ..., \Delta x_{lm}\}_{l\in\Ch^\prec}\right): U_0\to \R^{sm}$$
around $p$. Write $(w_1, ..., w_m)$ for the weight sequence of $H_\bullet$.
Then there is a map
$$  \Phi:U\to \R^{|\Rt^\prec|\cdot m} \times \prod_{N\in\cN}\left( \bS^{|\Ct^\prec_N|\cdot m-1} \times [0,\infty) \right) $$
defined on an open subset $U\subseteq\Conf^{[s]}(M,H_\bullet)$ and consisting of components
\begin{itemize}
    \item $(\hat x_{k1}, ..., \hat x_{km}):U\to \R^m$ for each $k\in\Rt^\prec$
    \item $\{\Delta \hat x_{l1}, ..., \Delta \hat x_{lm}\}_{l\in\Ct^\prec_N}:U\to \bS^{|\Ct^\prec_N|\cdot m-1}$ for each $N\in\cN$
    \item $t_N: U\to [0,\infty)$ for each $N\in\cN$
\end{itemize}
that satisfy the following:
\begin{enumerate}
    \item $U$ is the intersection of $b^{-1}(U_0)$ and $$\Conf^{[s]}_{\subseteq\cN}(M,H_\bullet):=\{p\in \Conf^{[s]}(M,H_\bullet)\;|\; \cN_p\subseteq\cN\}.$$
    \item $\Phi$ is a diffeomorphism of manifolds with corners onto its image\footnote{It is not generally surjective - indeed (1) specifically excludes points with any collisions beyond $\cN$, which can naturally occur for sequences contained in the boundary of the image.}.
    \item The blow-down map can be written in $\prec$-offset coordinates by the components of $\Phi$ as
        \begin{align*}
        x_{ki} &= \hat x_{ki},\\
        \Delta x_{li} &= \Delta \hat x_{li} \cdot\prod_{\{l,\operatorname{par}^\prec(l)\}\subseteq N\in\cN} t_N^{w_i},
        \end{align*}
        for $k\in\Rt^\prec$, $l\in\Ch^\prec$ and $i\in\upto{m}$.
    \item For $\hat p\in U,$ $$\cN_{\hat p} = \{N\in\cN\;|\; t_N(\hat p) = 0\}.$$
\end{enumerate}
\end{proposition}

As a sanity check, the fact that $\{\Rt^\prec\}\cup\{\Ct^\prec_N\}_{N\in\cN}$ forms a partition of $\upto{s}$ implies that the codomain of $\Phi$ is an $sm$-dimensional manifold with corners. Before we detail the construction of this model, let us build some intuition first:

\begin{npar}[Visualization]\label{par:visualization}
Consider a collided configuration $p\in \Conf^{[s]}(M, H_\bullet)$ and its control set $\cN\subseteq\cI$. To be concrete, say the nest is given as in Example~\ref{ex:fm-nest-forest} by $$\cN=\{\{1,2,3\},\quad\{5,6\},\quad\{7,8,9\},\quad\{5,6,7,8,9\}\},$$ and the configuration $p$ is such that the first three points are collided at $A\in M$, the fourth point is located at $B\in M$ and the remaining points are collided at $C\in M$. Take the forest from Example~\ref{ex:fm-nest-forest} together with coordinates adapted to the filtration around $A, B$ and $C$. Our local model then yields parameters
$$(\Delta \hat x_1, \hat x_2, \Delta \hat x_3, \hat x_4, \Delta \hat x_5, \Delta \hat x_6, \hat x_7, \Delta \hat x_8, \Delta \hat x_9, t_{123}, t_{56}, t_{789}, t_{56789})$$
where $\hat x_2, \hat x_4, \hat x_7\in \R^m$ specify $A,B$ and $C$ in the chosen coordinates, and $t_{123}, t_{56}, t_{789}$ and $ t_{56789}\in[0,\infty)$ control the collisions of the infinitesimal configurations described by 
$$(\Delta \hat x_1, \Delta \hat x_3)\in\bS^{2m-1},\qquad
\Delta \hat x_5\in \bS^{m-1}, \qquad
\Delta \hat x_6\in \bS^{m-1}, \qquad
(\Delta \hat x_8, \Delta \hat x_9)\in\bS^{2m-1}.$$
Since the control set of the point $p$ is all of $\cN$, we know that
$$t_{123}= t_{56}= t_{789}= t_{56789}=0$$
must hold at $p$. The remaining parameters can be visualized as follows:
\begin{center}
    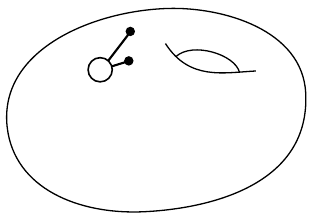
\end{center}
As a graph, this is exactly the description of the forest in Example~\ref{ex:fm-nest-forest}. But now for every subtree corresponding to some nest $N\in\cN$, the (normalized) offset of the child nodes in $\Ct^\prec_N$ (i.e., those connected to the root of the subtree by a solid line) corresponds to an element in $\bS^{|\Ct^\prec_N|\cdot m-1}$. For example, the two solid lines starting at 7 and pointing at 8 and 9 are exactly $(\Delta \hat x_8, \Delta \hat x_9)$ when interpreted as coordinates of a tangent vector at $C$. The dotted lines are drawn arbitrarily and do not reflect any parameters.

Changing the position of any root node or any endpoint of a solid line without creating further collisions corresponds to moving within the stratum labeled by $\cN$. Given that the coordinates specifying offsets are normalized, this provides $m|\Rt^\prec|+\sum_{N\in\cN}(m|\Ct^\prec_N|-1)$ degrees of freedom. The remaining $|\cN|$ degrees of freedom are given by the control parameters $t_N$.
Increasing $t_N$ to a non-zero value has the effect of undoing one level of the collision.
This moves the configuration into the larger stratum labeled by $\cN\setminus\{N\}$. For example, increasing the parameter $t_{56}$ associated with $N=\{5,6\}$ would modify the tree at $C$ as follows, with the distance between the fifth and six point controlled by $t_{56}$:
\begin{center}
    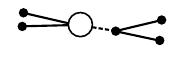
\end{center}
This behavior is also reflected when we write out the blow-down map for this example:
\begin{align*}
x_1 &= \hat{x}_2 + t_{123}\cdot \Delta\hat{x}_1,\\
x_2 &= \hat{x}_2,\\
x_3 &= \hat{x}_2 + t_{123}\cdot \Delta\hat{x}_3,\\
x_4 &= \hat{x}_4,\\
x_5 &= \hat{x}_7 + t_{56789}\cdot \Delta\hat{x}_6 + (t_{56789}t_{56})\cdot \Delta\hat{x}_5,\\
x_6 &= \hat{x}_7 + t_{56789}\cdot \Delta\hat{x}_6,\\
x_7 &= \hat{x}_7,\\
x_8 &= \hat{x}_7 + (t_{56789}t_{789})\cdot \Delta\hat{x}_8,\\
x_9 &= \hat{x}_7 + (t_{56789}t_{789})\cdot \Delta\hat{x}_9,
\end{align*}
where all control parameters act according to the weights for the chosen coordinates.
\end{npar}

\begin{proof}[Proof of Proposition~\ref{prop:FM-local-model}]
For every $k\in\Rt^\prec$, we set
$$(\hat x_{k1}, ..., \hat x_{km}):= (x_{k1}, ..., x_{km})\circ b,$$
which is clearly well-defined on $b^{-1}(U_0)$.
Define $U$ as in condition (1). Every point $p\in U$ must lie in some chart domain $U_{\chi\cN\mathbf{hs}}$ where $\chi$ is given by the $\prec$-offset coordinates. We can assume that for every $N\in\cN$, $\mathbf{h}(N)$ is an index corresponding to an offset coordinate $\Delta x_{l_0i}$ with $l_0\in\Ct^\prec_N$:
After all, whether $p$ lies on the exceptional divisor of $N$ is exactly controlled by the $N$-controls in $\Ct^\prec_N.$
Letting this induced blow-up coordinate corresponding to $\Delta x_{l_0i}$ act on the rest according to the weights, where the gap it leaves is filled by $\mathbf{s}(N)$, yields an element $v\in\R^{|\Ct^\prec_N|\cdot m}$ depending on the point $p$. We can now define the remaining components of the map $\Phi$ as the unique $$\{\Delta \hat x_{l1}, ..., \Delta \hat x_{lm}\}_{l\in\Ct^\prec_N}\in\bS^{|\Ct^\prec_N|\cdot m-1} \qquad\text{and}\qquad t_N \in [0,\infty)$$ such that letting $t_N$ act on this element of the sphere according to the weights produces the vector $v\in\R^{|\Ct^\prec_N|\cdot m}$. This definition amounts to passing from coordinates on a sphere constructed by intersecting rays with a hyperplane back to the original sphere. It is straightforward but notationally extravagant to check that this yields a diffeomorphism $\Phi$ onto its image by explicitly writing out its local representations for every good perspective arising from the fixed offset coordinates and nest. The third property is a direct consequence of Eqs.~\eqref{eq:coord-lemma-three} and~\eqref{eq:local-blow-down} in this situation.
The fourth property follows by Proposition~\ref{prop:strat-local-coords} since our new parameters vanish precisely when the induced coordinate corresponding to $\Delta x_{l_0i}$ does.
\end{proof}

To discuss the relation to Fulton and MacPherson's \textit{screen} terminology, we want to investigate for some collided configuration $p\in\Delta_I$ what the fiber over $p$ in the exceptional divisor of $\Bl_{\cW^{FM}_{\Delta_I}}(M)$ looks like:

\begin{lemma}
For $p=(p_1, ..., p_s)\in\Delta_I\in\cG^{FM}$ and $i\in I$ it holds that
$$
\bS\nu\cW^{FM}_{\Delta_I} |_p \simeq \left(\gr\left(H^{s-|I|}_\bullet|_{p_i}\right) / \sim\right)\setminus\{[0_{p_i}]\}.
$$
The equivalence relation on the right hand side identifies $(s-|I|)$-fold configurations of graded tangent vectors at $p_i$ if they can be made to match by dilating with the $\R_{>0}$-action and shifting by a common graded vector. The excluded point $[0_{p_i}]$ is the class of the completely collided configuration. The isomorphism may depend on a choice of local coordinates and is thus not in general canonical.
\end{lemma}

\begin{proof}
By combining Equation~\eqref{eq:noncanonical-weighted-normal} (which may depend on a choice of coordinates) with Equation~\eqref{eq:filtration-quotient}, we obtain
$$ \nu\cW^{FM}_{\Delta_I} \simeq \gr(H^s_\bullet|_{\Delta_I})\,/\,T\Delta_I$$
By definition of the diagonal $\Delta_I$, the right-hand side is just
$$ \left.\gr\left(H^{s-|I|}_\bullet\right)\right|_{\Delta_I}\,/\,\sim', $$
where the equivalence relation $\sim'$ identifies configurations of graded tangent vectors that are shifted by a common graded vector. Restricting to the fiber over $p$ yields
$$
\nu\cW^{FM}_{\Delta_I} |_p \simeq \gr\left(H^{s-|I|}_\bullet|_{p_i}\right) / \sim'.
$$
The equation in the statement of the Lemma then follows immediately by taking the sphere bundle over this, i.e. removing the origin and taking the quotient by the $\R_{>0}$-action.
\end{proof}

\begin{npar}[Visualization by screens]\label{par:screens}
Consider a blown-up collided configuration given by $p\in \Conf^{[s]}(M, H_\bullet)$ and its control set $\cN\subseteq\cI$. The configuration $p$ is completely determined by the underlying $b_{\cW^{FM}}(p)$ together with the components in the exceptional divisor of $\Bl_{\cW^{FM}_{\Delta_N}}(M)$ for each $\Delta_N\in\cG^{FM}$ with $b_{\cW^{FM}}(p)\in\Delta_N$. In the language of Fulton and MacPherson, such elements of $\bS\nu\cW^{FM}_{\Delta_N}$ are called a \textbf{screen}, and the screens corresponding to components of $p$ over the nest $\cN$ are the \textbf{essential screens} of $p$. In the light of the previous Lemma, a screen for some $N\in\cN$ is exactly an equivalence class of a configuration of $|N|$ points in $\gr(H_\bullet)|_A$ where $A\in M$ is the location of the collision.

To be concrete, consider the configuration $p$ and nest $\cN$ as in Paragraph~\ref{par:visualization}, with the first three points collided at $A\in M$, the fourth point located at $B\in M$ and the remaining points collided at $C\in M$. We may visualize this $p$ by drawing the blown-down configuration and a screen for every element of the nest:
\begin{center}
    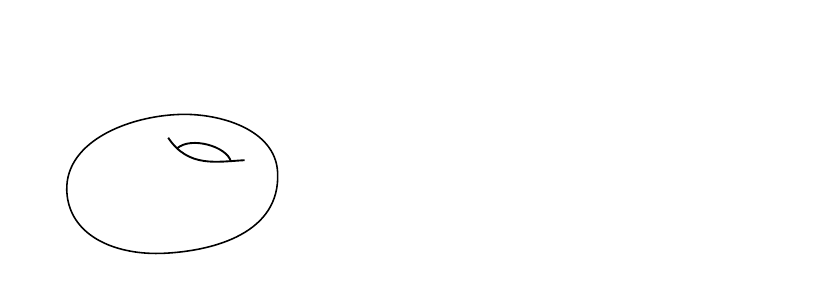
\end{center}

This coincides with the pictures from~\cite{FM94} since we do not depict the filtered structure here. Indeed for the trivial filtered structure, $\gr(H_\bullet)$ is just $TM$.
Alternatively, one may think of all screens encoded in trees as follows:

\begin{center}
    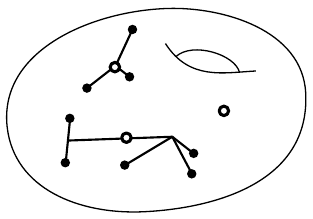
\end{center}

Here, the circled roots of each tree mark the locations $A,B$ and $C$ of the collisions. The non-leaf nodes may be labeled by all the elements of the nest. The positions of child nodes relative to their parent encode a screen, but their common scaling and offset is arbitrary. Indeed we chose to add an offset to the vectors of $\{7,8,9\}\in\cN$ for legibility.

This visualization by trees is of course closely related to the one we have given before in Paragraph~\ref{par:visualization}. The main difference is that the latter takes into account a forest as data required to build our local model, while this visualization only depends on choices of coordinates around $A, B$ and $C$. As before, changing the location of the nodes without further collisions moves within the stratum, while increasing a control parameter modifies the shape of the tree by removing an intermediary point.
\end{npar}

We also want to provide a local model for the projective case analogous to Proposition~\ref{prop:FM-local-model}. The presence of weights when taking the quotient means that this involves a \textit{weighted} version of the tautological bundle of projective space:

\begin{definition}\label{def:tautological-bundle}
For a weight sequence $(w_1, ..., w_n)$, we define the \textbf{weighted projective tautological bundle} as
$$\cT^{n}:= (\bS^{n-1}\times\R)/\sim,$$
where we quotient out the natural weighted action of $(-1)$ on $\bS^{n-1}\times\R\subset \R^n\times \R$, thus identifying any point $(s_{1}, ..., s_{n},t)$ with $((-1)^{w_1}s_{1}, ..., (-1)^{w_n}s_{n},-t).$
\end{definition}

This is in general an orbifold, but recovers the usual tautological bundle of $\R\bP^{n-1}$ for trivial weights $w_1=...=w_n=1$. With it, we obtain the following local model of the projective configuration space:

\begin{proposition}\label{prop:FM-local-model-proj}
Let $p=(p_1, ..., p_s)\in M^s$ and take a nest $\cN\subseteq\cI$ with $p\in\Delta_\cN$, a forest $\prec$ covering the nest and $\prec$-offset coordinates 
$$\left(\{x_{k1}, ..., x_{km}\}_{k\in\Rt^\prec}, \{\Delta x_{l1}, ..., \Delta x_{lm}\}_{l\in\Ch^\prec}\right): U_0\to \R^{sm}$$
around $p$. Write $(w_1, ..., w_m)$ for the weight sequence of $H_\bullet$ and consider weighted tautological bundles $\cT^{k\cdot m}$ that arise from repeating this sequence $k$ times.
Then there is a map
$$  \Phi:U\to \R^{|\Rt^\prec|\cdot m} \times \prod_{N\in\cN} \cT^{|\Ct^\prec_N|\cdot m} $$
defined on an open subset $U\subseteq\PConf^{[s]}(M,H_\bullet)$ and represented on all sufficiently small $U'\subsetneq U$ by components
\begin{itemize}
    \item $(\hat x_{k1}, ..., \hat x_{km}):U\to \R^m$ for each $k\in\Rt^\prec$
    \item $\{\Delta \hat x_{l1}, ..., \Delta \hat x_{lm}\}_{l\in\Ct^\prec_N}:U'\to \bS^{|\Ct^\prec_N|\cdot m-1}$ for each $N\in\cN$
    \item $t_N: U'\to \R$ for each $N\in\cN$
\end{itemize}
that satisfy the following:
\begin{enumerate}
    \item $U$ is the intersection of $b^{-1}(U_0)$ and $$\PConf^{[s]}_{\subseteq\cN}(M,H_\bullet):=\{p\in \PConf^{[s]}(M,H_\bullet)\;|\; \cN_p\subseteq\cN\}.$$
    \item $\Phi$ is an homeomorphism onto its image.
    \item The blow-down map can be written in $\prec$-offset coordinates by the components of $\Phi$ as
        \begin{align*}
        x_{ki} &= \hat x_{ki},\\
        \Delta x_{li} &= \Delta \hat x_{li} \cdot\prod_{\{l,\operatorname{par}^\prec(l)\}\subseteq N\in\cN} t_N^{w_i},
        \end{align*}
        for $k\in\Rt^\prec$, $l\in\Ch^\prec$ and $i\in\upto{m}$.
    \item For $\hat p\in U,$ $$\cN_{\hat p} = \{N\in\cN\;|\; t_N(\hat p) = 0\}.$$
\end{enumerate}
\end{proposition}

\begin{proof}
The canonical quotient $\Conf^{[s]}(M,H_\bullet)\to\PConf^{[s]}(M,H_\bullet)$ identifies points in the domain $U$ of the local model of the spherical configuration space from Proposition~\ref{prop:FM-local-model} only with other points in $U$, so we can argue within that local model. The proposition then largely follows by noting that the weighted tautological bundle can be written 
$$\cT^{dm}=(\bS^{dm-1}\times [0,\infty))/\sim$$
where the quotient identifies any point $(s_{11}, ..., s_{dm},0)$ with
\begin{equation*}((-1)^{w_1}s_{11}, ..., (-1)^{w_m}s_{dm},0).\qedhere\end{equation*}
\end{proof}

As a corollary to this proof we have:

\begin{corollary}
    $\PConf^{[s]}(M,H_\bullet)$ carries a canonical orbifold structure such that the map $\Phi$ from Proposition~\ref{prop:FM-local-model-proj} is an isomorphism onto its image.
\end{corollary}

Note that the singular points of the orbifold $\cT^{dm}$ are points $(s_{11}, ..., s_{dm},0)$ with $(-1)^{w_j}s_{ij}=s_ij$ for all $i\in\upto{d}, j\in\upto{m}$, allowing us to read off the singular points of $\PConf^{[s]}(M,H_\bullet)$. The discussion of Paragraph~\ref{par:visualization} also largely applies to the projective configuration space, just that moving off $t_N=0$ into the negative direction produces a configuration that is mirrored in a weighted sense.

\startSubchaption{Adaptations for fiber bundles with horizontally trivial filtration}\label{ssec:filtered-bundle-fulton-macpherson}

From here on out, fix $s\geq 1$ and a fiber bundle $E\xrightarrow{\pi}M$ whose total space is equipped with a Lie filtration $H_\bullet$. Write $(\cG^{FM}_E,\cW^{FM}_E)$ for the weighted Fulton-MacPherson building set for $(E^s,H^s_\bullet)$ and $(\cG^{FM},\cW^{FM})$ for the trivial weighting over the Fulton-MacPherson building set $\cG^{FM}$ of $M^s$. For every $I\in\cI$, we use the notations $\Delta_I\in\cG^{FM}$ and $\underline{\Delta}_I\in\cG_E^{FM}$ to distinguish the diagonals.

Our goal is to construct a blown-up $s$-fold configuration space for $E$ that comes with a canonical projection map to the Fulton-MacPherson configuration space $\Conf^{[s]}(M)=\Bl_{\cG^{FM}}(M^s)$. The crucial observation is the following:

\begin{lemma}
Consider the $s$-fold power $\pi^s:E^s\to M^s$ of the bundle projection and the map $\phi:\cG^{FM}\to\cG^{FM}_E$ that sends each ${\Delta}_I$ to $\underline{\Delta}_I$. Then $(\pi^s,\phi):(E^s,\cW^{FM}_E)\to(M^s,\cW^{FM})$ is a morphism of weighted building sets.
\end{lemma}

\begin{proof}
The map $\pi^s$ is clearly smooth. For every $I\in\cI$ the diagonal $\underline{\Delta}_I$ in $E^s$ pushes forward along $\pi^s$ to the diagonal $\Delta_I$ in $M^s$. In other words, $\pi^s$ maps the support of $(\cW^{FM}_E)_{\phi(\Delta_I)}$ exactly onto that of $\cW^{FM}_{\Delta_I}$. Since the latter weighting is trivial, this suffices to conclude that $\pi^s$ is a weighted morphism from $(\cW^{FM}_E)_{\phi(\Delta_I)}$ to $\cW^{FM}_{\Delta_I}$. Since $I$ was arbitrary, $(\pi^s,\phi)$ is a morphism of weighted building sets.
\end{proof}

This allows us to make the following definition:

\begin{definition}\label{def:blow-up-bundle}
We define the \textbf{$s$-fold blown-up configuration space of $(E\xrightarrow{\pi}M,H_\bullet)$} as
$$ \Conf^{[s]}(E\xrightarrow{\pi}M,H_\bullet):= \Bl_\pi (E^s) \xrightarrow{\hat\pi} \Conf^{[s]}(M), $$
where we write $\hat\pi$ for the induced map $b_{(\pi^s,\phi)}$. Similarly, define
$$ \PConf^{[s]}(E\xrightarrow{\pi}M,H_\bullet):= \PBl_\pi (E^s) \xrightarrow{p\hat\pi} \PConf^{[s]}(M) $$
for the projective case.
\end{definition}

Note that pushing a nest of $\cG^{FM}_E$ forward along $\pi^s$ preserves the identification of nests with nested subsets of $\cI$, such that the stratifications of the total space and the base space are compatible.

Our motivation is to give a configuration space for the $r$-th order jet bundle $J^r(M,N)\to M$ between two manifolds $M$ and $N$, which carries a filtration induced by the Cartan distribution. This has the special property that its projection to the base is surjective, i.e. it gives a trivial filtered structure on the base. More generally, we can consider the following:

\begin{definition}\label{def:horiz-triv}
We say a smooth fiber bundle $E\xrightarrow{\pi}M$ together with a Lie filtration $H_\bullet$ on the total space $E$ is a \textbf{fiber bundle with horizontally trivial filtration} if it holds that $$d\pi(H_{-1})=TM.$$
\end{definition}

As a consequence, the weight sequence of $H_\bullet$ must start with at least $\dim M$ ones.

\begin{remark}\label{rem:points-removed}
For a bundle with a horizontally trivial filtration, the restriction to the domain $\Conf^{[s]}(M,H_\bullet)$ of the induced map removes two kinds of points: On one hand, every element $\{p_G\}_{G\in\cG^{FM}_E}$ with a component $p_G$ such that $b_G(p_G)\not\in G$, i.e. not on the exceptional divisor, such that $\pi(b_G(p_G))$ does lie in $\pi(G)$. In other words, they encode elements whose configuration of base points in $M$ is \textit{more} collided than the configuration in $E$ itself.
On the other hand, we remove $\{p_G\}_{G\in\cG^{FM}_E}$ with some component $p_G$ encoding a weighted normal vector along $G$ that is vertical with respect to the projection $\pi$, i.e. that encodes a collision along the fibers of $E\to M$. If we did not assume that $d\pi(H_{-1})=TM$, then we would have to remove additional points that do not come from a vertical collision, but have no component along $H_{-1}$. 

When $d>0$, the restriction to $\Conf^{[s]}(M,H_\bullet)$ creates an additional open boundary. Thus while the classic Fulton-MacPherson configuration space may be called compactification at least in the sense of not creating an additional open boundary as one moves to the diagonals, this construction certainly should not.
\end{remark}

For the case with a horizontally trivial filtration, we want to give a local model that preserves the fiber structure. Thus consider offset coordinates on $E^s$ compatible with the bundle structure:

\begin{definition}\label{def:coords-horiz-triv} Let $(E\xrightarrow[]{\pi}M,H_\bullet)$ be a fiber bundle with horizontally trivial filtration with fiber dimension $d$.
We say offset coordinates 
$$(x_1, ..., x_m, y_1, ..., y_d, \Delta x_1, ..., \Delta x_m, \Delta y_1, ..., \Delta y_d):U\to\R^{2m+2d}$$
for $U\subseteq E^2$ around a point on the diagonal of $(E^2,H_\bullet^2)$ are \textbf{adapted to $\pi$} if they additionally satisfy that
$$(x_1, ..., x_m, \Delta x_1, ..., \Delta x_m):U\to\R^{2m}$$
can be written as the composition of $\pi^2$ and offset coordinates on $M^2$ for the trivial filtration. Given a forest $\prec$ on $\upto{s}$, we call $\prec$-offset coordinates \textbf{adapted to $\pi$} if they were constructed using only adapted offset coordinates on $E^2$.
\end{definition}

\begin{example}
Consider the jet bundle $J^2(\R^m,\R)\xrightarrow{\pi}\R^m$ together with the Cartan distribution $$\xi = \left\langle \left.\partial x_a+y'_a \,\partial y + \sum_{b=1}^my''_{ab}\,\partial y'_b\;\right|\; 1\leq a\leq m \right\rangle + \left\langle \partial y''_{ab}\;|\; 1\leq a,b \leq m\right\rangle.$$
By taking Lie derivatives, this generates the depth-3 Lie filtration
$$
H_{-1} = \xi, \qquad H_{-2} = \xi+\langle \partial y'_a\;|\; 1\leq a \leq m\rangle, \qquad H_{-3} = TJ^2(\R^m,\R).
$$
Note that $J^2(\R^m,\R)\to\R^m$ is indeed a fiber bundle with horizontally trivial filtration, since every horizontal vector in $T\R^m$ can be lifted to $\xi$.
In the introduction of this article, we defined the two-fold second-order jet configuration space $J^{2,2}(\R^m,\R):= (J^{2,2}(\R^m,\R))^2\to \R^{2m}$ with coordinates
$$(x_1,y_1,y'_1,y''_1;\;x_2,y_2,y'_2,y''_2).$$
One may check that replacing the second half of these with the offsets
$(\Delta x, \Delta y, \Delta y', \Delta y'')$
defined in Equation~\eqref{eq:adapted-coords} yields offset coordinates adapted to the fiber bundle projection $\pi$.
\end{example}

We can always find such adapted offset coordinates:

\begin{lemma}
Let $(E\xrightarrow[]{\pi}M,H_\bullet)$ be a fiber bundle with horizontally trivial filtration.
Then there exist offset coordinates adapted to $\pi$ close to every point $p$ of the diagonal of $E^2$.
\end{lemma}

\begin{proof}
Take offset coordinates 
$$\tilde{\chi}=(\tilde{x}_1, ..., \tilde{x}_{m+d}, \Delta \tilde{x}_1, ..., \Delta \tilde{x}_{m+d}):\tilde{U}\to\R^{2m+2d}$$
of the diagonal $\underline{\Delta}\subseteq(E^2,H^2_\bullet)$ around $p$ as well as offset coordinates
$$\chi'=(x'_1, ..., x'_{m}, \Delta x'_1, ..., \Delta x'_{m}):U'\to\R^{2m}$$
of the diagonal $\Delta\subseteq M^2$ around $\pi^2(p)$ for the trivial filtration.
Define on a small enough open $U$ around $p$ the functions
$$x_j:=x'_j\circ\pi^2\qquad\text{and}\qquad \Delta x_j:=\Delta x'_j\circ\pi^2 \qquad\text{for }j=1,...,m $$
and reorder the indices of the $\tilde{x}_j$ and $\Delta\tilde{x}_j$ such that extending the above with
$$y_j:=\tilde{x}_{m+j}\qquad\text{and}\qquad \Delta y_j:=\tilde{x}_{m_j}\qquad\text{for }j=1,...,d$$
yields a system of coordinates. These are clearly adapted to $\pi$, but we must show that they are still offset coordinates. $(x_1,...,x_m,y_1,...,y_d)$ only depend on the first factor of $E^2$ since composing $\pi^2$ with the projection onto that first factor is the same as composing the projection onto the first factor of $M^2$ with $\pi$. So we only need to show the three conditions that our coordinates are adapted to the filtered submanifold $\underline{\Delta}$.

\textit{Regarding (1):} Assume $q\in U\cap\underline{\Delta}$. It then holds that $\Delta\tilde{x}_j(q)=0$ for all $j=1,...,m+d$. Since $\pi^2(q)\in\Delta$ we further have $\Delta x_j(q)=0$ for all $j=1,...,m$. Together, we can say for $q\in U$ that $$q\in \underline{\Delta}\qquad\implies\qquad
(\Delta x_1,...,\Delta x_m;\Delta y_1, ..., \Delta y_d)(q)=0.$$
Since $\underline{\Delta}$ is a submanifold of codimension $m+d$ and the functions are part of a coordinate system, we know that the reverse implication is also true.

\textit{Regarding (2):} Let $\{w_j\}_{j=1,...,m+d}$ be the weight sequence of the filtered submanifold $\underline{\Delta}$. We first argue that the components $\Delta \tilde{x}_j$ that we replace with the $\Delta x_j$ when constructing our coordinates all carry a weight $w_j=1$. Assume otherwise, i.e. that some $d_q\Delta\tilde{x}_{j_0}$ at $q\in\underline{\Delta}$ with positive weight is a linear combination of the $d_q\Delta x_j$. By condition (2) for $\tilde{\chi}$ there is some $v\in T_q E^2$ such that $v\not\in H^2_{-1}$ and $d_q\Delta\tilde{x}_{j_0}(v)\neq 0$. This means there must also be a $j_1$ such that $d_q\Delta x_{j_1}(v)\neq0$. However, by the assumption that $d\pi H_{-1}=TM$, any $v\not\in H^2_{-1}$ must be mapped by $d\pi^2$ to zero, which is a contradiction to the previous inequality by $\Delta x_{j_1}=\Delta x'_{j_1}\circ\pi^2$

To now show (2), we have to establish for any $i\in\{0, ..., r\}$ and $q\in \underline{\Delta}\cap U$ that
\begin{align*}
H^2_{-i}|_q+T_q\underline{\Delta} = \{v\in T_q E^2 \;|\; &d\Delta x_j(v) = 0 \text{ for all $j\in\{1,...,m\}$ with } w_j>i \text{ and }\\
&d\Delta y_j(v)=0\text{ for all $j\in\{1,...,d\}$ with } w_{m+j}>i\}.     
\end{align*}
Given the weight assignment above, for $i=0$ this follows immediately from (1), while for $i>0$ is a consequence of condition (2) for $\tilde{\chi}$ since only its components appear on the right-hand side.

\textit{Regarding (3):} For each $\Delta y_j$, this follows immediately from the corresponding condition for $\tilde\chi$. As argued above, the weights assigned to each $\Delta x_j$ must be one. Therefore the condition for the remaining $\Delta x_j$ is just that they vanish on $\underline{\Delta}$, which we have already seen in (1).
\end{proof}

We are now ready to give the modified local model. It largely coincides with that of Proposition~\ref{prop:FM-local-model}, except that we normalize only in the horizontal offset coordinates:

\begin{proposition}\label{prop:FM-local-bundle-model}
The blown-up configuration space
$$\Conf^{[s]}(E\xrightarrow{\pi}M,H_\bullet)\to \Conf^{[s]}(M) $$
is a smooth fiber bundle. It can locally be described as follows:

Let $p=(p_1, ..., p_s)\in M^s$ and take a nest $\cN\subseteq\cI$ with $p\in\Delta_\cN\subseteq E^s$, a forest $\prec$ covering the nest and $\prec$-offset coordinates 
$$\left(\{x_{k1}, ..., x_{km}, y_{k1}, ..., y_{kd}\}_{k\in\Rt^\prec}, \{\Delta x_{l1}, ..., \Delta x_{lm}, \Delta y_{l1}, ..., \Delta y_{ld} \}_{l\in\Ch^\prec}\right): U_0\to \R^{s(m+d)}$$
around $p$ that are adapted to $\pi$. Write $(1, ..., 1, w_1, ..., w_d)$ for the weight sequence of $H_\bullet$.
Then there is a map
$$  \Phi:U\to \R^{|\Rt^\prec|\cdot (m+d)}  \times \prod_{N\in\cN}\left( \bS^{|\Ct^\prec_N|\cdot m-1}\times [0,\infty) \times \R^{|\Ct^\prec_N|\cdot d} \right) $$
defined on an open subset $U\subseteq\Conf^{[s]}(E\xrightarrow{\pi}M,H_\bullet)$ and consisting of components
\begin{itemize}
    \item $(\hat x_{k1}, ..., \hat x_{km},\hat y_{k1}, ..., \hat y_{kd}):U\to \R^{m+d}$ for each $k\in\Rt^\prec$
    \item $\{\Delta \hat x_{l1}, ..., \Delta \hat x_{lm}\}_{l\in\Ct^\prec_N}:U\to \bS^{|\Ct^\prec_N|\cdot m-1}$ for each $N\in\cN$
    \item $t_N: U\to [0,\infty)$ for each $N\in\cN$
    \item $\{\Delta \hat y_{l1}, ..., \Delta \hat y_{ld}\}_{l\in\Ct^\prec_N}:U\to \R^{|\Ct^\prec_N|\cdot d}$ for each $N\in\cN$
\end{itemize}
that satisfy the following:
\begin{enumerate}
    \item $U$ is the intersection of $b^{-1}(U_0)$ and $$\Conf^{[s]}_{\subseteq\cN}(E\xrightarrow{\pi}M,H_\bullet):=\{p\in \Conf^{[s]}(E\xrightarrow{\pi}M,H_\bullet)\;|\; \cN_p\subseteq\cN\}.$$
    \item The $x_{ki}$, $\Delta x_{li}$ and $t_N$ components of $\Phi$ are the composition of $b_\pi$ with a local model of $\Conf^{[s]}(M)$ in the sense of Proposition~\ref{prop:FM-local-model}. 
    \item $\Phi$ is a diffeomorphism of manifolds with corners onto its image.
    \item The blow-down map can be written in $\prec$-offset coordinates by the components of $\Phi$ as
        \begin{align*}
        x_{ki} &= \hat x_{ki},\\
        y_{kj} &= \hat y_{kj},\\
        \Delta x_{li} &= \Delta \hat x_{li} \cdot\prod_{\{l,\operatorname{par}^\prec(l)\}\subseteq N\in\cN} t_N,\\
        \Delta y_{lj} &= \Delta \hat y_{lj} \cdot \left(\prod_{\{l,\operatorname{par}^\prec(l)\}\subseteq N\in\cN} t_N\right)^{w_j},
        \end{align*}
        for $k\in\Rt^\prec$, $l\in\Ch^\prec$, $i\in\upto{m}$ and $j\in\upto{d}$.
    \item For $\hat p\in U,$ $$\cN_{\hat p} = \{N\in\cN\;|\; t_N(\hat p) = 0\}.$$
\end{enumerate}
\end{proposition}

\begin{proof}
We first build a local model with Proposition~\ref{prop:FM-local-model} and the adapted offset coordinates. Since we restricted the configuration space for bundles with a horizontally trivial filtration to hold only collisions with a horizontal component as discussed in Remark~\ref{rem:points-removed}, we know that
$$\{\Delta \hat x_{l1}, ..., \Delta \hat x_{lm}\}_{l\in\Ct^\prec_N}$$
do not all vanish and can thus rescale all offset coordinates such that only these horizontal components are normalized. It only remains to argue for (2), but this follows since the horizontal components of the rescaled local model now coincide with the local model of $\Conf^{[s]}(M)$ constructed with Proposition~\ref{prop:FM-local-model} and only the horizontal components of the adapted offset coordinates.
\end{proof}

Once again, we can also formulate a projective version of this:

\begin{proposition}\label{prop:FM-local-bundle-model-proj}
The blown-up projective configuration space
$$\PConf^{[s]}(E\xrightarrow{\pi}M,H_\bullet)\to \PConf^{[s]}(M) $$
is a smooth fiber bundle with a local model as follows:

Let $p=(p_1, ..., p_s)\in M^s$ and take a nest $\cN\subseteq\cI$ with $p\in\Delta_\cN\subseteq E^s$, a forest $\prec$ covering the nest and $\prec$-offset coordinates 
$$\left(\{x_{k1}, ..., x_{km}, y_{k1}, ..., y_{kd}\}_{k\in\Rt^\prec}, \{\Delta x_{l1}, ..., \Delta x_{lm}, \Delta y_{l1}, ..., \Delta y_{ld} \}_{l\in\Ch^\prec}\right): U_0\to \R^{s(m+d)}$$
around $p$ that are adapted to $\pi$.
Write $(1, ..., 1, w_1, ..., w_d)$ for the weight sequence of $H_\bullet$ and consider the projective tautological bundle $\cT^{\bullet m}$ from Definition~\ref{def:tautological-bundle} for trivial weights.
Then there is a map
$$  \Phi:U\to \R^{|\Rt^\prec|\cdot (m+d)} \times \prod_{N\in\cN}\left( \cT^{|\Ct^\prec_N|\cdot m}  \times \R^{|\Ct^\prec_N|\cdot d}\right) $$
defined on an open subset $U\subseteq\PConf^{[s]}(E\xrightarrow{\pi}M,H_\bullet)$ and represented on all sufficiently small $U'\subsetneq U$ by components
\begin{itemize}
    \item $(\hat x_{k1}, ..., \hat x_{km},\hat y_{k1}, ..., \hat y_{kd}):U\to \R^{m+d}$ for each $k\in\Rt^\prec$
    \item $\{\Delta \hat x_{l1}, ..., \Delta \hat x_{lm}\}_{l\in\Ct^\prec_N}:U\to \bS^{|\Ct^\prec_N|\cdot m-1}$ for each $N\in\cN$
    \item $t_N: U\to \R$ for each $N\in\cN$
    \item $\{\Delta \hat y_{l1}, ..., \Delta \hat y_{ld}\}_{l\in\Ct^\prec_N}:U\to \R^{|\Ct^\prec_N|\cdot d}$ for each $N\in\cN$
\end{itemize}
that satisfy the following:
\begin{enumerate}
    \item $U$ is the intersection of $b^{-1}(U_0)$ and $$\PConf^{[s]}_{\subseteq\cN}(E\xrightarrow{\pi}M,H_\bullet):=\{p\in \PConf^{[s]}(E\xrightarrow{\pi}M,H_\bullet)\;|\; \cN_p\subseteq\cN\}.$$
    \item The $x_{ki}$, $\Delta x_{li}$ and $t_N$ components of $\Phi$ are the composition of $b_\pi$ with a local model of $\PConf^{[s]}(M)$ in the sense of Proposition~\ref{prop:FM-local-model}. 
    \item $\Phi$ is a diffeomorphism of manifolds with corners onto its image.
    \item The blow-down map can be written in $\prec$-offset coordinates by the components of $\Phi$ as
        \begin{align*}
        x_{ki} &= \hat x_{ki},\\
        y_{kj} &= \hat y_{kj},\\
        \Delta x_{li} &= \Delta \hat x_{li} \cdot\prod_{\{l,\operatorname{par}^\prec(l)\}\subseteq N\in\cN} t_N,\\
        \Delta y_{lj} &= \Delta \hat y_{lj} \cdot \left(\prod_{\{l,\operatorname{par}^\prec(l)\}\subseteq N\in\cN} t_N\right)^{w_j},
        \end{align*}
        for $k\in\Rt^\prec$, $l\in\Ch^\prec$, $i\in\upto{m}$ and $j\in\upto{d}$.
    \item For $\hat p\in U,$ $$\cN_{\hat p} = \{N\in\cN\;|\; t_N(\hat p) = 0\}.$$
\end{enumerate}
\end{proposition}

The blown-up configuration space $\PConf^{[s]}(M,H_\bullet)$ of a filtered manifold was not necessarily smooth when the filtration was non-trivial, i.e. has weights larger than one. For bundles with a horizontally trivial filtration, this proposition establishes that $\PConf^{[s]}(E\xrightarrow{\pi}M,H_\bullet)$ is always smooth.

\begin{proof}
The proof proceeds entirely analogously to that of Proposition~\ref{prop:FM-local-model-proj} when taking into account the modifications of Proposition~\ref{prop:FM-local-bundle-model}. Note that the assumption on the weights along horizontal directions being trivial also means that we can use the unweighted tautological bundle, which is smooth without orbifold singularities.
\end{proof}

%% file: diagonals.pdf_tex
\begingroup%
  \makeatletter%
  \providecommand\color[2][]{%
    \errmessage{(Inkscape) Color is used for the text in Inkscape, but the package 'color.sty' is not loaded}%
    \renewcommand\color[2][]{}%
  }%
  \providecommand\transparent[1]{%
    \errmessage{(Inkscape) Transparency is used (non-zero) for the text in Inkscape, but the package 'transparent.sty' is not loaded}%
    \renewcommand\transparent[1]{}%
  }%
  \providecommand\rotatebox[2]{#2}%
  \newcommand*\fsize{\dimexpr\f@size pt\relax}%
  \newcommand*\lineheight[1]{\fontsize{\fsize}{#1\fsize}\selectfont}%
  \ifx\svgwidth\undefined%
    \setlength{\unitlength}{283.46456693bp}%
    \ifx\svgscale\undefined%
      \relax%
    \else%
      \setlength{\unitlength}{\unitlength * \real{\svgscale}}%
    \fi%
  \else%
    \setlength{\unitlength}{\svgwidth}%
  \fi%
  \global\let\svgwidth\undefined%
  \global\let\svgscale\undefined%
  \makeatother%
  \begin{picture}(1,0.55)%
    \lineheight{1}%
    \setlength\tabcolsep{0pt}%
    \put(0,0){\includegraphics[width=\unitlength,page=1]{diagonals.pdf}}%
    \put(0.68279221,0.99560276){\color[rgb]{0.8745098,0.47058824,0.07058824}\makebox(0,0)[t]{\lineheight{0.80000001}\smash{\begin{tabular}[t]{c}$N$\end{tabular}}}}%
    \put(0.40960369,0.83474022){\color[rgb]{0.05490196,0.52156863,0.6627451}\makebox(0,0)[lt]{\lineheight{0.80000001}\smash{\begin{tabular}[t]{l}$\bS\nu_\cW M$\end{tabular}}}}%
    \put(0.61815155,1.17067688){\color[rgb]{0.04705882,0.58431373,0.27058824}\makebox(0,0)[t]{\lineheight{0.80000001}\smash{\begin{tabular}[t]{c}$x_3=1$\end{tabular}}}}%
    \put(0.76465958,0.93807331){\color[rgb]{0,0,0}\makebox(0,0)[t]{\lineheight{0.80000001}\smash{\begin{tabular}[t]{c}0\end{tabular}}}}%
    \put(-0.00794622,1.18724693){\color[rgb]{0,0,0}\makebox(0,0)[t]{\lineheight{0.80000001}\smash{\begin{tabular}[t]{c}2\end{tabular}}}}%
    \put(0.26204367,0.45691128){\color[rgb]{0,0,0}\makebox(0,0)[rt]{\lineheight{0.80000001}\smash{\begin{tabular}[t]{r}${\large M^3}$\end{tabular}}}}%
    \put(0.50412052,0.01828445){\color[rgb]{0,0,0}\makebox(0,0)[lt]{\lineheight{0.80000001}\smash{\begin{tabular}[t]{l}$p_1$\end{tabular}}}}%
    \put(0.37300171,0.04537632){\color[rgb]{0,0,0}\makebox(0,0)[rt]{\lineheight{0.80000001}\smash{\begin{tabular}[t]{r}$p_2$\end{tabular}}}}%
    \put(0.26092089,0.26967247){\color[rgb]{0,0,0}\makebox(0,0)[rt]{\lineheight{0.80000001}\smash{\begin{tabular}[t]{r}$p_3$\end{tabular}}}}%
    \put(0,0){\includegraphics[width=\unitlength,page=2]{diagonals.pdf}}%
    \put(0.15733635,1.14809479){\color[rgb]{0,0,0}\makebox(0,0)[t]{\lineheight{0.80000001}\smash{\begin{tabular}[t]{c}1\end{tabular}}}}%
    \put(0,0){\includegraphics[width=\unitlength,page=3]{diagonals.pdf}}%
    \put(0.51657314,0.33761892){\color[rgb]{0,0,0}\makebox(0,0)[lt]{\lineheight{0.80000001}\smash{\begin{tabular}[t]{l}$\Delta_{123}$\end{tabular}}}}%
    \put(0,0){\includegraphics[width=\unitlength,page=4]{diagonals.pdf}}%
  \end{picture}%
\endgroup%

%% file: collided-forest.pdf_tex
\begingroup%
  \makeatletter%
  \providecommand\color[2][]{%
    \errmessage{(Inkscape) Color is used for the text in Inkscape, but the package 'color.sty' is not loaded}%
    \renewcommand\color[2][]{}%
  }%
  \providecommand\transparent[1]{%
    \errmessage{(Inkscape) Transparency is used (non-zero) for the text in Inkscape, but the package 'transparent.sty' is not loaded}%
    \renewcommand\transparent[1]{}%
  }%
  \providecommand\rotatebox[2]{#2}%
  \newcommand*\fsize{\dimexpr\f@size pt\relax}%
  \newcommand*\lineheight[1]{\fontsize{\fsize}{#1\fsize}\selectfont}%
  \ifx\svgwidth\undefined%
    \setlength{\unitlength}{150.23622047bp}%
    \ifx\svgscale\undefined%
      \relax%
    \else%
      \setlength{\unitlength}{\unitlength * \real{\svgscale}}%
    \fi%
  \else%
    \setlength{\unitlength}{\svgwidth}%
  \fi%
  \global\let\svgwidth\undefined%
  \global\let\svgscale\undefined%
  \makeatother%
  \begin{picture}(1,0.69811321)%
    \lineheight{1}%
    \setlength\tabcolsep{0pt}%
    \put(0.44217246,0.58977298){\color[rgb]{0,0,0}\makebox(0,0)[lt]{\lineheight{0.80000001}\smash{\begin{tabular}[t]{l}1\end{tabular}}}}%
    \put(0.43767929,0.48356036){\color[rgb]{0,0,0}\makebox(0,0)[lt]{\lineheight{0.80000001}\smash{\begin{tabular}[t]{l}3\end{tabular}}}}%
    \put(0.16149228,0.3228135){\color[rgb]{0,0,0}\makebox(0,0)[lt]{\lineheight{0.80000001}\smash{\begin{tabular}[t]{l}5\end{tabular}}}}%
    \put(0.14784088,0.22674039){\color[rgb]{0,0,0}\makebox(0,0)[lt]{\lineheight{0.80000001}\smash{\begin{tabular}[t]{l}6\end{tabular}}}}%
    \put(0.22771694,0.61917208){\color[rgb]{0,0,0}\makebox(0,0)[rt]{\lineheight{0.80000001}\smash{\begin{tabular}[t]{r}$\huge M$\end{tabular}}}}%
    \put(0.51568003,0.16053794){\color[rgb]{0,0,0}\makebox(0,0)[t]{\lineheight{0.80000001}\smash{\begin{tabular}[t]{c}7\end{tabular}}}}%
    \put(0.71187556,0.24754513){\color[rgb]{0,0,0}\makebox(0,0)[t]{\lineheight{0.80000001}\smash{\begin{tabular}[t]{c}8\end{tabular}}}}%
    \put(0.71286018,0.17964626){\color[rgb]{0,0,0}\makebox(0,0)[t]{\lineheight{0.80000001}\smash{\begin{tabular}[t]{c}9\end{tabular}}}}%
    \put(0,0){\includegraphics[width=\unitlength,page=1]{collided-forest.pdf}}%
    \put(0.30084011,0.45396378){\color[rgb]{0,0,0}\makebox(0,0)[lt]{\lineheight{0.80000001}\smash{\begin{tabular}[t]{l}2\end{tabular}}}}%
    \put(0,0){\includegraphics[width=\unitlength,page=2]{collided-forest.pdf}}%
    \put(0.75562697,0.3313202){\color[rgb]{0,0,0}\makebox(0,0)[lt]{\lineheight{0.80000001}\smash{\begin{tabular}[t]{l}4\end{tabular}}}}%
    \put(0.40393657,0.23097546){\color[rgb]{0,0,0}\makebox(0,0)[t]{\lineheight{0.80000001}\smash{\begin{tabular}[t]{c}7\end{tabular}}}}%
  \end{picture}%
\endgroup%

%% file: collided-forest-explode.pdf_tex
\begingroup%
  \makeatletter%
  \providecommand\color[2][]{%
    \errmessage{(Inkscape) Color is used for the text in Inkscape, but the package 'color.sty' is not loaded}%
    \renewcommand\color[2][]{}%
  }%
  \providecommand\transparent[1]{%
    \errmessage{(Inkscape) Transparency is used (non-zero) for the text in Inkscape, but the package 'transparent.sty' is not loaded}%
    \renewcommand\transparent[1]{}%
  }%
  \providecommand\rotatebox[2]{#2}%
  \newcommand*\fsize{\dimexpr\f@size pt\relax}%
  \newcommand*\lineheight[1]{\fontsize{\fsize}{#1\fsize}\selectfont}%
  \ifx\svgwidth\undefined%
    \setlength{\unitlength}{87.87401575bp}%
    \ifx\svgscale\undefined%
      \relax%
    \else%
      \setlength{\unitlength}{\unitlength * \real{\svgscale}}%
    \fi%
  \else%
    \setlength{\unitlength}{\svgwidth}%
  \fi%
  \global\let\svgwidth\undefined%
  \global\let\svgscale\undefined%
  \makeatother%
  \begin{picture}(1,0.32258065)%
    \lineheight{1}%
    \setlength\tabcolsep{0pt}%
    \put(0.02807616,0.23648154){\color[rgb]{0,0,0}\makebox(0,0)[lt]{\lineheight{0.80000001}\smash{\begin{tabular}[t]{l}5\end{tabular}}}}%
    \put(0.02249741,0.13058415){\color[rgb]{0,0,0}\makebox(0,0)[lt]{\lineheight{0.80000001}\smash{\begin{tabular}[t]{l}6\end{tabular}}}}%
    \put(0.62601114,0.03008599){\color[rgb]{0,0,0}\makebox(0,0)[t]{\lineheight{0.80000001}\smash{\begin{tabular}[t]{c}7\end{tabular}}}}%
    \put(0.96144221,0.17883982){\color[rgb]{0,0,0}\makebox(0,0)[t]{\lineheight{0.80000001}\smash{\begin{tabular}[t]{c}8\end{tabular}}}}%
    \put(0.9631256,0.06275497){\color[rgb]{0,0,0}\makebox(0,0)[t]{\lineheight{0.80000001}\smash{\begin{tabular}[t]{c}9\end{tabular}}}}%
    \put(0,0){\includegraphics[width=\unitlength,page=1]{collided-forest-explode.pdf}}%
    \put(0.43496599,0.15051102){\color[rgb]{0,0,0}\makebox(0,0)[t]{\lineheight{0.80000001}\smash{\begin{tabular}[t]{c}7\end{tabular}}}}%
  \end{picture}%
\endgroup%

%% file: collided-config.pdf_tex
\begingroup%
  \makeatletter%
  \providecommand\color[2][]{%
    \errmessage{(Inkscape) Color is used for the text in Inkscape, but the package 'color.sty' is not loaded}%
    \renewcommand\color[2][]{}%
  }%
  \providecommand\transparent[1]{%
    \errmessage{(Inkscape) Transparency is used (non-zero) for the text in Inkscape, but the package 'transparent.sty' is not loaded}%
    \renewcommand\transparent[1]{}%
  }%
  \providecommand\rotatebox[2]{#2}%
  \newcommand*\fsize{\dimexpr\f@size pt\relax}%
  \newcommand*\lineheight[1]{\fontsize{\fsize}{#1\fsize}\selectfont}%
  \ifx\svgwidth\undefined%
    \setlength{\unitlength}{396.8503937bp}%
    \ifx\svgscale\undefined%
      \relax%
    \else%
      \setlength{\unitlength}{\unitlength * \real{\svgscale}}%
    \fi%
  \else%
    \setlength{\unitlength}{\svgwidth}%
  \fi%
  \global\let\svgwidth\undefined%
  \global\let\svgscale\undefined%
  \makeatother%
  \begin{picture}(1,0.35714286)%
    \lineheight{1}%
    \setlength\tabcolsep{0pt}%
    \put(0.25990268,0.10795217){\color[rgb]{0,0,0}\makebox(0,0)[lt]{\lineheight{0.80000001}\smash{\begin{tabular}[t]{l}$B$\end{tabular}}}}%
    \put(0.16257988,0.08066251){\color[rgb]{0,0,0}\makebox(0,0)[lt]{\lineheight{0.80000001}\smash{\begin{tabular}[t]{l}$C$\end{tabular}}}}%
    \put(0.15074938,0.15077744){\color[rgb]{0,0,0}\makebox(0,0)[lt]{\lineheight{0.80000001}\smash{\begin{tabular}[t]{l}$A$\end{tabular}}}}%
    \put(0.11810778,0.21190962){\color[rgb]{0,0,0}\makebox(0,0)[lt]{\lineheight{0.80000001}\smash{\begin{tabular}[t]{l}\Large $M$\end{tabular}}}}%
    \put(0.51909524,0.33127815){\color[rgb]{0,0,0}\makebox(0,0)[rt]{\lineheight{0.80000001}\smash{\begin{tabular}[t]{r}$\gr(H_\bullet)|_A$\end{tabular}}}}%
    \put(0.57858734,0.15787246){\color[rgb]{0,0,0}\makebox(0,0)[rt]{\lineheight{0.80000001}\smash{\begin{tabular}[t]{r}$\gr(H_\bullet)|_C$\end{tabular}}}}%
    \put(0.86701649,0.26993896){\color[rgb]{0,0,0}\makebox(0,0)[rt]{\lineheight{0.80000001}\smash{\begin{tabular}[t]{r}$\gr(H_\bullet)|_C$\end{tabular}}}}%
    \put(0.86664499,0.01031829){\color[rgb]{0,0,0}\makebox(0,0)[rt]{\lineheight{0.80000001}\smash{\begin{tabular}[t]{r}$\gr(H_\bullet)|_C$\end{tabular}}}}%
    \put(0,0){\includegraphics[width=\unitlength,page=1]{collided-config.pdf}}%
    \put(0.48407184,0.29395471){\color[rgb]{0,0,0}\makebox(0,0)[lt]{\lineheight{0.80000001}\smash{\begin{tabular}[t]{l}1\end{tabular}}}}%
    \put(0.40711803,0.23198378){\color[rgb]{0,0,0}\makebox(0,0)[lt]{\lineheight{0.80000001}\smash{\begin{tabular}[t]{l}2\end{tabular}}}}%
    \put(0.47847672,0.10182443){\color[rgb]{0,0,0}\makebox(0,0)[t]{\lineheight{0.80000001}\smash{\begin{tabular}[t]{c}56\end{tabular}}}}%
    \put(0.54386428,0.10400284){\color[rgb]{0,0,0}\makebox(0,0)[t]{\lineheight{0.80000001}\smash{\begin{tabular}[t]{c}789\end{tabular}}}}%
    \put(0.48187757,0.24215954){\color[rgb]{0,0,0}\makebox(0,0)[lt]{\lineheight{0.80000001}\smash{\begin{tabular}[t]{l}3\end{tabular}}}}%
    \put(0,0){\includegraphics[width=\unitlength,page=2]{collided-config.pdf}}%
    \put(0.76828156,0.06087464){\color[rgb]{0,0,0}\makebox(0,0)[t]{\lineheight{0.80000001}\smash{\begin{tabular}[t]{c}7\end{tabular}}}}%
    \put(0.84521849,0.09051947){\color[rgb]{0,0,0}\makebox(0,0)[t]{\lineheight{0.80000001}\smash{\begin{tabular}[t]{c}8\end{tabular}}}}%
    \put(0.84468907,0.06548579){\color[rgb]{0,0,0}\makebox(0,0)[t]{\lineheight{0.80000001}\smash{\begin{tabular}[t]{c}9\end{tabular}}}}%
    \put(0,0){\includegraphics[width=\unitlength,page=3]{collided-config.pdf}}%
    \put(0.78487745,0.22398039){\color[rgb]{0,0,0}\makebox(0,0)[lt]{\lineheight{0.80000001}\smash{\begin{tabular}[t]{l}5\end{tabular}}}}%
    \put(0.78087774,0.17844846){\color[rgb]{0,0,0}\makebox(0,0)[lt]{\lineheight{0.80000001}\smash{\begin{tabular}[t]{l}6\end{tabular}}}}%
    \put(0,0){\includegraphics[width=\unitlength,page=4]{collided-config.pdf}}%
  \end{picture}%
\endgroup%

%% file: collided-config-tree.pdf_tex
\begingroup%
  \makeatletter%
  \providecommand\color[2][]{%
    \errmessage{(Inkscape) Color is used for the text in Inkscape, but the package 'color.sty' is not loaded}%
    \renewcommand\color[2][]{}%
  }%
  \providecommand\transparent[1]{%
    \errmessage{(Inkscape) Transparency is used (non-zero) for the text in Inkscape, but the package 'transparent.sty' is not loaded}%
    \renewcommand\transparent[1]{}%
  }%
  \providecommand\rotatebox[2]{#2}%
  \newcommand*\fsize{\dimexpr\f@size pt\relax}%
  \newcommand*\lineheight[1]{\fontsize{\fsize}{#1\fsize}\selectfont}%
  \ifx\svgwidth\undefined%
    \setlength{\unitlength}{150.23622047bp}%
    \ifx\svgscale\undefined%
      \relax%
    \else%
      \setlength{\unitlength}{\unitlength * \real{\svgscale}}%
    \fi%
  \else%
    \setlength{\unitlength}{\svgwidth}%
  \fi%
  \global\let\svgwidth\undefined%
  \global\let\svgscale\undefined%
  \makeatother%
  \begin{picture}(1,0.69811321)%
    \lineheight{1}%
    \setlength\tabcolsep{0pt}%
    \put(0.45339178,0.57610072){\color[rgb]{0,0,0}\makebox(0,0)[lt]{\lineheight{0.80000001}\smash{\begin{tabular}[t]{l}1\end{tabular}}}}%
    \put(0.3328829,0.49113004){\color[rgb]{0,0,0}\makebox(0,0)[rt]{\lineheight{0.80000001}\smash{\begin{tabular}[t]{r}$A$\end{tabular}}}}%
    \put(0.3986246,0.29827978){\color[rgb]{0,0,0}\makebox(0,0)[t]{\lineheight{0.80000001}\smash{\begin{tabular}[t]{c}$C$\end{tabular}}}}%
    \put(0.68359807,0.33239969){\color[rgb]{0,0,0}\makebox(0,0)[rt]{\lineheight{0.80000001}\smash{\begin{tabular}[t]{r}$B$\end{tabular}}}}%
    \put(0.2146268,0.39236591){\color[rgb]{0,0,0}\makebox(0,0)[lt]{\lineheight{0.80000001}\smash{\begin{tabular}[t]{l}2\end{tabular}}}}%
    \put(0.44425428,0.41273369){\color[rgb]{0,0,0}\makebox(0,0)[lt]{\lineheight{0.80000001}\smash{\begin{tabular}[t]{l}3\end{tabular}}}}%
    \put(0.1541343,0.29863674){\color[rgb]{0,0,0}\makebox(0,0)[lt]{\lineheight{0.80000001}\smash{\begin{tabular}[t]{l}5\end{tabular}}}}%
    \put(0.13591236,0.15639082){\color[rgb]{0,0,0}\makebox(0,0)[lt]{\lineheight{0.80000001}\smash{\begin{tabular}[t]{l}6\end{tabular}}}}%
    \put(0.74862563,0.30206051){\color[rgb]{0,0,0}\makebox(0,0)[lt]{\lineheight{0.80000001}\smash{\begin{tabular}[t]{l}4\end{tabular}}}}%
    \put(0.22771694,0.61917208){\color[rgb]{0,0,0}\makebox(0,0)[rt]{\lineheight{0.80000001}\smash{\begin{tabular}[t]{r}$\huge M$\end{tabular}}}}%
    \put(0.39477158,0.08183044){\color[rgb]{0,0,0}\makebox(0,0)[t]{\lineheight{0.80000001}\smash{\begin{tabular}[t]{c}7\end{tabular}}}}%
    \put(0.66463701,0.18646784){\color[rgb]{0,0,0}\makebox(0,0)[t]{\lineheight{0.80000001}\smash{\begin{tabular}[t]{c}8\end{tabular}}}}%
    \put(0.66562163,0.11856896){\color[rgb]{0,0,0}\makebox(0,0)[t]{\lineheight{0.80000001}\smash{\begin{tabular}[t]{c}9\end{tabular}}}}%
    \put(0,0){\includegraphics[width=\unitlength,page=1]{collided-config-tree.pdf}}%
  \end{picture}%
\endgroup%

%% file: manifold-intersections.tex
We require various notions of regularity of submanifold intersections:

\begin{definition}\label{rem:transversal}\label{def:transversal}
Consider a collection $\cS$ of submanifolds of $M$ and its intersection $R=\bigcap\cS$. We say $\cS$...
\begin{enumerate}
    \item \textbf{intersects cleanly} if $R$ is a submanifold with $T_pR=\bigcap\limits_{S\in\cS}T_pS$ at every $p\in R$,
    \item \textbf{intersects like coordinate subspaces} if around each $p\in R$ there exist charts of $M$ under which the elements of $\cS$ are coordinate subspaces,
    \item \textbf{intersects transversely} if it intersects cleanly and further $$\operatorname{codim} R = \sum\limits_{S\in\cS} \operatorname{codim} S.$$
\end{enumerate}
\end{definition}

Transverse intersection means that each of the submanifolds represent an independent constraint, or equivalently that their conormal bundles are linearly independent.
It implies intersection like coordinate subspaces, which in turn implies clean intersection. Perhaps deceptively, any collection consisting of a single submanifold trivially intersects transversely. If the number of proper submanifolds intersecting in a point is higher than the dimension of the ambient space, they cannot intersect transversely.

Also recall that we can check whether an intersection is clean using the vanishing ideals of a submanifold:

\begin{lemma}\label{lem:clean-vanishing-ideal}
A collection $\cS$ of submanifolds of $M$ intersects cleanly if and only if they intersect in a submanifold and $$I_{\cap\cS}=\sum_{S\in\cS}I_S$$ holds\footnote{Note that on the right hand side, we use the convention that the sum of ideals $\sum_{S\in\cS}I_S$ is defined as the ideal generated by sums of elements in each $I_S$.} for their vanishing ideals.
\end{lemma}

\begin{proof}\ \\
\textit{Regarding $(\Rightarrow)$:} Let $\cS$ be a collection of manifolds that cleanly intersect in $\cap\cS$. We only need to show $$I_{\cap\cS}\subseteq\sum_{S\in\cS}I_S,$$ as the other inclusion is automatic. Close to some point $p\in\cap\cS$, each $S\in\cS$ is cut out by some functions $\{f_{S,i}\}_{i=1, ..., \codim S}$. Due to clean intersection, we have $$\bigcap_{S,i}\ker d_p f_{S,i} = \bigcap_{S} T_p S = T_p (\cap \cS).$$
As the right-hand side has the codimension of $\cap\cS$, there must be a subset of that many $f_{S,i}$ that are linearly independent at $p$. As $\cap\cS$ is a submanifold, it must also be cut out by that subset, and any function that vanishes on $\cap\cS$ is locally a $C^\infty(M)$-linear combination of these functions by a version of Hadamard's Lemma for submanifolds.

\textit{Regarding $(\Leftarrow)$:} Assume that $\cap\cS$ is a submanifold, its vanishing ideal is generated by the sum $\sum_{S\in\cS}I_S$ and $p\in\cap\cS$. We only need to show that $$\bigcap_{S\in\cS}T_pS \subseteq  T_p(\cap\cS),$$
as the other inclusion is automatic. Thus it suffices to show for any vector $v$ on the left-hand side and function $f\in I_{\cap\cS}$ that $d_pf(v)=0$. By assumption, there must be a local decomposition $f=\sum_{S,j}\alpha_{S,j}f_{S,j}$ for $f_{S,j}$ as above. Then $$d_pf(v)=\sum_{S,j} \alpha_{S,j} d_p f_{S,j}(v) = 0,$$ where we used $f_{S,j}\in I_{\cap\cS}$ and $v\in T_pS$.
\end{proof}

%% file: proof-weighted-submanifolds.tex
We tackle the main goal of this appendix in the first \subchaption. This is to prove Theorem~\ref{thm:weighted-blow-up-atlas}, i.e. that the charts from Definition~\ref{def:smooth-blow-up} yield a nice manifold structure on the blow-up of a weighted submanifold. In \Subchaption~\ref{ssec:induced-smooth}, we prove that the induced maps between these blow-ups are defined on an open domain (Corollary~\ref{cor:domain-induced-open}) and are smooth (Proposition~\ref{prop:induced-map-smooth}).

\startSubchaption{Smoothness of transition functions}

Let us consider two charts $\chi_{hs}$ and $\tilde\chi_{\tilde h\tilde s}$ from Def.~\ref{def:smooth-blow-up} and assume that their domains overlap. From here on out we will write $w_i$ and $\tilde w_i$ for the weights assigned to the $i$-th coordinate direction under $\chi$ and $\tilde\chi$, respectively. We want to consider the transition map
\begin{equation}\label{eq:transition-function}
(y_1, ..., y_m)\mapsto(\tilde y_1, ..., \tilde y_m) := \tilde\chi_{\tilde h\tilde s}\circ\chi^{-1}_{hs}(y_1, ..., y_m)
\end{equation}
that sends the coordinates $(y_1, ..., y_m)$ under $\chi_{hs}$ of a point in $\Bl_\cW(M)$ to its coordinates $(\tilde y_1, ..., \tilde y_m)$ under $\tilde\chi_{\tilde h\tilde s}$.

Recall from Equation~\eqref{eq:local-blow-down} that the blow down map locally takes the form
$$\chi\circ b_\cW \circ \chi_{hs}^{-1} (y_1, ..., y_m) = y_h \cdot q,$$
where we use the shorthand $$q:=(y_1, ..., s, ..., y_m)$$ here and in the rest of this \subchaption. Note that $q$ can be seen as the coordinates of a weighted normal vector $\left(\chi^{(\cW)}\right)^{-1}(q)$ in $\nu\cW$ (compare Paragraphs~\ref{par:smooth-weighted-normal-local} and~\ref{smooth-structure-local-case}).

To start, let us compute an explicit expression for the transition map:

\begin{lemma}\label{lem:local-trans-blowdown}
The transition map from Equation~\eqref{eq:transition-function} is given for $i\neq \tilde h$ by
$$
\tilde y_i = \begin{cases}
\frac{
(\tilde \chi\circ\chi^{-1})_i(y_h\cdot q)
}{
\left( \tilde s\;\,(\tilde \chi\circ\chi^{-1})_{\tilde h}(y_h\cdot q) \right)^{\tilde w_i/\tilde w_{\tilde h}}
}
&\text{if }y_h>0,\\
\frac{
\left(\tilde \chi^{(\cW)}\circ\left(\chi^{(\cW)}\right)^{-1}\right)_i(q)
}{
\left( \tilde s\,\left(\tilde \chi^{(\cW)}\circ\left(\chi^{(\cW)}\right)^{-1}\right)_{\tilde h}(q) \right)^{\tilde w_i/\tilde w_{\tilde h}}
}
&\text{if }y_h=0
\end{cases}
$$
as well as
$$
\tilde y_{\tilde h} = \begin{cases}
 \left(\tilde s\, (\tilde\chi\circ\chi^{-1})_{\tilde h}(y_h\cdot q)  \right)^{1/\tilde w_{\tilde h}}
&\text{if }y_h>0, \\
0 &\text{if }y_h=0.
\end{cases}
$$
\end{lemma}

\begin{proof}
We use the characterization of the charts from Equation~\eqref{eq:inv-coords-local} to relate the $y_i$ and $\tilde y_i$ and make a case distinction based on whether $y_{h}$ vanishes:

\textit{Case 1:} If $y_{h}>0$, then we can solve the $\tilde h$-th component of the resulting equations for $\tilde y_{\tilde h}$ to get
$$
\tilde y_{\tilde h}=\left( \tilde s\;\,(\tilde \chi\circ\chi^{-1})_{\tilde h}(y_h\cdot q) \right)^{1/\tilde w_{\tilde h}},
$$
where for convenience we used $q :=(y_1, ..., s, ..., y_m)$.
Note that the sign $\tilde s$ ensures that we are taking a root of something positive, given that we are considering points in $\tilde U_{\tilde h\tilde s}$. We can insert this into the remaining components to obtain
\begin{equation}\label{eq:yi-in-bulk}
\tilde y_i = \frac{
(\tilde \chi\circ\chi^{-1})_i(y_h\cdot q)
}{
\left( \tilde s\;\,(\tilde \chi\circ\chi^{-1})_{\tilde h}(y_h\cdot q) \right)^{\tilde w_i/\tilde w_{\tilde h}}
}\qquad\text{for $i\neq \tilde h$.}
\end{equation}

\textit{Case 2:} If $y_h=0$, then Equation~\eqref{eq:inv-coords-local} tells us that there exists a $\lambda\in\R_{>0}$ such that $$\left( \tilde \chi^{(\cW)} \right)^{-1}(\tilde q) = \left( \chi^{(\cW)} \right)^{-1}(\lambda\cdot q)$$
with $\tilde q=(\tilde y_1, ..., \tilde s, ..., \tilde y_m)$. Solving the $\tilde h$-th component of $\tilde q$ for $\lambda$ gives
$$
\lambda=\left( \tilde s\,\left(\tilde \chi^{(\cW)}\circ\left(\chi^{(\cW)}\right)^{-1}\right)_{\tilde h}(q) \right)^{-1/\tilde w_{\tilde h}}.
$$
Inserted into the remaining components, we get
\begin{equation}\label{eq:yi-in-divisor}
\tilde y_i = \frac{
\left(\tilde \chi^{(\cW)}\circ\left(\chi^{(\cW)}\right)^{-1}\right)_i(q)
}{
\left( \tilde s\,\left(\tilde \chi^{(\cW)}\circ\left(\chi^{(\cW)}\right)^{-1}\right)_{\tilde h}(q) \right)^{\tilde w_i/\tilde w_{\tilde h}}
}
\end{equation}
for $i\neq \tilde h$. Since $y_{h}=0$ corresponds to points on the exceptional divisor, we must also have $\tilde y_{\tilde h}=0$.
\end{proof}

We recall a weak version of the Malgrange Preparation Theorem:

\begin{theorem}\label{thm:malgrange}
Let $f$ be a smooth real-valued function defined in a neighbourhood of the origin in $\R\times\R^m$. Assume that $\partial^i f/\partial t^i(t,x)$ vanishes at $t=0$ for all $x$ and $0\leq i<w$.

Then there exists a smooth and real-valued function $g$ close to the origin of $\R\times\R^m$ such that $$f(t,x)=t^w\cdot g(t,x).$$ It follows immediately that $$g(0,x)=\frac{1}{w!}\frac{\partial^w f}{\partial t^w}(0,x).$$
In particular, $g$ is non-zero at the origin exactly when the $w$-th derivative of $f$ does not vanish.
\end{theorem}

Compared to other formulations like~\cite[][Theorem 2.1, Note (1)]{GG73}, all terms with powers of $t$ at a lower order than $w$ vanish since our assumption on the differentials hold for all $x$ instead of just the origin.

The transition functions repeatedly include the expression $\tilde{s}\,(\tilde\chi\circ\chi^{-1})_i(t\cdot q)$ for $t\neq0$. The following Lemma uses Malgrange Preparation to determine the order of vanishing of this expression at $t=0$ and shows that the extension to zero is smooth: 

\begin{lemma}\label{lem:useful-transition-relation}
    For each sign $\tilde{s}$ and index $i\in\{1, ..., m\}$ there exists a smooth function $g(t,q)$, defined for all those $t\in[0,\infty)$ and $q\in\R^m$ with $\chi^{-1}(t\cdot q)$ well-defined and contained in the domain of $\tilde\chi$, such that
$$
\tilde{s}\,(\tilde\chi\circ\chi^{-1})_i(t\cdot q) = t^{\tilde w_i} \cdot g(t,q).
$$
Here $t\cdot q$ is the action associated with $\cW$ and we have
$$
g(0,q) = \tilde{s}\,\left( \tilde\chi^{(\cW)}\circ\left(\chi^{(\cW)}\right)^{-1} \right)_i(q).
$$
In particular, $g(0,q)>0$ if $\left(\chi^{(\cW)}\right)^{-1}(q)\in \tilde U_{i\tilde{s}}$.
\end{lemma}

This lemma can be understood to mean the following: Applying the weighted action of $\cW$ in the argument of $\tilde\chi\circ\chi^{-1}$ is the same as applying the weighted action of $\widetilde{\cW}$ in the target, at least up to leading order.

\begin{proof}
We apply Theorem~\ref{thm:malgrange} to the map $$f(t,q):=s\,(\tilde\chi\circ\chi^{-1})_i(t\cdot q).$$ We therefore have to verify for all $q$ that 
\begin{equation}\label{eq:deriv}
\frac{1}{\tilde w_i!}\left.\frac{d^{\tilde w_i}}{dt^{\tilde w_i}}\right|_{t=0}\; s\,(\tilde\chi\circ\chi^{-1})_i(t\cdot q)
=
s\,\left( \tilde\chi^{(\cW)}\circ\left(\chi^{(\cW)}\right)^{-1} \right)_i(q)
\end{equation}
holds, and that all lower-order derivatives vanish. By Equation~\eqref{eq:inverse-coords-normal},
$$
\left(\chi^{(\cW)}\right)^{-1} (q) = [t\mapsto \chi^{-1}(t\cdot q)]
$$
and thus for all $i=1, ..., m$ we have
\begin{align*}
\left(\tilde \chi^{(\cW)}\circ\left(\chi^{(\cW)}\right)^{-1}\right)_i(q)
&= \tilde x_{i}^{(\tilde w_i)}\left( \left(\chi^{(\cW)}\right)^{-1} (q) \right) \\
 &= \frac{1}{\tilde w_i!} \left.\frac{d^{\tilde w_i}}{dt^{\tilde w_i}}\right|_{t=0} \left(\tilde \chi\circ\chi^{-1}\right)_i(t\cdot q)
\end{align*}
All derivatives of lower order $j<\tilde w_i$ are proportional to $\tilde x_i^{(j)}\left( \left(\chi^{(\cW)}\right)^{-1} (q) \right)$ instead, which must vanish for all elements in the weighting $\cW$ by the definition of weightings.

The statement about the sign of $g(0,q)$ is immediate from Equation~\eqref{eq:deriv} and the definition of the chart domain $\tilde{U}_{i\tilde s}.$
\end{proof}

We are now ready to prove that the charts from Def.~\ref{def:smooth-blow-up} induce a manifold structure on the blow-up of a weighted submanifold:

\begin{proof}[Proof of Thm.~\ref{thm:weighted-blow-up-atlas}]
We need our candidate for coordinate charts to be bijective, induce smooth transition maps and a second countable Hausdorff topology. We have already given the inverse of the charts in Equation~\eqref{eq:inv-coords-local}, so bijectivity is clear.
Any two points $p_1, p_2$ in the blow-up either lie in some common chart domain or can be separated by chart domains\footnote{If $b_\cW(p_1)\neq b_\cW(p_2)$, one may simply choose disjoint charts on a small enough domain. If $b_\cW(p_1)= b_\cW(p_2)\not\in\supp\cW$, then $p_1=p_2$ can be covered by a single chart. If $b_\cW(p_1)= b_\cW(p_2)\in\supp\cW$ and the points correspond to antipodal weighted normal vectors, then one may choose a chart $\chi$ in which they are coordinate vectors such that the induced charts with opposite signs separate them. If they are not antipodal vectors, then any chart of $M$ with a coordinate direction along which both $p_1$ and $p_2$ have a positive component will induce a chart that covers both.}. As the atlas clearly has a countable subatlas, this already suffices to conclude the topological properties (compare e.g. Proposition~1.32 of~\cite{Lee}).
So we are left to prove only smoothness of the transition maps.

Assuming that the domains of two charts $\chi_{hs}$ and $\tilde\chi_{\tilde h \tilde s}$ overlap, we get transition functions as in Lemma~\ref{lem:local-trans-blowdown}. Smoothness is clear away from $y_h=0.$

Consider the component $\tilde{y}_{\tilde{h}}$ for $y_h\neq0$: According to Lemma~\ref{lem:useful-transition-relation}, we can write
$$\tilde{y}_{\tilde{h}}=\left(  y_h^{\tilde{w}_{\tilde{h}}}\cdot g(y_1,...,y_m)  \right)^{1/\tilde{w}_{\tilde{h}}} = y_h\cdot g(y_1,...,y_m)^{1/\tilde{w}_{\tilde{h}}}$$ for some function $g$ with $g(y_1,...,y_m)>0$ whenever $y_h=0$ due to $\left(\chi^{(\cW)}\right)^{-1}(y_1,...,s,...,y_m)\in \tilde U_{\tilde{h}\tilde{s}}$. In particular, the component $\tilde{y}_{\tilde{h}}$ smoothly extends to $y_h=0$ by setting it to 0, as required.

We can similarly apply Lemma~\ref{lem:useful-transition-relation} for the other components $\tilde{y}_i$ with $i\neq\tilde{h}$ at $y_h\neq0$ to obtain
$$\tilde{y}_i = \frac{y_h^{\tilde{w}_i}\cdot k_1(y_1,...,y_m)}{
\left( y_h^{\tilde{w}_{\tilde{h}}}\cdot k_2(y_1, ..., y_m) \right)^{\tilde{w}_i/\tilde{w}_{\tilde{h}}}}
=
\frac{k_1(y_1,...,y_m)}{
k_2(y_1, ..., y_m)^{\tilde{w}_i/\tilde{w}_{\tilde{h}}}}
$$
for smooth maps $k_1$ and $k_2$, where $k_2>0$ for $y_h=0$ due to $\chi^{-1}_{hs}(y_1,...,y_m)\in \tilde U_{\tilde{h}\tilde{s}}$. This expression is clearly smooth also when $y_h=0$. To conclude smoothness of the transition functions, we can easily check that evaluating the $k_i$ at $y_h=0$ according to Lemma~\ref{lem:useful-transition-relation} matches our expression for the transition maps from Lemma~\ref{lem:local-trans-blowdown}. We have thus shown that the charts define a manifold structure on the blow-up.

Finally, note that the local expression in Equation~\eqref{eq:local-blow-down} for the blow-down map $b_\cW$ is clearly smooth and a closed map. As the fibers of the blow-down map are compact (being either points or spheres), $b_\cW$ is also proper.
\end{proof}

\startSubchaption{Smoothness of induced maps}\label{ssec:induced-smooth}

We want to conclude the discussion of smooth structures by proving smoothness of induced maps $b_\phi$ and openness of their domain. Thus for the rest of this \subchaption, let $\phi:(M,\cW)\to(\widetilde M, \widetilde\cW)$ be a morphism of weighted manifolds and consider weighted coordinates $(\chi,U)$ and $(\tilde\chi,\tilde U)$ on the domain and codomain. Let $w_i$ and $\tilde w_i$ be the respective weights of the $i$-th coordinate direction.

We first need a straightforward generalization of Lemma~\ref{lem:useful-transition-relation}:

\begin{lemma}\label{lem:useful-transition-relation-phi}
    For each sign $\tilde{s}$ and index $i\in\{1, ..., m\}$, there exists a smooth function $g(t,q)$ defined for all those $t\in[0,\infty)$ and $q\in\R^m$ with $\phi\circ\chi^{-1}(t\cdot q)$ well-defined and contained in the domain of $\tilde\chi$ such that we can write 
$$
\tilde{s}\,(\tilde\chi\circ\phi\circ\chi^{-1})_i(t\cdot q) = t^{\tilde w_i} \cdot g(t,q).
$$
Here $t\cdot q$ is the action associated with $\cW$, and we have
$$
g(0,q) = \tilde{s}\,\left( \tilde\chi^{(\widetilde{\cW})}\circ\nu_\phi\circ\left(\chi^{(\cW)}\right)^{-1} \right)_i(q).
$$
In particular, $g(0,q)>0$ if $\nu_\phi\circ\left(\chi^{(\cW)}\right)^{-1}(q)\in \tilde U_{i\tilde{s}}$.
\end{lemma}

\begin{proof}
Proceed exactly as in the proof of Lemma~\ref{lem:useful-transition-relation}, just inserting $\phi$ and $\nu_\phi$ in the appropriate spots.
\end{proof}

We now turn towards proving openness of $\Bl_\phi$. To that end, we consider for a concrete choice of adapted charts $(\chi,U)$ and $(\tilde\chi,\tilde U)$ the helper map
$$f:\chi_{hs}\left(U_{hs}\cap\Bl_\cW(\phi^{-1}(\tilde U))\right) \to \R^{I}$$
for $I:=\{i\in\{1,...,\dim \widetilde M\}\;|\; \tilde w_i>0 \}$
defined by 
    $$
f(y_1, ..., y_m)_i:=\begin{cases}
y_h^{-\tilde w_{i}} \left(\widetilde \chi\circ\phi\circ\chi^{-1}\right)_{i}(y_h\cdot q) 
 &\text{for }y_h>0, \\
\left(\widetilde\chi^{(\widetilde\cW)}\circ\nu_\phi\circ(\chi^{(\cW)})^{-1}\right)_{i}(q),
 &\text{for }y_h=0,
\end{cases}
    $$
    where we again use the shorthand $q=(y_1, ..., s, ..., y_m)$.

This map cuts out the part of the subset $\Bl_\phi$ covered by the chosen charts:

\begin{lemma}\label{lem:helper-cuts}
The map $f$ satisfies
    $$
(y_1, ..., y_m)\in \chi_{hs}\left(\Bl_\phi(M)\cap U_{hs}\cap\Bl_\cW(\phi^{-1}(\tilde U)) \right) \quad\iff\quad (y_1, ..., y_m)\in f^{-1}(\R^{I}\setminus\{0\}).
    $$
\end{lemma}

\begin{proof}
Let us characterize the coordinates $(y_1, ..., y_m)$ which correspond to points in $\Bl_\phi(M)$.
If $y_h>0$, they are exactly those coordinates which do not get mapped into $\supp\widetilde\cW$ by $\phi$, i.e. those that satisfy
$$ \left(\widetilde \chi\circ\phi\circ\chi^{-1}\right)_{i}(y_h\cdot q)\neq0$$ for some $i\in I$, or equivalently $f(y_1,...,y_m)\neq 0$.
If on the other hand $y_h=0$, the coordinates describe a point $[n]$ on the exceptional divisor of $\Bl_\cW(M)$ that lies in $\Bl_\phi(M)$ if and only if $\nu_\phi(n)\neq 0$, or equivalently $$\left(\widetilde\chi^{(\widetilde\cW)}\circ\nu_\phi\circ(\chi^{(\cW)})^{-1}\right)_{i}(q)\neq0$$ for some $i\in I.$
Thus in both cases $\Bl_\phi(M)$ is the subset where $f>0$.
\end{proof}

\begin{lemma}\label{lem:helper-smooth}
The map $f$ is smooth.
\end{lemma}

\begin{proof}
This is an immediate consequence of smoothness of $g$ in Lemma~\ref{lem:useful-transition-relation-phi}.
\end{proof}

This already allows us to conclude openness of the domain:

\begin{corollary}\label{cor:domain-induced-open}
The subset $\Bl_\phi$ of $\Bl_\cW(M)$ is open.
\end{corollary}

\begin{proof}
Away from the exceptional divisor of $(M,\cW)$, openness of $\Bl_\phi$ is automatic since the blow-down map is a diffeomorphism. Close to the exceptional divisor, openness can be checked in $U_{hs}\cap\Bl_\cW(\phi^{-1}(\tilde U))$ for appropriately chosen charts. But there it is by Lemmas~\ref{lem:helper-cuts} and~\ref{lem:helper-smooth} exactly the continuous preimage of the open set $\R^I\setminus\{0\}$.
\end{proof}

To prove smoothness of the induced map $b_\phi$, we first consider its local representation in induced charts:

\begin{lemma}\label{lem:local-rep-induced-map}
The local representation
$$
    (y_1, ..., y_m)\mapsto(\tilde y_1, ..., \tilde y_m) := \tilde\chi_{\tilde h\tilde s}\circ b_\phi\circ\chi^{-1}_{hs}(y_1, ..., y_m)$$
of $b_\phi$ is for $i\neq\tilde h$ given by
$$
\tilde y_i = \begin{cases}
\frac{
(\tilde \chi\circ\phi\circ\chi^{-1})_i(y_h\cdot q)
}{
\left( \tilde s\;\,(\tilde \chi\circ\phi\circ\chi^{-1})_{\tilde h}(y_h\cdot q) \right)^{\tilde w_i/\tilde w_{\tilde h}}
}
&\text{if }y_h>0,\\
\frac{
\left(\tilde\chi^{(\widetilde\cW)}\circ\nu_\phi\circ(\chi^{(\cW)})^{-1}\right)_{i}(q)

}{
\left(\tilde s\,
\left(\tilde\chi^{(\widetilde\cW)}\circ\nu_\phi\circ(\chi^{(\cW)})^{-1}\right)_{\tilde h}(q)
\right)^{\tilde w_{i}/\tilde w_{\tilde h}}

}
&\text{if }y_h=0,
\end{cases}
$$
as well as
$$
\tilde y_{\tilde h} = \begin{cases}
\left(\tilde s\,
\left(\tilde\chi\circ\phi\circ\chi^{-1}\right)_{\tilde h}(y_h\cdot q)
\right)^{1/\tilde w_{\tilde h}}

&\text{if }y_h>0, \\
0 &\text{if }y_h=0.
\end{cases}
$$
\end{lemma}

\begin{proof}
Using Eq~\eqref{eq:inv-coords-local}, the definition of $b_\phi$ and the expression for the inverse of the coordinates from Equation~\eqref{eq:inverse-coords-normal}, we obtain
$$
(\tilde y_1, ..., \tilde y_m)
= \widetilde\chi_{\tilde h\tilde s}
\begin{cases}
\phi\circ\chi^{-1}(y_h\cdot q) &\text{for }y_h>0,\\
\left[ t\mapsto \phi\circ\chi^{-1}(t\cdot q) \right] &\text{for }y_h=0,
\end{cases}
$$
where still $q=(y_1, ..., s, ..., y_m)$. Using the explicit expression of $\widetilde\chi_{\tilde h\tilde s}$ from Definition~\ref{def:smooth-blow-up}, it follows for $i\neq \tilde h$ that
$$
\tilde y_i = 
\begin{cases}
\frac{

(\tilde\chi\circ\phi\circ\chi^{-1})_i(y_h\cdot q)

}{

\left(
\tilde s\, (\tilde\chi\circ\phi\circ\chi^{-1})_{\tilde h}(y_h\cdot q)
\right)^{\tilde w_i/\tilde w_{\tilde h}}

} &\text{for } y_h>0,\\
\frac{

\frac{1}{\tilde w_i!} \left.\frac{d^{\tilde w_i}}{dt^{\tilde w_i}}\right|_{t=0}
(\tilde\chi\circ\phi\circ\chi^{-1})_i(y_h\cdot q)

}{
\left(

\frac{1}{\tilde w_{\tilde h}!} \left.\frac{d^{\tilde w_{\tilde h}}}{dt^{\tilde w_{\tilde h}}}\right|_{t=0}
(\tilde\chi\circ\phi\circ\chi^{-1})_{\tilde h}(t \cdot q)

\right)^{\tilde w_i/\tilde w_{\tilde h}}
} &\text{for } y_h=0,
\end{cases}
$$
as well as
$$
\tilde y_{\tilde h}=
\begin{cases}
\left(
\tilde s\, (\tilde\chi\circ\phi\circ\chi^{-1})_{\tilde h}(y_h\cdot q)
\right)^{1/\tilde w_{\tilde h}}
&\text{for } y_h>0,\\
0
&\text{for } y_h=0.
\end{cases}
$$
The derivatives in the expression for $\tilde{y}_i$ when $y_h=0$ can be rewritten in terms of weighted normal vectors exactly as in the proof of Lemma~\ref{lem:useful-transition-relation}.
\end{proof}

\begin{proposition}\label{prop:induced-map-smooth}
The induced map $b_\phi$ is smooth.
\end{proposition}

\begin{proof}
Smoothness of the local representatives of $b_\phi$ near the exceptional divisor is clear by applying Lemma~\ref{lem:useful-transition-relation-phi} to the expression for $\tilde{y}_{\tilde{h}}$ as well as to both the numerator and denominator of $\tilde{y}_i$, canceling powers of $y_h$ and noting that the denominator is then non-zero. Smoothness of $b_\phi$ away from the exceptional divisors follows immediately since it coincides with $\phi$ and the blow-down map is a diffeomorphism there.
\end{proof}

%% file: proof-weighted-building-set.tex
In this appendix, we will prove our main Theorem~\ref{thm:manifold-structure}, i.e. that the charts from Definition~\ref{def:building-local-charts} yield the structure of a smooth manifold with corners on the blow up $\Bl_{\cW}(M)$ of a uniformly aligned weighting $\cW$ along a separated building set $\cG$, and furthermore that the blow-down map $b_\cW:\Bl_\cW(M)\to M$ is smooth and proper. We recommend the reader to first familiarize themselves with the simpler proof for a single weighted submanifold in Appendix~\ref{app:proof-weighted-submanifolds}. In particular, we will start by giving adapted versions of the preparatory Lemmas before tackling smoothness of the transition functions.

\startSubchaption{Preparatory lemmas}

We first establish a technical lemma that allows us to commute induced maps on normal bundles with charts as well as commute the weighted action of some $t\in\R$ with application of $\tilde\chi^{(\widetilde\cW)}\circ\nu_{\widetilde\cW,\cW}\circ \left( \chi^{(\cW)} \right)^{-1}$, i.e. the representative of the induced map between different weighted normal bundles in local charts.
As such, this will serve the same function as Lemma~\ref{lem:useful-transition-relation}, but now in the presence of further relevant weightings.

\begin{lemma}\label{lem:projection-in-charts}
Let $\cW\subseteq\widetilde\cW\subseteq\cU$ be weightings on a manifold $M$ of dimension $m$. Assume $\cU$ and $\widetilde\cW$ are aligned with a chart $\tilde\chi$ and that $\cU$ and $\cW$ are aligned with a chart $\chi$. Let $i\in\{1, ..., m\}$ be the index of a coordinate direction under $\tilde\chi$ such that the weights of $\cU$ and $\widetilde\cW$ match. Then for any $q\in\R^m$ in the image of $\chi^{(\cW)}$ it holds that
$$
\left( \tilde\chi^{(\widetilde\cW)}\circ\nu_{\widetilde\cW,\cW}\circ \left( \chi^{(\cW)} \right)^{-1} \right)_i (q)
=\left(  \tilde\chi^{(\cU)}\circ\left(\chi^{(\cU)}\right)^{-1} \circ\nu_{\chi_*\cU,\chi_*\cW} \right)_i(q).
$$

Furthermore, if $\cW'$ is another weighting such that
\begin{enumerate}
\item $\cW\subseteq\cW'\subseteq \cU$,
\item $\cW'\subseteq \cU$ and $\cW$ are uniformly aligned under the chart $\chi$, and
\item there is some $t\in\R, q_0\in\R^m$ such that $q=t\cdot q_0$ according the action of $\chi_*\cW'$,
\end{enumerate}
then the expressions above are equal to
$$
\left(t\cdot \tilde\chi^{(\cU)}\circ\left(\chi^{(\cU)}\right)^{-1} \circ\nu_{\chi_*\cU,\chi_*\cW} \right)_i (q_0),
$$
where $t$ acts according to $\tilde\chi_*\widetilde\cW$.
\end{lemma}

\begin{proof}
Since the weightings are nested and aligned with the charts, the following diagram is a diagram of weighted morphisms:
$$
\begin{tikzcd}
(\R^m,\tilde\chi_*\widetilde\cW)\arrow[d,"\operatorname{Id}"] & (M,\widetilde\cW)\arrow[l,"\tilde\chi"']\arrow[d,"\operatorname{Id}"] & (M,\cW)\arrow[r,"\chi"]\arrow[l,"\operatorname{Id}"']\arrow[d,"\operatorname{Id}"] & (\R^m,\chi_*\cW)\arrow[d,"\operatorname{Id}"] \\
(\R^m,\tilde\chi_*\cU) & (M,\cU)\arrow[l,"\tilde\chi"'] & (M,\cU)\arrow[r,"\chi"]\arrow[l,"\operatorname{Id}"'] & (\R^m,\chi_*\cU)
\end{tikzcd}
$$

By functoriality of the weighted normal bundle construction we get the following commuting diagram:
$$
\begin{tikzcd}
\nu_{\tilde\chi_*\widetilde\cW} \R^m \arrow[d,"\nu_{\tilde\chi_*\cU,\tilde\chi_*\widetilde\cW}"'] & \nu_{\widetilde\cW}M\arrow[d]\arrow[l,"\tilde\chi^{(\widetilde\cW)}"'] & \nu_\cW M\arrow[d]\arrow[r,"\chi^{(\cW)}"]\arrow[l,"\nu_{\widetilde\cW,\cW}"'] & \nu_{\chi_*\cW}\R^m\arrow[d,"\nu_{\chi_*\cU,\chi_*\cW}"] \\
\nu_{\tilde\chi_*\cU} \R^m & \nu_\cU M\arrow[l,"\tilde\chi^{(\cU)}"'] & \nu_\cU M\arrow[r,"\chi^{(\cU)}"]\arrow[l,"\operatorname{Id}"'] & \nu_{\chi_*\cU}\R^m
\end{tikzcd}
$$
The left hand side of the equation we want to show is the $i$-th component of the upper row applied to $q$. By the description of the induced map in aligned coordinates from Example~\ref{example:induced-map-weighted-normal-bundle} and the condition on the $i$-th weights, the arrow $\nu_{\tilde\chi_*\cU,\tilde\chi_*\widetilde\cW}$ is a projection that leaves this component unaffected. Thus following the opposite path in this diagram yields the right-hand side of the equation and we are done with the first part of the lemma.

The second part of the Lemma follows by noting that the description from Example~\ref{example:induced-map-weighted-normal-bundle} also applies to the map $\nu_{\chi_*\cU,\chi_*\cW}$. In particular, the action of $t$ according to $\chi_*\cW'$ is indistinguishable from the action according to $\chi_*\cW$ in the subset that we are projecting to. We can then commute the action to the front of the expression as the induced maps intertwine it.
\end{proof}

Furthermore, we will need an iterated version of Malgrange Preparation:

\begin{lemma}\label{lem:iterated-malgrange}
Let $f$ be a smooth real-valued function defined in a neighbourhood of the origin in $\R^k\times\R^m$ and $w\in\N_0^k$. Assume that $\partial^{|\alpha|} f/\partial t^\alpha(t,x)$ vanishes when the following holds: $\alpha$ is a multi-index with $\alpha_j\leq w_j$, $\sum_j\alpha_j<\sum_j w_j$ and $t_j=0$ for all $j$ with $\alpha_j>0$. Then there exists a smooth and real-valued function $g$ close to the origin of $\R^k\times\R^m$ such that
$$f(t,x)=\prod_{j=1}^k t_j^{w_j}\;\cdot g(t,x).$$
It follows immediately that
$$g(0,x)=\prod_{j=1}^k\left(\frac{1}{w_j!}\left.\frac{d^{w_j}}{dt_j^{w_j}}\right|_{t_j=0}\right)\; f(t,x).$$
\end{lemma}

\begin{proof}
Repeatedly applying Theorem~\ref{thm:malgrange} and taking $t$ to be $t_j$ at each step from $j=1$ to $k$ yields the statement.
\end{proof}

\startSubchaption{Proof of the main theorem}

To prove Theorem~\ref{thm:manifold-structure}, we need our coordinate charts to be bijective, induce smooth transition maps and a second countable Hausdorff topology. Bijectivity is clear from Corollary~\ref{cor:comp-coords} as it provides the inverse to our charts. Once again the atlas clearly has a countable subatlas and every pair of points either lie in a common chart domain or can be separated by chart domains, so we are left with smoothness of transition maps wherever the chart domains of two good perspectives overlap.

So let $(\chi,\cN,\mathbf{h},\mathbf{s})$ and $(\tilde\chi,\tilde\cN,\mathbf{\tilde h},\mathbf{\tilde s})$ be two overlapping good perspectives and consider the transition map
\begin{equation}\label{eq:transition-map-building}
(y_1, ..., y_m)\mapsto (\tilde y_1, ..., \tilde y_m) := \left(\tilde\chi_{\bf\tilde h \tilde s}\circ \chi^{-1}_{\bf hs}\right)(y_1, ..., y_m).
\end{equation}

Our goal is to show for an arbitrary point $p'=\chi^{-1}_{\bf hs}(y')$ in the overlap of our chart domains that the transition map is smooth close to $y'\in\R^m$. We thus fix such a point from here on out. For all components that are not a control parameter, this is straightforward:

\begin{lemma}
For all $i=1,...,m$ such that there are no $\tilde N\in\tilde\cN$ with $i=\mathbf{\tilde h}(\tilde N)$, the component $\tilde y_i$ is a smooth function of $(y_1,...,y_m)$.
\end{lemma}

\begin{proof}
These components are given by the composition $\tilde x_{i:\tilde\mu(i)}\circ \chi^{-1}_{\bf hs}$ for the appropriate selector map $\tilde\mu$ and thus factor through the $\tilde\mu(i)$-th component of the ambient product $\prod_{G\in\cG}\Bl_{\cW_G}(M)$. Restricted to this component, the map $\chi^{-1}_{\bf hs}$ on the right-hand side is smooth by the coordinate description in Corollary~\ref{cor:comp-coords}. The map on the left-hand side is by definition just a component of a chart on the $\tilde\mu(i)$-th component, making the composition smooth as well.
\end{proof}

So from now on we fix some $\tilde N\in\tilde\cN$ and consider the corresponding control parameter $\tilde y_i$ for $i:=\mathbf{\tilde h}(\tilde N)$. By definition it takes the form
$$
\tilde y_i(y_1, ..., y_m) :=
\left(\mathbf{\tilde s}(\tilde N)\,\tilde x_{i:\tilde\mu(i)}\circ\chi^{-1}_{\bf hs}(y_1, ..., y_m) \right)^{1/\tilde w_{\tilde N,i}}.
$$
The expression
\begin{equation}\label{eq:f}
f(y_1, ..., y_m):= \mathbf{\tilde s}(\tilde N)\,\tilde x_{i:\tilde\mu(i)}\circ\chi^{-1}_{\bf hs}(y_1, ..., y_m)
\end{equation}
in the parentheses is certainly smooth according to the explicit formulas in Corollary~\ref{cor:comp-coords}. However, due to the presence of the root, $\tilde y_i (y_1, ..., y_m) = (f(y_1, ..., y_m))^{1/\tilde w_{\tilde N,i}}$ may fail to be smooth at $y'$ if $f(y')$ vanishes. Our strategy from here on out is to write $f$ as the product of a $\tilde w_{\tilde N,i}$-th power of a smooth function and a smooth function that does not vanish at $y'$. The following lemmas first establish smoothness under additional simplifying assumptions:

\begin{lemma}
If $\tilde\mu(i)=\emptyset$, then $\tilde y_i(y_1,...,y_m)$ is smooth.
\end{lemma}

\begin{proof}
By inspecting Corollary~\ref{cor:comp-coords}, we know that the $\tilde{N}$-th component $\pi_{\tilde N}(\chi^{-1}_{\bf hs}(y_1, ..., y_m))\in\Bl_{\cW_{\tilde N}}(M)$ depends smoothly on $(y_1,...,y_m).$ Thus the same must hold when using different smooth coordinates $\tilde x_{i:\tilde{N}}\circ\chi^{-1}_{\bf hs}(y_1, ..., y_m)$ on that component. But when evaluating this expression using Equation~\eqref{eq:coord-lemma-three}, we see that it is exactly equal to $\tilde y_i(y_1,...,y_m)$.
\end{proof}

\begin{lemma}
If the point $p'=\chi^{-1}_{\bf hs}(y')$ does not lie on the exceptional divisor of the blow-up of $\tilde\mu(i)$, i.e. $b_\cW(p)\not\in\tilde\mu(i)$, then $\tilde y_i(y_1,...,y_m)$ is smooth at $y'$.
\end{lemma}

\begin{proof}
By definition of the coordinates on the blow-up of $\tilde\mu(i)$, we can write
$$\tilde x_{i:\tilde\mu(i)}(p)=\frac{\tilde x_i(p_{\tilde\mu(i)})}{
\left( \mathbf{\tilde s}(\tilde\mu(i))\,\tilde x_{\mathbf{\tilde h}(\tilde\mu(i))}(p_{\tilde\mu(i)})
\right)^{\tilde w_{\tilde\mu(i),i}/\tilde w_{\tilde\mu(i),\mathbf{\tilde h}(\tilde\mu(i))}}
}$$
for any $p=\chi_{\bf hs}^{-1}(y_1,...,y_m)$ close to $p'$,
where the denominator does not vanish precisely because we are not on the exceptional divisor. Furthermore we must have $p_{\tilde\mu(i)}=b_{\cW_{\tilde N}}(p_{\tilde N})$. Let
$$\tilde u_l = \tilde x_{l:\tilde N}(p)=\tilde x_{l:\mathbf{\tilde h}(\tilde N)\mathbf{\tilde s}(\tilde N)}(p_{\tilde N})$$ be $p$ under the coordinates associated to $\tilde N$, which depend smoothly on $(y_1, ..., y_m)$ since we know the components of $\chi_{\bf hs}^{-1}$ to be smooth. Using the local coordinate expression for the blow-down map from Equation~\eqref{eq:local-blow-down}, we get
$$
\tilde x_i(p_{\tilde\mu(i)}) = \tilde x_i\left(b_\cW\left(\tilde\chi^{-1}_{\mathbf{\tilde h}(\tilde N)\mathbf{\tilde s}(\tilde N)}(\tilde y_1, ..., \tilde y_m)\right)\right) = \tilde u_i^{\tilde w_{\tilde N,i}}\cdot \mathbf{\tilde s}(\tilde N).
$$
where we used $i=\mathbf{\tilde h}(\tilde N)$ on the right-hand side. We can conclude that $f$ of Equation~\eqref{eq:f} can indeed be written as the correct power of a smooth expression times something non-vanishing close to $y'$ and are thus done.
\end{proof}

Thus from now on we may further \textbf{assume without loss of generality} that $b(p')\in\tilde\mu(i)\neq\emptyset$, and as a consequence $b(p')\in\tilde N$ also holds. We so far have established how $p'$ relates to the relevant elements $\tilde N$ and $\tilde\mu(i)$ of the nest $\tilde\cN$ in the target, but these elements are not necessarily contained in the nest $\cN$ in the domain. To write down an expression in terms of the coordinates $(y_1,...,y_m)$ in the domain, we will make use of the fact that both these elements are determined by some element of $\cN$:

\begin{lemma}
There is a $N\in\cN\cap\tilde\cN$ such that $N\subseteq \tilde N$ and $p'_N\in\Bl_{\tilde N, N}(M)$. Similarly, there is a $\mu_0\in\cN\cap\tilde\cN$ such that $\mu_0\subseteq \tilde\mu(i)$ and $p'_{\mu_0}\in\Bl_{\tilde \mu(i), \mu_0}(M)$.
\end{lemma}

\begin{proof}
If $\tilde N\in\cN,$ this is trivially satisfied for $N=\tilde N$. If on the other hand $\tilde N\not\in\cN$, note that $\cG_{p'}\subseteq \cN\cap\tilde\cN$ by Lemma~\ref{lem:control-set-in-nest} and the assumption that our good perspectives cover $p'$, so $\tilde N\not\in\cG_{p'}$. By definition of the control set and since $b(p')\in\tilde N$, there is at least one $N\in\cG$ with $N\subsetneq \tilde N$ and $p'_N\in\Bl_{\tilde N, N}(M)$. A minimal such $N$ must itself lie in $\cG_{p'}\subseteq\cN\cap\tilde\cN$ and thus satisfies our claim.

Since $b(p')\in\tilde\mu(i)$, the exact same argument also yields existence of $\mu_0$.
\end{proof}

If $N$ was strictly smaller than $\tilde N$, we would be done:

\begin{lemma}
If the $N\subseteq\tilde N$ from the previous lemma is strictly smaller than $\tilde N$, then $\tilde y_i(y_1,...,y_m)$ is smooth at $y'$.
\end{lemma}

\begin{proof}
We consider three cases based on whether $N$ is smaller, larger or incomparable to $\tilde\mu(i)$.

\textit{Case 1, $N$ and $\tilde\mu(i)$ incomparable:} Since $\cN$ is a nest, by Proposition~\ref{prop:char-nest}(1) the set $\cP=\{N,\tilde\mu(i)\}$ is a set of $\cG$-factors. These must intersect transversely since the building set is separated. As $\tilde N$ has positive codimension and contains both elements of $\cP$, this is impossible.

\textit{Case 2, $N\subseteq\tilde\mu(i)$:} As $p'_N\in\Bl_{\tilde N, N}(M)$, this implies $p'_{\tilde\mu(i)}\in\Bl_{\tilde N, \tilde\mu(i)}(M)$. Consider again the $f$ from Equation~\eqref{eq:f}. Smoothness can only fail if $f(y')=0$. Assuming this was true and thus $\tilde x_{i:\tilde\mu(i)}(p')=0$, the induced map in the coordinates $\tilde\chi$ would give that $\left(\tilde\chi^{(\cW_{\tilde N})}_{i\mathbf{\tilde s}(\tilde N)}\right)_i(n)=0$ for $p'_{\tilde N}=[n]$, in contradiction to $p'$ being in the domain of these charts.

\textit{Case 3, $N\supsetneq\tilde\mu(i)$:} Due to $N\subsetneq\tilde N$, this yields a contradiction to the definition of $\tilde\mu(i)$ for $i=\mathbf{\tilde h}(\tilde N)$.
\end{proof}

To recap, \textbf{we can now assume} that $N=\tilde N$  (in particular, $\tilde N$ lies in $\cN\cap\tilde\cN$) and that there is a $\mu_0\in\cN\cap\tilde\cN$ with $\mu_0\subseteq\tilde\mu(i)$ and $p'_{\mu_0}\in\Bl_{\tilde \mu(i), \mu_0}(M)$. When drawing the coordinate tableau of the nest $\tilde\cN$ under $\tilde\chi$ we are in the following situation:

\begin{center}
\begin{tikzpicture}
  \matrix (m) [matrix of nodes, nodes in empty cells, column sep=5pt, row sep=5pt] {
    $\vdots$ \\
    $\;$ & $\;$ & $\;$ & $\npboxed{\phantom{A}}$ & $\dots$\\
    $\;$ & $\;$ & $\;$ & $\;$ & $\;$ & $\;$ & $\npboxed{\phantom{A}}$ & $\dots$ \\
    $\vdots$ &&&&&& $\vdots$\\
    $\;$ & $\;$ & $\;$ & $\;$ & $\;$ & $\;$ & $\;$ & $\;$ & $\;$ & $\npboxed{\phantom{A}}$ & $\dots$ \\
    $\vdots$ & & & $\uparrow$ \\
  };

  \tableaubox(m)($\tilde{N}$)(2:1:4)
  \tableaubox(m)($\tilde{\mu}(i)$)(3:1:7)
  \tableaubox(m)($\mu_0$)(5:1:10)
\end{tikzpicture}
\end{center}

The $i=\mathbf{\tilde h}(\tilde N)$-th column is marked with an arrow and $\mu_0$ may equal $\tilde{\mu}(i)$ - namely, when $\tilde\mu(i)$ is also an element of $\cN$.

Our goal remains to show smoothness of $\tilde y_i$ by showing that whenever $f(y')=0$, $f$ vanishes at the correct order. By treating a number of simpler cases in the previous lemmas, we were able to make additional assumptions. But now we are finally ready to give an explicit formula for $f$ using Lemma~\ref{lem:local-rep-induced-map} and Corollary~\ref{cor:comp-coords}. To be able to do this, we first need to slightly rewrite $f$:
\begin{align*}
f(y_1, ..., y_m) &= \mathbf{\tilde s}(\tilde N)\,\left(\tilde\chi_{\mathbf{\tilde h}(\tilde\mu(i))\mathbf{\tilde s}(\tilde\mu(i))}\circ \pi_{\tilde\mu(i)}\circ\chi^{-1}_{\bf hs}\right)_i(y_1, ..., y_m) \\
&= \mathbf{\tilde s}(\tilde N)\,\left(
\tilde\chi_{\mathbf{\tilde h}(\tilde\mu(i))\mathbf{\tilde s}(\tilde\mu(i))}
\circ b_{\tilde\mu(i),\mu_0}
\circ \pi_{\mu_0}
\circ\chi^{-1}_{\bf hs}
\right)_i(y_1, ..., y_m) \\
&= \mathbf{\tilde s}(\tilde N)\,\left(
\tilde\chi_{\mathbf{\tilde h}(\tilde\mu(i))\mathbf{\tilde s}(\tilde\mu(i))}
\circ b_{\tilde\mu(i),\mu_0}
\circ \chi^{-1}_{\mathbf{h}(\mu_0)\mathbf{s}(\mu_0)}
\right)_i(z_1, ..., z_m).
\end{align*}
For the first equality, we simply rewrote the definition of $f$ in terms of the $i$-component of the chart. The second equality follows by $p_{\tilde\mu(i)}=b_{\tilde\mu(i),\mu_0}(p_{\mu_0})$ for $p:=\chi_{\bf hs}^{-1}(y_1,...,y_m)$ close to $p'$. Finally, we sneak in the identity $ \chi^{-1}_{\mathbf{h}(\mu_0)\mathbf{s}(\mu_0)}\circ \chi_{\mathbf{h}(\mu_0)\mathbf{s}(\mu_0)} $ after the projection and define the short hand
$$ z_l := x_{l:\mu_0}(p) = 
\left(
\chi_{\mathbf{h}(\mu_0)\mathbf{s}(\mu_0)}
\circ \pi_{\mu_0}
\circ\chi^{-1}_{\bf hs}
\right)_l (y_1, ..., y_m)
\quad \text{for }l\in 1,...,m.
$$
What we have achieved now is that the dependence of $f$ on $(z_1,...,z_m)$ is given by the local expression of the induced chart from Lemma~\ref{lem:local-rep-induced-map}, while the dependence of $z_l$ on $(y_1, ..., y_m)$ is the parametrization given in Corollary~\ref{cor:comp-coords}. Putting these together, we find for $z_{\mathbf{h}(\mu_0)}\in\R$ and $q,q_0\in\R^m$ defined in terms of $(y_1,...,y_m)$ by
\begin{align*}
z_{\mathbf{h}(\mu_0)} &= \prod_{\mu_0\supseteq N'\in\cN} y_{\mathbf{h}(N')},\\
q_l &= \prod_{N'\in\cN} (y_{\mathbf{h}(N')})^{w_{N',l}}\cdot
\begin{cases}
\mathbf{s}(N'') &\text{if } l=\mathbf{h}(N'') \text{ for some }N''\in\cN,\\
y_l &\text{else,}
\end{cases}\\
(q_0)_l &= \prod_{\mu_0\not\supseteq N'\in\cN} (y_{\mathbf{h}(N')})^{w_{N',l}}\cdot
\begin{cases}
\mathbf{s}(N'') &\text{if } l=\mathbf{h}(N'') \text{ for some }N''\in\cN,\\
y_l &\text{else,}
\end{cases}
\end{align*}
that
\begin{equation}\label{eq:oversized-f}\resizebox{.88\hsize}{!}{$
f(y_1, ..., y_m) = \begin{cases}

\frac{
\mathbf{\tilde s}(\tilde N)\, \left(\tilde\chi\circ\chi^{-1}\right)_i\left( q\right)
}{
\left( \mathbf{\tilde s}(\tilde\mu(i))\,\left(\tilde\chi\circ\chi^{-1}\right)_{\mathbf{\tilde h}(\tilde\mu(i))}( q)  \right)^{\tilde w_{\tilde\mu(i),i}/\tilde w_{\tilde\mu(i),\mathbf{\tilde h}(\tilde\mu(i))}}
}

&\text{if }z_{\mathbf{h}(\mu_0)}>0\\

\frac{
\mathbf{\tilde s}(\tilde N)\, \left(\tilde\chi^{(\cW_{\tilde\mu(i)})}\circ\nu_{\tilde\mu(i),\mu_0}\circ\left(\chi^{(\cW_{\mu_0})}\right)^{-1}\right)_i\left(q_0\right)
}{
\left( \mathbf{\tilde s}(\tilde\mu(i))\,\left(\tilde\chi^{(\cW_{\tilde\mu(i)})}\circ\nu_{\tilde\mu(i),\mu_0}\circ\left(\chi^{(\cW_{\mu_0})}\right)^{-1}\right)_{\mathbf{\tilde h}(\tilde\mu(i))}(q_0)  \right)^{\tilde w_{\tilde\mu(i),i}/\tilde w_{\tilde\mu(i),\mathbf{\tilde h}(\tilde\mu(i))}}
}

&\text{if }z_{\mathbf{h}(\mu_0)}=0
\end{cases}$}
\end{equation}
holds.

To make sense of these expressions, note that by Corollary~\ref{cor:comp-coords} applied to $N=\emptyset$, we have that $b_\cW(p)=\chi^{-1}(q)$, i.e. $q$ are the coordinates of the blown-down point. The point $b_\cW(p)$ is an element of $\mu_0$ (and, equivalently, $\tilde\mu(i)$) if and only if $z_{\mathbf{h}(\mu_0)}>0$. According to the above, this happens when any of the control parameters $y_{\mathbf{h}(N')}$ associated with $N'\subseteq\mu_0$ in $\cN$ vanishes. Thus it makes sense that in this case, our expression for $f$ depends on $q_0$, which is the same as $q$ without including these vanishing parameters in the sense that
\begin{equation}\label{eq:q0-to-q}
q_l = \prod_{\mu_0\supseteq N'\in\cN} (y_{\mathbf{h}(N')})^{w_{N',l}}\cdot (q_0)_l.\end{equation}
Note that we could have concluded that the expression for $z_{\mathbf{h}(\mu_0)}>0$ holds in the bulk without spending the amount of effort we have, but that would not have allowed us to understand for which $y$ the denominator vanishes and what $f$ looks like at these points:

\begin{lemma}
The function $f$ vanishes at a point $(y_1,...,y_m)$ if and only if there is an $N'\in\cN$ such that $\mu_0\subsetneq N'\subseteq \tilde N$ and $y_{\mathbf{h}(N')}=0$.
\end{lemma}
\begin{proof}
In the case that $z_{\mathbf{h}(\mu_0)}>0$, we can see this along the following chain of equivalences:
\begin{align*}
\left(\tilde\chi\circ\chi^{-1}\right)_i\left( q\right)=0 &\iff 
\chi^{-1}(q)\in\tilde N \\
&\iff q_{\mathbf{h}(\tilde N)} = \mathbf{s}(\tilde N) 
\cdot\prod_{N'\in\cN} (y_{\mathbf{h}(N')})^{w_{N',\mathbf{h}(\tilde N)}}  = 0 \\
&\iff \exists N'\in\cN:\, \mu_0\subsetneq N'\subseteq \tilde N \text{ and } y_{\mathbf{h}(N')}=0.
\end{align*}
For the first equivalence we use that $i=\mathbf{\tilde h}(\tilde N)$, thus the left-hand side measures exactly the parameter controlling collapse onto $\tilde N$ in the coordinates $\tilde\chi_{\mathbf{\tilde h}(\tilde N)\mathbf{\tilde s}(\tilde N)}$. The same argument for the coordinates $\chi_{\mathbf{h}(\tilde N)\mathbf{s}(\tilde N)}$ yields the following equivalence, where we have also spelled out the definition of $q_{\mathbf{h}(\tilde N)}$. For the last equivalence, we use two facts: Firstly, for control parameters associated with an $N'$ such that either $N'\not\subseteq\tilde N$ or such that $N'$ is incomparable to $\mu_0$, the associated weight is zero anyway and thus they cannot contribute to $q_{\mathbf{h}(\tilde N)}$ vanishing.
Secondly, that the control parameters for $N'\subseteq\mu_0$ are non-zero by the assumption $z_{\mathbf{h}(\mu_0)}>0$.

For $z_{\mathbf{h}(\mu_0)}=0$ we can make an analogous chain of equivalences, but now involving the induced charts of the weighted normal bundle and $q_0$:
\begin{align*}
& \left(\tilde\chi^{(\cW_{\tilde\mu(i)})}\circ\nu_{\tilde\mu(i),\mu_0}\circ\left(\chi^{(\cW_{\mu_0})}\right)^{-1}\right)_i\left(q_0\right) =0 \\
\iff \; & \left(\tilde\chi^{(\cW_{\tilde N})}\circ\left(\chi^{(\cW_{\tilde N})}\right)^{-1}\right)_i\left(\nu_{\chi_*\cW_{\tilde N}, \chi_*\cW_{\mu_0}} (q_0)\right) =0\\
\iff \; &  \left(\nu_{\chi_*\cW_{\tilde N}, \chi_*\cW_{\mu_0}} (q_0)\right)_{\mathbf{h}(\tilde N)} =0\\
\iff\; & (q_0)_{\mathbf{h}(\tilde N)} = \mathbf{s}(\tilde N) 
\cdot\prod_{\mu_0\not\supseteq N'\in\cN} (y_{\mathbf{h}(N')})^{w_{N',\mathbf{h}(\tilde N)}}  = 0 \\
\iff\;& \exists N'\in\cN:\, \mu_0\subsetneq N'\subseteq \tilde N \text{ and } y_{\mathbf{h}(N')}=0.
\end{align*}

The first equivalence follows by application of Lemma~\ref{lem:projection-in-charts} for $\cU=\cW_{\tilde N}$, $\widetilde{\cW}=\cW_{\tilde\mu(i)}$ and $\cW=\cW_{\mu_0}$. For the second equivalence, we just check whether the control parameter vanishes in the original charts $\chi^{(\cW_{\tilde N})}$ instead of in $\tilde\chi^{(\cW_{\tilde N})}$. This means we must look at the $\mathbf{h}(\tilde N)$-th component instead of the $i=\mathbf{\tilde h}(\tilde N)$-th component. For the third equivalence, note that the induced map of the identity in local coordinates has no effect since we are considering a component in which the weights of $\chi_*\cW_{\tilde N}$ and $\chi_*\cW_{\mu_0}$ match (compare Example~\ref{example:induced-map-weighted-normal-bundle}). We write out the definition of $(q_0)_{\mathbf{h}(\tilde N)}$ for convenience. Finally, note that due to separation of the building set, $w_{N',\mathbf{h}(\tilde N)}$ vanishes whenever $N'$ and $\mu_0$ are incomparable, and also when $N'$ is not a subset of $\tilde N$, thus the last equivalence follows.
\end{proof}

Let $M_1\supsetneq ...\supsetneq M_k$ be the elements of $\cN$ with $\tilde N\supseteq M_j\supsetneq\mu_0$ and $y'_{\mathbf{h}(M_j)}=0$. Here we used that $\cN$ is a nest and thus subsets of $\tilde N$ are comparable. We claim that close to $y'$ we can write$$
f(y_1,...,y_m) = g(y_1,...,y_m) \cdot \prod\limits_{j=1}^k y_{\mathbf{h}(M_j)}^{\tilde w_{\tilde N,i}}
$$
where $g$ is a smooth, non-vanishing function, thus proving smoothness of the remaining components of the transition functions. To prove this, we luckily only need to consider the numerators in our oversized Equation~\eqref{eq:oversized-f} for $f$, as the denominators do not vanish. For this, we use the iterated version of Malgrange Preparation, Lemma~\ref{lem:iterated-malgrange}. We know $f$ vanishes exactly if any $y_{\mathbf{h}(M_j)}=0$, so to apply Malgrange it suffices to check that the derivatives
\begin{equation}\label{eq:remainder}\resizebox{.88\hsize}{!}{$
0\neq  \prod\limits_{j=1}^k\left( \frac{1}{\tilde w_{\tilde N,i}!}
\left.\frac{d^{\tilde w_{\tilde N,i}}}{dt_j^{\tilde w_{\tilde N,i}}}\right|_{t_j=0}\right) 
\begin{cases}
\left(\tilde\chi\circ\chi^{-1}\right)_i\left( \mathbf{t}\cdot u\right)
&\text{if }z_{\mathbf{h}(\mu_0)}>0,\\
\left(\tilde\chi^{(\cW_{\tilde\mu(i)})}\circ\nu_{\tilde\mu(i),\mu_0}\circ\left(\chi^{(\cW_{\mu_0})}\right)^{-1}\right)_i\left(\mathbf{t}\cdot u_0\right)
&\text{if }z_{\mathbf{h}(\mu_0)}=0,
\end{cases}$}
\end{equation}
of the numerators (up to signs and decorative combinatorial factors) don't vanish when $y=y'$, and that all lower order derivatives do vanish. This will constitute the rest of our proof, so we can assume $y=y'$ from now. Here the notation $\mathbf{t}\cdot u$ is shorthand for letting each $t_j$ act on the vector $u$ according to the action associated with $\cW_{M_j}$ under $\chi$, and the vectors $u$ and $u_0$ are the result of extracting the action of the $y_{\mathbf{h}(M_j)}$ from $q$ and $q_0$, i.e. 
$$
u_l= \prod_{N'\in\cN\setminus\{M_1, ..., M_k\}} (y_{\mathbf{h}(N')})^{w_{N',l}}\cdot
\begin{cases}
\mathbf{s}(N'') &\text{if } l=\mathbf{h}(N'') \text{ for some }N''\in\cN,\\
y_l &\text{else,}
\end{cases}
$$
and
$$
(u_0)_l= \prod_{\substack{\mu_0\not\supseteq N'\in\cN \\ N'\not\in\{M_1, ..., M_k\}}} (y_{\mathbf{h}(N')})^{w_{N',l}}\cdot
\begin{cases}
\mathbf{s}(N'') &\text{if } l=\mathbf{h}(N'') \text{ for some }N''\in\cN,\\
y_l &\text{else}
\end{cases}
$$
such that $$(y_{\mathbf{h}(M_1)}, ..., y_{\mathbf{h}(M_k)})\cdot u = q \qquad\text{and}\qquad (y_{\mathbf{h}(M_1)}, ..., y_{\mathbf{h}(M_k)})\cdot u_0 = q_0.$$
This is entirely analogous to how $q_0$ arises by extracting the action of the $N'\subseteq\mu_0$ from the expression for $q$ in Equation~\ref{eq:q0-to-q}.

To prepare for proving the inequality in both cases, note that since we are on the exceptional divisor of $\tilde N$ there is a representative $n$ of $p'_{\tilde N}$ such that for $l\neq \mathbf{h}(\tilde N)$
$$
\left( \chi^{(\cW_{\tilde N})}(n) \right)_l = \prod_{\substack{ N'\in\cN \\ w_{\tilde N,l=0}\lor \tilde N\subsetneq N' }} (y_{\mathbf{h}(N')})^{w_{N',l}}\cdot
\begin{cases}
\mathbf{s}(N'') &\text{if } l=\mathbf{h}(N'') \text{ for some }N''\in\cN,\\
y_l &\text{else}
\end{cases}
$$
as well as $\left( \chi^{(\cW_{\tilde N})}(n) \right)_{\mathbf{h}(\tilde N)}=\mathbf{s}(\tilde N)$. We can see this by inspecting the expression from Corollary~\ref{cor:comp-coords} and using the definition of the blow-up coordinates in terms of the coordinates on the exceptional divisor.

\begin{lemma}
The Inequality~\eqref{eq:remainder} holds and any lower order derivatives vanish.
\end{lemma}

\begin{proof}[Proof for $z_{\mathbf{h}(\mu_0)}=0$]
We can use Lemma~\ref{lem:projection-in-charts} with $\cW_{\mu_0}\subseteq\cW_{\tilde\mu(i)}\subseteq\cW_{\tilde N}$ for $\cW\subseteq\widetilde\cW\subseteq\cU$ and repeatedly apply the second part of the Lemma to pull the action of $\mathbf{t}$ to the front. Evaluating the derivatives then yields the equivalent inequality
$$
0\neq  
\left(\tilde\chi^{(\cW_{\tilde N})}\circ\left(\chi^{(\cW_{\tilde N})}\right)^{-1}\right)_i\left(\nu_{\chi_*\cW_{\tilde N},\chi_*\cW_{\mu_0}}u_0\right).
$$
Note that if we had computed any $t_j$-derivative at a lower order than $\tilde w_{\tilde N, i}$, the evaluation at $t_j=0$ would yield zero in line with what we set out to prove.

We claim that this inequality is in turn equivalent to
$$
0\neq  
\left(\tilde\chi^{(\cW_{\tilde N})}\circ\left(\chi^{(\cW_{\tilde N})}\right)^{-1}\right)_i\left( 
\chi^{(\cW_{\tilde N})}(n)  \right)=
\left( \tilde\chi^{(\cW_{\tilde N})}(n)\right)_{i}.
$$
In this case the inequality must be satisfied since $i=\mathbf{\tilde h}(\tilde N)$. To see the equivalence, note first that $u_0$ differs from $\chi^{(\cW_{\tilde N})}(n)$ solely by the actions of $y_{\mathbf{h}(N')}$ for all $\mu_0\subsetneq N'\subseteq \tilde N$ such that $y_{\mathbf{h}(N')}\neq 0$. These actions can similarly be pulled to the front, where they must act by multiplying with a non-zero number on the $i$-th component, which does not affect whether the inequality holds. Furthermore, applying $\nu_{\chi_*\cW_{\tilde N},\chi_*\cW_{\mu_0}}$ has no effect on $\chi^{(\cW_{\tilde N})}(n)$: The components that get projected to zero must vanish anyway because $n$ represents a weighted normal vector that is based at a point in $\mu_0.$
\end{proof}

\begin{proof}[Proof for $z_{\mathbf{h}(\mu_0)}>0$]
We can rewrite the differentiation by $t_k$ by using the local form of the map induced for $\phi=\operatorname{Id}$ from~\ref{props-weighted-normal}(5) to obtain
$$
0\neq  \prod\limits_{j=1}^{k-1}\left( \frac{1}{\tilde w_{\tilde N,i}!}
\left.\frac{d^{\tilde w_{\tilde N,i}}}{dt_j^{\tilde w_{\tilde N,i}}}\right|_{t_j=0}\right) 
\left(\tilde\chi^{(\cW_{\tilde N})}\circ
\nu_{\cW_{\tilde N},\cW_{M_k}}\circ
\left(\chi^{(\cW_{M_k})}\right)^{-1}\right)_i\left( \mathbf{t'}\cdot u\right)
$$
where $\mathbf{t'}\cdot u$ is the result of letting $t_1,...,t_{k-1}$ act on $u$ according to the actions of the $\chi_*M_j$. We can argue analogously to the case $z_{\mathbf{h}(\mu_0)}=0$ by applying Lemma~\ref{lem:projection-in-charts} where we set $\cW\subseteq\widetilde\cW\subseteq\cU$ to $\cW_{M_k}\subseteq\cW_{\tilde N}\subseteq\cW_{\tilde N}$. This allows us to pull out all remaining $t_j$ and compute the derivatives to get
$$
0\neq  
\left(\tilde\chi^{(\cW_{\tilde N})}\circ\left(\chi^{(\cW_{\tilde N})}\right)^{-1}\right)_i\left(\nu_{\chi_*\cW_{\tilde N},\chi_*\cW_{M_k}}u_0\right).
$$
By an analogous argument to the previous case, this is equivalent to the true statement
\begin{equation*}
0\neq  
\left( \tilde\chi^{(\cW_{\tilde N})}(n)\right)_{i}.
\qedhere
\end{equation*}
\end{proof}

This finally concludes our proof of smoothness of the transition maps.

Lest we forget, our main Theorem makes the additional claim that the blow-down map $$b_\cW:\Bl_\cW(M)\to M$$
is smooth and proper. We can make short work of this: Smoothness follows by Corollary~\ref{cor:comp-coords}, which gives the local coordinate expression for the blow-down map when choosing $N=\emptyset$. For properness, consider a compact $K\subseteq M$. By properness of the component-wise blow-down maps $$b_G:\Bl_{\cW_G}(M)\to M,$$ we know that $b_G^{-1}(K)$ is compact. By construction, $\Bl_\cW(M)$ is a closed subset of $\prod_{G\in\cG}\Bl_{\cW_G}(M)$. We now see that $b_\cW^{-1}(K)$ is a closed subset of the compact $\prod_{G\in\cG} b_G^{-1}(K)$, and thus itself compact.